\documentclass[a4paper,11pt]{amsart}
\usepackage{amssymb,amsmath,latexsym, upgreek,comment}
\usepackage{epsfig,caption}
\usepackage{mathrsfs}
\usepackage{hyperref}
\usepackage{color,graphics,subfigure}

\usepackage{epic,eepic,wrapfig,ifthen}
\usepackage{url}
\newtheorem{thm}{Theorem}[section]
\newtheorem{corollary}[thm]{Corollary}
\newtheorem{lemma}[thm]{Lemma}
\newtheorem{prop}[thm]{Proposition}
\newtheorem{defn}[thm]{Definition}

\newtheorem{remark}[thm]{Remark}
\newtheorem{example}[thm]{Example}
\numberwithin{equation}{section}

\newcommand{\formula}[2][nolabel]
{\ifthenelse{\equal{#1}{nolabel}}
 {\begin{align*} #2 \end{align*}}
 {\ifthenelse{\equal{#1}{}}
  {\begin{align} #2 \end{align}}
  {\begin{align} \label{#1} #2 \end{align}}
 }
}

\def\qed{{\hfill $\Box$ \bigskip}}

\DeclareMathOperator*{\esssup}{ess\,sup}

\renewcommand{\bar}{\overline}
\newcommand{\cal}[1]{\mathcal{#1}}
\def\sA {{\cal A}} \def\sB {{\cal B}} \def\sC {{\cal C}}
\def\sD {{\cal D}} \def\sE {{\cal E}} \def\sF {{\cal F}}
  \def\sI {{\cal I}}
  
 \def\sN {{\cal N}} \def\sO {{\cal O}}

 \def\bH {{\mathbb H}} 
  
 \def\bN {{\mathbb N}}

\def\R {{\mathbb R}}
\def\N {{\mathbb N}}

\def\P{{\mathbb P}}
\def\E{{\mathbb E}}

\def\e{{\mathbf e}}
\def\EE{{\mathcal E}}
\def\eps{\varepsilon}
\def\e{\mathbf e}

\def\FF{{\mathcal F}}
\def\EE{{\mathcal E}}

\def\R{{\mathbb R}}

\def\E{{\mathbb E}}
\def\F{{\mathbf F}}
\def\P{{\mathbb P}}
\def\N{{\mathbb N}}
\def\eps{\varepsilon}
\def\wh{\widehat}
\def\wt{\widetilde}

\def\ub{{\overline{\beta}}}
\def\lb{{\underline{\beta}}}

\def\1{{\bf 1}}

\def\nn{\nonumber}

\def\wt{\widetilde}
\def\wh{\widehat}

\def\eps{{\varepsilon}}

\allowdisplaybreaks

\makeatletter
\@addtoreset{equation}{section}

\makeatother

\begin{document}
\bibliographystyle{plain}

\title[Markov processes with jump kernels decaying at the boundary]
{ \bf    Markov processes with jump kernels decaying at the boundary}

\author{Soobin Cho, \quad Panki Kim, \quad Renming Song \quad and \quad Zoran Vondra\v{c}ek}

\address[Cho]{Department of Mathematics, University of Illinois Urbana-Champaign, Urbana, IL 61801, USA}
\curraddr{}
\email{soobinc@illinois.edu}

\address[Kim]{Department of Mathematical Sciences and Research Institute of Mathematics,
	Seoul National University,	Seoul 08826, Republic of Korea}
\thanks{Panki Kim is supported by the National Research Foundation of Korea (NRF) grant funded by the Korea government (MSIP) (No. 2022H1D3A2A01080536).}
\curraddr{}
\email{pkim@snu.ac.kr}

\address[Song]{
	Department of Mathematics, University of Illinois Urbana-Champaign, Urbana, IL 61801,
	USA}
\curraddr{}
\thanks{Research of Renming Song is supported in part by a grant from
	the Simons Foundation \#960480.}
\email{rsong@illinois.edu}

\address[Vondra\v{c}ek]
{
	Department of Mathematics, Faculty of Science, University of Zagreb, Zagreb, Croatia
}
\curraddr{}
\thanks{Zoran Vondra\v{c}ek is supported by the National Research Foundation of Korea (NRF) grant funded by the Korea government (MSIP) (No. 2022H1D3A2A01080536).  Supported in part by the Croatian Science Foundation under the project IP-2022-10-2277.} 
\email{vondra@math.hr}

\date{}

\begin{abstract}
The goal of this work is to develop a general theory for non-local singular operators of the type
$$
	L^{\mathcal{B}}_{\alpha}f(x)=\lim_{\epsilon\to 0} \int_{D,\, |y-x|>\epsilon}\big(f(y)-f(x)\big) \mathcal{B}(x,y)|x-y|^{-d-\alpha}\,dy,
$$
and
$$
L f(x)=L^{\mathcal{B}}_{\alpha}f(x) - \kappa(x) f(x),	
$$ 
in case $D$ is  a $C^{1,1}$ open set in $\mathbb{R}^d$, $d\ge 2$.
The function $\mathcal{B}(x,y)$ above may 
vanish at the boundary of $D$, and the killing potential $\kappa$ 
may be subcritical or critical. 

From a probabilistic point of view we study the reflected process on the closure $\overline{D}$ with 
infinitesimal generator $L^{\mathcal{B}}_{\alpha}$, 
and its part process on $D$ obtained by either killing at the boundary $\partial D$, or by killing via the killing potential 
$\kappa(x)$. 
The general theory developed in this work (i)  contains  subordinate killed stable processes in $C^{1,1}$ open sets as a special case, 
(ii) covers the case when $\mathcal{B}(x,y)$ is bounded between two positive constants and is well approximated by certain H\"older continuous functions, and (iii) extends the main results known for the half-space in $\mathbb{R}^d$.
The main results of the work are the boundary Harnack principle and its possible failure, and sharp two-sided Green function estimates. Our results on the boundary Harnack principle completely cover the corresponding earlier results in the case of half-space. Our
Green function estimates extend the corresponding earlier estimates in the case of half-space to  bounded $C^{1, 1}$ open sets.
\end{abstract}

\maketitle

\bigskip

\maketitle



\bigskip
\tableofcontents


\section{Introduction}\label{ch:intro}

The fractional Laplacian $\Delta^{\alpha/2}:=-(-\Delta)^{\alpha/2}$, $\alpha\in (0,2)$,
is one of the most important and most studied non-local operators. 
It appears in various branches of mathematics -- partial differential equations (see \cite{BV16, Ros16} for extensive surveys), probability theory (\cite{BGR62, Sat99}), potential theory (\cite{BBKRSV, Lan72}), 
harmonic analysis (\cite{Ste71}), semigroup theory 
(\cite{St}), 
numerical analysis (\cite{Lis-et-al}), as well as in applications involving long range dependence. One of its several equivalent definitions, see \cite{Kwa17}, is the singular integral definition: The fractional Laplacian in $\R^d$, $d\ge 1$, is the principal value integral
\begin{align}\label{e:intro-frac-lap}
\Delta^{\alpha/2}f(x)&=\text{p.v.}\int_{\R^d}  c_{d,-\alpha}
\big(f(y)-f(x)\big)|x-y|^{-d-\alpha}dy \\
&=\lim_{\eps\to 0}  \int_{\R^d, \, |y-x|>\eps} c_{d,-\alpha}
\big(f(y)-f(x)\big)|x-y|^{-d-\alpha}dy,\nonumber
\end{align}
where $c_{d,\beta}=2^{-\beta}\pi^{-d/2}\Gamma((d-\beta)/2)/|\Gamma(\beta/2)|$. 
The fractional Laplacian is the infinitesimal generator of the isotropic $\alpha$-stable L\'evy process in 
$\R^d$, which is a prototype of a purely discontinuous Markov process. For probabilists, the singular kernel 
$j(x,y)= c_{d,-\alpha} |x-y|^{-d-\alpha}$, 
$x,y\in \R^d$, serves as the jump kernel of the $\alpha$-stable process. Both the fractional Laplacian and the isotropic stable process have been studied for a long time. 

Of more recent interest is the investigation of fractional Laplacians
in a (proper) open subset $D$ of $\R^d$. One possible definition is  obtained from \eqref{e:intro-frac-lap} by taking 
$f_{ |  \R^d\setminus D}=0$,
leading to the operator
\begin{equation}\label{e:intro-generator}
L f(x)=\text{p.v.} \int_D c_{d,-\alpha}
(f(y)-f(x))|x-y|^{-d-\alpha}\,dy - \kappa(x)f(x), \quad x\in D,
\end{equation}
where $\kappa(x)=c_{d,-\alpha}\int_{D^c} |x-y|^{-d-\alpha}dy $  is the (critical) killing potential.
In the PDE literature the operator $L$ is usually called the \emph{restricted fractional Laplacian}. For probabilists, it is the infinitesimal generator of the part of the $\alpha$-stable process in $D$ (that is, of the $\alpha$-stable process killed upon first exit from $D$). By removing the killing part $\kappa$ from the operator $L$, one obtains the so-called \emph{censored} (or \emph{regional})  \emph{fractional Laplacian}. The corresponding Markov process -- the \emph{censored stable process} -- was introduced and thoroughly studied in \cite{BBC}. 
Note that on account of \eqref{e:intro-generator}, the restricted fractional Laplacian can be viewed as a (critical) Schr\"odinger perturbation of the censored fractional Laplacian. By changing 
this perturbation, one gets a different operator, and may hope to see different potential-theoretic behaviors. 
Such line of reasoning was employed in \cite{CKSV20}, and will be important in this work as well.

Two of the most important potential-theoretic results related to the fractional Laplacian and its variants in proper open sets are the boundary Harnack principle and the Green function estimates. 

The boundary Harnack principle (BHP) is the result roughly stating that 
all non-negative harmonic functions vanishing at a common part of the boundary of an open subset in $\R^d$ decay at the 
same rate. The first such result for 
$\alpha$-harmonic functions (functions harmonic with respect to the isotropic $\alpha$-stable process)
 in Lipschitz domains was proved in \cite{Bog97} in 1997. The extension to the so-called $\kappa$-open sets
was given two years later in \cite{SW99}, and all restrictions on the boundary were removed in \cite{BKK08}. By use of an extension method, another proof in case of Lipschitz domains was given in \cite{CS07}.
A stronger form of BHP is the BHP with exact decay rate, which requires a certain smoothness of the boundary -- typically $C^{1,1}$ smoothness. For  the fractional Laplacian, the  exact decay rat is $\delta_D(x)^{\alpha/2}$, which means that all non-negative $\alpha$-harmonic functions
 vanishing at a part of the boundary of a $C^{1,1}$ open set $D$ decay at the rate of $\delta_D(x)^{\alpha/2}$.
 Here $\delta_D(x)$ denotes the distance of the point $x$ to the boundary $\partial D$. 
For the censored $\alpha$-stable process in $C^{1,1}$ open set and $\alpha\in (1,2)$,  it was proved in \cite{BBC} that the BHP with exact decay rate $\delta_D(x)^{\alpha-1}$ holds. 
The paper \cite{CKSV20} studied how perturbations of the censored fractional Laplacian  by critical killings
affect the exact decay rates of the corresponding  harmonic functions.

For $d>\alpha$, the 
 potential 
of the fractional Laplacian in $\R^d$ is the Riesz potential:
$$
Gf(x)=\int_{\R^d} c_{d,\alpha} f(y)|x-y|^{-d+\alpha} dy.
$$
It is (at least formally) 
the inverse operator of the fractional Laplacian $\Delta^{\alpha/2}$. The Riesz kernel 
$G(x,y)=c_{d,\alpha}|x-y|^{-d+\alpha}$
is for probabilists the density of the occupation time measure of the $\alpha$-stable process. For the restricted, respectively censored, fractional  Laplacian, there is no explicit formula for the density of the occupation time measure of the killed, respectively censored, $\alpha$-stable process in an open set $D$. The best one can hope for 
is  sharp two-sided estimates. Investigation of the Green function $G^D(x,y)$ of the part of the (isotropic) $\alpha$-stable process in a $C^{1,1}$ open set $D$ also started in the late 1990's. The sharp two-sided estimates of the Green function $G^D(x,y)$, independently obtained in \cite{CS98} and \cite{Kul97}, state that when $d>\alpha$, 
\begin{equation}\label{e:intro-gfe}
	G^D(x,y)\asymp \left(\frac{\delta_D(x)}{|x-y|}\wedge 1\right)^{\alpha/2}
	\left(\frac{\delta_D(y)}{|x-y|}\wedge 1\right)^{\alpha/2}|x-y|^{\alpha-d}, \quad x,y\in D.
\end{equation}
Here $a\asymp b$ means that the ratio $a/b$ is bounded between two positive constants. 
For the censored process and $\alpha\in (1,2)$, \cite{CK02} established the sharp two-sided Green function estimates of the form \eqref{e:intro-gfe} with the power $\alpha/2$ replaced by $\alpha-1$. 

The jump kernel of the part of the process is inherited from its parent process in $\R^d$, and is for $x,y\in D$ 
still equal to $c_{d,-\alpha}  |x-y|^{-d-\alpha}$. 
The same is also true for the censored stable process. This obvious fact  highly facilitates the analysis of both the part process and the censored process. The jump kernel of the stable process in $\R^d$ is spatially homogeneous, hence the fractional Laplacian can be viewed as an operator with constant coefficients. This property is inherited by the censored fractional Laplacian and is also true for the integral part of the restricted Laplacian.
One possibility to introduce non-constant coefficients versions of the fractional Laplacian is to define the kernel 
$J(x,y):=c(x,y)|x-y|^{-d-\alpha}$ (with $x,y$ in the appropriate state space), with the function $c(x,y)$ bounded between two  positive constants. Such operators can be thought 
 as  non-local counterparts of uniformly elliptic differential operators. The pioneering work in this direction is \cite{CK03} (on metric measure spaces), 
which has led to many subsequent developments (see, for instance, \cite{ Acu16, CKK09, CKKW, CKS10-2, CK08, CKW20, CKW21, KS12}).
In the case of Euclidean space, a non-constant coefficients version of the regional fractional Laplacian (and the related reflected process) were studied in \cite{CKS10-2, Gu06, Gu07, GM06}.
In particular, under certain regularity conditions on the function $c(x,y)$, 
\cite{Gu07} proved a boundary Harnack principle and
 and \cite{CKS10-2} established its Green and heat kernel estimates.

We  now describe  another, quite natural, way of  introducing  
 non-constant coefficients into the fractional Laplacian. Let 
$D\subset \R^d$ be a $C^{1,1}$ open set, and let $X^D$ denote the part of an  isotropic $\gamma$-stable process in $D$, 
$\gamma\in(0,2]$ (for $\gamma=2$, $X^D$ is a Brownian motion killed upon exiting $D$). Let $S=(S_t)_{t\ge 0}$ be an independent (of $X^D$) $\beta$-stable subordinator. The subordinate process $Y^D_t:=X^D_{S_t}$ is called a subordinate killed stable process.  
It is worth mentioning that, unlike the part of a stable process in $D$, the process $Y^D$ is \emph{not} part of a larger process in $\R^d$, and is intrinsically connected with its state space $D$.
In case $\gamma=2$ (subordinate killed Brownian motion), 
its infinitesimal generator is the \emph{spectral fractional Laplacian} 
$-(-\Delta_{|D})^{\beta}$ -- the $\beta$-power of the Dirichlet Laplacian. 
This operator has been intensively studied in the PDE literature 
(\cite{AD17, AGCV22, BSV,Gr,SV}).
Similarly, in case $\gamma\in (0,2)$, the infinitesimal generator is the $\beta$-power of the restricted $\gamma$-Laplacian. By setting $\alpha:=\gamma \beta\in (0,2)$, we can regard these operators as versions of the $\alpha$-fractional Laplacian in the open set $D$. They are non-local integral operators of the form
	\begin{equation}\label{e:intro-spectral-gen}
		\text{p.v.}\int_D (f(y)-f(x))J^D(x,y)\, dy -\kappa_D(x)f(x), \quad x\in D,
	\end{equation}
where $\kappa_D(x)\asymp \delta_D(x)^{-\alpha}$, 
and the singular kernel $J^D(x,y)$ enjoys the following sharp two-sided estimates (see \cite{KSV18-a} and \cite{KSV18-b} for more general results, and \cite{KSV} for a version pertinent to this setting): For $\gamma=2$,
	\begin{equation}\label{e:intro-skbm}
		J^D(x,y)\asymp \left(\frac{\delta_D(x)}{|x-y|}\wedge 1\right)\left(\frac{\delta_D(y)}{|x-y|}\wedge 1\right)|x-y|^{-d-\alpha}\, 
	\end{equation}
and for $\gamma\in (0,2)$,
	\begin{equation}\label{e:intro-skst}
		\begin{split}
		& J^D(x,y)\asymp\\
		& \begin{cases} 
		\left(\frac{\delta_D(x)\wedge \delta_D(y)}{|x-y|}\wedge 1\right)^{\gamma(1-\beta)}|x-y|^{-d-\alpha} 
			&\mbox{if }\beta\in (1/2,1),\\[5pt]
		\left(\frac{\delta_D(x)\wedge \delta_D(y)}{|x-y|}\wedge 1\right)^{\gamma/2}
			\log\left(1+\frac{(\delta_D(x)\vee \delta_D(y))\wedge |x-y|}{\delta_D(x)\wedge \delta_D(y)\wedge |x-y|}\right)
			|x-y|^{-d-\alpha} &\mbox{if }\beta=1/2, \\[5pt]
			\left(\frac{\delta_D(x)\wedge \delta_D(y)}{|x-y|}\wedge 1\right)^{\gamma/2}\left(\frac{\delta_D(x)\vee \delta_D(y)}{|x-y|}
			\wedge 1\right)^{ (\gamma/2)(1-\beta/2) }|x-y|^{-d-\alpha}&\mbox{if } \beta\in (0,1/2).
		\end{cases}  
		\end{split}
	\end{equation} 
Thus we see that the kernel $J^D(x,y)$ depends not only on the distance between $x$ and $y$, but also on the distance of these points to the boundary. By defining $\sB(x,y)=J^D(x,y)|x-y|^{d+\alpha}$, we can write the jump kernel of $Y^D$ in the form 
$J^D(x,y)=\sB(x,y)|x-y|^{-d-\alpha}$. It is clear from \eqref{e:intro-skbm} and \eqref{e:intro-skst} that $\sB(x,y)$ decays to 0 as 
$\delta_D(x)\to 0$ or $\delta_D(y)\to 0$, hence it is \emph{not} bounded between two positive constants.
As a consequence, the infinitesimal generator of a subordinate killed L\'{e}vy process is degenerate near the boundary and is not ``uniformly elliptic". In case of local operators, partial differential
equations degenerate at the boundary have been studied intensively in the PDE literature;
see, for instance, \cite{DL03, FP14, Kim07},  and the references therein.

An interesting and important feature of the estimates \eqref{e:intro-skst} is that there is a phase transition at  $\beta=1/2$ which is responsible for qualitatively different, and quite unexpected, potential-theoretic properties. It turned out, cf.~\cite{KSV18-b}, that when $\gamma\in (0,2)$, the scale invariant BHP with exact decay rate $\delta_D(x)^{\gamma/2}$ holds when 
$\beta\in (1/2,1)$, while even the non-scale invariant BHP fails when $\beta\in (0,1/2]$ (the scale invariant BHP holds for $\gamma=2$ regardless of the value of $\beta$, see \cite{KSV18-a}).

As mentioned earlier, the subordinate killed Brownian motion 
is the  probabilistic counterpart of the spectral fractional Laplacian.
The subordinate killed Brownian motion 
and, more generally, subordinate killed L\'evy processes are natural 
and important, and 
there are many papers in the  literature on these. 
They
can be viewed as prototypes of singular non-local integral operators degenerate at the boundary. 
Thus, it is very important, both theoretically and from an application point of view,
to build a general framework for  singular operators degenerate at the boundary of the type \eqref{e:intro-spectral-gen} with or without killing potential.

The first step in this direction was taken in \cite{KSV, KSV20, KSV21} where such operators were studied for the open half-space $\bH=\{x=(\wt{x},x_d): \wt{x}\in \R^{d-1}, x_d>0\}\subset \R^d$ under the assumptions that the underlying singular operator (and consequently, the related process) is  
invariant under horizontal translations and appropriate scaling.

The goal of this work is to develop a general theory for  singular non-local  operators of the type
	\begin{equation}\label{e:intro-gen-oper}
		L^{\sB}_{\alpha}f(x)=\lim_{\eps
		\downarrow 0} \int_{D,\, |y-x|>\eps}\big(f(y)-f(x)\big) \sB(x,y)|x-y|^{-d-\alpha}\,dy,
	\end{equation}
and
	\begin{equation}\label{e:intro-gen-oper-kill}
		L f(x)=L^{\sB}_{\alpha}f(x) - \kappa(x) f(x),
	\end{equation}
in case $D$ is a $C^{1,1}$ open set, the function $\sB(x,y)$ may decay at the boundary of $D$, and the killing potential $\kappa$ is subcritical or critical. 
As  (very)  special cases, such type of singular operators contain
spectral, restricted and censored fractional Laplacian.
From a probabilistic point of view we will study the reflected process on the closure $\overline{D}$ with  
infinitesimal generator $L^{\sB}_{\alpha}$, 
and its part process on $D$ obtained by either killing at the boundary $\partial D$ (this happens only when $\alpha\in (1,2)$ and the obtained process is an analog of the censored process), or by killing via the killing potential $\kappa(x)$. This general theory should (i) include as a special case subordinate killed stable processes in
$C^{1,1}$ open sets; (ii) cover the case when $\sB(x,y)$ is bounded between two positive constants and is well approximated by certain H\"older continuous functions (thus extending main result in \cite{Gu07}), and (iii) contain as a special case the main results obtained in \cite{KSV, KSV20, KSV21}  for the half-space.
The key ingredient in developing such a general theory is to find good and reasonable assumptions on the functions $\sB$ and $\kappa$.

There are two major obstacles towards this goal. The first one is that the 
\emph{flattening the boundary} 
method does not work, hence one cannot 
 use  the half-space results to  get results for  regular smooth open sets.
Flattening the boundary of $D$ is a common way of proving certain results for non-local  operators (or \emph{part} processes) 
in $C^{1,1}$ open sets, and amounts to setting up an orthonormal coordinate system at a boundary point of the $C^{1,1}$ opens set, and 
ingeniously using the results known for the half-space in the local coordinate system, 
see e.g.~\cite{ROS14}. What makes this method work in the nondegenerate case is that the kernels for $D$ and for the half-space are the same – namely $c_{d,-\alpha}|x-y|^{-d-\alpha}$. 
In the axiomatic framework we intend to build, the kernel for $D$ is intrinsically connected to the set itself -- 
it is $\sB(x,y)|x-y|^{-\alpha-d}$, where the function $\sB(x,y)$ is defined only on $D\times D$ (and will usually decay at the boundary). The flattening of the boundary method does not work directly -- 
the function $\sB$ (and thus the jump kernel) is intrinsically connected with distances of the points to the boundary of $D$, while its counterpart in the case of the half-space $\bH$ should be defined in terms of the distances of the points to the boundary of $\bH$. When one flattens the boundary of $D$, 
 distance to the boundary changes. 
Thus, flattening destroys the structure of the function $\sB$ in terms of  distances to the boundary, and one cannot make connections with the half-space case directly.

The first and foremost challenge is to find an appropriate condition on 
$\sB$ that somehow circumvents and replaces the flattening of the boundary method.
We address this challenge by introducing the assumption \hyperlink{Q5}{{\bf (B5)}} and use the whole Section \ref{ch:examples} for its justification.

The second obstacle in developing the theory is the lack of scaling in general $C^{1,1}$ open sets. 
In the half-space case, the operator \eqref{e:intro-gen-oper} (with $D=\bH$), denoted by $L^{\sB_{\bH}}_{\alpha}$, is invariant under horizontal translations and scaling. By using scaling  and horizontal translation invariance in a fundamental way, one can calculate the action of the operator $L^{\sB_{\bH}}_{\alpha}$ on the powers of the distance function to the boundary.
More precisely, for a parameter $p$ in a certain range, one gets that
\begin{equation}\label{e:intro-L-on-dist}
	L^{\sB_{\bH}}_{\alpha} x_d^p =C(\alpha, p, \sB)x_d^{p-\alpha},
\end{equation}
with a semi-explicit constant 
$C(\alpha, p, \sB_{\bH})$.
For general $D$, there is no hope for such a formula.
A substitute for such a result is a good estimate of
the action of $L^{\sB}_{\alpha}$ on the so-called \emph{barrier functions}. The key Proposition \ref{p:barrier} contains such an estimate on the power of the cutoff distance function $\1_V(x)\delta_D(x)^q$ with $V$ a Borel subset of $D$, and relies on the assumption \hyperlink{Q5}{{\bf (B5)}} in a crucial way.

The form of the function $\sB(x,y)$ is motivated by 
the estimates \eqref{e:intro-skbm} and 
\eqref{e:intro-skst} -- we assume that it is comparable to the product 
\begin{equation}\label{e:intro-Phi}
\Phi_1\left(\frac{\delta_D(x)\wedge \delta_D(y)}{|x-y|}\right) \Phi_2\left(\frac{\delta_D(x)\vee \delta_D(y)}{|x-y|}\right)
	\ell\bigg(\frac{\delta_D(x)\wedge \delta_D(y)}{(\delta_D(x)\vee\delta_D(y))\wedge |x-y|}\bigg),
\end{equation}
where 
$\Phi_1$, $\Phi_2$ and  $\ell$ are functions satisfying certain weak scaling conditions (with some parameters)
 -- see the assumption \hyperlink{Q4-c}{{\bf (B4-c)}} for the precise definition. 

Our main results are the boundary Harnack principle with exact decay rate, and the sharp two-sided Green function estimates. We prove that the BHP holds for certain values of 
a parameter $p$ related to the killing potential $\kappa$ and 
the parameters entering the functions $\Phi_1$ 
and $\Phi_2$ in \eqref{e:intro-Phi} 
(and it may fail for the other values).  
In fact, when $\Phi_1$ and $\Phi_2$ are power functions,  and $\ell$ is a slowly varying function, 
we completely  determine the  region of the parameters where the boundary Harnack principle holds. 
Moreover, we also completely cover the boundary Harnack principle results of \cite{KSV20}.

We establish sharp two-sided estimates on the Green functions of these processes for all admissible values of 
the parameters involved.
The sharp two-sided Green function estimates are in terms of 
the quantity on the right-hand side of
\eqref{e:intro-gfe} (with the decay rate parameter $p$ replacing $\alpha/2$) multiplied by an integral involving functions $\Phi_1$ and $\Phi_2$. Depending on the parameters in these functions, 
these estimates may exhibit an anomalous behavior, 
see Corollary \ref{c:Green-2}. 
Recently in \cite{KSV20}, such anomalous behavior of Green function in the half-space has been proved 
under stronger assumptions on the function $\sB$.
Our work on Green function estimates extends the results in \cite{KSV20} to bounded $C^{1,1}$ open sets under weaker assumption 
on $\sB$. 

Examples are an integral part of this paper. They serve as a justification of our assumptions on $\sB$ and $\kappa$, and at the same time show the versatility of the theory. The last section is fully devoted to several types of examples. Besides covering 
subordinate killed stable processes and their variants, we provide an example extending the setting in \cite{Gu07}.

Organization of the work: In the next section we give a detailed overview of the work. We provide the set-up and gradually introduce the assumptions on the functions $\sB(x,y)$ and $\kappa(x)$. We explain and justify these assumptions, and show what type of results they imply. 
Sections \ref{ch:geometry}--\ref{ch:operator} employ only some of the assumptions, and partly use some known results from the literature.
For finer results we need stronger assumptions that supersede the ones already introduced.  
Starting from Section \ref{ch:key-estimates} the presentation is mostly self-contained and does not rely on the half-space results from \cite{KSV, KSV20, KSV21}. 

\medskip

We end this section with a few words on notation. Throughout this 
work, we use ``$:=$" to denote a definition, which is read as ``is defined to be''. We use the notation
$a\wedge b:=\min \{a, b\}$ and  $a\vee b:=\max\{a, b\}$. 
 $\N$ denotes the set of natural number and $\N_0$ denotes the set of non-negative integers.  
The notation $C=C(a,b,\ldots)$ indicates that the constant $C$
depends on $a, b, \ldots$. 
 The dependence on  $d, \alpha$, the localization characteristics $\wh R$, $\Lambda_0$ and $\Lambda$ (see Definition \ref{df:lipschitz}), and the constants in conditions ${\bf (B)}$ and ${\bf (K)}$ (see Section \ref{ch:set-up}) may not be mentioned explicitly. 
 Upper case letters $C_i$, 
 $i \in \N$, 
  with subscripts denote  strictly positive constants whose values are
 fixed throughout this work.
A lower case letter $c$ without subscript denotes a strictly positive constant whose 
value is
unimportant and which may change even within a line, 
while the values of $c_i$, 
$i \in \N_0$,
are fixed in each statement and proof, and the labeling of these constants starts anew in each proof.
We denote $x\in \R^d$ as $x=(\wt{x}, x_d)$ with $\wt{x}\in \R^{d-1}$.
We  use $m_d$ to denote the Lebesgue measure on $\R^d$. 
For a Borel subset $A\subset \R^d$, $\delta_A(x)$ denotes the Euclidean distance between $x$ and $\partial A$. 
  For a  subset $A \subset \R^d$,  we define
\begin{align*}	B_{A}(x,r):=A \cap B(x,r), \quad x \in \R^d, \, r>0.
\end{align*}
 For a given function $f$ defined on $(0,\infty)$, we set $f(\infty):=\lim_{r \to \infty}f(r)$, if the limit exists.  For a Borel set $A\subset \R^d$ and a Borel function $f$ defined on $A\times A \setminus \{(x,x):x\in A\}$, the principal value integral is defined by    
 \begin{align*}
 	 \text{p.v.}\int_A f(x,y)dy = \lim_{\eps 
	 \downarrow
	  0}\int_{A, \, |x-y|>\eps} f(x,y)dy, \quad x \in A.
 \end{align*}
    We adopt the convention $c/0=\infty$ for $c>0$.


\section{Set-up and main results}\label{ch:set-up}

In this work, we study  some
analytic and  potential-theoretic properties of Markov processes in proper open subsets $D$ of $\R^d$, $d\ge 2$, defined through their jump kernels and killing potentials. 

The jump kernels are of the form $\sB(x,y)|x-y|^{-d-\alpha}$, 
$\alpha\in (0,2)$, with a positive function $\sB$ on $D\times D$ which is allowed to decay to zero at the boundary of $D$. The killing 
potentials $\kappa:D\to [0,\infty)$ are either critical or sub-critical. It is clear that the properties of the underlying Markov process depend on the assumptions imposed on $\sB$ and $\kappa$. In this section we gradually introduce these assumptions and explain the type of results that follow. The assumptions pertaining to the function $\sB$ will be denoted as \textbf{(B)}, while those related to $\kappa$ will have the letter \textbf{(K)}. 

We will always assume that $D$ is a Lipschitz open set.
For our main results, we need further regularity of the boundary of $D$. 
Starting from Section \ref{ch:key-estimates},  we  assume that $D$ is a $C^{1,1}$  open set.

\subsection{Construction and some properties of the processes}

We begin with the construction of three processes -- the conservative process $\overline{Y}$ in the closure $\overline{D}$ of $D$, 
the process $Y^0$ in $D$, and  $Y^{\kappa}$ obtained by killing $Y^0$ via the  killing potential $\kappa$. The construction of these processes is carried through Dirichlet form theory and is quite standard.

Let $D\subset \R^d$, $d\ge 2$, be a Lipschitz open set (see 
Definition \ref{df:lipschitz} in Section \ref{ch:geometry} 
for the precise definition).  Denote by 
$\overline{D}$ the closure of $D$, 
and by $\delta_D(x)$ the Euclidean distance between $x\in \R^d$ and the boundary $\partial D$. 
We will assume  that  the jump measure of the process $\overline{Y}$ 
which we will construct is absolutely continuous with respect to the Lebesgue measure on  $\overline{D}$. Since $D$ is Lipschitz, the Lebesgue measure of
$\overline{D}$ is zero and the value of the jump kernel on $\partial D$ does not matter. 
For $\alpha \in (0,2)$ we consider the bilinear form
\begin{equation*}
	\sE^0(u,v):= \frac{1}{2}\iint_{D \times D}  (u(x)-u(y))(v(x)-v(y)) \frac{\sB(x,y)}{|x-y|^{d+\alpha}}dxdy,
\end{equation*}
where $\sB:D \times D \to (0,\infty)$ is a Borel function satisfying the following assumptions:

\bigskip

\noindent \hypertarget{Q1}{{\bf (B1)}} $\sB(x,y)=\sB(y,x)$ for all $x,y \in D$.

\smallskip

\noindent \hypertarget{Q2-a}{{\bf (B2-a)}}  There exists a constant $C_1>0$ such that $\sB(x,y) \le C_1$ for all $x,y \in D$.

\smallskip

\noindent \hypertarget{Q2-b}{{\bf (B2-b)}} For any $a\in (0,1]$, there exists a constant $C_2=C_2(a)>0$ such that
	\begin{align*}	\sB(x,y) \ge C_2 \quad \text{for all }  x,y \in D \text{ with } \delta_D(x) \wedge \delta_D(y) \ge a |x-y|.
	\end{align*}

\medskip

\emph{Assumptions \hyperlink{Q1}{{\bf (B1)}}, 
\hyperlink{Q2-a}{{\bf (B2-a)}} and \hyperlink{Q2-b}{{\bf (B2-b)}} will be in force throughout this work
except in Section \ref{ch:examples}.
}

\smallskip

Assumption \hyperlink{Q1}{{\bf (B1)}} is natural as it ensures the symmetry of the form $\EE^0$.
Note that \hyperlink{Q2-a}{{\bf (B2-a)}} implies
	\begin{align}\label{e:stochastic-complete}	
	\sup_{x \in  D} \int_{D} (1 \wedge |x-y|^2) \frac{\sB(x,y)}{|x-y|^{d+\alpha}}dy <\infty.
	\end{align} 
Moreover, \hyperlink{Q2-a}{{\bf (B2-a)}} and  \hyperlink{Q2-b}{{\bf (B2-b)}} imply that for $C_2=C_2(1)$,
	\begin{align}\label{e:B(x,x)}
		C_2\le 	\sB(x,x) \le C_1 \quad \text{for all} \;\, x \in D.
	\end{align}
Observe that assumptions \hyperlink{Q1}{{\bf (B1)}}, \hyperlink{Q2-a}{{\bf (B2-a)}} and \hyperlink{Q2-b}{{\bf (B2-b)}} do not specify the behavior of $\sB$ at the boundary of $D$.

For a Borel set $A \subset \R^d$ and  $p \in [1,\infty]$, we denote by $L^p(A)$ the $L^p$-space $L^p(A,
m_d)$, and  
by  $\mathrm{Lip}_c(A)$ the family of all Lipschitz functions on $A$ with compact support.  It follows from \eqref{e:stochastic-complete} that $\sE^0(u,u)<\infty$ for any $u \in \mathrm{Lip}_c(\overline D)$. Let  $\overline \sF$ be the closure of 
$\mathrm{Lip}_c(\overline D)$ in 
$L^2(\overline D)=L^2(D)$ under  the norm $(\sE^0_1)^{1/2}$ where  $\sE^0_1:=\sE^0 + \lVert \cdot \rVert_{L^2(D)}^2$.  Then  $(\sE^0, \overline \sF)$ is a regular Dirichlet form on $L^2(\overline D)$,
see \cite[Chapter 1]{FOT}.  Since $\sB(x, y)>0$ for all $x, y\in D$, using \cite[Theorem 1.6.1]{FOT}, one can easily see that   the Dirichlet form $(\sE^0, \overline \sF)$ is irreducible. 
Moreover,  since the  form $(\sE^0, \overline \sF)$ has no killing  and satisfies \eqref{e:stochastic-complete}, it is conservative by \cite[Theorem 1.3]{GHM} or \cite[Theorem 1.1]{MUW}. Associated with the regular Dirichlet form $(\sE^0, \overline \sF)$, there is a conservative Hunt process $\overline Y=(\overline Y_t,t \ge 0; \P_x,x \in \overline D\setminus \sN')$. Here $\sN'$ is an  exceptional set for $\overline Y$. 

Let $\sF^0$ be the closure of  $\mathrm{Lip}_c(D)$ in $L^2(D)$ under $\sE^0_1$. Then $(\sE^0,\sF^0)$ is a regular Dirichlet form. Let $Y^0=(Y^0_t,t \ge 0; \P_x,x \in \overline D\setminus \sN_0)$ be the Hunt process associated with $(\sE^0,\sF^0)$, where  $\sN_0$ is an  
exceptional set for $Y^0$.

The third process is obtained by killing $Y^0$ via  a killing potential $\kappa$. We assume that $\kappa$ is a non-negative Borel function on $D$ satisfying the following assumption:

\bigskip

\noindent \hypertarget{W1}{{\bf (K1)}}  There exists a constant $C_3>0$ such that
	\begin{equation*}
		\kappa(x) \le C_3 (\delta_D(x) \wedge 1)^{-\alpha}.
	\end{equation*}
 If $\alpha \le 1$, then we also assume that $\kappa$ is non-trivial, namely,
\begin{equation}\label{e:K}
	m_d(\{x\in D:\kappa(x) >0\})>0.
\end{equation}

\medskip

Assumption \hyperlink{W1}{{\bf (K1)}} says that  
the killing through $\kappa$
 is sub-critical or critical. 
 Note that $\kappa$ can be identically zero when $\alpha>1$.	
 If $\alpha \le 1$ and $\kappa\equiv 0$, then $Y^0= Y^\kappa=\overline{Y}$ so it is conservative.
 The additional assumption \eqref{e:K} in \hyperlink{W1}{{\bf (K1)}} 
 guarantees that $Y^\kappa$ is not conservative (see Proposition \ref{p:notconservative}).

\medskip

\emph{
	Assumption \hyperlink{W1}{{\bf (K1)}} will be in force throughout this work
	except in Section \ref{ch:examples}.}
	
	\medskip

We consider a symmetric form $(\sE^\kappa, \sF^\kappa)$ defined by
	\begin{align*}
		\sE^\kappa(u,v)&=\sE^0(u,v) + 
		\int_{D} u(x)v(x)\kappa(x)dx, \\
		\sF^\kappa&= 	\wt \sF^0 \cap L^2(D, 
			\kappa(x) dx),\nn
	\end{align*}
	where $\wt \sF^0$ is the family of all $\sE^0_1$-quasi-continuous functions in $\sF^0$. Then $(\sE^\kappa,\sF^\kappa)$ is a regular Dirichlet form on $L^2(D)$ with Lip$_c(D)$ as a special standard core,  see \cite[Theorems 6.1.1 and 6.1.2]{FOT}.
Let  $Y^\kappa=(Y^\kappa_t,t \ge 0; \P_x,x \in D \setminus  \sN_\kappa)$ be the Hunt process associated with 
$(\sE^\kappa, \sF^\kappa)$ where $\sN_\kappa$ is an exceptional set for $Y^\kappa$. We denote by $\zeta^\kappa$ the lifetime of $Y^\kappa$, and define $Y^\kappa_t=\partial$ for $t \ge \zeta^\kappa$, where $\partial$ is a cemetery point added to the state space $D$. 
Note that $Y^\kappa$ includes $Y^0$, when $\alpha\in (1, 2)$, as a special case.
	
In Section \ref{ch:processes} we establish several important properties of the processes $\overline{Y}$, $Y^0$ and $Y^{\kappa}$. 
In Subsection \ref{s-processes-Y-bar} we first look at $\overline{Y}$,  establish a Nash-type inequality (Proposition \ref{p:Nash}) which leads to the existence and some preliminary  upper bound of the  transition densities (Proposition \ref{p:upper-heatkernel}). An important preliminary lower bound of the transition densities of 
$\overline{Y}$ killed upon exiting $\overline{D}\cap B(x_0,r)$ is given in Proposition \ref{p:ndl}. 
Relying on methods from \cite{CKW20, CKW21} we then establish joint H\"older continuity of bounded caloric functions (parabolic H\"older regularity). As a consequence we get that $\overline{Y}$ can be refined to be a strongly Feller process starting from every point in $\overline{D}$ (hence the exceptional set $\sN'$ can be taken to be empty set). Finally, we show that the parabolic Harnack inequality holds true for non-negative caloric functions for $\overline{Y}$. For this property,  we need the following additional assumption on $\sB$:

\medskip
\noindent \hypertarget{UBS2}{{\bf (UBS)}} There exists $C>0$ such that  for  a.e. $x,y\in  D$,
\begin{align}\label{e:UBS-2}
	\sB(x,y)\le \frac{C}{r^d}\int_{ \overline D \cap B(x,r) } \sB(z,y)dz 
	\quad  \text{whenever} \;\, 0< r \le \frac12 ( |x-y| \wedge \wh R).
\end{align}
Here $\wh{R}$ is the localization radius of the Lipschitz open set $D$, see Definition \ref{df:lipschitz} for details.
Assumption \hyperlink{UBS2}{{\bf (UBS)}} implies the usual {\bf (UJS)} condition, see e.g.~\cite[Definition 1.16]{CKW20}.

In Subsection \ref{s-processes-Y-0} we analyze properties of $Y^0$ and $Y^{\kappa}$. We first establish that 
$\sF^0=\overline \sF$ if and only if $\alpha \le 1$ (Proposition \ref{p:alpha>1}). This implies that $Y^0=\overline{Y}$ when $\alpha\le 1$, while in case $\alpha\in (1,2)$, $Y^0$ can be regarded as the part process of $\overline{Y}$ in $D$ with a.s.~finite lifetime $\zeta^0$ such that $Y^0_{\zeta^0-}\in \partial D$. Similarly, the process $Y^{\kappa}$ can be regarded as the part process of $\overline{Y}$ killed  at the a.s.~finite lifetime $\zeta^{\kappa}$.
In this case we have that $Y^{\kappa}_{\zeta^{\kappa}-}\in D$. 
As a consequence of the fact that $Y^\kappa$ is a part process of $\overline{Y}$, we conclude that 
the exceptional set $\sN_{\kappa}$  can be taken to be an empty set.
In the remaining part of the subsection we establish the existence and an upper bound of the transition densities of $Y^\kappa$, a lower bound similar to the one described above, parabolic H\"older regularity, and parabolic Harnack inequality for non-negative caloric functions of $Y^\kappa$.  
  In order to get uniform large time estimates (Proposition \ref{p:upper-heatkernel-3}), we introduce the following assumption for $\kappa$:

\medskip

\noindent \hypertarget{W2}{{\bf (K2)}} If $\alpha\le 1$, then there exist constants $\wh r \in(0, \wh{R})$ and $C_4>0$ such that for every bounded connected component $D_0$ of $D$,
\begin{align*}
	\kappa(x) \ge C_4 \quad \text{for all $x \in D_0$ with $\delta_{D_0}(x)<\wh r$.}
\end{align*}
Note that when $\alpha\le 1$,  without extra condition for $\kappa$, the assertion of Proposition \ref{p:upper-heatkernel-3} does not hold, as demonstrated in Example \ref{ex:alpha<1}.
For the  parabolic Harnack inequality,
 we need the assumption \hyperlink{IUBS}{{\bf (IUBS)}} on $\sB$ saying that \eqref{e:UBS-2} holds when 
 $0< r \le \frac12 ( |x-y| \wedge \delta_D(x) \wedge\wh R)$.

By using the upper and lower bounds on the transition densities of $Y^\kappa$, in Subsection \ref{s-int-green} we establish (interior) estimates on the Green function $G^\kappa(x,y)$ of the process $Y^\kappa$. In case of bounded $D$, we see that 
$G^\kappa(x,y)\le C|x-y|^{-d+\alpha}$ for all $x,y\in D$, and the same lower bound is valid if $x$ and $y$ are away from the boundary
(see Corollary \ref{c-int-green-bounded} for the precise statement).

\subsection{The operator $L^\sB_\alpha$}\label{s-operator-2}

In order to study finer properties of the process $Y^\kappa$ we need additional assumptions on the function $\sB$ that we now describe. We still assume that $D$ is a Lipschitz open set. 
The following assumption 
is needed to make sure that $C^1_c(D)$, the space of continuous functions with compact support in $D$, is contained
in the domain of definition of the operator $L^\sB_\alpha$ introduced below.

\medskip

\noindent \hypertarget{Q3}{{\bf (B3)}} If $\alpha \ge 1$, then there exist constants $\theta_0>\alpha-1$ and $C_5>0$ such that
\begin{align}\label{e:ass-B3-2}
	|\sB(x,x)-\sB(x,y)| \le C_5 \bigg(\frac{|x-y|}{\delta_D(x) \wedge \delta_D(y) \wedge  \wh R} \bigg)^{\theta_0} \quad \text{for all }  x,y \in D.
\end{align}

\medskip

Consider a non-local operator $(L^\sB_\alpha, \sD(L^\sB_\alpha))$ of the form
\begin{align}\label{e:def-L-alpha-2}
	L^\sB_\alpha f(x)&=\text{p.v.}\int_D (f(y)-f(x)) \frac{\sB(x,y)}{|x-y|^{d+\alpha}}dy, \quad x \in D,
\end{align}
where  $\sD(L^\sB_\alpha)$ consists of  all functions $f:D \to \R$ for which the above principal value integral makes sense.  
Note that if $f\in C^1_c(D)$,
the integral above is absolutely convergent for 
$\alpha\in (0,1)$. In case $\alpha\ge 1$, principal value is needed to make sense of the integral. When $\sB(x,y)$ is a constant, a symmetry argument guarantees that the principal value integral is well defined for 
$f\in C^1_c(D)$. When $\sB$ is not a constant, the symmetry argument breaks down, but \hyperlink{Q3}{{\bf (B3)}} guarantees that $L^\sB_\alpha f$ is well defined for $f\in C^1_c(D)$.

Recall that $\kappa$ is a  non-negative Borel  function on $D$ satisfying \hyperlink{W1}{{\bf (K1)}}.
We define an operator $(L^\kappa, \sD(L^\sB_\alpha))$  by 
\begin{align}\label{e:def-operator-2}
	L^\kappa f(x)=L^\sB_\alpha f(x) - \kappa(x)f(x), \quad  x \in D.
\end{align}
Let $(\sA^\kappa, \sD(\sA^\kappa))$ be the $L^2$-generator of $(\sE^\kappa,\sF^\kappa)$. 
Under the assumptions \hyperlink{Q1}{{\bf (B1)}}, \hyperlink{Q2-a}{{\bf (B2-a)}}, \hyperlink{Q2-b}{{\bf (B2-b)}} and \hyperlink{Q3}{{\bf (B3)}}, 
we will establish in Proposition \ref{p:generator-C11} that $\sA^\kappa f = L^\kappa f$ for all $f$ in an appropriate class of functions, showing that $L^\kappa $ is the infinitesimal generator of the semigroup corresponding to $Y^{\kappa}$. Additionally, we will prove a Dynkin-type formula for not necessarily smooth and compactly supported functions, see Corollary \ref{c:Dynkin-local}.

Let $\Phi_0$ be a Borel  function on $(0,\infty)$ such that $\Phi_0(r)=1$ for $r \ge 1$ and 
\begin{align}
 	c_L \bigg( \frac{r}{s}\bigg)^{\lb_0}\le 	\frac{\Phi_0(r)}{\Phi_0(s)} \le c_U \bigg( \frac{r}{s}\bigg)^{\ub_0} 
	\quad \text{for all} \;\, 0<s\le r\le 1, \label{e:scale-2}
\end{align}
for some constants $\ub_0\ge\lb_0\ge0$ and $c_L,c_U>0$.  
Let $\beta_0$ be the lower Matuszewska index of $\Phi_0$ (see \cite[pp. 68--71]{BGT}):
\begin{equation}\label{e:Matuszewska}
 \beta_0=\sup\left \{\beta:\exists \, a>0 \,\mbox{ s.\ t. } \Phi_0(r)/\Phi_0(s)\ge a(r/s)^\beta \, \mbox{ for } \,0<s\le r\le 1 \right\}.
\end{equation}
Typical examples of such a function $\Phi_0$ include $\Phi_0(r)=(r\wedge 1)^{\beta}$ for $\beta \ge 0$. In this case, the lower Matuszewska index of $\Phi_0$  is equal to $\beta$.  
The property \eqref{e:scale-2} of $\Phi_0$ is usually referred to as a \emph{weak scaling condition at zero}. It clearly implies that $\Phi_0$ is \emph{almost increasing}, namely, for all $0\le s\le r<\infty$, $c_L\Phi_0(s)\le \Phi_0(r)$.  The precise value of $\beta_0$ will appear in our results, while the precise value of the upper scaling index $\ub_0$  remains insignificant for most of the content presented in this work.

We next consider the following two assumptions on $\sB$:

\medskip

\noindent \hypertarget{Q4-a}{{\bf (B4-a)}}
There exists a constant $C_6>0$ such that 
\begin{align*}
	\sB(x,y) \le C_6 \Phi_0\bigg(\frac{\delta_D(x) \wedge \delta_D(y)}{|x-y|}\bigg) \quad \text{for all }  x,y \in D.
\end{align*}

\noindent \hypertarget{Q4-b}{{\bf (B4-b)}}  There exists a constant $C_7>0$ such that
\begin{align*}
	\sB(x,y) \ge C_7 \Phi_0\left(\frac{\delta_D(x) \wedge \delta_D(y)}{|x-y|}\right) \quad \text{for all }  x,y \in D \text{ with } \delta_D(x) \vee \delta_D(y) \ge \frac{|x-y|}{2}.
\end{align*}

\medskip

Assumptions \hyperlink{Q4-a}{{\bf (B4-a)}} and \hyperlink{Q4-b}{{\bf (B4-b)}} are inspired by \eqref{e:intro-skst} -- instead of the explicit function there, we use the function $\Phi_0$. Clearly, \hyperlink{Q4-a}{{\bf (B4-a)}} implies that the jump kernel $\sB(x,y)|x-y|^{-d-\alpha}$ may decay to zero at the boundary. Note that \hyperlink{Q4-a}{{\bf (B4-a)}} implies \hyperlink{Q2-a}{{\bf (B2-a)}}.

In the remainder of this section, we assume that
$$
\textit{$\sB$ satisfies \hyperlink{Q1}{{\bf (B1)}}, \hyperlink{Q2-b}{{\bf (B2-b)}}, \hyperlink{Q3}{{\bf (B3)}}, \hyperlink{Q4-a}{{\bf (B4-a)}} and \hyperlink{Q4-b}{{\bf (B4-b)}}.}
$$

A usual way to estimate  exit probabilities of 
a Markov process is to construct appropriate functions, called barriers, which are either superharmonic or subharmonic for the infinitesimal generator (and may have some additional desired properties). Applying a Dynkin-type formula to such barriers provides useful information on the exit probabilities. In Subsection 
\ref{s-barrier} we construct a family of such barriers, $\psi^{(r)}$, and in Proposition \ref{p:compensator} 
give an upper bound on $L^{\sB}_{\alpha}\psi^{(r)}$ in terms of the function $\Phi_0$. In case when $D$ is a half-space, a similar barrier is constructed in \cite[Section 8]{KSV} -- this was the key technical result of that paper. The construction and the estimate given here are simpler, and independent of the half-space result in \cite{KSV}. It is worth mentioning that all subsequent results of this work are independent of the results proved in \cite{KSV, KSV22b}  
in case of the half-space, thus making this work essentially 
self-sufficient. 

\subsection{Key assumptions on $\sB$ and $\kappa$}\label{s-key-assumptions}

We first give the definition of a $C^{1,1}$ open set.
The description of additional assumptions on $\sB$ will be given in local coordinates. 

\begin{defn}\label{df-c11}
	We say that $D$ is a $C^{1,1}$ open set with characteristics  $(\wh R, \Lambda)$, if 
	for each $Q \in \partial D$, there exist a $C^{1,1}$ function $\Psi=\Psi^Q:\R^{d-1} \to \R$ with 
	\begin{align}\label{e:local-map-2}
		\Psi(\wt 0)= |\nabla \Psi(\wt 0)|=0 \quad \text{and} \quad |\nabla \Psi(\wt y)-\nabla\Psi(\wt z)| \le \Lambda |\wt y-\wt z| \; \text{ for all } \,  \wt y, \wt z \in \R^{d-1},
	\end{align}
	 and an orthonormal coordinate system CS$_{Q}$ with origin at $Q$ such that
	\begin{align}\label{e:local-coordinate_2}
		 B_D(Q, \wh R) = \left\{ y= (\wt y, y_d) \in B(0, \wh R) \text{ in CS$_{Q}$} : y_d>\Psi(\wt y) \right\}.
	\end{align}
\end{defn}
From now on we assume that $D\subset \R^d$ is a $C^{1,1}$ open set with characteristics $(\wh{R}, \Lambda)$. Without loss of generality, we assume that   $\wh R\le 1 \wedge (1/(2\Lambda))$.

For $Q\in \partial D$, $\nu\in (0,1]$ and $r\in (0, \wh R/4]$, we introduce the set
	\begin{align}\label{e:EQ}
E^Q_\nu(r)= \left\{ y=(\wt y, y_d) \text{ in CS$_{Q}$}:  |\wt y|<r/4, \,    4r^{-\nu}|\wt y|^{1+\nu} <y_d < r/2 \right\}.
	\end{align}

Here is our key assumption on the killing potential $\kappa$:

\medskip
\noindent \hypertarget{W3}{{\bf (K3)}} There exist constants $\eta_0>0$ and 
	$C_8, C_9\ge 0$ 
	such that for all $x \in D$,
	\begin{align}\label{e:kappa-explicit-2}
		\begin{cases}
			|\kappa(x) - C_9 \sB(x,x) \delta_D(x)^{-\alpha}| \le C_8\delta_D(x)^{-\alpha+\eta_0} &\text{ if }  \delta_D(x) < 1,\\[4pt]
			\kappa(x) \le C_8 &\text{ if } \delta_D(x) \ge 1.
		\end{cases}
	\end{align}
When $\alpha \le 1$, we further assume that $C_9>0$. 

\bigskip 
We note that assumption \hyperlink{W3}{{\bf (K3)}} implies \hyperlink{W1}{{\bf (K1)}} and \hyperlink{W2}{{\bf (K2)}}.
Observe that when $C_9>0$, the killing 
potential $\kappa(x)$ is  comparable to $\delta_D(x)^{-\alpha}$ near the boundary, so we have critical killing. In case $C_9=0$ (which by assumption is allowed only when $1<\alpha <2$), we see that $\kappa(x)\le C_8\delta_D(x)^{-\alpha+\eta_0}$, hence the killing is subcritical (which includes the case of no killing at all). In the next assumption on $\sB$ we will discuss these two cases separately.

For $a\in\R$, let $\bH_a=\{(\wt y,y_d)\in \R^d: y_d>a\}$, and denote $\bH_0$ by $\bH$. Let further $\e_d=(\wt 0,1) \in \R^{d}$ be the unit vector in the vertical direction. 

\medskip
	
	\subsubsection{Case $C_9>0$ -- critical killing.}
	 The assumption that we are going to introduce may be viewed as a substitute for the \emph{flattening of the boundary} method which, as described in the introduction, does not work in the current setting. In order to motivate the assumption, we look at the process $Y^{\bH}$ obtained by subordinating a $\gamma$-stable process killed upon exiting the half-space $\bH$, via an independent $\beta$-stable 
subordinator (non-decreasing L\'evy process), where $\gamma\in (0,2)$ and  $\beta\in (0,1)$.  
Set $\alpha=\gamma \beta$. Let $J^{\bH}(x,y)$, $x,y\in \bH$, denote the jump kernel of $Y^{\bH}$. It can be written in the form 
$J^{\bH}(x,y)=\sB^{\bH}(x,y)|x-y|^{-d-\alpha}$, with $\sB^{\bH}(x,x)=
c_{d, -\alpha}$.
Due to the scale and horizontal translation invariance of $Y^{\bH}$, the function $\sB^{\bH}(x,y)$ satisfies for all $a>0$ and all $\wt{z}\in \R^{d-1}$,
$$
\sB^{\bH}(x,y)=\sB^{\bH}(ax,ay)=\sB^{\bH}(x+(\wt{z},0),y+(\wt{z},0)).
$$
If we define $F_0^{\gamma, \beta}:\bH_{-1}\to [0,\infty)$ by 
$F_0^{\gamma, \beta}(z)=c_{d, -\alpha}^{-1}\sB^{\bH}(\e_d, \e_d+z)$, 
then it is straightforward that (see Lemma \ref{l:ass-1})
\begin{equation}\label{e:B=F_0}
\sB^{\bH}(x,y)
=c_{d, -\alpha}F_0^{\gamma, \beta}\left(\frac{y-x}{x_d}\right), \quad x,y\in \bH,
\end{equation}
and, by the symmetry of $\sB^{\bH}(x,y)$,
\begin{equation}\label{e:F_0-symmetry}
F_0^{\gamma, \beta}(z)=F_0^{\gamma, \beta}(-z/(1+z_d)), \quad z\in \bH_{-1}.
\end{equation}
For a $C^{1,1}$ open set $D$ with characteristics $(\wh{R},\Lambda)$, let $Y^D$ be a process constructed analogously to $Y^{\bH}$ -- we subordinate a $\gamma$-stable process killed upon exiting $D$ by an independent $\beta$-stable subordinator. Its jump kernel can be written as 
$J^D(x,y)=\sB^D(x,y)|x-y|^{-d-\alpha}$
with $\sB^{D}(x,x)=c_{d, -\alpha}$.
 Fix a point $Q\in \partial D$ and consider the orthonormal coordinate system CS$_{Q}$ with 
origin at $Q$ (as in Definition \ref{df-c11}) and recall
that $E^Q_\nu(r)$ is defined in \eqref{e:EQ}.
 Then, under the assumption that $D$ is either (1) bounded or (2) the domain above the graph of a bounded $C^{1,1}$ function in $\R^{d-1}$,
 one can show (see Lemma \ref{l:gamma-beta-extra-decay}) that
  there exists $C>0$ such that for all $\nu\in (0,1)$ and $x,y \in E^Q_{\nu}(\wh{R}/8)$,
	\begin{align*}
		|J^D(x,y)-J^\bH(x,y)| \le  C\bigg(\frac{\delta_D(x) \vee \delta_D(y)}{
		\wh{R}}\bigg)^{(1-\beta)(1-\nu)\gamma/(2+2\nu)
		} 
		\frac{1}{(\delta_D(x)\vee\delta_D(y))^{d+\alpha}}.
	\end{align*}
Taking into account that $J^{\bH}(x,y)=\sB^{\bH}(x,y)|x-y|^{-d-\alpha}$, $J^D(x,y)=\sB^D(x,y)|x-y|^{-d-\alpha}$ and \eqref{e:B=F_0}, we get that for all $\nu\in (0,1)$ and $x,y \in E^Q_{\nu}(\wh{R}/8)$,
\begin{align*}
&\Big|\sB^D(x,y)- \sB^{D}(x,x)F_0^{\gamma,\beta}((y-x)/x_d)\Big|\\
& \quad =\Big|\sB^D(x,y)- c_{d, -\alpha}F_0^{\gamma,\beta}((y-x)/x_d)\Big|\\
& \quad \le  c \bigg(\frac{ |x-y|}{\delta_D(x) \vee \delta_D(y)} \bigg)^{d+\alpha}\bigg( \frac{\delta_D(x)\vee\delta_D(y)}{ \wh{R}}\bigg)^{(1-\beta)(1-\nu)\gamma/(2+2\nu)
	}\\
& \quad \le  c \bigg(\frac{\delta_D(x) \vee \delta_D(y) \vee  |x-y|}{\delta_D(x) \wedge \delta_D(y) \wedge |x-y|} \bigg)^{d+\alpha}\bigg( \frac{\delta_D(x)\vee\delta_D(y) 
\vee |x-y|}{ \wh{R}}\bigg)^{
(1-\beta)(1-\nu)\gamma/(2+2\nu)},
\end{align*}
where the constant 
$c>0$  depends only on $d$ and $\gamma$. 
This calculation serves as one motivation for the following assumption:

\bigskip
	\noindent \hypertarget{Q5-I}{{\bf (B5-I)}} There  exist  constants $\nu \in (0,1]$, $\theta_1,\theta_2,C_{10}>0$, and a non-negative Borel function $\F_0$ on $\bH_{-1}$   such that  for any $Q \in \partial D$ and    $x,y \in E^Q_\nu(\wh R/8)$ with $x=(\wt x,x_d)$ in CS$_Q$,
	\begin{align}\label{e:ass-B5-2}
			& \big|\sB(x,y)- \sB(x,x)\F_0((y-x)/x_d) \big| + \big|\sB(x,y)- \sB(y,y)\F_0((y-x)/x_d) \big|\nn\\	&\le C_{10} \bigg(\frac{ \delta_D(x)\vee\delta_D(y) \vee |x-y|}{ \delta_D(x) \wedge \delta_D(y) \wedge |x-y|} \bigg)^{\theta_1}
			\big( \delta_D(x)\vee\delta_D(y) \vee|x-y|\big)^{\theta_2}.
\end{align}
We allow that constants above depend on $\wh R$.
\bigskip

Under condition \hyperlink{Q5-I}{{\bf (B5-I)}}, we  define a  function $\F$ on $\bH_{-1}$ by 
	\begin{align}\label{e:def-F0-transform-2}
		\F(y) =\frac{\F_0(y) + \F_0(-y/(1+y_d))}{2}, \quad \;\;y  = (\wt y,y_d) \in \bH_{-1}.
	\end{align}
We will see in Lemma \ref{l:F-basic} that $\F$ is a bounded function. Moreover, we observe that
	\begin{align}\label{e:A5-F-2}
		\F(y)=\F(-y/(1+y_d)) \quad \text{for all} \;\, y \in \bH_{-1}.
	\end{align}
This property is in a crucial way related to the symmetry of $\sB$, see \eqref{e:F_0-symmetry}. 
With the function $\F$ above and $q \in [(\alpha-1)_+, \alpha+\beta_0)$, 
we associate a  constant $C(\alpha,q,\F)$ defined by
\begin{align}\label{e:def-killing-constant-2}
		&	C(\alpha,q,\F)\\
		&=\int_{\R^{d-1}}\frac{1}{(|\wt u|^2+1)^{(d+\alpha)/2}}\int_0^1  
		\frac{(s^q -1)(1-s^{\alpha-1-q})}{(1-s)^{1+\alpha}} \F\big(((s-1)\wt u, s-1)\big) ds \, d \wt u.\nn
\end{align}
We additionally assume that
\begin{align}\label{e:K4-2}	
	C_9 < \lim_{q \to \alpha+\beta_0} C(\alpha,q,\F).
\end{align}
		
We will show in Lemma \ref{l:constant} that	$q\mapsto C(\alpha,q,\F)$ is a well-defined 
strictly increasing continuous function on   
$[(\alpha-1)_+, \alpha+\beta_0)$ and $C(\alpha,(\alpha-1)_+,\F)=0$. 
Therefore, under \eqref{e:K4-2},  there exists a unique constant 
$p\in ((\alpha-1)_+, \alpha + \beta_0)$ such that 
	\begin{align}\label{e:C(alpha,p,F)-2}
		C_9=C(\alpha,p,\F).
	\end{align}	
The one-to-one correspondence between the  positive constants $C_9$ in \eqref{e:kappa-explicit-2} 
that multiply $\sB(x,x)\delta_D(x)^{-\alpha}$,  and the parameters 
$p\in ((\alpha-1)_+, \alpha + \beta_0)$ plays a fundamental role in this work.

The process $Y^D$ described above is a prime example of a process satisfying \hyperlink{Q5-I}{{\bf (B5-I)}}, \hyperlink{W3}{{\bf (K3)}} and \eqref{e:K4-2} (as well as the other assumptions on $\sB$ introduced before). This is shown in Subsection \ref{s:subordinate-killed}, which also contains two other examples satisfying all 
our assumptions. 

\medskip
	
\subsubsection{Case $C_9=0$ -- subcritical killing.}
	In this case, we assume the constant $C_9$ is zero.
 In this case,	 instead of \hyperlink{Q5-I}{{\bf (B5-I)}}, we will introduce a weaker assumption
	\hyperlink{Q5-II}{{\bf (B5-II)}}. The motivation for this assumption comes from the following 
	example.
	
	\begin{example}
	Assume that $\alpha \in (1,2)$ and 
	\begin{align}\label{e:example-censored-1-2}
	C^{-1}\le 	\sB(x,y)=\sB(y,x) \le C \quad \text{for all} \;\, x,y \in D
	\end{align}			
for some $C\ge 1$. 
When $\sB(x,y) \equiv c$ is a constant,  the operator 
$L^\sB_\alpha$  in \eqref{e:def-L-alpha-2}
is called the \textit{regional  (or censored)  fractional Laplacian} in $D$ and 
the process $Y^0$ corresponding to $L^\sB_\alpha$
 is called the \textit{censored $\alpha$-stable process} on $D$.

Let  $\theta \in (\alpha-1,1)$. 
Since $\wh R\le 1 \wedge (1/(2\Lambda))$,  for all $y\in D$ with $\delta_D(y)<\wh R/8$, there is 
a unique $Q_y \in \partial D$ 
such that $\delta_D(y)=|y-Q_y|$, see Lemma \ref{l:C11}(ii).  For $y\in D$ with $\delta_D(y)<\wh R/8$, let $\overline y$ be the reflection of $y$ with respect to $\partial D$, that is, 
$\overline y= 2Q_y -y$.

 Suppose that there exist $C>0$ and  
$\theta$-H\"older continuous functions $h_1:D \times D \to [0,\infty)$, $h_2:D \times D \to [0,\infty)$, and  $\Theta: [0,\infty) \to [0,\infty)$   such that $\sup_{x \in D}h_2(x,x)<\infty$ and  for all $x,y \in D$,
\begin{align}\label{e:censored-condition-2}
\!	\begin{cases}
			\left| \sB(x,x)-\sB(x,y) \right| \le C|x-y|^{\theta} &\mbox{if }\delta_D(x) \wedge \delta_D(y) >
			\wh R/16,\\[4pt]
			 \displaystyle 	\bigg| 	\sB(x,y) - h_1(x,y) - h_2(x,y) 
			 \Theta\bigg(  \frac{|x-y|}{|x-\overline y|}\bigg) \bigg| \le C|x-y|^{\theta}  &\mbox{if }\delta_D(x) \vee \delta_D(y) 
			 < \wh R/8.
	\end{cases}
\end{align}
In case  $\Theta(r)=r^{d+\alpha}$  and $h_1,h_2\in C^1(\overline{D}\times \overline{D})$, such a condition was considered in \cite{Gu07} to establish a unified framework that incorporates both the regional fractional Laplacian and the formal generator of \textit{subordinate reflected Brownian motions} on $D$. The main result of that paper was the boundary Harnack principle for non-negative harmonic functions with respect to $L^\sB_\alpha$.

From \eqref{e:censored-condition-2}, one can see that one function $\F_0$ is not enough to approximate $\sB(x,y)$ as in \eqref{e:ass-B5-2}, and that we need two functions. Indeed, by setting $\mu^1(x)=h_1(x,x)$ and $\mu^2(x)=h_2(x,x)$ for $x \in D$, and
$$
F^1_0(z)=1 \quad \text{and} \quad F^2_0(z)=\Theta(|z|/|(\wt z, - z_d-2)|) \quad \text{for}  \;\,z \in 
\bH_{-1},
$$
we show in Example \ref{ex:censored} that if $Q\in \partial D$ and $x,y\in E^Q_{1/2}(\wh{R}/8)$, then
\begin{align*}
&	\bigg|\sB(x,y)- \sum_{i=1}^{2}\mu^i(x)F_0^i((y-x)/x_d) \bigg| + \bigg|\sB(x,y)- \sum_{i=1}^{2}\mu^i(y)F_0^i((y-x)/x_d) \bigg|\\
&\le c \bigg(\frac{\delta_D(x)\vee\delta_D(y)\vee|x-y|}{\delta_D(x) \wedge \delta_D(y) \wedge |x-y|} \bigg)^{2\theta}
\big( \delta_D(x)\vee\delta_D(y)\vee|x-y|\big)^{\theta/3}.
\end{align*}
	\end{example}
This example motivates the following assumption:

\bigskip
	
	\noindent \hypertarget{Q5-II}{{\bf (B5-II)}} There  exist  constants $\nu \in (0,1]$, $\theta_1,\theta_2,C_{10}>0$, $C_{11}>1$, 
	$i_0 \in\bN$, and  non-negative  Borel functions $\F_0^i:\bH_{-1} \to [0,\infty)$ and $\mu^i:D \to (0,\infty)$, $1\le i \le i_0$, such that 
	\begin{align}\label{e:mu-bound-2}
	C_{11}^{-1} \le	\mu^i(x) \le C_{11} \quad \text{for all} \;\, x \in D,
	\end{align}
	and for any $Q \in \partial D$ and    $x,y \in E^Q_\nu(\wh R/8)$ with $x=(\wt x,x_d)$ in CS$_Q$,
	\begin{align}\label{e:ass-B5'-2}	
	\begin{split}
		&\bigg|\sB(x,y)- \sum_{i=1}^{i_0}\mu^i(x)\F_0^i((y-x)/x_d) \bigg| +  \bigg|\sB(x,y)- \sum_{i=1}^{i_0}\mu^i(y)\F_0^i((y-x)/x_d) \bigg| \\
		&\le C_{10} \bigg(\frac{ \delta_D(x)\vee\delta_D(y) \vee |x-y|}{ \delta_D(x) \wedge \delta_D(y) \wedge |x-y|} \bigg)^{\theta_1}
			\big( \delta_D(x)\vee\delta_D(y) \vee|x-y|\big)^{\theta_2}.
	\end{split}
	\end{align}
	For each $1\le i\le i_0$, we define $\F^i(y):=(\F_0^i(y)+\F_0^i(-y/(1+y_d)))/2$ and  $C(\alpha, q,\F^i)$ for 
	$q \in [(\alpha-1)_+, \alpha+\beta_0)$ analogously to \eqref{e:def-killing-constant-2}.
	
\bigskip
	
Note that if \hyperlink{Q5-I}{{\bf (B5-I)}} holds, then also \hyperlink{Q5-II}{{\bf (B5-II)}} holds with $i_0=1$, $\F_0^1=\F_0$ and $\mu^1(x)=\sB(x,x)$.
	
We combine the assumptions \hyperlink{Q5-I}{{\bf (B5-I)}} and \hyperlink{Q5-II}{{\bf (B5-II)}} in the assumption
	
\bigskip
	
\noindent \hypertarget{Q5}{{\bf (B5)}}  If $C_9>0$, then	\hyperlink{Q5-I}{{\bf (B5-I)}} and \eqref{e:K4-2} hold, and  if $C_9=0$, then \hyperlink{Q5-II}{{\bf (B5-II)}} holds. 
	
\bigskip
Also, we treat \hyperlink{Q5-I}{{\bf (B5-I)}} as a special case of \hyperlink{Q5-II}{{\bf (B5-II)}} with $i_0=1$.
In the remainder of this section, we assume that
$$
\textit{$\kappa$ satisfies \hyperlink{W3}{{\bf (K3)}};}
$$
$$
\textit{$\sB$ satisfies \hyperlink{Q1}{{\bf (B1)}}, \hyperlink{Q2-b}{{\bf (B2-b)}}, \hyperlink{Q3}{{\bf (B3)}}, \hyperlink{Q4-a}{{\bf (B4-a)}}, \hyperlink{Q4-b}{{\bf (B4-b)}} and \hyperlink{Q5}{{\bf (B5)}}. }
$$

\medskip
We explain now the key result that follows from the assumption \hyperlink{Q5}{{\bf (B5)}}. Recall from the introduction that in case of the 
half-space one can calculate the action of the operator $L^B_{\alpha}$ on the power of the distance function - see 
\eqref{e:intro-L-on-dist}. This calculation uses the scaling properties of the associated process in an  essential way. Instead of such an exact formula, assumption \hyperlink{Q5}{{\bf (B5)}} allows for a weaker, but sufficient, 
substitute. 
For $Q\in \partial D$ and 
$a,b \in (0, \wh{R}/2)$, let 
	$$
		U^Q(a,b)=\big\{x\in D: x=(\wt{x},x_d)\text{ in CS}_Q \mbox{ with } |\wt{x}|< a,\, 0<\rho_D(x)<b\big\}
	$$	
denote the box of width $a$ and height $b$ based at $Q$. 
Here $\rho_D(x)=\rho^Q_D(x)=x_d-\Psi(\wt{x})$ is the ``vertical distance" 
 of the point $x$ to the boundary in the local coordinate system CS$_Q$ with $C^{1,1}$ function $\Psi$. 
 (See \eqref{e:local-coordinate_2}.)
 We denote 
 $U^Q(r,r)$  as $U^Q(r)$. 
For $r< \wh{R}/8$, let $V$ be a Borel set satisfying $U^Q(3r)\subset V\subset 
B_D(Q, \wh R)$.
Let $h_{q,V}(y)=\1_V(y) \delta_D(y)^q$ be the $q$-th power of the cutoff distance function where 
$q\in [(\alpha-1)_+,\alpha+\beta_0) \cap (0,\infty)$. Then for any $x\in U^Q(r/4)$,
	\begin{align}\label{e:prop-barrier}
		\big|L_\alpha^\sB h_{q,V}(x)-\sum_{i=1}^{i_0} \mu^i(x) C(\alpha,q,\F^i)  \delta_D(x)^{q-\alpha}\big| 
		\le C (\delta_D(x)/r)^{\eta_1} \delta_D(x)^{q-\alpha},
	\end{align}
with constants $C>0$ and $\eta_1>0$ independent of $Q$, $r$ and $V$, see Proposition \ref{p:barrier}. This shows that the operator $L^{\sB}_{\alpha}$ essentially acts on the power of the cutoff distance function by decreasing the power by $\alpha$ (up to a lower order term). 

In Subsection \ref{s-refined-barriers} we construct more refined barrier functions by using combinations of cutoff functions of the type $h_{q, U(r)}(y)=\1_{U(r)}(y) \delta_D(y)^q$ and the already constructed barrier $\psi^{(r)}$. Estimates of the action of 
the operator $L^\kappa$  (and a related operator) on these barriers are based on the estimate \eqref{e:prop-barrier}.
Combined with 
the Dynkin-type formula  in Corollary \ref{c:Dynkin-local}, 
 these estimates lead in Subsection \ref{s:decay-spec-har}  to various exit probability estimates and decay rates of some special harmonic functions. Before describing these estimates, let us recall that a non-negative Borel function $f$ on $D$ is said to be \emph{harmonic} in an open set $V\subset D$ with respect to the process $Y^\kappa$ if for every open 
$U\subset \overline{U}\subset V$,
	$$
	f(x)=\E_x[f(Y^\kappa_{\tau_U})], \quad \text{for all } x\in U, 
	$$
where $\tau_U:=\inf\{t>0:Y^\kappa_t \notin U\}.$
Important examples of non-negative harmonic functions are 
$$
	x\mapsto \P_x(Y^\kappa_{\tau_{U(\epsilon_2 r)}}\in U(r)\setminus U(r,r/2)) \quad \text{and} 
	\quad x\mapsto 	\P_x(Y^\kappa_{\tau_{U(\epsilon_2 r)}}\in D).
$$
Here $\epsilon_2$ is some small constant, and $r\in (0, \wh{R}/24)$. These two harmonic functions continuously decay to zero at the boundary of $D$. The key result of Subsection \ref{s:decay-spec-har} is Theorem \ref{t:Dynkin-improve} stating that 
their decay rates are
 comparable to $(\delta_D(x)/r)^p$. Here $p\in [(\alpha-1)_+, \alpha + \beta_0)$ is the parameter corresponding to
$C_9$ through \eqref{e:C(alpha,p,F)-2} if $C_9>0$, and $p=\alpha-1$ if $C_9=0$. 

The exact decay rate of these two special harmonic functions is used in Section 
 \ref{ch-killed-potentials} to establish the Green function estimates of the process $Y^\kappa$ killed upon exiting $D\cap B(x_0, R_0)$ with $x_0\in \overline{D}$ and $R_0>0$. These estimates improve the ones from Subsection \ref{s-int-green} in the sense that the preliminary boundary decay is now included. The main result of the section is Proposition \ref{p:bound-for-integral-new} which gives sharp estimates of the Green potentials of powers of distance functions. To be more precise, for any $Q\in \partial D$, any $R\in (0,\wh{R}/24)$, and any Borel set $A$ satisfying $D\cap B(Q,R/4)\subset A \subset B(Q,R)$, we establish sharp bounds of $\int_A G^A (x,y)\delta_D(y)^{\gamma}\,dy$ for $\gamma >-p-1$. Here $G^A$ is the Green function of $Y^\kappa$ killed upon exiting $A$.

\subsection{Final assumption and main results}\label{s-main-results}

Our final assumption \hyperlink{Q4-c}{{\bf (B4-c)}} below 
 replaces \hyperlink{Q4-a}{{\bf (B4-a)}} and \hyperlink{Q4-b}{{\bf (B4-b)}}, and gives precise upper and lower bounds on the decay rate of $\sB$.

Let  $\Phi_1$ and $\Phi_2$ be Borel  functions on $(0,\infty)$ such that $\Phi_1(r)= \Phi_2(r)=1$ for $r \ge 1$
and that
	\begin{align}
		c_L' \bigg( \frac{r}{s}\bigg)^{\lb_1}\le 	\frac{\Phi_1(r)}{\Phi_1(s)} \le c_U' \bigg( \frac{r}{s}\bigg)^{\ub_1} 
		\quad \text{for all} \;\, 0<s\le r\le 1,\label{e:scale1new-2}
	\end{align}
and
	\begin{align}
		c_L''\bigg( \frac{r}{s}\bigg)^{\lb_2}\le 	\frac{\Phi_2(r)}{\Phi_2(s)} \le c_U'' \bigg( \frac{r}{s}\bigg)^{\ub_2} 
		\quad \text{for all} \;\, 0<s\le r\le 1\label{e:scale2-2}
	\end{align}
for some $\ub_1\ge \lb_1\ge 0$, $\ub_2\ge \lb_2\ge 0$ and $c_L', c_U', c_L'', c_U''>0$.
Let $\beta_1$ and $\beta_2$ be the lower Matuszewska indices of $\Phi_1$ and $\Phi_2$ respectively.

Let $\ell$ be a Borel  function on $(0,\infty)$ with the following properties: (i)   $\ell(r)=1$ for $r \ge 1$,  
and (ii)
for every $\eps>0$, there exists a constant  $c(\eps)>1$ such that 
	\begin{align}	
		c(\eps)^{-1} \bigg( \frac{r}{s}\bigg)^{-\eps  \wedge   \beta_1}\le	\frac{\ell(r)}{\ell(s)} 
		\le 	c(\eps) \bigg( \frac{r}{s}\bigg)^{\eps  		\wedge \beta_2} 
		\quad \text{for all} \;\, 0<s\le r\le 1.	\label{e:scale3-2}
	\end{align}
Note that $\ell$ is almost increasing if $\beta_1=0$, and $\ell$ is almost decreasing if $\beta_2=0$.

We consider the following assumption which should be compared with \eqref{e:intro-skst} and an analogous  assumption in the half-space case, see \cite[(1.2)]{KSV20} and Remark \ref{r:compare-KSV} below.

\bigskip

\noindent \hypertarget{Q4-c}{{\bf (B4-c)}} There exist  comparison constants such that for all $x,y \in D$,
\begin{align*}
	\sB(x,y) \asymp  \Phi_1\bigg(\frac{\delta_D(x) \wedge \delta_D(y)}{|x-y|}\bigg)\,
	\Phi_2\bigg(\frac{\delta_D(x) \vee \delta_D(y)}{|x-y|}\bigg)\,
	\ell\bigg(\frac{\delta_D(x)\wedge \delta_D(y)}{(\delta_D(x)\vee\delta_D(y))\wedge |x-y|}\bigg).
\end{align*}

\bigskip

Define a function $\Phi_0$ on $(0, \infty)$ by
\begin{align}\label{e:phi0new-2}
\Phi_0(r):=\Phi_1(r)\ell(r), \qquad r>0.
\end{align}
By \eqref{e:scale1new-2}, \eqref{e:scale3-2} and the definition of the lower Matuszewska index, since both $\Phi_1$ and $\ell$ are almost increasing if $\beta_1=0$, we see that for any $\eps>0$, there exists $\wt c(\eps)>1$ such that 
\begin{align}\label{e:scalenew-2}
	 \wt c(\eps)^{-1}\bigg( \frac{r}{s}\bigg)^{\beta_1-\eps\wedge\beta_1}
	\le \frac{\Phi_0(r)}{\Phi_0(s)}\le \wt c(\eps) \bigg( \frac{r}{s}\bigg)^{\ub_1+\eps \wedge\beta_2}
	\quad \mbox{for all } 0<s\le r\le 1.
\end{align}
Hence, the function $\Phi_0$ defined in \eqref{e:phi0new-2} satisfies \eqref{e:scale-2}
and is thus almost increasing.
We emphasize that by \eqref{e:scale3-2},
$$
	\text{the lower Matuszewska index of $\Phi_0$ equals to $\beta_1$,}
$$
which is the lower Matuszewska index of $\Phi_1$.

As will be proved in Lemma \ref{l:A3-c}, assumption \hyperlink{Q4-c}{{\bf (B4-c)}} implies 
\hyperlink{Q2-a}{{\bf (B2-a)}}, \hyperlink{Q2-b}{{\bf (B2-b)}}, \hyperlink{UBS2}{{\bf (UBS)}} and \hyperlink{Q4-a}{{\bf (B4-a)}}--\hyperlink{Q4-b}{{\bf (B4-b)}} (with function $\Phi_0$ defined in \eqref{e:phi0new-2}).
Hence, in the remainder of this section,  we assume that
$$
	\textit{$\sB$ satisfies \hyperlink{Q1}{{\bf (B1)}}, \hyperlink{Q3}{{\bf (B3)}}, \hyperlink{Q4-c}{{\bf (B4-c)}} and \hyperlink{Q5}{{\bf (B5)}} and $\kappa$ satisfies \hyperlink{W3}{{\bf (K3)}},}
$$
and that $\Phi_0$ will be the function defined in \eqref{e:phi0new-2} and so 
$$
	\beta_0=\beta_1.
$$
	
Under these assumptions we first prove Carleson's estimate, see Theorem \ref{t:Carleson}, which is used in the proof of our first main result -- the boundary Harnack principle.

Recall that 	the constant  $p \in [(\alpha-1)_+, \alpha+\beta_1) \cap (0,\infty)$ 
	denotes the constant satisfying \eqref{e:C(alpha,p,F)-2} if $C_9>0$ and  $p=\alpha-1$ if $C_9=0$ where $C_9$ is the constant in \hyperlink{W3}{{\bf (K3)}}. 
\begin{thm}\label{t:BHPnew-2}
 \!{\rm(Boundary Harnack principle)} 
 Suppose that $D$ is a $C^{1,1}$ open set 
 and that \hyperlink{Q1}{{\bf (B1)}}, \hyperlink{Q3}{{\bf (B3)}}, \hyperlink{Q4-c}{{\bf (B4-c)}},  \hyperlink{W3}{{\bf (K3)}} and  \hyperlink{Q5}{{\bf (B5)}} hold.
Suppose also that $p<\alpha+(\beta_1\wedge \beta_2)$.	
Then  for any $Q \in \partial D$, $0<r \le \wh R$, and any non-negative Borel function $f$ in $D$ which is harmonic in $D\cap B(Q,r)$ with respect to $Y^{\kappa}$ and vanishes continuously on 
	$\partial D\cap B(Q,r)$, 	we have 
	\begin{equation}\label{e:TAMSe1.8new-2}
		\frac{f(x)}{\delta_D(x)^p}\asymp 	\frac{f(y)}{\delta_D(y)^p} \quad \text{for} \;\, x,y\in D\cap B(Q,r/2),
	\end{equation}
where the comparison constants are independent of $Q,r$ and $f$.
\end{thm}

Proof of Theorem \ref{t:BHPnew-2} uses the Harnack inequality, Carleson's estimate, some exit time estimates, Theorem \ref{t:Dynkin-improve} on the decay rate of some special harmonic functions, upper estimates of killed potentials from Proposition \ref{p:bound-for-integral-new}, and some delicate estimates of the jump kernel obtained in Lemma \ref{l:estimates-of-J-for-BHP}.

Under the setting of Theorem \ref{t:BHPnew-2}, there exists $C>0$ such that the following holds: For any $Q \in \partial D$ and $0<r \le \wh R$, if two Borel functions $f, g$ in $D$ are harmonic in $B_D(Q,r)$ with respect to $Y^\kappa$ and vanish continuously on $\partial D \cap B(Q,r)$, then
\begin{align}\label{e:def-BHP-2}
	\frac{f(x)}{f(y)}\,\le C\,\frac{g(x)}{g(y)}\quad \text{for all} \;\, x,y\in B_D(Q,r/2).
\end{align}
The inequality \eqref{e:def-BHP-2} is referred to as the \textit{scale-invariant boundary Harnack principle} for $Y^\kappa$.

We say that the  \textit{inhomogeneous non-scale-invariant boundary Harnack principle}  holds for $Y^\kappa$, if there is a constant $r_0\in (0,\wh R]$  such that for any $Q \in \partial D$ and 
$0<r\le r_0$, there exists  a constant $C=C(Q,r)\ge 1$ such that \eqref{e:def-BHP-2} holds for any two Borel functions $f, g$ in $D$ which are harmonic in $B_D(Q,r)$ with respect to $Y$ and vanish continuously on $\partial D \cap B(Q,r)$.

Note that Theorem \ref{t:BHPnew-2} is stated for $p<\alpha+(\beta_1\wedge \beta_2)$ only.  
In particular, if $\beta_1\le \beta_2$, then BHP holds for all admissible values of the parameter $p$, while if $\beta_2 < \beta_1$, it holds when $p<\alpha+\beta_2$.
We will show that without this extra condition, even inhomogeneous non-scale-invariant BHP may not hold for $Y^\kappa$.  Consider the following condition:

\medskip

\noindent \hypertarget{F2}{{\bf (F)}} For any   $0<r\le \wh R$, there exists a constant $C=C(r)$ such that
\begin{align}\label{e:condition-F-2}
	\liminf_{s\to 0} \frac{\Phi_2(b/r) \ell(s/b)}{\ell(s)} \ge C b^{p-\alpha} \quad \text{for all} \;\, 0<b \le r.
\end{align}

\medskip

\begin{thm}\label{t:BHPfail-general2}
	 Suppose that $D$ is a $C^{1,1}$ open set 
	 and  that \hyperlink{Q1}{{\bf (B1)}}, \hyperlink{Q3}{{\bf (B3)}}, \hyperlink{Q4-c}{{\bf (B4-c)}},  \hyperlink{W3}{{\bf (K3)}} and \hyperlink{Q5}{{\bf (B5)}} hold.  Suppose also that   \hyperlink{F2}{{\bf (F)}} holds.
	Then the inhomogeneous non-scale-invariant boundary Harnack principle fails for $Y^\kappa$.
\end{thm}

We will see in Remark \ref{r:condition-F} that \hyperlink{F2}{{\bf (F)}} implies that $p\ge \alpha+\beta_2$. Conversely, \hyperlink{F2}{{\bf (F)}} holds if (i) $p>\alpha+\ub_2$, or (ii) $p=\alpha+\beta_2$,  $\ell$ is  slowly varying at zero, and there exists $c_0>0$ such that  $\Phi_2(r) \ge c_0 \Phi_2(1)r^{\beta_2}$ for all $0<r\le 1$, see Lemma \ref{l:condition-F}. These two sufficient conditions for \hyperlink{F2}{{\bf (F)}} together with Remark \ref{r:compare-KSV}   show that Theorems \ref{t:BHPnew-2} and \ref{t:BHPfail-general2} completely cover the boundary Harnack principle results of \cite{KSV20}.

Suppose that
 $\Phi_2$ is regularly varying with index $\beta_2$. If either
 (1) $p>\alpha+\beta_2$, or (2) $p=\alpha+\beta_2$  
and the right-hand side inequality of \eqref{e:scale2-2} holds with $\beta_2$, 
then \hyperlink{F2}{{\bf (F)}} holds true. Hence, in this case we can completely determine the region of 
the parameters $\beta_1$, $\beta_2$ and $p$ for which BHP holds true. 
If $\Phi_2$ is \emph{not} regularly varying at zero, i.e., $\beta_2<\beta_2^*$ (where $\beta_2^*$ denotes the upper Matuszewska index),  the oscillation of $\Phi_2$ near zero is an obstacle to completely determining when the BHP holds.

\medskip 
The second main result is about sharp Green function estimates. We first introduce a positive function $\Upsilon$ on $(0,\infty)$ by
\begin{equation}\label{e:def-of-Theta-2}
\Upsilon	(t):= \int_{t\, \wedge\, 1}^2 u^{2\alpha-2p-1}\Phi_1(u)\Phi_2(u)\, du.
\end{equation}

\begin{thm}\label{t:Green-2}
Suppose that $D$ is a bounded $C^{1,1}$ open set 
and that \hyperlink{Q1}{{\bf (B1)}}, \hyperlink{Q3}{{\bf (B3)}}, \hyperlink{Q4-c}{{\bf (B4-c)}},   \hyperlink{W3}{{\bf (K3)}} and  \hyperlink{Q5}{{\bf (B5)}} hold.
Let $p \in [(\alpha-1)_+, \alpha+\beta_1) \cap (0,\infty)$
denote the constant satisfying \eqref{e:C(alpha,p,F)-2} if $C_9>0$ and let $p=\alpha-1$ if $C_9=0$ where $C_9$ is the 
constant in \hyperlink{W3}{{\bf (K3)}}.  
Then for all $x,y \in D$,
\begin{align*}
		G^\kappa(x, y)& \asymp  \left(\frac{\delta_D(x)\wedge \delta_D(y)}{|x-y|} \wedge 1\right)^p 
		\left(\frac{\delta_D(x)\vee \delta_D(y)}{|x-y|} \wedge 1\right)^p \\
		&\qquad \times \Upsilon\left(\frac{\delta_D(x)\vee \delta_D(y)}{|x-y|}\right)\frac1{|x-y|^{d-\alpha}}.
\end{align*}
\end{thm}

Theorem \ref{t:Green-2} covers all admissible values of the parameters involved so clearly it 
 includes 
the region of the parameters where the boundary Harnack principle may fail.
We note that the sharp bounds above involve the function $\Upsilon$ defined through an integral. For certain regions of the involved parameters, the integral can be estimated, leading to the following corollary.

\begin{corollary}\label{c:Green-2}
Under the setting of Theorem \ref{t:Green}, the following statements hold true.
	
	\noindent (i)  
	Suppose that $p<\alpha + (\beta_1+\beta_2)/2$. 
	Then  for all $x,y \in D$,
	\begin{align*}
		&G^\kappa(x, y) \asymp
		 \left(\frac{\delta_D(x)\wedge \delta_D(y)}{|x-y|} \wedge 1\right)^p 
		 \left(\frac{\delta_D(x)\vee \delta_D(y)}{|x-y|} \wedge1\right)^p\frac1{|x-y|^{d-\alpha}}.
	\end{align*}
	\noindent (ii)  
	Suppose that $\alpha + (\ub_1+\ub_2)/2<p<\alpha+\beta_1$. 
	Then for all $x,y \in D$,
	\begin{align*}
		G^\kappa(x, y)& \asymp
		\left(\frac{\delta_D(x)\wedge \delta_D(y)}{|x-y|} \wedge 1\right)^p 
		\left(
		\frac{\delta_D(x)\vee \delta_D(y)}{|x-y|} \wedge1\right)^{2\alpha-p} \\
		&\qquad \times  \Phi_1\left(\frac{\delta_D(x)\vee \delta_D(y)}{|x-y|}\right) 
		\Phi_2\left(\frac{\delta_D(x)\vee \delta_D(y)}{|x-y|}\right)\frac1{|x-y|^{d-\alpha}}.
	\end{align*}
\end{corollary}

\medskip
Examples form an important part of this work. We have already explained in Subsection \ref{s-key-assumptions} that a subordinate killed stable process, i.e.~the process with generator  
$-((-\Delta)^{\gamma/2}|_D)^{\beta}$, is one natural example satisfying all of the introduced assumptions. An independent sum of such processes is another example. To be more precise, 
let $\alpha \in (0,2)$, $m \ge 2$ and $0<\gamma_1< \cdots< \gamma_m \le 2$. Set $\beta_i:=\alpha/\gamma_i$ for $1\le i\le m$. Consider a process  $\wt Y$  corresponding to the generator $L=\sum_{i=1}^m-((-\Delta)^{\gamma_i/2}|_D)^{\beta_i}$. We show in Example \ref{ex:sum-of-processes} that all assumptions are 
satisfied.
Example \ref{ex:subordinator} gives another modification in which the L\'evy measure of the subordinator behaves  near zero as that of the $\beta$-subordinator, but may decay at infinity at  a  faster rate than polynomial. This family of subordinators contains relativistic $\beta$-stable subordinators. By subordinating the killed $\gamma$-stable process, and letting $\beta \gamma=\alpha$, we again arrive at a process satisfying all the assumptions. 
Another example that we have discussed in Subsection \ref{s-key-assumptions} is the censored process.

In Subsection \ref{s-example-general} we describe a different family of examples  motivated by assumption \hyperlink{Q4-c}{{\bf (B4-c)}}. Suppose that the function $\sB$ is defined by
\begin{align}\label{e:B-expression-2}
	\sB(x,y)&=a(x,y) \Phi_1\bigg(\frac{\delta_D(x) \wedge \delta_D(y)}{|x-y|}\bigg)
	\Phi_2\bigg(\frac{\delta_D(x) \vee \delta_D(y)}{|x-y|}\bigg)\\
	&\qquad  \times
	\ell\bigg(\frac{\delta_D(x)\wedge \delta_D(y)}{(\delta_D(x)\vee\delta_D(y))\wedge |x-y|}\bigg),\nn
\end{align}
where $a:D\times D\to (0,\infty)$ satisfies certain assumptions. For the time being, it suffices to note that these assumptions hold true provided that $a$ is the restriction to $D\times D$ of a $C^{\theta_0'}(\overline D \times \overline D)$ symmetric function bounded above and below by positive constants, where $\theta_0'>(\alpha-1)_+$. Then $\sB$ satisfies \hyperlink{Q1}{{\bf (B1)}}, \hyperlink{Q3}{{\bf (B3)}}, \hyperlink{Q4-c}{{\bf (B4-c)}} and \hyperlink{Q5}{{\bf (B5)}}.


\section{On Lipschitz and $C^{1,1}$ open sets}\label{ch:geometry}

Throughout this work we will need various, mostly elementary, properties of Lipschitz and $C^{1,1}$ opens sets, and their subsets. In this section we collect these properties. The reader may wish to skip the details and come back to them later when needed. We begin with the definition.

\begin{defn}\label{df:lipschitz}
{\rm
	Let   $D\subset \R^d$ be an open set  and let $\wh R, \Lambda_0$ and $\Lambda$ be positive constants.
		
	\noindent (i) We say that $D$ is a Lipschitz open set  with \textit{localization radius}  $\wh R$
		 and \textit{Lipschitz constant} $\Lambda_0$,
		 if for each $Q \in \partial D$, there exist a Lipschitz function $\Psi=\Psi^Q:\R^{d-1}\to \R$ with $\Psi(\wt 0)=0$ 
		 and $|\Psi(\wt y)-\Psi(\wt z)| \le \Lambda_0|\wt y - \wt z|$ for all $\wt y, \wt z\in \R^{d-1}$, 
		 and an orthonormal coordinate system CS$_{Q}$ with origin at $Q$ such that
		\begin{align}\label{e:local-coordinate}
		 B_D(Q, \wh R) = \left\{ y= (\wt y, y_d) \in B(0, \wh R) \text{ in CS$_{Q}$} : y_d>\Psi(\wt y) \right\}.
		\end{align}
	When $D$ is additionally assumed to be connected, then $D$ is called a \textit{Lipschitz domain}.

	\noindent (ii) We say that $D$ is a $C^{1,1}$ open set with characteristics  $(\wh R, \Lambda)$, if
	for each $Q \in \partial D$, there exist a $C^{1,1}$ function $\Psi=\Psi^Q:\R^{d-1} \to \R$ with 
	\begin{align}\label{e:local-map}
		\Psi(\wt 0)= |\nabla \Psi(\wt 0)|=0 \quad \text{and} \quad |\nabla \Psi(\wt y)-\nabla\Psi(\wt z)| \le \Lambda |\wt y-\wt z| \text{ for all }  \wt y, \wt z \in \R^{d-1},
	\end{align}
	 and an orthonormal coordinate system CS$_{Q}$ with origin at $Q$ such that \eqref{e:local-coordinate} holds.
}
\end{defn}

\subsection{Lipschitz open sets}
In this subsection, we assume that $D \subset \R^d$ is a Lipschitz  open set with localization radius $\wh R$ and Lipschitz constant $\Lambda_0$.
It is known that $D$ satisfies \textit{the measure density condition}, that is, there exists  $C>0$  
depending only  on $d$ and $\Lambda_0$
  such that 
	\begin{align}\label{e:VD}		
	m_d(	B_{ D}(x_0,r)) \ge Cr^d \quad \text{for all} \;\, x_0 \in \overline D, \;\, 0<r\le \wh R.		
	\end{align} 
	
	For $Q \in \partial D$ and  $x=(\wt x, x_d)\in B_D(Q, \wh{R})$ in CS$_Q$, we define
	$$\rho_D(x)=\rho^Q_{D}(x):=x_d-\Psi(\wt{x}),$$
where $\Psi=\Psi^Q$ is the function in \eqref{e:local-coordinate}.	For   $a,b\in (0, \wh{R}/(2+\Lambda_0)]$ and $Q \in \partial D$, we let
	\begin{align}\label{e:def-UQ}
		U^Q(a,b):=
		\big\{x\in D: x=(\wt{x},x_d)\text{ in CS}_Q \mbox{ with } |\wt{x}|< a,\, 0<\rho_D(x)<b\big\},
	\end{align}
and refer to $U^Q(a,b)$ as the box based at $Q$ of width $a$ and height $b$, see Figure \ref{f-Urho2}.
For the half space, the box is simply 
\begin{align}\label{e:def-UH}
	U_{\bH}(a,b):=\big\{y=(\wt{y}, y_d)\in \bH: |\wt{y}|<a,\, 0< y_d<b\big\}.
\end{align}
We write  $U_\bH(a)$ for $U_\bH(a,a)$ and $U^Q(a)$ for $U^Q(a,a)$.
\begin{figure}[!h]
	\centering
	
	\includegraphics[width=0.52\textwidth]{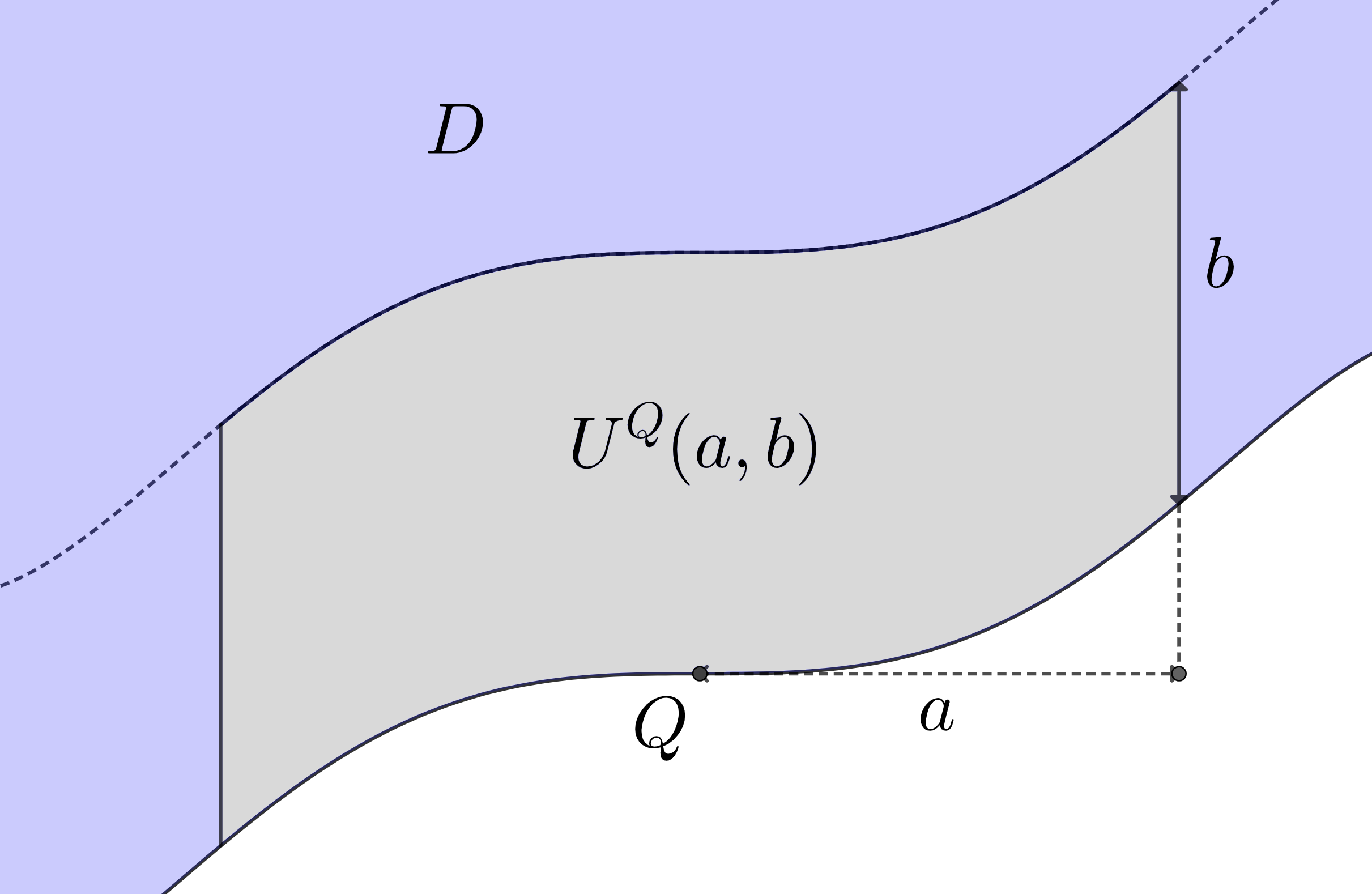}
	
	\vspace{-2mm}
	
	\caption{The set $U^Q(a,b)$}\label{f-Urho2}
\end{figure}

When we work with a fixed $Q\in \partial D$, we will
sometimes write $U(a, b)$ for $U^Q(a, b)$, and $U(a)$ for $U^Q(a)$.

\begin{lemma}\label{l:U-rho-Lipschitz}
	Let $Q \in\partial D$.  The following statements hold.
	
	\noindent (i) For any $0<r\le \wh R/(2+\Lambda_0)$, 
	$$	
	B_D(Q,(1+\Lambda_0)^{-1}r) \subset U(r) \subset B_D(Q,(\sqrt 2+\Lambda_0)r).
	$$
	
	\noindent	(ii) Set $\Lambda_1:=\Lambda_0 \vee (1/2)$.  For any $x \in U(\wh R/(6+4\Lambda_0))$, 
	$$
	(1+\Lambda_1^2)^{-1/2} \rho_D(x) \le \delta_D(x) \le \rho_D(x).
	$$
\end{lemma}
\begin{proof} (i) For any $x=(\wt x, x_d) \in B_D(Q,(1+\Lambda_0)^{-1}r)$ in CS$_Q$,  we have $|\wt x| \le |x|<r$ and $\rho_D(x) \le |x_d| + |\Psi (\wt x)| \le  |x_d| + \Lambda_0|\wt x| \le (1+\Lambda_0)|x|<r$.  Hence, $B_D(Q,(1+\Lambda_0)^{-1}r) \subset U(r)$. On the other hand, for any $x=(\wt x,\Psi(\wt x) + \rho_D(x)) \in U(r)$,  it holds that 
\begin{align*}
	|x|^2 \le |\wt x|^2 + ( \Lambda_0 |\wt x| + \rho_D(x) )^2< r^2 +  (1+\Lambda_0)^2 r^2< (\sqrt 2 + \Lambda_0)^2 r^2 .
\end{align*}
Thus $U(r) \subset B_D(Q,(\sqrt 2+\Lambda_0)r)$.

(ii) Let $x=(\wt x, \Psi(\wt x) + \rho_D(x)) \in U(\wh R/(6+4\Lambda_0))$  in CS$_Q$. It is clear that $\delta_D(x) \le \rho_D(x)$. 
To prove the other inequality, we  define
$$
	A=\{(\wt x + \wt z, z_d): \Lambda_1|\wt z|<z_d -\Psi(\wt x)<  \wh R/4\}.
$$
We claim that $A \subset D$. Indeed,  for any $(\wt x + \wt z, z_d) \in A$, since $|\wt z| \le \Lambda_1^{-1}\wh R/4 \le \wh R/2$ and $|\Psi(\wt x)| \le \Lambda_0 |\wt x|$, we have
\begin{align*}
	&|(\wt x + \wt z, z_d) | \le |\wt x| + |\wt z| + |z_d - \Psi(\wt x)| + |\Psi(\wt x)| \\
	&< ((6+4\Lambda_0)^{-1} +1/2 + 1/4 + \Lambda_0 (6+4\Lambda_0)^{-1})\wh R  < \wh R
\end{align*} 
and $\Psi(\wt x + \wt z) \le \Psi(\wt x) + \Lambda_0|\wt z| < z_d$. Therefore, $(\wt x + \wt z, z_d) \in D$ and the claim holds true. 
Using the claim, we obtain $\delta_D(x) \ge \delta_A(x) = (1+\Lambda_1^2)^{-1/2} \rho_D(x).$
\end{proof}

For $Q \in \partial D$ and  $r \in (0,\wh R/(6+3\Lambda_0)]$, we define $f^{(r)}=f^{(r)}_Q:U_\bH(3)\to U(3r)$ by
	\begin{equation}\label{e:def-fr}
		f^{(r)}(v)=f^{(r)}(\wt{v}, v_d):=(r\wt{v}, rv_d+\Psi(r\wt{v})).
	\end{equation}
Then $f^{(r)}$ is a diffeomorphism from $U_\bH(3)$ onto $U(3r)$ with Jacobian $|Df^{(r)}(v)|=r^d$ for every $v\in U_\bH(3)$. Moreover,  $\rho_{D}(f^{(r)}(v))=rv_d$ for all $v\in U_\bH(3)$.
	
	\begin{lemma}\label{l:diffeo}
		Let $Q \in \partial D$ and $0<r\le \wh R/(6+3\Lambda_0)$. For any  $v,w \in U_\bH(3)$, 
		\begin{equation}\label{e:diffeo}
			(1+\Lambda_0)^{-1} r |v-w|\le 
			|f^{(r)}(v)-f^{(r)}(w)|\le (1+\Lambda_0)r |v-w|.
		\end{equation}
	\end{lemma}
	\begin{proof}  Let $v,w \in U_\bH(3)$. Using the Lipschitz property of $\Psi$, we have
	\begin{align*}
		& |f^{(r)}(v)-f^{(r)}(w)| \le r|v-w|+|\Psi(r\wt{v})-\Psi(r\wt{w})|\\
		& \le r|v-w|+r\Lambda_0 |\wt v - \wt w| \le r(1+\Lambda_0)|v-w|,
	\end{align*}
	which proves the second inequality in \eqref{e:diffeo}. 
	For the first inequality in \eqref{e:diffeo}, we note that if $|\wt v- \wt w| \ge |v-w|/(1+\Lambda_0)$, then
	\begin{align*}
			& |f^{(r)}(v)-f^{(r)}(w)| \ge r|\wt v-\wt w| \ge r|v-w|/(1+\Lambda_0)
	\end{align*}
	and if $|\wt v- \wt w| < |v-w|/(1+\Lambda_0)$, then
	\begin{align*}
		& |f^{(r)}(v)-f^{(r)}(w)| \ge r|v-w|-|\Psi(r\wt{v})-\Psi(r\wt{w})|\\
	& \ge r|v-w|-r\Lambda_0 |\wt v - \wt w| \ge r|v-w|/(1+\Lambda_0).
	\end{align*}
	The proof is complete.	
\end{proof}
	
A connected open set $A \subset \R^d$ is called a \textit{$c_0$-John domain}, $c_0 \ge 1$, if any $x,y \in A$ can be joined by a rectifiable curve $g:[0,l] \to A$ parameterized by arc length such that $\delta_A(g(s)) \ge (s \wedge (l-s))/c_0$ for all $s \in [0,l]$.

	\begin{lemma}\label{l:John}
		For any $x_0 \in \overline D$ and $0<r \le (2+\Lambda_0)^{-2}\wh R/3$,  there is a $ 2(1+\Lambda_0)^{4}$-John domain $A$ such that	$B_{D}(x_0,r) \subset A \subset B_{ D}(x_0, 2(2+\Lambda_0)^2r)$.
	\end{lemma}
	\begin{proof}  
	When $\delta_D(x_0) \ge r$, one can simply take $A:= B(x_0,r)=B_{D}(x_0,r)$, which is  1-John domain. 
	
	Now, we assume  that $\delta_D(x_0)<r$. 
	 Let $Q_{x_0} \in \partial D$ be such that $\delta_D(x_0)=|x_0-Q_{x_0}|$. 
	Set  
	$$ r':=(1+\Lambda_0)r \in (0, \wh R/(6+3\Lambda_0)]\quad \text{and}\quad A:=U^{Q_{x_0}}(2r').$$ 
	Then  $A \subset B_D(Q_{x_0}, 2(\sqrt 2+\Lambda_0)r') \subset B_D(x_0, 2(2+\Lambda_0)^2r)$ and $A \supset B_D(Q_{x_0}, 2r) \supset B_D(x_0,r)$ by Lemma \ref{l:U-rho-Lipschitz}(i).

	 Let $x \in A$ and $v:=(f^{(r')})^{-1}(x) \in U_\bH(2)$, where $f^{(r')}$ is the function defined in \eqref{e:def-fr}. For every $u \in [0,1]$, by the convexity of $U_\bH(2)$,  we see that  $(1-u)v + u \e_d \in U_\bH(2)$.
	 Define a function $\wh g_x:[0,1] \to A$ by 
	 $$
	 \wh g_x(u) = f^{(r')} ( (1-u)v + u\e_d).
	 $$  
	 Let $g_x(s)=\wh g_x(h_x(s))$, $s \in [0,l_x]$, be the reparametrization of $\wh g_x$  by arc length. 
	 Note that $g_x(0)=\wh g_x(0)= x$ and $g_x(l_x)=\wh g_x(1)=\e_d$. Moreover, by Lemma \ref{l:diffeo},  
	 	$$ \sup_{ u \in [0,1]} |\nabla \wh g_x (u)| \le (1+\Lambda_0)|v-\e_d|\, r' \le 2(1+\Lambda_0)^2 r.$$ 
	 	Hence, we have
	\begin{align*}
		\inf_{s \in (0,l_x)}	h_x'(s) \ge \Big(\sup_{ u \in [0,1]}  |\nabla \wh g_x (u)|\Big)^{-1} 
		\ge 2^{-1}(1+\Lambda_0)^{-2}r^{-1}|v-v_0|^{-1},
	\end{align*}
which yields that  $h_x(s) \ge 2^{-1}(1+\Lambda_0)^{-2}r^{-1}s$ for all $s \in [0,l_x]$. 
Using this and Lemma \ref{l:diffeo}, we get that for all $s \in (0,l_x)$,
	\begin{align}\label{e:John-property}
		\delta_A(g_x(s))&= \inf\left\{|g_x(s)- f^{(r')}(w)|: w \in \partial U_\bH(2)\right\}\\
		& = \inf\left\{|f^{(r')} ( (1-h_x(s))v + h_x(s)\e_d )- f^{(r')}(w)|: w \in \partial U_\bH(2)\right\} \nn\\
		&\ge \frac{r}{(1+\Lambda_0)^{2}} \inf\big\{| (1-h_x(s))v + h_x(s)\e_d - w|: w \in \partial U_\bH(2)\big\} \nn\\
		&=\frac{r}{(1+\Lambda_0)^{2}}  \Big[ \big(2- (1-h_x(s))|\wt v| \big) \wedge \big(2 - (1-h_x(s))v_d - h_x(s) \big) \nn\\
		&\qquad \qquad \qquad \;\; \wedge \big((1-h_x(s))v_d + h_x(s)\big)\Big] \nn\\
			&\ge\frac{r}{(1+\Lambda_0)^{2}}  \big[  (2h_x(s))\wedge h_x(s)  \wedge h_x(s)\big]   \ge 2^{-1}(1+\Lambda_0)^{-4} s.\nn
	\end{align}

	Pick any $x,y \in A$. Define $g:[0, l_x +l_y] \to A$ by $g(s)=g_x(s)$ if $s \in [0,l_x]$ and $g(s)=g_y(l_x+l_y-s)$ if $s \in [l_x, l_x+l_y]$.  By \eqref{e:John-property}, we have
	$\delta_A(g(s)) \ge 2^{-1}(1+\Lambda_0)^{-4} (s \wedge (l_x+l_y-s))$ for all $ s \in [0,l_x+l_y]$. Thus, $A$ is a $2(1+\Lambda_0)^4$-John domain. 
	The proof is complete.
	 \end{proof}

	 \begin{lemma}\label{l:boundary-decomposition}
	 	There exists a family $\{A_i:i\ge 1\}$ of $2(1+\Lambda_0)^4$-John domains satisfying the following properties:
	 	\begin{gather}
	 		c_1 \wh R^d \le m_d(A_i) \le c_2 \wh R^d \quad \text{for all $i \ge 1$},\label{e:boundary-decomposition1}\\
	 		 \{x \in D: \delta_D(x)<(2+\Lambda_0)^{-2}\wh R/18\} \subset \cup_{i\ge 1} A_i \subset 	\{x \in D: \delta_D(x)<2\wh R/3\},\label{e:boundary-decomposition2}\\
	 		\sum_{i\ge1}\1_{A_i}\le c_3 \quad \text{on $D$},\label{e:boundary-decomposition3}
	 	\end{gather}where  $c_1,c_2,c_3>0$ are constants depending only on $d$ and $\Lambda_0$.
	 \end{lemma}
	 \begin{proof}
	 	Let $r_0:= (2+\Lambda_0)^{-2}\wh R/18$.	By the Vitali covering lemma, there exists a family of disjoint open balls  $\{B(Q_i,r_0):i \ge 1\}$ with $Q_i \in \partial D$ for all $i \ge 1$ such that  $\partial D \subset \cup_{i\ge 1} B(Q_i,5r_0)$. For each $i \ge 1$, by Lemma \ref{l:John}, there exists a $ 2(1+\Lambda_0)^{4}$-John domain $A_i$ such that	$B_{D}(Q_i,6r_0) \subset A_i \subset B_{ D}(Q_i, 12(2+\Lambda_0)^2r_0)$.

	\eqref{e:boundary-decomposition1}  follows from \eqref{e:VD}. 
	We have $\cup_{i\ge 1} A_i \subset 	\{x \in D: \delta_D(x)<12(2+\Lambda_0)^2r_0\}$.
	Let $x \in D$ be such that  $\delta_D(x)<r_0$ and $Q_x \in \partial D$ be such that $|x-Q_x| = \delta_D(x)$. Since $\partial D \subset \cup_{i\ge 1} B(Q_i,5r_0)$, $Q_x \in B(Q_i,5r_0)$ for some $i \ge 1$. Then  $x \in B_D(Q_i,6r_0) \subset A_i$ so that  \eqref{e:boundary-decomposition2} holds. For \eqref{e:boundary-decomposition3}, suppose that $y \in D$ is in $N$ of the sets $A_i$, $i \ge1$. Then $y$ is in at least $N$ of the sets $B_{ D}(Q_i, 12(2+\Lambda_0)^2r_0)$. Consequently, $B(y,  (12(2+\Lambda_0)^2 + 1) r_0)$ contains at least  $N$ of the sets $B_{D}(Q_i,r_0)$. Since $B_{D}(Q_i,r_0)$, $i \ge 1$, are disjoint, using \eqref{e:VD}, we get that 
	 	\begin{align*}
	 		c_3 N r_0^d \le  	\sum_{i: y \in A_i}m_d(B_{D}(Q_i,r_0)) \le m_d (B(y,  (12(2+\Lambda_0)^2 + 1) r_0)) \le c_4 r_0^d.
	 	\end{align*}
	 	Hence  $N \le c_4/c_3$, proving that \eqref{e:boundary-decomposition3} holds.
	 \end{proof}

	 \begin{lemma}\label{l:interior-decomposition}
	 	There exists a family $\{B_i:i\ge 1\}$ of open balls of radius $ (2+\Lambda_0)^{-2}\wh R/36$ satisfying the following properties:
	 	\begin{gather}
	 		\{x \in D: \delta_D(x)\ge (2+\Lambda_0)^{-2}\wh R/36\} \subset \cup_{i\ge 1} B_i \subset D,\label{e:interior-decomposition1}\\
	 		\sum_{i\ge1}\1_{B_i}\le c_1\quad \text{on $D$},\label{e:interior-decomposition2}
	 	\end{gather} where  $c_1>0$ is a constant depending only on $d$ and $\Lambda_0$.
	 \end{lemma}
	 \begin{proof}
	 	Let $r_0:= (2+\Lambda_0)^{-2}\wh R/18$ and $D_0:=	\{x \in D: \delta_D(x)\ge r_0/2\}$.	By the Vitali covering lemma, there exists a family of disjoint open balls  $\{B(x_i,r_0/10):i\ge 1\}$ with $x_i \in D_0$ for all $i \ge 1$ such that  $D_0 \subset \cup_{i\ge 1} B(x_i,r_0/2)$.  Let $B_i:=B(x_i,r_0/2)$ for $i\ge 1$. Then  \eqref{e:interior-decomposition1} holds. Moreover,  by repeating the argument for \eqref{e:boundary-decomposition3} in the proof of Lemma \ref{l:boundary-decomposition}, we deduce that \eqref{e:interior-decomposition2} holds. 
	 \end{proof}

\subsection{$C^{1,1}$ open sets}

In this subsection, we assume that $D \subset \R^d$ is a $C^{1,1}$ open set with characteristics $(\wh R,\Lambda)$ 
such that $\wh R\le 1 \wedge (1/(2\Lambda))$. 
	See Definition \ref{df-c11}.
Note that
\begin{align}\label{e:Lipschitz-constant}
\text{the Lipschitz constant $\Lambda_0$ of $\partial D$ is at most }	\Lambda \wh R\le 1/2.
\end{align}
 
It follows from  Lemma \ref{l:U-rho-Lipschitz} that for any $Q \in \partial D$ and $0<r\le \wh R/8$,
\begin{align}\label{e:U-rho-C11-1}
	B_D(Q,2r/3) \subset U^Q(r) \subset B_D(Q,2r)
\end{align}
and
\begin{align}\label{e:U-rho-C11-2}
	(2/\sqrt 5) \rho_D(x) \le \delta_D(x) \le \rho_D(x) \quad \text{for all} \;\, x \in U^Q(\wh R/8).
\end{align}

Let  $Q \in \partial D$. Let   $\Psi=\Psi^Q:\R^{d-1} \to \R$ be a $C^{1,1}$ function and CS$_Q$ be an orthonormal coordinate system with origin at $Q$ such that \eqref{e:local-coordinate} and  \eqref{e:local-map} hold. 
Note that $|\nabla \Psi(\wt y)| \le \Lambda   |\wt y| \le |\wt y|/\wh R$ for any $\wt y \in \R^{d-1}$. Hence,   we have
\begin{align}\label{e:Psi-bound}
	|\Psi (\wt y)| \le |\wt y| \sup_{|\wt z| \le |\wt y|} | \nabla \Psi(\wt z)| 
	\le  \wh R^{-1}|\wt y|^2 \quad \text{for all} \;\, \wt y \in \R^{d-1}.
\end{align}
Let $\nu\in (0,1]$. We define for $r \in (0, \wh R/4]$,
\begin{equation}\label{e:def-Enu}
	\begin{split} 
	E^Q_\nu(r)
	&:= \left\{ y=(\wt y, y_d) \text{ in CS$_{Q}$}:  
	|\wt y|<r/4, \,    4r^{-\nu}|\wt y|^{1+\nu} <y_d < r/2 \right\},\\	
	\wt E^Q_\nu(r)
	&:= \left\{ y=(\wt y, y_d) \text{ in CS$_{Q}$}: |\wt y|<r/4, \,  4r^{-\nu}|\wt y|^{1+\nu} <-y_d < r/2 \right\},\\
	S^Q(r)
	&:=\big\{y = (\wt y, y_d) \text{ \rm  in CS$_{Q}$}: 
	|(\wt y,y_d) - r\e_d |<r  \big\},\\
	\wt S^Q(r)
	&:=\big\{y = (\wt y, y_d) \text{ \rm  in CS$_{Q}$}: |(\wt y,y_d) +r \e_d |<r  \big\},
	\end{split}
\end{equation}
see Figure \ref{f-Enu-3}.
For any $0<\nu\le \nu'\le 1$, $r \in (0,\wh R/4]$ and $\wt y \in \R^{d-1}$ with $|\wt y|<r/4$, by \eqref{e:Psi-bound}, 
$$
4r^{-\nu}|\wt y|^{1+\nu} \ge 4r^{-\nu'}|\wt y|^{1+\nu'} \ge 4r^{-1}|\wt y|^{2} \ge |\Psi(\wt y)|.
$$
Hence, we have	
\begin{align}\label{e:E-nu-inclusion}
	E_\nu^{Q}(r) \subset 	E_{\nu'}^{Q}(r) \subset 	E_{1}^{Q}(r) \subset D.
\end{align}
 When we work with a fixed $Q\in \partial D$,   
 we write $E_\nu(r)$,  $\wt E_\nu(r)$, $S(r)$ and $\wt S(r)$  
 instead of $E^Q_\nu(r)$, $\wt E^Q_\nu(r)$, $S^Q(r)$ and $\wt S^Q(r)$ respectively. 
\begin{figure}[!h]
	\begin{tabular}{cc}
	\includegraphics[width=0.45\textwidth]{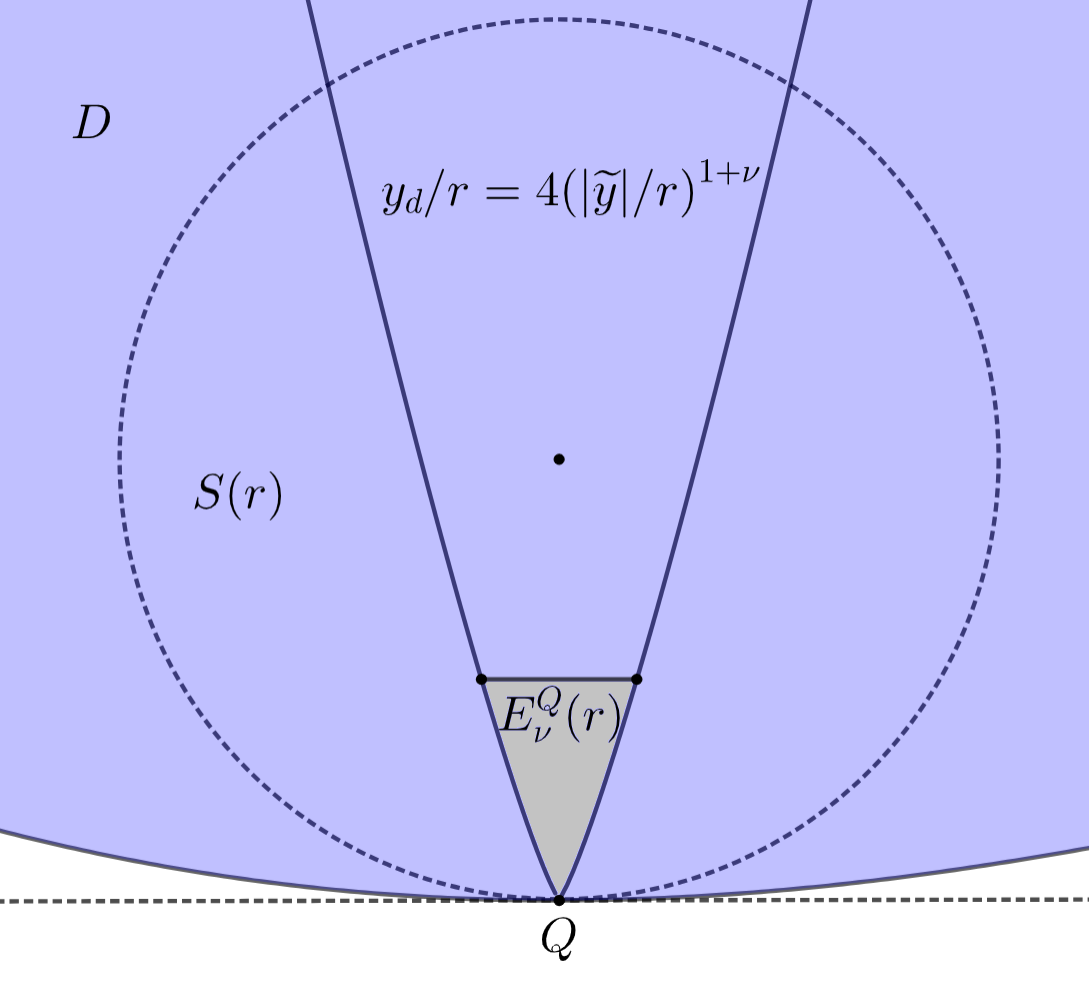}\;\;&\includegraphics[width=0.45\textwidth]{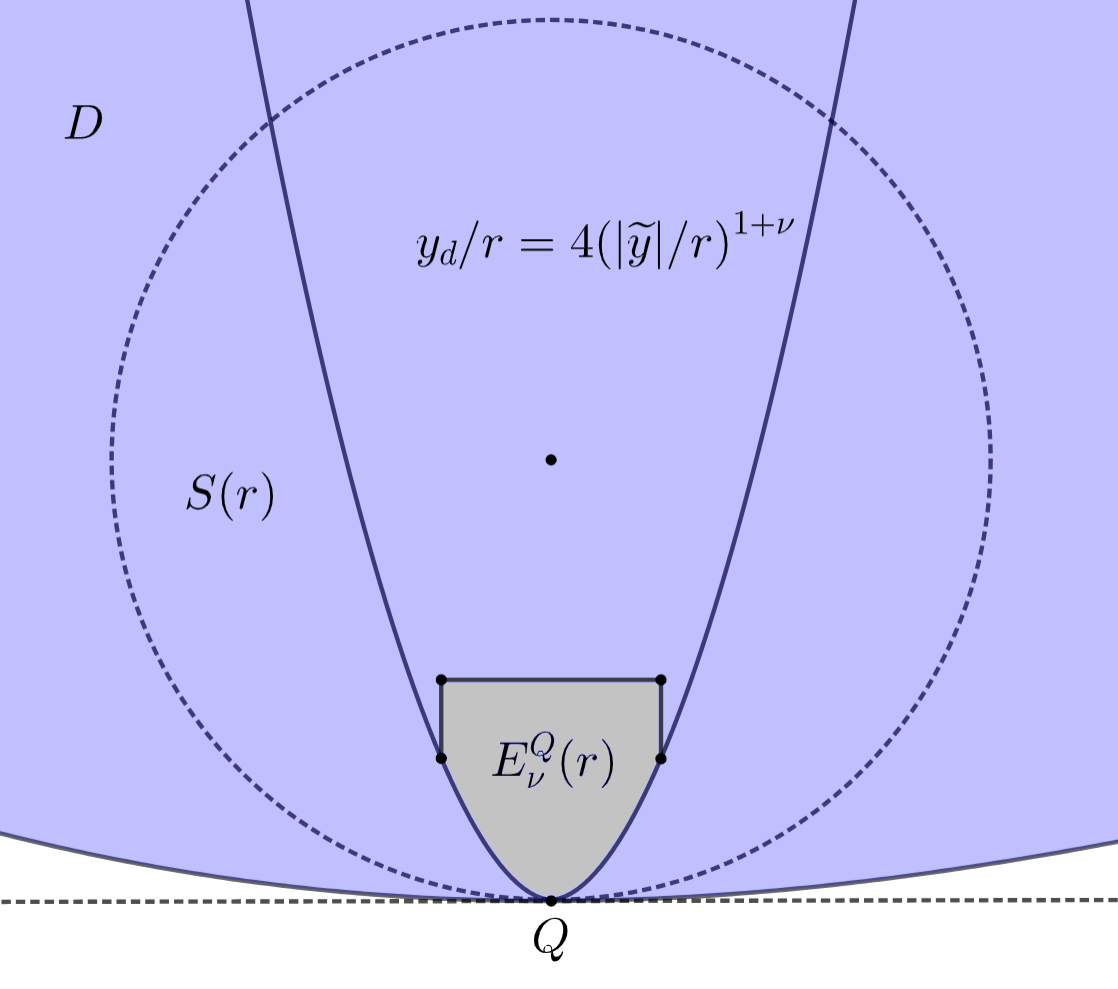}
	\end{tabular}
		
	\caption{The set $E_\nu^Q(r)$. Left $\nu=0.2$; Right $\nu=0.8$.}\label{f-Enu-3}
	
	\vspace{-1mm}
\end{figure}

\begin{lemma}\label{l:C11}
	Let $Q \in \partial D$, $\nu \in (0,1]$ and $r \in (0, \wh R/4]$.   The following statements hold in the coordinate system CS$_Q$.
	
	\noindent (i) For any $y=(\wt y, y_d)\in B_D(Q, \wh R)$, 
	we have $\delta_D(y) \le y_d+\wh R^{-1}|\wt y|^2$.
	
	\noindent (ii)  $E_\nu(r) \subset S(r) \subset D$ and $\wt E_\nu(r) \subset \wt S(r) \subset \R^d \setminus D$.
	
	\noindent (iii) For any  $y=(\wt y, y_d) \in E_\nu(r)$, we have 
	$$
	3y_d/4 \le y_d-r^{-1}|\wt y|^2\le \delta_{S(r)}(y) \le 	\delta_D(y)\le  2y_d.
	$$
	In particular, 
	$$
	\delta_D(y) \asymp \delta_{S(r)}(y) \asymp y_d \quad \text{for} \;\; y \in E_\nu(r).
	$$
\end{lemma}
\begin{proof} (i) Since $\delta_D(y)\le y_d + |\Psi(\wt y)|$ for all $y=(\wt y,y_d) \in B_D(Q,\wh R)$, we get the result from \eqref{e:Psi-bound}.
	
	(ii) For any $y \in E_\nu(r)$, since $2r^{-1}(r-y_d)\ge 1$, we have
	\begin{equation*}
	\begin{split}
		&\delta_{S(r)}(y) = r- \sqrt{|\wt y|^2 + (r-y_d)^2} \ge r- \sqrt{(r-y_d+r^{-1}|\wt y|^2)^2}\\ 
		& = y_d - r^{-1}|\wt y|^2\ge y_d - r^{-\nu}|\wt y|^{1+\nu} \ge 3y_d/4 >0.
	\end{split}
	\end{equation*}
	Hence, $E_\nu(r) \subset S(r)$. 
	Besides, by \eqref{e:Psi-bound}, we see  that for any $\wt y \in \R^{d-1}$ with $|\wt y| <r$, 
	\begin{align*}
		r- \sqrt{r^2-|\wt y|^2} \ge  r- \sqrt{(r-r^{-1} |\wt y|^2/2)^2} =r^{-1}  |\wt y|^2 /2\ge 2\wh R^{-1} |\wt y|^2 
		\ge  |\Psi(\wt y)|.
	\end{align*}
	Hence,  $E_\nu(r) \subset S(r)\subset D$.  Since $\R^d \setminus D$ is also a $C^{1,1}$ open set with characteristics 
	$(\wh R, \Lambda)$, we also get that $\wt E_\nu(r) \subset \wt S(r)\subset \R^d \setminus D$.
	
	(iii) For $y \in E_\nu(r)$, we have $\wh R^{-1}|\wt y|^2 \le 4r^{-1} |\wt y|^2 \le 4r^{-\nu}|\wt y|^{1+\nu}<y_d$.
	Now we get the result  from (i) and (ii).\end{proof}


\section{Properties of processes $\overline{Y}$ and $Y^{\kappa}$}\label{ch:processes}

The following are our standing assumptions 
on $\sB(x,y)$ in Sections \ref{ch:processes} through \ref{ch:green} of this work:

\medskip

\noindent \hypertarget{B1}{{\bf (B1)}} $\sB(x,y)=\sB(y,x)$ for all $x,y \in D$.

\smallskip

\noindent \hypertarget{B2-a}{{\bf (B2-a)}}  There exists a constant $C_1>0$ such that $\sB(x,y) \le C_1$ for all $x,y \in D$.

\smallskip

\noindent \hypertarget{B2-b}{{\bf (B2-b)}} For any $a\in (0,1]$, there exists a constant $C_2=C_2(a)>0$ such that
	\begin{align*}	\sB(x,y) \ge C_2 \quad \text{for all }  x,y \in D \text{ with } \delta_D(x) \wedge \delta_D(y) \ge a |x-y|.
	\end{align*}

\medskip

In this section  we assume that  $D \subset \R^d$ is a Lipschitz  open set with localization radius $\wh R$ and Lipschitz constant $\Lambda_0$ and we study the  processes $\overline{Y}$ and $Y^{\kappa}$ in  $D$.

 For the process $Y^{\kappa}$, we introduce the  conditions \hyperlink{K1}{{\bf (K1)}} and \hyperlink{K2}{{\bf (K2)}} 
 on the killing potential $\kappa$
 and work under these conditions. 
 The main goal is to establish 
 the parabolic H\"older regularity and parabolic Harnack inequality for these processes, 
and interior estimates  of the Green function of $Y^\kappa$.

\subsection{Analysis and properties of $\overline Y$}\label{s-processes-Y-bar}

Recall from Section \ref{ch:set-up} that $\overline{Y}$ is a Hunt process in $\overline{D}$ associated with the regular Dirichlet form $(\EE^0, \overline{\FF})$ and 
the exceptional 
set $\sN'$.  
Since the jump kernel $\sB(x,y)|x-y|^{-d-\alpha}dxdy$ of $(\sE^0, \overline \sF)$ is absolutely continuous with respect to $m_d \otimes m_d$, by using \cite[(5.3.15)]{FOT} and repeating the arguments in \cite[p.~40]{CK03}, one sees that  $\overline Y$ satisfies the following L\'evy system formula: For any $x \in  
\overline D
$, any non-negative Borel function $f$ on 
$\overline D \times \overline D$
vanishing on the diagonal, and any stopping time $\tau$,
\begin{align}\label{e:Levysystem-Y-bar}
	\E_x\bigg[  \sum_{s \le \tau} f( \overline Y_{s-},  \overline Y_s)  \bigg] 
	= \E_x \left[ \int_0^\tau   \int_{D} \frac{f( \overline Y_s, y) 
		\sB( \overline Y_s, y)}{| \overline Y_s-y|^{d+\alpha}} dy  \, ds \right].
\end{align}

For each $\rho>0$, define a bilinear form $(\sE^{0,(\rho)}, \overline\sF)$ by
\begin{align}\label{e:def-Erho}
	\sE^{0,(\rho)}(u,v)= \frac{1}{2}\iint_{D \times D, \, |x-y|<\rho}  (u(x)-u(y))(v(x)-v(y)) \frac{\sB(x,y)}{|x-y|^{d+\alpha}}dxdy.
\end{align}
By \hyperlink{B2-a}{{\bf (B2-a)}},    for all $\rho>0$  and $u \in \overline \sF$, we have
\begin{align}\label{e:E-and-Erho} 
	\begin{split} 
		\sE^{0}(u,u)-	\sE^{0,(\rho)}(u,u)	& \le  C_1 \int_D u(x)^2 \int_{D, \, |x-y|\ge \rho} \frac{dy}{|x-y|^{d+\alpha}}  dx  \le 	 \frac{c_1}{\rho^\alpha}\lVert u \rVert_{L^2(D)}^2.
	\end{split}
\end{align}
In particular, we have,
\begin{align*}
	\sE_1^{0,(\rho)}(u,u)\le \sE_1^{0}(u,u)\le 		(1 + c_1\rho^{-\alpha})	\sE_1^{0,(\rho)}(u,u),\end{align*}implying that $\sE^0$ and $\sE^{0,(\rho)}$ have same sets of capacity zero, and therefore, by \cite[Theorem 4.2.1(ii)] {FOT}, same exceptional sets.

For a Borel set $A \subset \R^d$ with $m_d(A) \in (0,\infty)$ and $u \in L^1(A)$, we let 
$$
\overline u_A:=\frac{1}{m_d(A)} \int_A u\,dx.
$$

In the following two propositions we establish a Nash-type inequality and, consequently, the existence and a preliminary upper bound of the transition densities of $\overline{Y}$ (or the heat kernel of the corresponding semigroup).

\begin{prop}\label{p:Nash}
	There exists $C>0$ depending only on $d,\alpha,\wh R$ and $\Lambda_0$ such that
	\begin{equation}\label{e:Nash}
		\lVert u \rVert_{L^{2}(D)}^{2+2\alpha/d} 	\le C   	\sE_1^0 (u,u) \quad \text{for all $u \in \overline \sF$ with $\lVert u \rVert_{L^{1}(D)}=1$}.
	\end{equation}
\end{prop}
\begin{proof}
	By Lemma \ref{l:boundary-decomposition}, there exists a family $\{A_i:i\ge 1\}$ of $2(1+\Lambda_0)^4$-John domains satisfying 
	\eqref{e:boundary-decomposition1}--\eqref{e:boundary-decomposition3}, and by Lemma \ref{l:interior-decomposition}, there exists a family $\{B_i:i\ge 1\}$ of open balls of radius  $(2+\Lambda_0)^{-2}\wh R/36$ satisfying \eqref{e:interior-decomposition1} and \eqref{e:interior-decomposition2}. Write $\{D_i:i\ge 1\}:=\{A_i:i\ge 1\} \cup \{B_i:i\ge 1\}$. Then $\{D_i:i\ge 1\}$ is an open covering of $D$, and by \eqref{e:boundary-decomposition3} and \eqref{e:interior-decomposition2},
	\begin{align}\label{e:Sobolev-overlap}
		\sum_{i\ge 1} \1_{D_i}\le c_1 \quad \text{on $D$}.
	\end{align} 
	Moreover, since every open ball in $\R^d$ is a $1$-John domain, $D_i$ are $2(1+\Lambda_0)^4$-John domains.
	
	Let  $u \in \overline \sF$ be such that  $\lVert u \rVert_{L^{1}(D)}=1$.
	By \eqref{e:boundary-decomposition2}, for all $i \ge 1$ and $z \in D_i$,  we have $\delta_{D_i}(z)/2 < \wh R$. Hence, by using
	\hyperlink{B2-b}{{\bf (B2-b)}} and \eqref{e:Sobolev-overlap}, we see that
	\begin{align}\label{e:interior-1}
		\begin{split}
			\sE^{0,(\wh R)}(u,u)&\ge \frac{C_2}{2} \int_{D}\int_{B(z,   (\delta_D(z)/2) \wedge \wh R)}\frac{(u(z)-u(y))^2}{|z-y|^{d+\alpha}} dydz\\
			&\ge c_2 \sum_{i=1}^\infty \int_{D_i}\int_{B(z,   (\delta_D(z)/2) \wedge \wh R)}\frac{(u(z)-u(y))^2}{|z-y|^{d+\alpha}} dydz\\
			&\ge c_2 \sum_{i=1}^\infty \int_{D_i}\int_{B(z, \delta_{D_i}(z)/2)}\frac{(u(z)-u(y))^2}{|z-y|^{d+\alpha}} dydz.
		\end{split}
	\end{align}
	Observe that 
	\begin{equation}\label{e:Sobolev-1}
		\begin{split}
			\lVert u \rVert_{L^{2}(D)}^{2}  &\le 
	\sum_{i=1}^\infty	\lVert u \rVert_{L^{2}(D_i)}^{2}
			\le 2\sum_{i=1}^\infty (\overline u_{D_i})^2 m_d(D_i)+ 2\sum_{i=1}^\infty \lVert u - \overline u_{D_i} \rVert_{L^2(D_i)}^2.
		\end{split}
	\end{equation}
	By  \eqref{e:Sobolev-overlap},	for all $i \ge 1$,
	\begin{align}\label{e:Sobolev-overlap-L1}
	\lVert u \rVert_{L^1(D_i)}\le \sum_{j=1}^\infty 	\lVert u \rVert_{L^1(D_j)} \le c_1 \lVert u \rVert_{L^1(D)}= c_1.
	\end{align}
	Using H\"older's  inequality in the second line below, and  \eqref{e:Sobolev-overlap-L1} in the third, we get
	\begin{equation}\label{e:Sobolev-first}		\begin{split}
				&\sum_{i=1}^\infty (\overline u_{D_i})^2 m_d(D_i)  \le  \sum_{i=1}^\infty m_d(D_i)^{-1}\lVert u \rVert_{L^1(D_i)}^2 \\		
				&  \le \left[\sum_{i=1}^\infty m_d(D_i)^{-(d+\alpha)/d} \lVert u \rVert_{L^1(D_i)}^{(2d+\alpha)/d}   \right] ^{d/(d+\alpha)}\left[	\sum_{i=1}^\infty \lVert u \rVert_{L^1(D_i)} \right] ^{\alpha/(d+\alpha)}\\		
					&  \le c_1 \left[	\sum_{i=1}^\infty m_d(D_i)^{-(d+\alpha)/d} \lVert u \rVert_{L^1(D_i)}^{2}   \right] ^{d/(d+\alpha)}\\&  \le c_1\left[\sum_{i=1}^\infty m_d(D_i)^{-\alpha/d} \lVert u \rVert_{L^2(D_i)}^2 \right] ^{d/(d+\alpha)}.		\end{split}	
	\end{equation}
	By \eqref{e:boundary-decomposition1} and since $B_i$ are open balls of radius $(2+\Lambda_0)^{-2}\wh R/36$, we have  $m_d(D_i) \ge c_3 \wh R^d$ for all $i \ge 1$. Hence,  it follows from \eqref{e:Sobolev-first} that
	\begin{equation}\label{e:Sobolev-2}
		\begin{split} 
			&\sum_{i=1}^\infty (\overline u_{D_i})^2 m_d(D_i) \le \frac{c_1}{(c_3\wh R^d)^{\alpha/(d+\alpha)} }\left[ \sum_{i=1}^\infty \lVert u \rVert_{L^2(D_i)}^2  \right] ^{d/(d+\alpha)}\le  c_4 \lVert u \rVert_{L^2(D)}^{2d/(d+\alpha)},
		\end{split} 
	\end{equation}
	where we used  \eqref{e:Sobolev-overlap} in the second inequality above.
	Since $D_i$ are $2(1+\Lambda_0)^4$-John domains, by  \cite[Theorem 3.1]{HV2}, there exists $c_5>0$ such that for all $i \ge 1$,
	\begin{align}\label{e:HV2thm3.1}
	 \lVert u - \overline u_{D_i} \rVert_{L^{2d/(d-\alpha)}(D_i)}^2 \le c_5 \int_{D_i} \int_{B(z,  \delta_{D_i}(z)/2)}\frac{(u(z)-u(y))^2}{|z-y|^{d+\alpha}} dydz.
	\end{align}
	Using  H\"older's inequality in the first and the third inequalities below, and \eqref{e:HV2thm3.1} in the second, we obtain
		\begin{equation*}		\begin{split}			&\sum_{i=1}^\infty \lVert u - \overline u_{D_i} \rVert_{L^2(D_i)}^2\\& \le 	\sum_{i=1}^\infty \lVert u - \overline u_{D_i} \rVert_{L^1(D_i)}^{2\alpha/(d+\alpha)}   	 \lVert u - \overline u_{D_i} \rVert_{L^{2d/(d-\alpha)}(D_i)}^{2d/(d+\alpha)} \\		&\le 	c_6\sum_{i=1}^\infty \left( 2	\lVert u \rVert_{L^1(D_i)}	\right)^{2\alpha/(d+\alpha)} \bigg( \int_{D_i} \int_{B(z,  \delta_{D_i}(z)/2)}\frac{(u(z)-u(y))^2}{|z-y|^{d+\alpha}} dydz\bigg)^{d/(d+\alpha)}   \\	
					&\le c_6 \left[\sum_{i=1}^\infty   \left( 2		\lVert u \rVert_{L^1(D_i)} \right)^{2} \right]^{\alpha/(d+\alpha)} \left[\sum_{i=1}^\infty\int_{D_i} \int_{B(z,  \delta_{D_i}(z)/2)}\frac{(u(z)-u(y))^2}{|z-y|^{d+\alpha}} dydz \right]^{d/(d+\alpha)}.	\end{split}	
		\end{equation*}
By \eqref{e:Sobolev-overlap-L1},
	\begin{equation*}
		\begin{split}
			\sum_{i=1}^\infty \left( 2		\lVert u \rVert_{L^1(D_i)} \right)^{2}\le 4c_1 \sum_{i\ge 1} 	\lVert u \rVert_{L^1(D_i)}  \le 4c_1^2.  
		\end{split}
	\end{equation*}
	Therefore,  it holds that 
	\begin{equation}\label{e:Sobolev-3}
		\begin{split} 
			\sum_{i=1}^\infty \lVert u - \overline u_{D_i} \rVert_{L^2(D_i)}^2 \le c_7 \left[\sum_{i=1}^\infty \int_{D_i} \int_{B(z,  \delta_{D_i}(z)/2)}\frac{(u(z)-u(y))^2}{|z-y|^{d+\alpha}} dydz \right]^{d/(d+\alpha)}.
		\end{split}
	\end{equation}
	Combining  \eqref{e:Sobolev-1}, \eqref{e:Sobolev-2}, \eqref{e:Sobolev-3} and  \eqref{e:interior-1}, and using \eqref{e:E-and-Erho}, we arrive at
	\begin{align*}
		\lVert u \rVert_{L^{2}(D)}^{2+2\alpha/d} & \le c_7 \sum_{i=1}^\infty \int_{D_i} \int_{B(z,  \delta_{D_i}(z)/2)}\frac{(u(z)-u(y))^2}{|z-y|^{d+\alpha}} dydz  +  c_7  \lVert u \rVert_{L^{2}(D)}^{2}  \\
		&\le c_8 \sE^{0,(\wh R)}(u,u)  +  c_7 \lVert u \rVert_{L^{2}(D)}^{2}\le c_8 \sE^{0}(u,u)  +  c_9(1+ \wh R^{-\alpha}) \lVert u \rVert_{L^{2}(D)}^{2}.
	\end{align*}
	The proof is complete.
\end{proof}

Denote by $(\overline P_t)_{t \ge 0}$ the semigroup  of $\overline Y$.

\begin{prop}\label{p:upper-heatkernel}
	The process $\overline Y$ has a transition density  $\overline p(t,x,y)$ defined on 
	$(0,\infty) \times (\overline D \setminus \sN) \times (\overline D \setminus \sN)$, where
	$\sN \supset \sN'$ 
	is a properly exceptional set for $\overline Y$.
	Moreover, for any $T>0$, there exists a constant $C=C(T)>0$ such that
	\begin{equation}\label{e:upper-heatkernel}
		\overline	p(t,x,y) \le C \left( t^{-d/\alpha} \wedge \frac{t}{|x-y|^{d+\alpha}}\right), 
		\quad 0<t\le T, \;\, x,y \in \overline D \setminus \sN.
	\end{equation}
\end{prop}
\begin{proof}  
	By Proposition \ref{p:Nash} and \cite[Theorem 2.1]{CKS87}, there exists  $c_1>0$ such that for any  $t>0$ and $f \in L^1(D)$,
	\begin{align}\label{e:Nash-CKS}
		\lVert \overline P_t f \rVert_{L^\infty(D)} \le c_1 t^{-d/\alpha} e^{t} 
		\lVert f \rVert_{L^1(D)}.
	\end{align}
	By \eqref{e:Nash-CKS} and	 \cite[Theorem 3.1]{BBCK},  
	one sees that $\overline Y$ has a transition density  $\overline p(t,x,y)$ on 
	$(0,\infty) \times (\overline D \setminus \sN'') \times (\overline D \setminus \sN'')$ for a properly exceptional set $\sN'' \supset \sN'$  and
	\begin{equation}\label{e:upper-heatkernel0}
 		\overline	p(t,x,y) \le  c_1 t^{-d/\alpha}e^{t},
 \quad  t>0, \; x,y \in 	\overline D \setminus \sN''.
	\end{equation} 
	Further, $\overline p(t,\cdot,y)$ and $\overline p(t,y,\cdot)$ are quasi-continuous in $\overline D$ for every $t>0$ and 
	$y\in \overline D \setminus \sN''$.

	To obtain the off-diagonal upper bounds for $\overline p(t,x,y)$, we follow the arguments given in \cite[Example 5.5]{CKKW}. Let  
	$$
	\delta=\frac{\alpha}{3(d+\alpha)}.
	$$ 
	By   Proposition \ref{p:Nash} and \eqref{e:E-and-Erho},  there exist $c_2,c_3>0$  such that for all $\rho\in \delta\mathbb{Q}_+$ and $u \in \overline \sF$ with $\lVert u \rVert_{L^1(D)}\le 1$,	
	\begin{equation}\label{e:p4.2-1}
		c_2	\lVert u \rVert_{L^{2}(D)}^{2+2\alpha/d} 
		\le  
		\sE^{0,(\rho)} (u,u) + (1 + c_3\rho^{-\alpha}) \lVert u \rVert_{L^2(D)}^2.
	\end{equation}
	Using the same argument as in \eqref{e:Nash-CKS} and \eqref{e:upper-heatkernel0},
	and \cite[Theorem 4.1.1]{FOT}, it follows from \eqref{e:p4.2-1} that   there exists a properly exceptional set $\sN_\rho$ (with respect to both $(\sE^0, \overline{\FF})$ and $(\sE^{0,(\rho)}, \overline{\FF})$)  contained in $\overline D$ such that  the Hunt process associated with $(\sE^{0,(\rho)}, \overline \sF)$ has a transition  density $\overline p^{(\rho)}(t,x,y)$ defined on 	 $(0,\infty) \times (\overline D \setminus \sN_\rho) \times (\overline D \setminus \sN_\rho)$ satisfying the following estimate:  There exists $c_4>0$ independent of  $\rho \in \delta\mathbb{Q}_+$ such that, for all $t>0$ and $x,y \in \overline D \setminus \sN_{\rho}$, 
	$$
	\overline p^{(\rho)}(t,x,y) \le c_4 t^{-d/\alpha}  \exp \left(  t + \frac{c_3t}{\rho^\alpha} \right).
	$$
	Let
	$$
	\sN:=\bigg(\bigcup_{\rho\in \delta\mathbb{Q}_+}\sN_{\rho}\bigg) \cup \sN''. 
	$$
	Then $\sN$ is a properly exceptional set.
	For $x_1,x_2\in D$ and $s>0$, define
	$$
	\psi^{x_1,x_2}_s(z):=\frac{s}{3}(|z-x_1|\wedge|x_1-x_2|), \quad\;\; z\in D
	$$
	and
	$$
	\Gamma_\rho[
	\psi ](z):=\frac12\int_{D, \,|z-y|<\rho}(e^{\psi(z)-\psi(y)}-1)^2\frac{\sB(z,y)}{|z-y|^{d+\alpha}}dy.
	$$
	Using \hyperlink{B2-a}{{\bf (B2-a)}} and repeating the elementary argument of \cite[p. 36]{CKKW}, we see that  for all  $\rho \in \delta\mathbb{Q}_+$,
	$x_1,x_2\in D$ and $s>0$, 
	\begin{align}\label{e:p4.2rs}
		H_\rho(\psi^{x_1,x_2}_s):=	\|\Gamma_\rho[\psi^{x_1,x_2}_s]\|_{L^\infty(D)}\vee \|\Gamma_\rho[-\psi^{x_1,x_2}_s]\|_{L^\infty(D)} \le  \frac{c_5e^{s\rho}}{\rho^\alpha}.
	\end{align}
	Hence, by \eqref{e:p4.2-1} and    \cite[Theorem 1.2]{CKKW},  there exists $c_6>0$ independent of $\rho$ such that, for all $t>0$ and $x,y \in \overline D \setminus \sN$, 
	\begin{align}\label{e:UHK-truncated}
		\begin{split}
			\overline p^{(\rho)}(t,x,y)& \le c_6t^{-d/\alpha}  \exp \left(  t + \frac{c_3t}{\rho^\alpha}  - \sup_{s>0} \bigg[  |\psi_s^{x,y}(y) -\psi_s^{x,y}(x) |+2t H_\rho(\psi^{x_1,x_2}_s) \bigg]  \right)\\
			& \le c_6 t^{-d/\alpha} \exp \left(t + \frac{c_3t}{\rho^\alpha}   - \sup_{s>0} \bigg[  \frac{s|x-y|}{3} - \frac{2c_5 te^{s\rho}}{\rho^\alpha} \bigg]  \right),
		\end{split}
	\end{align} 
	where   we used \eqref{e:p4.2rs}  in the second inequality above. 
	
	Let $t>0$
	and $x,y \in \overline D\setminus \sN$ with $|x-y|>2t^{1/\alpha}$.
	Let $q_{x,y}\in \mathbb{Q}_+$ such that $|x-y|/2\le q_{x,y}\le |x-y|$.
	By taking
	$$
	\rho = \delta q_{x,y}  \quad \text{ and } \quad s = \frac{1}{\delta  q_{x,y}} \log \bigg( \frac{q_{x,y}^\alpha}{t}\bigg),
	$$
	we get from \eqref{e:UHK-truncated} that
	\begin{align}\label{e:UHK-offdigonal}
		\begin{split}
			\overline p^{(\rho)}(t,x,y) &\le c_6 t^{-d/\alpha} \exp \left(  t + \frac{c_3t}{\delta^\alpha q_{x,y}^\alpha}    -  \frac{1}{3\delta}\frac{|x-y|}{q_{x,y}} \log \bigg( \frac{q_{x,y}^\alpha}{t}\bigg) +\frac{2c_5}{\delta^\alpha}  \right)\\
			&\le c_6 t^{-d/\alpha} \exp \left( t + \frac{c_3t}{\delta^\alpha q_{x,y}^\alpha}    -  \frac{1}{3\delta} \log \bigg( \frac{q_{x,y}^\alpha}{t}\bigg) +\frac{2c_5}{\delta^\alpha}  \right)\\
			&\le c_6e^{ t + (c_3+2c_5)/\delta^\alpha} t^{-d/\alpha} \bigg( \frac{t}{q_{x,y}^\alpha}\bigg)^{1/(3\delta)}= \frac{
				c_7e^{t}  t}{q_{x,y}^{d+\alpha}}
			\le  \frac{
				2^{d+\alpha}c_7e^{t}  t}{|x-y|^{d+\alpha}}.
		\end{split}
	\end{align}
	By \cite[Lemma 3.1(c)]{BGK09} and  the quasi-continuity of $\overline p(t,x, \cdot)$, using \eqref{e:UHK-offdigonal} 
	and \hyperlink{B2-a}{{\bf (B2-a)}},  we arrive at
	\begin{align*}
		\overline p(t,x,y) \le 	\overline p^{(\delta q_{x,y})}(t,x,y) + \sup_{z,w\in D:|z-w|>\delta q_{x,y}} \frac{t\sB(z,w)}{|z-w|^{d+\alpha}} \le  \frac{c_8	e^{t}  t}{|x-y|^{\alpha}}.
	\end{align*}
	Combining  this with \eqref{e:upper-heatkernel0}, we get the desired result.
\end{proof}

For an open set $U \subset \overline D$ relative to the  topology on $\overline D$, we let 
$$
\bar \tau_U:=\inf\{t>0:\overline Y_t \notin U\}.
$$ 
By a standard argument, since $(\sE^0,\overline \sF)$ is conservative, we get the following result from 
Proposition \ref{p:upper-heatkernel}, see, e.g., the proof of \cite[Lemma 2.7]{CKW21}.

\begin{lemma}\label{l:EP}
For any $T>0$,	there exists  $C=C(T)>0$ such that for all $x_0 \in \overline D \setminus \sN$,  $r>0$ and $0<t\le T$,
	\begin{align*}
		\P_{x_0}( \bar \tau_{B_{\overline D}(x_0,r)} \le t ) \le Ct r^{-\alpha}.
	\end{align*}
\end{lemma}

A consequence of this lemma is the following statement: 
For any $T>0$, there exists $c=c(T)>0$ such that for all 
$0<r\le (T/c)^{1/\alpha}$, $\P_{x_0}(\bar \tau_{B_{\overline D}(x_0,r)}\le c r^{\alpha})\le 1/2$.

In the next proposition, we obtain a local fractional Poincar\'e inequality for $\sE^0$. This inequality will be used to obtain a near diagonal lower estimate for Dirichlet heat kernels.

Recall that $\overline u_A:=\frac{1}{m_d(A)} \int_A u\,dx$.

\begin{prop}\label{p:PI}Set $k_0:=3(2+\Lambda_0)^2$.  There exists $C>0$ such that for all $x_0 \in \overline D$,  $0<r \le \wh R/k_0$ and any $u \in \overline \sF$,
	\begin{align*}
	&\int_{	B_{ D}(x_0,r)} (u(z)-\overline u_{	B_{ D}(x_0,r)})^2 dz\nn\\
	&\le	Cr^\alpha \iint_{	B_{ D}(x_0,k_0r) \times	B_{ D}(x_0,k_0r) } (u(z)-u(y))^2 \frac{\sB(z,y)}{|z-y|^{d+\alpha}}dzdy.
	\end{align*}
\end{prop}
\begin{proof} 
 Let $x_0 \in \overline D$ and $0<r \le \wh R/k_0$. We write $B:=B_D(x_0,r)$ and $B':=B_D(x_0,k_0r)$.	By Lemma \ref{l:John}, there is a $2(1+\Lambda_0)^4$-John domain $A$ such that $B \subset A \subset	B'$. Using \hyperlink{B2-b}{{\bf (B2-b)}} 
  in the second line,  \cite[Theorem 3.1]{HV2} in the third, 
  H\"older's  inequality in the fifth and  \eqref{e:VD} in the sixth, we get that for all $u \in \overline \sF$,
\begin{align*}
	\iint_{B' \times  B'} (u(z)-u(y))^2  \frac{\sB(z,y)}{|z-y|^{d+\alpha}}dzdy
	& \ge \iint_{A \times  A} (u(z)-u(y))^2  \frac{\sB(z,y)}{|z-y|^{d+\alpha}}dzdy\\
    &\ge C_2\int_{A} \int_{B(z,  \delta_{A}(z)/2)}\frac{(u(z)-u(y))^2}{|z-y|^{d+\alpha}} dydz\nn\\
	& \ge c_1\inf_{a \in \R}\bigg(\int_A |u(z) -a|^{2d/(d-\alpha)} dz \bigg)^{(d-\alpha)/d} \nn\\
	&\ge c_1\inf_{a \in \R}\bigg(\int_{B} |u(z) -a|^{2d/(d-\alpha)} dz \bigg)^{(d-\alpha)/d} \nn\\
	&\ge \frac{c_1}{m_d(B)^{\alpha/d}}\inf_{a \in \R}\int_{B} |u(z) -a|^{2} dz  \nn\\
	&\ge \frac{c_2}{r^\alpha}\inf_{a \in \R}\int_{B} |u(z) -a|^{2} dz.
\end{align*}
Since $\inf_{a \in \R}\int_{B} |u(z) -a|^{2} dz = \int_{B} (u(z)-\overline u_B)^2 dz$, we arrive at the result.
\end{proof}

Denote by $\overline{Y}^{U}$  the part of the process $\overline{Y}$ killed upon exiting $U$. By Proposition \ref{p:Nash} and \cite[Theorem 3.1]{BBCK}, $\overline Y^U$ has a transition density 
$\overline p^{U}(t,x,y)$ with respect to the Lebesgue measure on $U$.

Now we establish a near diagonal lower estimate on $\overline p^{B(x_0,r)}$ for $x_0 \in \overline D$ and $0<r \le R_0$. This estimate plays a crucial role in the probabilistic arguments for establishing parabolic H\"older regularity and parabolic Harnack inequality.

 \begin{prop}\label{p:ndl}
 Let $R_0>0$ and  $b \in (0,1)$. There exists  $C=C(R_0,b)>0$  such that for any $x_0 \in \overline D$, $0<r\le R_0$ and 
 $0< t\le (b r)^\alpha$, it holds that
	\begin{equation*}
		\overline	p^{	B_{\overline D}(x_0, r)}(t,z,y)  \ge  C t^{-d/\alpha} 
		\quad  \text{for all } y,z \in 	B_{\overline D}(x_0, b t^{1/\alpha}) \setminus \sN.
	\end{equation*}
\end{prop}
\begin{proof}   
Using  \eqref{e:VD}, \eqref{e:stochastic-complete},  Proposition  \ref{p:upper-heatkernel}, \cite[Remark 1.19]{CKW20} and 
\cite[Theorem 1.15]{CKW21}, 
we see that a local version of  the condition CSJ$(\phi)$ in \cite{CKW20} holds with $\phi(r)=r^\alpha$. Proposition \ref{p:PI} says that a local version of the Poincar\'e inequality PI$(\phi)$  in \cite{CKW20} holds for $\sE^0$
with $\phi(r)=r^\alpha$.  Now using  \hyperlink{B2-a}{{\bf (B2-a)}}, the local  CSJ$(\phi)$ condition, Proposition \ref{p:PI} and \cite[Remark 1.19]{CKW20}, we get the following near diagonal lower estimates:
There exist constants $c_0>0$, $c_1,c_2\in(0,1)$  such that  for any $x_0\in \overline D$,   $0<r\le c_1\wh R$ and $0< t\le (c_2 r)^\alpha$, it holds that
\begin{align}\label{e:ndl-step1}
		\overline	p^{	B_{\overline D}(x_0, r)}(t,z,y)  \ge  c_0 t^{-d/\alpha} 
		\quad \text{for all } y,z \in 	 B_{\overline D}(x_0, c_2t^{1/\alpha}) \setminus \sN.
\end{align}
By taking $c_1$ smaller if necessary, we assume that $c_1\wh R < R_0$.

It suffices to prove the 
proposition for $b\in ((1-c_1\wh R/R_0)^{1/2},1)$. 
Fix $x_0\in \overline D$ and  $0<r\le R_0$, and  write $B:=B_{\overline D}(x_0,r)$. 
 Let $b \in ((1-c_1\wh R/R_0)^{1/2},1)$, then
 $(1-b^2)r <(1-b^2)R_0 < c_1 \wh R$.
 Let
     $0< t\le (b r)^\alpha$ and $y,z \in B_{\overline D}(x_0, b t^{1/\alpha}) \setminus \sN$.  
Fix a constant $N \in \N$ such that  
 $$
 N^{-1/\alpha} < c_2(1-b^2)/b.
 $$
 Using this and the fact that $t\le (b r)^\alpha$, 
we have 
\begin{align}
		B_{\overline D}(y, (1-b^2)r)  \subset
		B \quad \text{and} \quad  (\eps t)^{1/\alpha} < b r N^{-1/\alpha} < c_2 (1-b^2)r \label{e:ndl-range-1}
\end{align}
where $\eps:=1/(N+1)$. Set $l_t=c_2(\eps t)^{1/\alpha}$, then
by the semigroup property, \eqref{e:ndl-range-1},  \eqref{e:ndl-step1} and \eqref{e:VD}, we obtain
\begin{align*}
	\overline	p^{B}(t,z,y) 
	&\ge  \int_{B_{\overline D}(y, l_t)}  \cdots \int_{B_{\overline D}(y, l_t)} \,	\overline p^{B}(\eps t,z,w_1) \,
	\overline	p^{	B_{\overline D}(y, (1-b^2)r)}(\eps t,w_1,w_2) \, \times  \nn\\
	& \;  \cdots  \times \overline	p^{	B_{\overline D}(y, (1-b^2)r)}(\eps t,w_{N-1},w_N) \,
	\overline	p^{	B_{\overline D}(y, (1-b^2)r)}(\eps t,w_N,y) \,dw_1 \cdots dw_N \nn\\
	&\ge ( c_0 (\eps t)^{-d/\alpha})^N m_d(B_{\overline D}(y, l_t))^{N-1}  
	\int_{B_{\overline D}(y, l_t)} \overline p^{	B}(\eps t,z,w_1) dw_1 \nn\\
	&\ge c_3 t^{-d/\alpha}\, \int_{B_{\overline D}(y, l_t)} \overline p^{B}(\eps t,z,w_1) dw_1.
\end{align*} 

Therefore, to obtain the desired result, it suffices to show that there exists a constant $c_4>0$ independent of 
$x_0\in \overline D$, $0<r\le R_0$,  $0< t\le (b r)^\alpha$ and  $y, z \in B_{\overline D}(x_0, b t^{1/\alpha}) \setminus \sN$ such that
\begin{align}\label{e:ndl-claim}
	\int_{B_{\overline D}(y, l_t)} \overline p^{B}(\eps t,z,w_1) dw_1 \ge c_4.
\end{align}

If $|y-z|< l_t$, then by using \eqref{e:ndl-step1} and \eqref{e:VD},  we get
\begin{align*}
		\int_{B_{\overline D}(y, l_t)}\overline  p^{	B}(\eps t,z,w_1) dw_1 
		&\ge 	\int_{B_{\overline D}(y, l_t)}\overline  p^{	B_{\overline D}(y, (1-b^2)r)}(\eps t,z,w_1) dw_1\\
		& \ge c_0(\eps t)^{-d/\alpha} m_d(B_{\overline D}(y, l_t)) \ge c_5.
\end{align*}
Hence, \eqref{e:ndl-claim} holds true in this case.

We now assume that $|y-z|\ge  l_t$.  Since $D$ a Lipschitz open set, there exist 
$y_0 \in  B_{\overline D}(y, 2^{-2/\alpha-2} \, l_t)$ and $c_6\in (0,1/4)$, depending only on 
$\Lambda_0$, $z_0 \in B_{\overline D}(z, 2^{-2/\alpha-2} \, l_t)$,
such that for $k_t:=c_6 2^{-2/\alpha-2} \,l_t$, it holds that
	$$
		B(z_0, 4k_t) \cup B(y_0,4k_t)  \subset D.
	$$
Note that 
	\begin{align}\label{e:ndl-range-2}
		B(z_0,2k_t) \subset B_{\overline D}(z, 2^{-2/\alpha-1}\, l_t) \quad \text{and} 
		\quad B(y_0,2k_t) \subset B_{\overline D}(y, 2^{-2/\alpha-1}\, l_t).
	\end{align}
In particular, $B(z_0,2k_t) \cap B(y_0,2k_t)=\emptyset$ since $|y-z| \ge l_t$. Set $c_7:=1\vee C$ where $C>0$ is the constant in Lemma \ref{l:EP} (with $T=R_0^\alpha$) and let $c_8:=2^{-3-2\alpha}(c_2c_6)^\alpha/c_7$.  
Since $c_8<1/2$, we have
	$$
		2^{-2/\alpha-1}\, l_t<2^{-1/\alpha}(1-c_8)^{1/\alpha}l_t= 
		c_2(2^{-1}(1-c_8)\eps t)^{1/\alpha}. 
	$$
Thus by \eqref{e:ndl-step1}, we have that, for all $v \in B(z_0, 2k_t) \setminus \sN$,
	\begin{align}\label{e:ndl-slide-1}
		\overline p^{B}(2^{-1}(1-c_8)\eps t,z,v) \ge \overline p^{	B_{\overline D}(z, (1-b^2)r)}(2^{-1}(1-c_8)\eps t,z,v) 
		\ge c_9 t^{-d/\alpha}
	\end{align}
and for all $w,w_1 \in B(y_0,2k_t)  \setminus \sN$,
	\begin{align}\label{e:ndl-slide-2}
		\overline p^{B}(2^{-1}(1-c_8)\eps t,w,w_1) 
		\ge \overline p^{	B_{\overline D}(y, (1-b^2)r)}(2^{-1}(1-c_8)\eps t,w,w_1) 
		\ge c_9 t^{-d/\alpha}.
	\end{align}

On the other hand, by the strong Markov property, we see that  for any $v \in B(z_0, 2k_t) \setminus \sN$,
\begin{align*}
	&\int_{B(y_0, 2k_t)} \overline p^{B}( c_8\eps t,v,w)dw =\P_v \left( \overline Y^{B}_{c_8\eps t} \in B(y_0, 2k_t) \right) \\
	&\ge \P_v \left(
	\overline Y^{B}_{ \bar \tau_{B(v, k_t)}} \in B(y_0, k_t),\;\;	|\overline Y^{B}_{c_8\eps t} - 
	\overline Y^{B}_{ \bar \tau_{B(v, k_t)}} |<k_t, \;\;   \bar \tau_{B(v, k_t)} \le c_8\eps t \right)\\
	&\ge  \P_v \left(
	\overline Y^{B}_{c_8\eps t \wedge  \bar \tau_{B(v,k_t)}} \in B(y_0, k_t) \right) 
	\inf_{w \in B_{\overline D}(y, k_t)}\P_w \left(\bar\tau_{B_{\overline D}( w,k_t)} \ge c_8\eps t\right) \\
	&=:I_1 \times I_2.
\end{align*}
By Lemma \ref{l:EP}, since $k_t=2^{-2/\alpha-2}c_2c_6(\eps t)^{1/\alpha}$ and  $ c_7 c_8 (c_2c_6)^{-\alpha} = 2^{-3-2\alpha}$, we get
	\begin{align}\label{e:ndl-improve-I2}
	I_2 \ge 1- c_7 c_8\eps t  k_t^{-\alpha}  = 1- 2^{2+2\alpha} c_7 c_8 (c_2c_6)^{-\alpha} = 2^{-1}.
	\end{align}
For $I_1$, note that  for any  $v \in B(z_0, 2k_t)$, $v' \in B(v,k_t)$ and $w \in B(y_0, k_t)$, we have 
$\delta_D(v') \ge \delta_D(z_0) - 3k_t \ge k_t$, $\delta_D(w) \ge \delta_D(y_0)  - k_t \ge 3k_t$ and 
$|v'-w| \le |z_0-y_0| + 4k_t \le |z-y| + 4k_t + l_t <6 t^{1/\alpha}$. Thus by \hyperlink{B2-b}{{\bf (B2-b)}}, 
\begin{align*}
	\sB(v',w)|v'-w|^{-d-\alpha} \ge  c |v'-w|^{-d-\alpha} \ge c t^{-1-d/\alpha}.
\end{align*}
Using this and the L\'evy system formula
\eqref{e:Levysystem-Y-bar}, since  $\overline Y^{B}_s=\overline Y^{B(v,k_t)}_s$ for  $s<\bar \tau_{B(v,k_t)}$,  we obtain
	\begin{align*}
		I_1&=  \E_v\bigg[ \int_0^{c_8\eps t \wedge \bar \tau_{B(v,k_t)}} \int_{B(y_0, k_t)}  
			\frac{\sB(\overline Y^{B(v,k_t)}_s, w)}{|\overline Y^{B(v,k_t)}_s - w |^{d+\alpha}} dw ds \bigg]\\
		&\ge c_{10} t^{-1-d/\alpha}  k_t^d\, \E_v[(c_8\eps t) \wedge \bar \tau_{B(v,k_t)}] = c_{11} t^{-1} \, 
			\E_v[(c_8\eps t) \wedge \bar \tau_{B(v,k_t)}].
	\end{align*}
By  Lemma \ref{l:EP} and \eqref{e:ndl-improve-I2}, it holds that
	\begin{align*}
	\E_v[(c_8\eps t) \wedge \bar \tau_{B(v,k_t)}] \ge  c_8\eps t \, \P_v ( \bar \tau_{B(v,k_t)} \ge c_8\eps t)  
	\ge c_8\eps t (1-c_7c_8\eps t k_t^{-\alpha}) = c_8\eps t/2.
	\end{align*}
Therefore, $I_1 \ge c_{12}$. Combining this with   \eqref{e:ndl-improve-I2}, \eqref{e:ndl-slide-1} and \eqref{e:ndl-slide-2}, and using the semigroup property, we arrive at
	\begin{align*}
	&	\int_{B_{\overline D}(y, l_t)}\overline  p^{	B}(\eps t,z,w_1) dw_1 
	\ge 	\int_{B(y_0,2k_t)} \int_{B(y_0,2k_t)}\int_{B(z_0,2k_t)} \overline  p^{	B}(2^{-1}(1-c_8)\eps t,z,v) \\
	&\qquad \qquad \qquad \qquad \qquad \quad	\times \, \overline  p^{	B}(c_8\eps t,v,w) \, 
	\overline p^{B}(2^{-1}(1-c_8)\eps t,w,w_1) dv\, dw\, dw_1\\
	&\ge (c_9 t^{-d/\alpha})^2 m_d(B(y_0,2k_t)) \int_{B(y_0,2k_t)}\int_{B(z_0,2k_t)} \overline  p^{B}(c_8\eps t,v,w) dv\, dw\\
	&\ge 2^{-1}c_{12}(c_9 t^{-d/\alpha})^2  m_d(B(y_0,2k_t)) \, m_d(B(z_0,2k_t)) = c_{13},
\end{align*}
which proves \eqref{e:ndl-claim}. The proof is complete. 
\end{proof}

By using \eqref{e:VD} and Proposition \ref{p:ndl}, one can follow the proof of  \cite[Proposition 3.5(ii)]{CKW20}  and obtain the following proposition.
\begin{prop}\label{p:E}
For any $R_0>0$,  there exists $C=C(R_0)>1$ such that
	\begin{align}\label{e:E}
		C^{-1}r^\alpha \le 	\E_{x_0}[\bar\tau_{	B_{\overline D}(x_0,r)}] \le C r^\alpha 
		\quad \text{for all} \;\, x_0 \in \overline D\setminus \sN, \; 0<r\le R_0.
	\end{align}
\end{prop}

Let $
Z := (V_s, \overline Y_s)_{s\ge 0}$ be the  time-space process  where $V_s = V_0 -s$. 
The law of the time-space process $s\mapsto 
Z_s$ starting from $(t, x)$ will be denoted by 
$\P_{(t,x)}$. 
For an open subset $U$ of $[0,\infty) \times \R^d$, define $
\tau^Z_{U} = \inf\{s > 0:\, 
Z_s\notin U\}$.

For $x_0 \in \overline D$ and $0\le a<b<\infty$, a Borel  function $u:[0,\infty)\times \overline D\to \R$ is said to be \textit{caloric} in $(a,b] \times B_{\overline D}(x_0,r)$ 
with respect to $\overline Y$ if for every relatively compact open set $U\subset 
(a,b] \times B_{\overline D}(x_0,r)$ with respect to the  topology on $[0,\infty) \times \overline D$,
it holds that 
$u(t,z)=\E_{(t,z)}u(
Z_{
\tau^Z_U})$ for all 
$(t,z)\in U$ with $z\notin \sN$.

By  \eqref{e:stochastic-complete}, \eqref{e:VD}  and Propositions \ref{p:ndl} and \ref{p:E}, we deduce the following joint H\"older regularity of bounded caloric functions from \cite[Proposition 3.8]{CKW20}.
\begin{thm}\label{t:phr}
 Let $R_0>0$ and  $b \in (0,1)$. 
	There exist constants $C=C(R_0,b)>0$ and  $\lambda=\lambda(R_0) \in (0,1]$ such that
	for all  $x \in \overline D$,  
	$0<r\le R_0$,  $t_0\ge 0$, and any bounded caloric function $u$ in $(t_0, 
	t_0+ r^\alpha]\times B_{\overline D}(x, r)$ with respect to $\overline Y$, there is a properly exceptional set $\sN_u \supset \sN$ such that 
	\begin{equation}\label{e:phr}
		|u(s,y) - u(t,z)| \le C \bigg(\frac{|s-t|^{1/\alpha} + |y-z|}{r}\bigg)^{\lambda } \esssup_{[t_0, t_0 + r^\alpha] \times D}\,|u|,
	\end{equation}
	for all $s,t \in (t_0 + (1-b^\alpha)r^\alpha, 
	t_0+ r^\alpha]$ and $y,z \in B_{\overline D}(x, b r)  \setminus \sN_u$.
\end{thm}

\begin{corollary}\label{c:phr-semigroup}
	Let $f \in L^1(D)$ and define $u(t,x)=\overline P_tf(x)$ for $t>0$ and $x \in \overline D  \setminus \sN$. Then  $u$ has a jointly  continuous version $\wt u$ in $(0,\infty) \times \overline D$ satisfying the following estimates: For any $T>0$, there exist constants  $C=C(T)>0$ and $\lambda=\lambda(T)\in (0,1]$ such that for all $0<t\le T$,
	\begin{equation}\label{e:semigroup-Linfty}
	\sup_{x \in \overline D}| \wt u(t,x) |  \le C t^{-d/\alpha}\lVert f \rVert_{L^1(D)}
	\end{equation}
	 and for any $y,z \in \overline D$ with $|y-z|\le (t/2)^{1/\alpha}/2$, 
		\begin{equation}\label{e:phr-semigroup}
		|\wt u(t,y) - \wt u(t,z)| \le \frac{C}{t^{d/\alpha}} \bigg(\frac{ |y-z|}{t^{1/\alpha}}\bigg)^{\lambda }\lVert f \rVert_{L^1(D)}.
	\end{equation}
\end{corollary}
\begin{proof}
Note that $u$ is caloric in $(0,\infty) \times \overline D$ by the Markov property. For any $T>0$, by Proposition \ref{p:upper-heatkernel}, there exists $c_1=c_1(T)>0$ such that for all $t>0$,
\begin{align}\label{e:semigroup-Linfty-u}
	\lVert u(t,\cdot) \rVert_{L^\infty(D)}  \le c_1 t^{-d/\alpha}\lVert f \rVert_{L^1(D)}. 
\end{align}  
In particular, $\lVert u(t,\cdot) \rVert_{L^\infty(D)}$ is locally bounded in $(0,\infty)$ as a function of $t$.  Thus, by Theorem \ref{t:phr}, one can deduce that  $u$ has a jointly  continuous version $\wt u$ in $(0,\infty) \times \overline D$. 
By \eqref{e:semigroup-Linfty-u}, $\wt u$ satisfies \eqref{e:semigroup-Linfty}.  Moreover, for each fixed $T>0$ and any $0<t\le T$ and $y,z \in \overline D$ with $|y-z|\le (t/2)^{1/\alpha}/2$, by applying \eqref{e:phr} with $u=\wt u$, $t_0=t/2$, $R_0=(T/2)^{1/\alpha}$,  $r=(t/2)^{1/\alpha}$ and $b=1/2$, we see from  \eqref{e:semigroup-Linfty} that \eqref{e:phr-semigroup} holds.
\end{proof}

\begin{remark} \label{r:conti}
By Corollary  \ref{c:phr-semigroup}, the transition density $\overline p(t,x,y)$ can be extended continuously to  $(0,\infty) \times \overline D \times \overline D$  by a standard argument. 
 See the proof of 
 \cite[Lemma 5.13]{GHH}  (although conditions (J) and (AB) are assumed in \cite{GHH},
 the arguments in the proof there only use \cite[(5.28) and (5.29)]{GHH} which can be replaced by \eqref{e:phr-semigroup} and \eqref{e:semigroup-Linfty}, respectively). 
Consequently, $\overline Y$ can be refined to be 
a strongly Feller process starting from every point in $\overline D$  and  the exceptional set $\sN$ in Propositions \ref{p:upper-heatkernel}, \ref{p:ndl} and \ref{p:E}, and  Lemma \ref{l:EP} can be taken to be  the empty set.
\end{remark}

For a closed subset $E\subset \overline D$, let 
	$$
	\bar \sigma_E:=\inf\{t>0:\overline Y_t \in E\}.
	$$

\begin{lemma}\label{l:ckw-3-7} 
 Let $R_0>0$ and $b \in (0,1)$.	There exists  $C=C(R_0, b)>0$  such that for all $x_0 \in \overline D$, $0<r \le  R_0$ and any compact set $K\subset B_{\overline D}(x_0,b r)$, 
	\begin{equation*}
		\P_{x_0}(
		\bar\sigma_K < \bar\tau_{	B_{\overline D}(x_0,r)}) 
		\ge  Cr^{-d} m_{d}(K).
	\end{equation*}
\end{lemma}
\begin{proof}  Using Proposition \ref{p:ndl} (with $b$ replaced by $b^{1/2}$), we get that for any compact set $K\subset 	B_{\overline D}(x_0, b r)$, 
\begin{align*}
	&\P_{x_0}(
	\bar\sigma_K < \bar\tau_{	B_{\overline D}(x_0,r)}) \ge	\P_{x_0}\big (\overline Y^{	B_{\overline D}(x_0,r)}_{(b^{1/2}r)^\alpha} \in K \big)\\
	& = \int_{K} \overline p^{	B_{\overline D}(x_0,r)} ( (b^{1/2}r)^\alpha, x_0, y) dy \ge c (b^{1/2}r)^{-d} m_d(K).
\end{align*}\end{proof}

For the next result, we need an additional assumption that implies the well-known condition \textbf{(UJS)}.

\medskip
\noindent \hypertarget{UBS}{{\bf (UBS)}} There exists $C>0$ such that  for  a.e. $x,y\in  D$,
	\begin{align}\label{e:UBS}
	\sB(x,y)\le \frac{C}{r^d}\int_{	B_{\overline D}(x,r)} \sB(z,y)dz 
	\quad  \text{whenever} \;\, 0< r \le \frac12 ( |x-y| \wedge \wh R).
	\end{align}
\medskip

Condition \hyperlink{UBS}{{\bf (UBS)}} implies 
the following local \textbf{(UJS)} condition:
for a.e. $x,y\in  D$  and $0< r \le 2^{-1}(|x-y| \wedge \wh R)$, 
\begin{align}\label{e:UJS}
	\frac{\sB(x,y)}{|x-y|^{d+\alpha}} \le \frac{
	C 2^{d+\alpha}}{r^d} \int_{	B_{\overline D}(x,r)} \frac{\sB(z,y)}{|z-y|^{d+\alpha}}dy,
\end{align}
since $|z-y| \le |x-y|+r <2|x-y|$ for all $ z \in 	B_{\overline D}(x,r)$. 
\begin{thm}\label{t:phi}
Suppose that $\sB$ 	satisfies \hyperlink{UBS}{{\bf (UBS)}}.  Let $R_0>0$. There exist constants $\eps>0$ and $C,K\ge1$  
depending 
on $R_0$ such that for all  $x\in \overline D$, $0<r\le R_0$, $t_0\ge 0$, and any non-negative function $u$ on $(0,\infty)\times \overline D$ which is caloric on 
	$(t_0, t_0+4\eps  r^\alpha] \times 	B_{\overline D}(x, r)$ 
	with respect to $\overline Y$, we have	
	\begin{equation*}
		\sup_{(t_1,y_1)\in Q_{-}}u(t_1,y_1)\le C \inf_{(t_2, y_2)\in Q_{+}}u(t_2, y_2),	
	\end{equation*}	
	where $Q_-=[t_0+\eps r^\alpha, t_0+2\eps r^\alpha]\times 	B_{\overline D}(x, r/K)$ and $Q_+=[t_0+3\eps  r^\alpha, t_0+4\eps  r^\alpha]\times 	B_{\overline D}(x, r/K)$.
\end{thm}
\begin{proof} 
 Proposition \ref{p:ndl} says that   a local version of the condition NDL$(\phi)$  in \cite{CKW20} holds with $\phi(r)=r^\alpha$. Hence, by  \eqref{e:VD}, the local \textbf{(UJS)} condition \eqref{e:UJS}, \cite[Remark 1.19]{CKW20}, and the equivalence between statements (1) and (4) in \cite[Theorem 1.18]{CKW20}, we immediately get our result (see also the proof of \cite[Theorem 5.2]{CKK09}). 
\end{proof}

	\subsection{Analysis and properties of  $Y^\kappa$}\label{s-processes-Y-0}

	Let $\sF^0$ be the closure of  Lip$_c(D)$ in $L^2(D)$ under $\sE^0_1$. Then $(\sE^0,\sF^0)$ is a regular Dirichlet form. Let $Y^0=(Y^0_t,t \ge 0; \P_x,x \in \overline D\setminus \sN_0)$ be the Hunt process associated with $(\sE^0,\sF^0)$, where  $\sN_0$ is an  exceptional set for $Y^0$.
	
	\begin{lemma}\label{l:Hardy}
Suppose that $\alpha>1$.  There exist constants $C>0$ and $M_0>1$ such that for any $Q \in \partial D$,  $0<r<\wh R$ and any 
$u \in C_c(B_{D} (Q,r/M_0))$,
\begin{align*}
	\int_{	B_{D}(Q,r/M_0)} u(z)^2 \delta_D(z)^{-\alpha} dz
	\le	C \iint_{	B_{D}(Q,r) \times 	B_{D}(Q,r)} (u(z)-u(y))^2 \frac{\sB(z,y)}{|z-y|^{d+\alpha}}dzdy.
\end{align*}
	\end{lemma}
\begin{proof} Let $Q \in \partial D$ and $0<r<\wh R$.  By \cite[p.45 (4)]{JK82}, there exist a constant $M_0>1$ independent of $Q$ and $r$, and a Lipschitz domain $A$   such that $B_{D}(Q,r/M_0) \subset A \subset B_{D}(Q,r)$.  Using \cite[Theorem 1.1]{Dyda04} in the second inequality,  \cite[(13)]{Dyda06} in the third and \hyperlink{B2-b}{{\bf (B2-b)}} in the fifth, we get that for any $u \in C_c(B_{D} (Q,r))$,
\begin{align*}		
	 &\int_{	B_{D}(Q,r/M_0)} u(z)^2 \delta_D(z)^{-\alpha} dz\le 	\int_{A} u(z)^2 \delta_A(z)^{-\alpha} dz\\
	 &\le c_1\int_{A}\int_A  \frac{(u(z)-u(y))^2}{|z-y|^{d+\alpha}}dydz \\	
	 	&\le 	c_2	\int_{A}  	\int_{B(z, \delta_A(z)/2)}  \frac{(u(z)-u(y))^2}{|z-y|^{d+\alpha}}dydz \\
	 	&\le	c_2	\int_{B_{D}(Q,r)}  	\int_{A \cap B(z, \delta_D(z)/2)}  \frac{(u(z)-u(y))^2}{|z-y|^{d+\alpha}}dydz\\				&\le 	c_2C_2^{-1}	\int_{B_{D}(Q,r)}  	\int_{A \cap B(z, \delta_D(z)/2)} (u(z)-u(y))^2 \frac{\sB(z,y)}{|z-y|^{d+\alpha}}dydz\\	& \le c_2C_2^{-1} \int_{	B_{D}(Q,r)} \int_{	B_{ D}(Q,r)} (u(z)-u(y))^2 \frac{\sB(z,y)}{|z-y|^{d+\alpha}}dydz.\end{align*} \end{proof}

\begin{lemma}\label{l:Hardy2}
	Suppose that $\alpha>1$. Let $Q \in \partial D$, $0<r< \wh R$ and $u \in \text{\rm Lip}_c(\overline D)$ be such that $u \ge 1$ on $B_{D}(Q,r)$. Then $u \in \overline \sF \setminus \sF^0$.
\end{lemma}
\begin{proof}  Clearly, $u \in \overline \sF$ since $\text{\rm Lip}_c(\overline D) \subset \overline \sF$. Suppose that $u \in \sF^0$. Then there exists an $\sE^0_1$-Cauchy sequence $(u_n)_{n \ge 1}$  of functions in $\text{\rm Lip}_c(D)$ such that $\lim_{n \to \infty} \sE^0_1(u-u_n,u-u_n)=0$.  Since  $\sup_{n \ge 1} \sE^0_1(u_n,u_n)<\infty$, by Lemma \ref{l:Hardy}, we have
 \begin{align}\label{e:Hardy2}
 \limsup_{n \to \infty}	\int_{	B_{D}(Q,r/M_0)} u_n(z)^2 \delta_D(z)^{-\alpha} dz\le	c_1 \limsup_{n \to \infty} \sE^0(u_n,u_n)<\infty,
 \end{align}
where $M_0>1$ is the constant in Lemma \ref{l:Hardy}. Note that $u_n$ converges to $u$ in $L^2(D)$. Thus, there is a subsequence $(u_{s_n})_{n \ge 1}$ such that $\lim_{n \to \infty}u_{s_n}(z)^2=u(z)^2$  for a.e. $z \in B_{D}(Q,r/M_0)$. Using Fatou's lemma and the facts that $u \ge 1$ on $B_{D}(Q,r)$, $D$ is a Lipschitz open set, and $\alpha >1$, we obtain
\begin{align*}
&\liminf_{n \to \infty} \int_{	B_{D}(Q,r/M_0)} u_{s_n}(z)^2 \delta_D(z)^{-\alpha} dz \\
&\ge  \int_{	B_{D}(Q,r/M_0)} u(z)^2 \delta_D(z)^{-\alpha} dz \ge \int_{	B_{ D}(Q,r/M_0)} \delta_D(z)^{-\alpha} dz=\infty.
\end{align*}
This contradicts  \eqref{e:Hardy2}. The proof is complete. \end{proof}

\begin{prop}\label{p:alpha>1}
		$\sF^0=\overline \sF$ if and only if $\alpha \le 1$.
\end{prop}
\begin{proof} 
Suppose $\alpha \le 1$. Define 
\begin{align*}
\widetilde{\sC}(u, v)&:=
\int_D\int_D
\frac{(u(x)-u(y))(v(x)-v(y))}{|x-y|^{d+\alpha}}dxdy,\\
{\mathcal D}(\widetilde{\sC})&:=\mbox{closure of Lip}_c(\overline{D}) \mbox{ in } L^2(D) \mbox{ under } \widetilde{\sC}+
(\cdot, \cdot)_{L^2(D)}. 
\end{align*}
Then $(\widetilde{\sC}, {\mathcal D}(\widetilde{\sC}))$ is a regular Dirichlet form associated with the reflected $\alpha$-stable process in $\overline{D}$ in the sense of \cite{BBC}. By \hyperlink{B2-a}{{\bf (B2-a)}}, there exists a constant $c>0$ such that  ${\mathcal E}^0(u, u)\le c\, \widetilde{\sC}(u, u)$ for all $u\in {\rm Lip}_c(\overline{D})$ and hence ${\mathcal D}(\widetilde{\sC})\subset \overline{\mathcal F}$.  By \cite[Theorem 2.5(i) and Remark 2.2(1)]{BBC}, since $\alpha \le 1$, $\partial D$ is 
$(\widetilde{\sC}, {\mathcal D}(\widetilde{\sC}))$-polar and hence is $({\mathcal E}^0, \overline{\mathcal F})$-polar. Therefore, when starting from $D$, $\overline{Y}$ will never exit $D$.  Hence $\overline{Y}$ and $Y^0$ are the same when they start from $x\in D$ and $\sF^0=\overline \sF$.
Combining this with Lemma \ref{l:Hardy2}, we arrive at the desired conclusion. \end{proof}

In the remainder of this work, we let $\kappa$ be a  non-negative Borel  function on $D$ with the following property:
	
	\medskip

	\noindent \hypertarget{K1}{{\bf (K1)}}  There exists a constant $C_3>0$ such that
	\begin{equation*}
		\kappa(x) \le C_3 (\delta_D(x) \wedge 1)^{-\alpha}.
	\end{equation*}
	 If $\alpha \le 1$, then we also assume that $\kappa$ is non-trivial, namely,
		\begin{equation}\label{e:kappa-non-trivial}m_d(\{x\in D:\kappa(x) >0\})>0.\end{equation}

\medskip

We consider a symmetric form $(\sE^\kappa, \sF^\kappa)$ defined by
	\begin{align*}
		\sE^\kappa(u,v)&=\sE^0(u,v) + 
		\int_{D} u(x)v(x)\kappa(x)dx, \\
		\sF^\kappa&= 	\wt \sF^0 \cap L^2(D, 
			\kappa(x) dx),\nn
	\end{align*}
where $\wt \sF^0$ is the family of all $\sE^0_1$-quasi-continuous functions in $\sF^0$. Then $(\sE^\kappa,\sF^\kappa)$ is a regular Dirichlet form on $L^2(D)$ with Lip$_c(D)$ as a special standard core,  see \cite[Theorems 6.1.1 and 6.1.2]{FOT}.
Let  $Y^\kappa=(Y^\kappa_t,t \ge 0; \P_x,x \in D \setminus  \sN_\kappa)$ be the Hunt process associated with $(\sE^\kappa, \sF^\kappa)$ 
where $\sN_\kappa$ is an exceptional set for $Y^\kappa$.
We denote by $\zeta^\kappa$ the lifetime of $Y^\kappa$. Define $Y^\kappa_t=\partial$ for $t \ge \zeta^\kappa$, where $\partial$ is a cemetery point added to the state space $D$. 
Since 	the jump kernel $\sB(x,y)|x-y|^{-d-\alpha}dxdy$ and the killing measure $\kappa(x)dx$ of  $(\sE^\kappa,  \sF^\kappa)$ are absolutely continuous with respect to  $m_d \otimes m_d$ and  $m_d$ respectively,   $Y^\kappa$ satisfies the following L\'evy system formula (cf. \eqref{e:Levysystem-Y-bar}): For any $x \in  D $, any non-negative Borel function $f$ on $D \times D_\partial$ vanishing on the diagonal, and any stopping time $\tau$,	
	\begin{align}\label{e:Levysystem-Y-kappa}	
		\E_x\bigg[  \sum_{s \le \tau} f( Y^\kappa_{s-},  Y^\kappa_s)  \bigg] 
		= \E_x \left[ \int_0^\tau \bigg(  \int_{D} 
		\frac{f( Y^\kappa_s, y) \sB( Y^\kappa_s, y)}{| Y^\kappa_s-y|^{d+\alpha}} dy  
		+ \kappa(Y^\kappa_s)f(Y_s^\kappa, \partial) \bigg) \,ds  \right].	
	\end{align}

The process $Y^\kappa$ can be regarded as the part process of $\overline Y$ killed at $\zeta^\kappa$. Hence, by Remark \ref{r:conti}, $Y^\kappa$ can be refined to be a Hunt process starting from every point $D$.
	Moreover, by Proposition \ref{p:upper-heatkernel}, we obtain the following result.
	
	\begin{prop}\label{p:upper-heatkernel-2}
		The process $Y^\kappa$ has a 	transition density  $p^\kappa(t,x,y)$ defined on $(0,\infty) \times D \times D$. 
		Moreover, for any $T>0$, there exists a constant $C=C(T)>0$ such that
		\begin{equation*}
			p^\kappa(t,x,y) \le C \left( t^{-d/\alpha} \wedge \frac{t}{|x-y|^{d+\alpha}}\right), \quad 0<t\le T, \; x,y \in D.
		\end{equation*}
	\end{prop}

	For an open set $U \subset  D$, we let $\tau_U:=\inf\{t>0: Y^\kappa_t \notin U\}$ and 
	we  denote by $Y^{\kappa,U}$ and $(P^{\kappa,U}_t)_{t \ge 0}$  the part of the process $Y^\kappa$ killed upon exiting $U$ and its semigroup, respectively.
 We denote   $(P^{\kappa,D}_t)_{t \ge 0}$ by $(P^{\kappa}_t)_{t \ge 0}$. By \cite[Theorem 6.1.1]{FOT}, the semigroup $(P^{\kappa,U}_t)_{t \ge 0}$ can be represented by
	\begin{align}\label{e:Feymann-Kac}
		P^{\kappa,U}_t f(x)= \E_x \left[\,\exp\bigg( - \int_0^t \kappa(\overline Y^{U}_s) ds \bigg)\, f(\overline Y^{U}_t) \right].
	\end{align}
	Denote by $p^{\kappa,U}(t,x,y)$ a transition density  of $Y^{\kappa,U}$.

	\begin{prop}\label{p:ndl2}
	Let $R_0>0$ and  $b \in (0,1)$. There exists  $C=C(R_0,b)>0$ such that for any $x_0 \in D$, $0<r<\delta_D(x_0) \wedge R_0$ and  $0< t\le (b r)^\alpha$, it holds that
		\begin{equation}\label{e:ndl-2}
			p^{\kappa,B(x_0, r)}(t,z,y)  \ge  C t^{-d/\alpha}\quad  \text{for all } z \in 	B_{\overline D}(x_0, b t^{1/\alpha}) \text{ and a.e. } y \in 	B_{\overline D}(x_0, b t^{1/\alpha}).
		\end{equation}
	\end{prop}
	\begin{proof}  Let $x_0 \in D$ and $0<r<\delta_D(x_0) \wedge R_0$. 
	For all $x \in B(x_0,b^{1/2}r)$, we have $\delta_D(x) \ge \delta_D(x_0) - b^{1/2}r > (1-b^{1/2})r$.  
	Thus, by \hyperlink{K1}{{\bf (K1)}}, we get that 
	\begin{align}\label{e:kappa-bound}
		\kappa(x) 
		 \le C_3(1-b^{1/2})^{-\alpha}(1+R_0)^\alpha r^{-\alpha}  \quad \text{for all} \;\, x \in B(x_0, b^{1/2}r).
	\end{align}
Using \eqref{e:Feymann-Kac}, \eqref{e:kappa-bound} and Proposition \ref{p:ndl} with $b^{1/2}$ (see Remark \ref{r:conti}),  we get that for all $0< t\le (b r)^\alpha$,  $z \in  B(x_0, b t^{1/\alpha}) \subset B(x_0, b^{1/2} t^{1/\alpha})$ and a.e. $y \in  B(x_0, b t^{1/\alpha})$,
\begin{align*}
	&p^{\kappa, B(x_0,r)}(t,z,y) \ge 	p^{\kappa, B(x_0,b^{1/2}r)}(t,z,y)  \\
	&\ge e^{-c_1r^{-\alpha}t}\,\overline p^{B(x_0,b^{1/2}r)}(t,z,y) \ge c_2 e^{-c_1b ^\alpha} t^{-d/\alpha}.
\end{align*}
 \end{proof}

	\begin{prop}\label{p:E2}
 For any $R_0>0$, 	there exists $C=C(R_0)>1$ such that
		\begin{align}\label{e:E2}
			C^{-1}r^\alpha \le 	\E_{x_0}[		\tau_{B(x_0,r)}] \le C r^\alpha \quad \text{for all} \;\, x_0 \in  D, \; 0<r<\delta_D(x_0) \wedge R_0.
		\end{align}
	\end{prop}
	\begin{proof} Let $x_0 \in D$ and $0<r<\delta_D(x_0) \wedge R_0$. Since $\tau_{B(x_0,r)} \le \bar \tau_{B(x_0,r)}$, the upper bound in \eqref{e:E2} follows from Proposition \ref{p:E}. On the other hand, by Lemma \ref{l:EP}, there exists $c_1>0$ independent of $x_0$ and $r$ such that 
	\begin{align}\label{e:E2-1}
	 \P_{x_0}(\overline Y^{B(x_0,r/2)}_{c_1r^\alpha} \in B(x_0,r/2) )=	\P_{x_0}(\bar \tau_{B(x_0,r/2)} > c_1r^\alpha) \ge 1/2.
	\end{align}
By \eqref{e:kappa-bound}, $\kappa(x) \le c_2r^{-\alpha}$ for all $x \in B(x_0, r/2)$.
Using this, \eqref{e:Feymann-Kac}  and \eqref{e:E2-1}, we obtain
\begin{align*}	
	&\P_{x_0}(\tau_{B(x_0,r/2)} > c_1r^\alpha)=	\P_{x_0}(Y^{\kappa, B(x_0,r/2)}_{c_1r^\alpha} \in B(x_0,r/2) )\\
		&\ge e^{- c_1r^\alpha ( c_2r^{-\alpha} )} \P_{x_0}(\overline Y^{B(x_0,r/2)}_{c_1r^\alpha} \in B(x_0,r/2) ) 
		  \ge e^{-c_1c_2}/2.	
\end{align*}
 Hence, $\E_{x_0} [
		\tau_{B(x_0,r)}] \ge  c_1r^\alpha	\P_{x_0}(
		\tau_{B(x_0,r/2)} > c_1r^\alpha) \ge c_3r^\alpha$. \end{proof}

Using \eqref{e:stochastic-complete},  \eqref{e:VD} and Propositions \ref{p:upper-heatkernel-2},  \ref{p:ndl2} and \ref{p:E2}, one obtains the following theorem by a standard argument. See the proof of \cite[Theorem 4.14]{CK03}.
We emphasize that conservativeness is not used in the proof of \cite[Theorem 4.14]{CK03}.
Caloric functions with respect to $Y^\kappa$ are defined analogously to those with respect to $\overline{Y}$.
	
	\begin{thm}\label{t:phr2}
	Let $R_0>0$ and $b \in (0,1)$. There exist constants  $\lambda \in (0,1]$ and $C=C(R_0,b)>0$
	such that
		for all  $x \in  D$,  
		$0<r<\delta_D(x) \wedge R_0$,  $t_0\ge 0$, and any bounded caloric function $u$ in $(t_0, 
		t_0+ r^\alpha]\times B(x, r)$ with respect to $Y^\kappa$, there is a properly exceptional set $\sN_u$ such that 
		\begin{equation*}
			|u(s,y) - u(t,z)| \le C \bigg(\frac{|s-t|^{1/\alpha} + |y-z|}{r}\bigg)^{\lambda } 
			\esssup_{[t_0, t_0 + r^\alpha] \times D}\,|u|,
		\end{equation*}
		for every $s,t \in (t_0 + (1-b^\alpha)r^\alpha, 
		t_0+ r^\alpha]$ and $y,z \in B(x,  br)  \setminus \sN_u$.
	\end{thm}

	\begin{remark}\label{r:strong-Feller}
			By Proposition \ref{p:upper-heatkernel-2} and  Theorem  \ref{t:phr2}, for any $f \in L^1(D)$, the result of
			Corollary \ref{c:phr-semigroup} holds for $u(t,x)=P^\kappa_t f(x)$. Thus,	the 	transition density  $p^\kappa(t,x,y)$  can be extended continuously to  $(0,\infty) \times  D \times  D$  by a standard argument (see Remark \ref{r:conti}). Similarly,
			for any open set $A \subset D$, the  transition density  $p^{\kappa,A}(t,x,y)$ can be extended continuously 
			to  $(0,\infty) \times  A \times  A$.  Consequently,   \eqref{e:ndl-2} holds for all 
			$z,y \in 	B_{\overline D}(x_0, b t^{1/\alpha})$ and $Y^\kappa$ and $Y^{\kappa, A}$ are strongly Feller.  
	\end{remark}

	In the remainder of this  work, we always take jointly continuous versions of $p^\kappa(t,x,y)$ and $p^{\kappa,A}(t,x,y)$.

	\begin{prop}\label{p:notconservative}
		The process $Y^\kappa$ is not conservative  in the sense that 
		$$
		\lVert \1_D - P^\kappa_t \1_D \rVert_{L^2(D)}>0 \quad \text{for all} \;\, t>0.
		$$
	\end{prop}
	\begin{proof}   Suppose that \eqref{e:kappa-non-trivial} holds. Since $Y^\kappa$ 
		is an irreducible Hunt process, by using \eqref{e:Feymann-Kac}, we get the result in this case.
		Suppose that $\kappa=0$ a.e. Then  $\alpha>1$  by 
		\hyperlink{K1}{{\bf (K1)}}.      
		Since $Y^0$ can be regarded as a part process of $\overline Y$ killed at $\zeta^0$, if $Y^0$ is conservative, then the process $\overline Y$ started from $D$ is equal to $Y^0$. By the one-to-one correspondence between regular Dirichlet forms and symmetric Hunt processes, it follows that $\overline \sF = \sF^0$. By Proposition \ref{p:alpha>1},  this is a contradiction and we conclude the desired result. \end{proof}

	\begin{lemma}\label{l:finite-lifetime} 
		For any $t>0$ and $x \in D$, we have
		\begin{align*}
			\int_D p^{\kappa}(t,x,y)dy <1.
		\end{align*}
	\end{lemma}
	\begin{proof} 	By Proposition \ref{p:notconservative} and symmetry,  we see that for any $t>0$,
		\begin{align*}\Big	\lVert \1_D(\cdot) - \int_D p^\kappa(t/2,y,\cdot)dy \Big\rVert_{L^2(D)} = \Big \lVert \1_D(\cdot) - \int_D p^\kappa(t/2,\cdot,y)dy \Big\rVert_{L^2(D)}>0.\end{align*}
		Therefore, since $p^\kappa(t/2,\cdot,\cdot)$ is jointly continuous, for any $t>0$,   there exist $x_0 \in D$ and constants $r_0>0$, $\eps_0\in (0,1)$  such that
		\begin{align}\label{e:notconserv}\sup_{z \in B_D(x_0, r_0)}	\int_D p^{\kappa}(t/2, z, y)dy \le 1-\eps_0.\end{align}
		Note that the semigroup $(P^{\kappa}_t)_{t>0}$ is irreducible by  \hyperlink{B2-b}{{\bf (B2-b)}}.
		 See Section \ref{ch:set-up}, the paragraph below \eqref{e:B(x,x)}.
		   Hence, we have
		\begin{align}\label{e:irreducible}
			p^{\kappa}(t,x,y)>0 \quad \text{for all} \;\, t>0, \, x,y \in D.
		\end{align}
		By the semigroup property,  \eqref{e:notconserv} and \eqref{e:irreducible}, we get that  for all $t>0$ and $x \in D$,
		\begin{align*}
			&\int_D	p^\kappa(t,x,y)dy= \int_D p^\kappa(t/2,x,z) \int_Dp^\kappa(t/2,z,y)dydz\\
			&\!\! \le  \int_{D}  p^\kappa(t/2,x,z)dz  +  \bigg( \! \sup_{z \in B_D(x_0, r_0)}	\int_D p^{\kappa}(t/2, z, y)dy -1\bigg)\int_{ B_D(x_0,r_0)} \! p^\kappa(t/2,x,z)dz \\
			&\!\! \le  \int_{D }  p^\kappa(t/2,x,z)dz  -\eps_0\int_{ B_D(x_0,r_0)} p^\kappa(t/2,x,z)dz< \int_{D}  p^\kappa(t/2,x,z)dz \le 1.
		\end{align*}
	\end{proof}
	
	\begin{prop}\label{p:largetime-1}
		There exists  $C>0$  such that for any bounded open subset $A$ of $D$,
		\begin{equation}\label{e:largetime-1}
			p^{\kappa,  A}(t,x,y) \le Cm_d(A)e^{-\lambda_1(t-2)}, \quad\; t\ge3, \; x,y \in A,
		\end{equation}
		where 
		\begin{align}\label{e:eigenvalue}
			\lambda_1:=\inf\left\{ \sE^\kappa(u,u):  u \in \text{\rm Lip}_c(A), \, \lVert u \rVert_{L^2(A)}=1\right\}.
		\end{align}
		Moreover, $\lambda_1$ is strictly positive and  there exists $C'>0$ depending on $A$ such that
		\begin{equation}\label{e:largetime-2}
			\sup_{x,y\in A}	p^{\kappa,  A}(t,x,y) \ge C'e^{-\lambda_1t}, \quad \; t>0.
		\end{equation}
	\end{prop}
	\begin{proof}
		By Proposition \ref{p:upper-heatkernel-2}, 
		the semigroup $(P^{\kappa, A}_t)_{ t> 0}$ consists of Hilbert-Schmidt operators, and hence compact operators in $L^2(A)$.
		Thus, since $(P^{\kappa,A}_t)_{t>0}$ is an  $L^2(A)$-contraction   symmetric semigroup, for each $t>0$,  $P^{\kappa,A}_t$ has discrete spectrum $(e^{-\lambda_nt})_{n \ge 1}$,  repeated according to their multiplicity, where $(\lambda_n)_{n \ge 1}$ is a non-decreasing non-negative sequence independent of $t$. By \cite[Theorem 4.4.3]{FOT}, $\text{\rm Lip}_c(A)$ is a core of the Dirichlet form associated with the semigroup $(P^{\kappa,A}_t)_{t>0}$. Hence, the bottom of the spectrum $\lambda_1$ is equal to the right-hand side of \eqref{e:eigenvalue}. Let $(v_n)_{n \ge 1}$ be eigenfunctions corresponding to $(\lambda_n)_{n \ge 1}$, constituting an orthonormal basis for $L^2(A)$. 
		For each  $t>0$ and $x \in A$, consider the eigenfunction expansion 
		$p^{\kappa,A}(t,x,\cdot )  =\sum_{k=1}^\infty a_{t,n}(x)v_n(\cdot)$ in $L^2(A)$. Since $(v_n)_{n \ge 1}$ is an  orthonormal basis for $L^2(A)$, for all $t>0$ and $n \ge 1$, we have
		\begin{align}\label{e:heatkernel-spectral-1}
				a_{t,n}(\cdot )=	\int_A	p^{\kappa,A}(t,\cdot ,y) v_n(y)  dy =  P^{\kappa,A}_t v_n(\cdot)= e^{-\lambda_n t} v_n(\cdot)\quad \text{in $L^2(A)$}.
		\end{align}
		Hence, since	the map $x\mapsto p^{\kappa,A}(t, x, y)$ is continuous and is bounded by Proposition \ref{p:upper-heatkernel-2}, we can assume that  $v_n(x) = e^{\lambda_nt} \int_A	p^{\kappa,A}(t,x,y) v_n(y)  dy$ are continuous functions on $A$ for all $n \ge 1$.	Consequently, we obtain
		\begin{align}\label{e:heatkernel-spectral-2}
			p^{\kappa,A}(t,x,y) = \sum_{n=1}^\infty e^{-\lambda_n t} v_n(x)v_n(y) \quad \text{for all} \;\, (t,x,y) \in (0,\infty) \times A \times A.
		\end{align}
		Using  the semigroup property  in the first equality below, Proposition \ref{p:upper-heatkernel-2}  and \eqref{e:heatkernel-spectral-2} in the first inequality, Fubini's theorem and the symmetry of $p^{\kappa,A}$ in the second equality, 
		and the fact that $\lVert f \rVert_{L^2(A)}^2 = \sum_{n=1}^\infty ( \int_A f(z)v_n(z) dz)^2$ in the third equality, we get that for all $t\ge3$ and $x,y \in A$,
		\begin{align*}
			p^{\kappa,A}(t,x,y)&=\int_{A\times A\times A \times A} p^{\kappa,A}(1/2,x,z_1)\,p^{\kappa,A}(1/2,z_1,z_2)\,p^{\kappa,A}(t-2,z_2,z_3)\nn\\
			&\qquad\qquad\qquad \;\;    \times p^{\kappa,A}(1/2,z_3,z_4)\,p^{\kappa,A}(1/2,z_4,y)\,dz_1\,dz_2\,dz_3\,dz_4 \nn\\
			&\le c_1^2\sum_{n=1}^\infty e^{-\lambda_n(t-2)}\int_{A\times A\times A \times A}  p^{\kappa,A}(1/2,z_1,z_2)\,v_n(z_2) v_n(z_3)\\
			&\qquad\qquad\qquad \qquad \qquad \qquad\qquad  \times  p^{\kappa,A}(1/2,z_3,z_4)\,dz_1\,dz_2\,dz_3\,dz_4 \\
			&= c_1^2\sum_{n=1}^\infty  e^{-\lambda_n(t-2)} \bigg( \int_{A} v_n(z_2)\int_A p^{\kappa,A}(1/2,z_2,z_1) \, dz_1\,dz_2 \bigg)^2 \\
			&\le c_1^2 e^{-\lambda_1(t-2)} \sum_{n=1}^\infty  \bigg( \int_{A} v_n(z_2) \int_A p^{\kappa,A}(1/2,z_2,z_1) \, dz_1 \, dz_2 \bigg)^2\\
			&=  c_1^2 e^{-\lambda_1(t-2)} \int_{A}  \bigg(\int_A p^{\kappa,A}(1/2,z_2,z_1) \, dz_1 \bigg)^2 \, dz_2 \le c_1^2\, m_d(A) e^{-\lambda_1(t-2)}.
		\end{align*}
		Therefore, \eqref{e:largetime-1} holds true.
		
		Next, we show that $\lambda_1>0$. 
		Suppose that $\lambda_1=0$. Then by H\"older's  inequality, symmetry, Fubini's theorem and Lemma \ref{l:finite-lifetime},  it holds that
		\begin{align*}
			1&=	\int_A v_1(x)^2dx =\int_A \bigg(\int_A p^{\kappa,A}(1,x,y)v_1(y)dy \bigg)^2 dx\\
			&\le \int_A \bigg(\int_A p^{\kappa,A}(1,x,y)dy \bigg) \bigg( \int_A p^{\kappa,A}(1,y,x)v_1(y)^2dy\bigg)dx\\
			&< \int_A v_1(y)^2  \int_A p^{\kappa,A}(1,y,x)dxdy <\int_D v_1(y)^2dy =1,
		\end{align*}
		which is a contradiction. Hence, $\lambda_1>0$.
		
		By Krein–Rutman theorem, we can assume that the eigenfunction $v_1$ is non-negative on $A$. Then from \eqref{e:heatkernel-spectral-1}, we obtain
		\begin{align*}
			\sup_{x,y\in A}	p^{\kappa,  A}(t,x,y) \ge e^{-\lambda_1 t}\frac{\sup_{x\in A} v_1(x)}{\int_A v_1(y)dy}=c_2e^{-\lambda_1t},
		\end{align*}
		proving that \eqref{e:largetime-2} holds. The proof is complete. 		
	\end{proof}

	\begin{lemma}\label{l:largetime-2}
		Suppose that $\alpha>1$.	For every $R_0>0$, there exists  a constant  $\lambda_0=\lambda_0(R_0)>0$  such that if $D_0$ is a bounded connected component of $D$ with $\text{\rm diam}(D_0) \le R_0$, then for all $u \in$ \text{\rm Lip}$_c(D_0)$,
		\begin{align}\label{e:eigenvalue-lowerbound}
		\sE^\kappa(u,u)\ge	\sE^0(u,u) \ge \lambda_0 \lVert u \rVert_{L^2(D_0)}^2.
		\end{align}
	\end{lemma}
	\begin{proof}
		  The first inequality in \eqref{e:eigenvalue-lowerbound} is evident. 		 According to  \cite[(13)]{Dyda06} and its proof, there exists $c_1>0$ depending only on $d,\alpha,\Lambda_0, \wh R$ and $R_0$ such that for all  $u \in$ \text{\rm Lip}$_c(D_0)$,
		\begin{align*}
			\int_{D_0}\int_{D_0}  \frac{(u(x)-u(y))^2}{|x-y|^{d+\alpha}} dydx \le c_1 	\int_{D_0}\int_{B(x, \delta_{D_0}(x)/2)}  \frac{(u(x)-u(y))^2}{|x-y|^{d+\alpha}} dydx.
		\end{align*} 
		Thus, by  \hyperlink{B2-b}{{\bf (B2-b)}}, we have
		\begin{equation}\label{e:eigen-1}
			\begin{split}
				\sE^0(u,u) &\ge 	\frac12\iint_{D_0 \times D_0} (u(x)-u(y))^2 \frac{\sB(x,y)}{|x-y|^{d+\alpha}} dxdy \\
				&\ge  C_2c_1^{-1} 	\iint_{D_0 \times D_0}  \frac{(u(x)-u(y))^2}{|x-y|^{d+\alpha}} dxdy.
			\end{split}	
		\end{equation}
		For $x \in D_0$ and  $w \in \R^d$ with $|w|=1$,  define
		\begin{align*}
			d^w_{D_0}(x):=\min\left\{|t|: x + tw \notin D_0\right\} \quad \text{and} \quad 	\delta^w_{D_0}(x):=\sup\left\{|t|: x + tw \in D_0\right\}.
		\end{align*}
		By \cite[Theorem 1.1]{LS10}, we have for all  $u \in$ \text{\rm Lip}$_c(D_0)$,
		\begin{align}\label{e:eigen-2}
			\frac12	\iint_{D_0 \times D_0}  \frac{(u(x)-u(y))^2}{|x-y|^{d+\alpha}} dxdy \ge c_2 \int_{D_0} \frac{u(x)^2}{M_\alpha(x)^\alpha} dx,
		\end{align}
		where 
		$$
		c_2:= \frac{\pi^{(d-1)/2}\Gamma((1+\alpha)/2)}{\alpha\Gamma((d+\alpha)/2)}
		\left[\frac{2^{1-\alpha}}{\sqrt \pi}  \Gamma\bigg( \frac{2-\alpha}{2}\bigg)  \Gamma\bigg( \frac{2-\alpha}{2}\bigg)  -1 \right]
		$$
		and
		$$
		\frac{1}{M_\alpha(x)^\alpha}:= \frac{\Gamma((d+\alpha)/2)}{2\pi^{(d-1)/2} \Gamma((1+\alpha)/2)} \int_{w \in \R^d: |w|=1}  \left[ \frac{1}{d^w_{D_0}(x)} +  \frac{1}{\delta^w_{D_0}(x)}  \right]^\alpha m_{d-1}(dw).
		$$
		Note that while \eqref{e:eigen-2} is proven for $u \in C_c^\infty(D_0)$ in \cite{LS10}, its extension to $\text{Lip}_c(D_0)$ is straightforward. 
		For all $x \in D_0$ and $w\in \R^d$ with $|w|=1$, we have $d^w_{D_0}(x) \le \text{\rm diam}(D_0) \le R_0$ so that $1/M_\alpha(x)^\alpha \ge c_3 R_0^{-\alpha}$ for $c_3>0$ depending only on $d$ and $\alpha$. Therefore, by \eqref{e:eigen-2}, we deduce that there exists $c_4>0$ depending only on $d,\alpha$ and $R_0$ such that for all  $u \in$ \text{\rm Lip}$_c(D_0)$,
		\begin{align*}
			\frac12	\iint_{D_0 \times D_0} \frac{(u(x)-u(y))^2}{|x-y|^{d+\alpha}} dxdy \ge c_4  \lVert u \rVert_{L^2(D_0)}^2.
		\end{align*}
		Combining this with \eqref{e:eigen-1}, we arrive at the desired result. 
	\end{proof} 
	
	\begin{example}\label{ex:alpha<1}
	 In this example, we provide a counterexample  showing that the conclusion of Lemma \ref{l:largetime-2} is not applicable when $\alpha\le 1$. Suppose that $\alpha \le 1$, $D=\cup_{n\ge 4} B(2^n \e_d,2)$, $\kappa(x)=|x|^{-1}$ and $\sB(x,y)=1$ for $x,y \in D$. By \cite[Theorems 1.1 and  2.4]{BBC}, there exists a sequence $(f_n)_{n \ge 4}$ in  $C_c^\infty(B(0,2))$ such that
	\begin{align}\label{e:censored-1}
	&	\lim_{n\to \infty}\left[  \iint_{B(0,2) \times B(0,2)} \frac{((f_n(x)-1)-(f_n(y)-1))^2}{|x-y|^{d+\alpha}} dxdy  + \lVert f_n - \1_{B(0,2)} \rVert_{L^2(B(0,2))}^2 \right]\nn\\
		&=	\lim_{n\to \infty}\left[  \iint_{B(0,2) \times B(0,2)} \frac{(f_n(x)-f_n(y))^2}{|x-y|^{d+\alpha}} dxdy  + \lVert f_n - \1_{B(0,2)} \rVert_{L^2(B(0,2))}^2 \right]=0.
	\end{align}
	Define for $n \ge 4$,
	\begin{align*}
		 u_n(x) =1 \wedge (f_n(x-2^n\e_d) \vee 0).
	\end{align*}
	Note that $u_n \in \text{\rm Lip}_c(B(2^n\e_d,2))$ and  by \eqref{e:censored-1},  
	\begin{align}\label{e:counterexample-1}
	\liminf_{n \to \infty} 	\,\lVert u_n \rVert_{L^2(B(2^n\e_d,2))}^2 = \lim_{n\to \infty} \lVert 1 \wedge (f_n \vee 0) \rVert_{L^2(B(0,2))}^2 = m_d(B(0,2)).
	\end{align}
Further, for all $n \ge 4$, we have 
	\begin{align*}
		\sE^\kappa(u_n,u_n) & \le \frac12\iint_{B(2^n\e_d,2) \times B(2^n\e_d,2)} \frac{(u_n(x)-u_n(y))^2}{|x-y|^{d+\alpha}} dxdy \\
		&\quad +  \iint_{B(2^n\e_d,2) \times B(2^n\e_d,2^{n-1}-4)^c} \frac{u_n(x)^2}{|x-y|^{d+\alpha}} dxdy\\
		&\quad + \int_{B(2^n\e_d,2)} u_n(x)^2 |x|^{-1} dx\\
		&=:I_{n,1}+I_{n,2}+I_{n,3}.
	\end{align*}
	Since $u_n^2 \le 1$, we have $I_{n,2}\le  \int_{B(0,2)}dx \int_{B(0, 2^{n-1}-4)^c} (|y|/2)^{-d-\alpha}dy \to 0 $  and  $	I_{n,3} \le  (2^n-2)^{-1}\int_{B(0,2)}dx \to 0$ as $n \to \infty$.
	Moreover, by using \eqref{e:censored-1}, we see that
	\begin{align*}
		I_{n,1} & = \iint_{B(0,2) \times B(0,2)} \frac{((1 \wedge (f_n(x) \vee 0))-(1 \wedge (f_n(y) \vee 0)))^2}{|x-y|^{d+\alpha}} dxdy \\
		& \le \iint_{B(0,2) \times B(0,2)} \frac{(f_n(x)-f_n(y))^2}{|x-y|^{d+\alpha}} dxdy \to 0 \quad \text{as $n \to \infty$.}
	\end{align*}
	Hence, $\lim_{n\to \infty} \sE^\kappa(u_n,u_n)=0$. By combining this with \eqref{e:counterexample-1}, we deduce that  $\inf_{n\ge 1}\{\sE^\kappa(u_n,u_n)/ \lVert u_n \rVert^2_{L^2(B(2^n\e_d,2))}\} =0$, leading to the failure of the conclusion of Lemma \ref{l:largetime-2}.
	\end{example}

	In order to get a counterpart of Lemma \ref{l:largetime-2} in case $\alpha\le 1$, we consider the following 
 additional condition on $\kappa$:
	
	\medskip

	\noindent \hypertarget{K2}{{\bf (K2)}} If $\alpha\le 1$, then there exist constants $\wh r \in(0, \wh{R})$ and $C_4>0$ such that for every bounded connected component $D_0$ of $D$,
	\begin{align*}
		\kappa(x) \ge C_4 \quad \text{for all $x \in D_0$ with $\delta_{D_0}(x)<\wh r$.}
	\end{align*} 
	
	\begin{lemma}\label{l:largetime-3}
		Suppose that $\alpha\le 1$ and  \hyperlink{K2}{{\bf (K2)}} holds.	For every $R_0>0$, there exists  a constant  $\lambda_0=\lambda_0(R_0)>0$   such that if $D_0$ is a bounded connected component of $D$ with $\text{\rm diam}(D_0) \le R_0$, then for all $u \in$ \text{\rm Lip}$_c(D_0)$,
		\begin{align*}
			\sE^\kappa(u,u) \ge \lambda_0 \lVert u \rVert_{L^2(D_0)}^2.
		\end{align*}
	\end{lemma}
	\begin{proof}  Let 
		$$
		A_0:=\left\{ x \in D_0: \delta_{D_0}(x)<\wh r\right\},
		$$
		where $\wh r \in (0,\wh R)$ is the constant in  \hyperlink{K2}{{\bf (K2)}}.		Since $D_0$ is a bounded Lipschitz domain, $m_d(D_0) \ge c_1 \wh R^d$ and $m_d(A_0) \ge c_2\wh rm_{d-1}(\partial D_0)$. Using this and the  isoperimetric inequality, we get that 
		\begin{align}\label{e:eigen-area}
			m_d(A_0) \ge c_3\wh r m_{d}(D_0)^{(d-1)/d} \ge c_4 \wh r\wh R^{d-1}
		\end{align}
		for a constant $c_4>0$ depending only on $d$ and $\Lambda_0$. 
		Let $x_0 \in \overline D$ be such that $D_0 \subset B_{\overline D}(x_0, 2R_0)$. By Proposition \ref{p:ndl} (with $b=1/2$), there exists $c_5>0$ depending on $R_0$ such that for all $t \in [(4R_0)^\alpha, (8R_0)^\alpha ]$ and $y,z \in D_0$,
		\begin{equation}\label{e:largetime-onestep}
			\overline	p(t,z,y)  \ge 	\overline	p^{B_{\overline D}(x_0,16R_0)}(t,z,y)  \ge  c_5t^{-d/\alpha} .
		\end{equation}
		
		Set $t_0:= (8R_0)^\alpha$. By \eqref{e:Feymann-Kac},  \hyperlink{K2}{{\bf (K2)}}, \eqref{e:largetime-onestep} and \eqref{e:eigen-area}, we have for all $x \in D_0$,
		\begin{equation}\label{e:largetime-onestep-1}
			\begin{split} 
				P^{\kappa,D_0}_{t_0} \1_{D_0}(x)&=  \E_x \left[\,\exp\bigg( - \int_0^{t_0} \kappa(\overline Y_s) ds \bigg) : t_0< 
				\bar \tau_{D_0} \right]\\
				&\le  \E_x \left[\,\exp\bigg( - C_4\int_0^{t_0} \1_{A_0}(\overline Y_s) ds \bigg)   \right]\\
				&\le   \exp\bigg( - C_4\int_{t_0/2^\alpha}^{t_0}  \int_{A_0} \overline p(s, x, y )dy ds \bigg)  \\
				&\le   \exp\bigg( -(1-2^{-\alpha}) C_4c_4c_5 t_0^{1-d/\alpha } 
				\wh r \wh R^{d-1}
				 \bigg)  =: \eps \in (0,1).
			\end{split} 
		\end{equation}
		Let $t > 2t_0$ and  $n \ge 1$ be such that   $t \in ((n+1)t_0, (n+2)t_0)$. Using the semigroup property, Proposition \ref{p:upper-heatkernel} and \eqref{e:largetime-onestep-1}, we get that  for all  $x,y \in D_0$,
		\begin{align*}
			p^{\kappa, D_0}(t,x,y) &=	\int_{D_0} \cdots \int_{D_0}	p^{\kappa, D_0}(t- nt_0,x,z_1) \\
			&\qquad\qquad \qquad  \times 	p^{\kappa, D_0}(t_0,z_1,z_2) \cdots p^{\kappa, D_0}(t_0,z_{n},y)\,dz_1 \cdots dz_{n} \\
			&\le c_6 t_0^{-d/\alpha} \bigg( \sup_{v \in D_0} 	\int_{D_0}p^{\kappa,D_0}(t_0,v,z)dz \bigg)^{n}\\
			& \le c_6 t_0^{-d/\alpha} \eps^{n} \le  c_6 t_0^{-d/\alpha}\eps^{-2} e^{-|\log \eps|\, t/t_0}.
		\end{align*}
		Comparing this with  \eqref{e:largetime-2} and letting $t \to \infty$, we conclude that 
		\begin{align*}
			\inf\left\{ \sE^\kappa(u,u):  u \in \text{\rm Lip}_c(D_0), \, \lVert u \rVert_{L^2(D_0)}=1\right\} \ge |\log \eps|/t_0,
		\end{align*}
		which yields the desired result.
	\end{proof}

	\begin{prop}\label{p:upper-heatkernel-3}
		 In addition to 
		 \hyperlink{B1}{{\bf (B1)}}, \hyperlink{B2-a}{{\bf (B2-b)}}, \hyperlink{B2-b}{{\bf (B2-b)}} and \hyperlink{K1}{{\bf (K1)}}, 
		we assume that, when  $\alpha\le 1$, 
		 \hyperlink{K2}{{\bf (K2)}}  holds.  
			Let $x_0 \in \overline D$ and $R_0>0$. There exist constants   $C=C(R_0)>0$ and $\lambda=\lambda(R_0)>0$ independent of $x_0$ such that
		\begin{equation}\label{e:upper-heatkernel-3}
			p^{\kappa,  B_D(x_0, R_0)}(t,x,y) \le Ce^{-\lambda t}, \quad t\ge3, \; x,y \in B_D(x_0,R_0).
		\end{equation}
		In particular, when $D$ is bounded, there exist constants  $C=C(\text{\rm diam}(D))>0$ and $\lambda=\lambda(\text{\rm diam}(D))>0$ such that
		\begin{equation}\label{e:upper-heatkernel-bounded}
			p^\kappa(t,x,y) \le Ce^{-\lambda t}, \quad t\ge3, \; x,y \in D.
		\end{equation}
	\end{prop}
	\begin{proof}
 	\eqref{e:upper-heatkernel-bounded} directly follows from \eqref{e:upper-heatkernel-3} by setting $R_0 = 2\,\text{diam}(D)$. We prove \eqref{e:upper-heatkernel-3}. Set $B:=B_D(x_0,R_0)$ and  $B':=B_D(x_0,R_0+\wh R)$. 
		We consider the following two cases separately.

		\smallskip
		
		Case 1:   $\partial B(x_0,R_0+\wh R) \cap \overline D \ne \emptyset$.   
	Let $z_0 \in 	\partial B(x_0,R_0+\wh R) \cap \overline D$. 	For all $x \in B$ and $z \in B_D(z_0,  \wh R)$, we have  $|x-z| \le
2(R_0+\wh R)=:r_0$. Hence, using Proposition \ref{p:ndl} (with $b=1/2$) and \eqref{e:VD}, we get that for all $x \in B$,
		\begin{align}\label{e:HKE-largetime}
			\begin{split} 
			\int_{B}	p^{\kappa, B}((2r_0)^\alpha,x,z)dz & \le  \int_{B}	\overline p((2r_0)^\alpha,x,z)dz\\
			& \le 1-   \int_{B_D(z_0,  \wh R)}	\overline p^{B_{\overline D}(x, 4r_0)}((2r_0)^\alpha,x,z)dz  \\
			& \le 1-c_1 (2r_0)^{-d}m_d(B_D(z_0,  \wh R))\le c_2,
			\end{split} 
		\end{align}
		where $c_2\in (0,1)$ is a constant independent of $x_0$.
		
		Let $t>2(2r_0)^\alpha$ and $n_0\ge 1$ be such that $t/(2r_0)^\alpha \in [n_0+1, n_0+2)$. By using the semigroup property, Proposition \ref{p:upper-heatkernel-2} and \eqref{e:HKE-largetime}, we get that for all  $x,y \in B$,
		\begin{align*}
			p^{\kappa, B}(t,x,y) &=	\int_{B} \cdots \int_{B}	p^{\kappa, B}(t- n_0(2r_0)^\alpha,x,z_1) \\
			&\qquad\qquad\qquad  \times 	p^{\kappa, B}((2r_0)^\alpha,z_1,z_2) \cdots p^{\kappa, B}((2r_0)^\alpha,z_{n_0},y)\,dz_1 \cdots dz_{n_0} \\
			&\le c_3 r_0^{-d} \bigg( \sup_{v \in B} 	\int_{B}p^{\kappa, B}((2r_0)^\alpha,v,z)dz \bigg)^{n_0}\\
			& \le c_3 r_0^{-d} c_2^{n_0} \le c_4 e^{-|\log c_2|\, t/(2r_0)^\alpha}.
		\end{align*}
		Since $c_2,c_4$ and $r_0$ are independent of $x_0$, combining this with Proposition \ref{p:upper-heatkernel-2}, we arrive at the result.	
		
		\smallskip

		Case 2:  $\partial B(x_0,R_0+\wh R) \cap \overline D = \emptyset$.
		 Since  $\partial B(x_0,R_0+\wh R) \cap \overline D = \emptyset$, we have	 $B' = \cup_{i=1}^N D_i$ for some bounded connected components $D_i$, $1\le i\le N$, of $D$. Note that  $m_d(D_i) \ge c_5 \wh R^d$ for all $ 1\le i \le N$. Hence,
		\begin{align}\label{e:largetime-N-bound}
			N \le \frac{m_d(B)}{\min\{m_d(D_i): 1\le i\le N\}} \le c_6 (1+R_0/\wh R)^d.
		\end{align}
		Furthermore, since we have assumed that \hyperlink{K2}{{\bf (K2)}}  holds if $\alpha\le 1$, by Lemmas \ref{l:largetime-2} and \ref{l:largetime-3}, 
		 there exists $\lambda_0=\lambda_0(R_0+\wh R)>0$  such that for all $u \in$ \text{\rm Lip}$_c( B')$,
		\begin{align*}
			\lambda_0 \lVert u  \rVert_{L^2(D_i)}^2 \le \sE^\kappa(u,u) \quad \text{for all} \;\, 1\le i\le N. 
		\end{align*}
		By \eqref{e:largetime-N-bound}, it follows that for all $u \in$ \text{\rm Lip}$_c( B')$,
		\begin{align*}
			\sE^\kappa(u,u) \ge N^{-1}\lambda_0\sum_{i=1}^N \lVert u  \rVert_{L^2(D_i)}^2 = N^{-1}\lambda_0 \lVert u  \rVert_{L^2(B')}^2 \ge c_6^{-1} (1+R_0/\wh R)^{-d}\lambda_0 \lVert u  \rVert_{L^2(B')}^2.
		\end{align*}		
		Using this, from Proposition \ref{p:largetime-1}, we conclude  that for all $t \ge 3$ and $x,y\in B$,
		\begin{align*}
			p^{\kappa,  B}(t,x,y) \le p^{\kappa,  B'}(t,x,y) \le c_7m_d(B')e^{- c_6^{-1} (1+R_0/\wh R)^{-d}\lambda_0 (t-2)}.
		\end{align*}
		Since $c_6,c_7$ and $\lambda_0$ are independent of $x_0$, we get \eqref{e:upper-heatkernel-3}.
		
		\smallskip

		The proof is complete. 
	\end{proof}
	
	\begin{remark}
The additional assumption \hyperlink{K2}{{\bf (K2)}} is only used in Case 2 in the proof of Proposition \ref{p:upper-heatkernel-3},
 whereas Case 1 remains valid independently of it.
	\end{remark}

	When $D$ has only finitely many components, we can drop the assumption 	\hyperlink{K2}{{\bf (K2)}} from Proposition \ref{p:upper-heatkernel-3}.
	
	\begin{prop}\label{p:upper-heatkernel-3+}
		 Suppose that $D$ has only finitely many components.	Let $x_0 \in \overline D$ and $R_0>0$. There exist constants   $C=C(D,R_0)>0$ and $\lambda=\lambda(D,R_0)>0$ independent of $x_0$ such that
		\begin{equation}\label{e:upper-heatkernel-3+}
			p^{\kappa,  B_D(x_0, R_0)}(t,x,y) \le Ce^{-\lambda t}, \quad t\ge3, \; x,y \in B_D(x_0,R_0).
		\end{equation}
		In particular, when $D$ is bounded, there exist constants  $C=C(D)>0$ and $\lambda=\lambda(D)>0$ such that
		\begin{equation}\label{e:upper-heatkernel-bounded+}
			p^\kappa(t,x,y) \le Ce^{-\lambda t}, \quad t\ge3, \; x,y \in D.
		\end{equation}
	\end{prop}
	\begin{proof} 
		\eqref{e:upper-heatkernel-bounded+} is a direct consequence of \eqref{e:upper-heatkernel-3+}. We prove \eqref{e:upper-heatkernel-3+}. Set $B:=B_D(x_0,R_0)$ and  $B':=B_D(x_0,R_0+\wh R)$.  If $B'=D$, then $D$ is bounded so that the result follows from Proposition \ref{p:largetime-1} (with $A=D$). 
		If   $\partial B(x_0,R_0+\wh R) \cap \overline D \ne \emptyset$, then  by applying the arguments for Case 1 in  Proposition \ref{p:upper-heatkernel-3}, we get the result.

		Suppose that $B'\neq D$ and   $\partial B(x_0, R_0+\wh R) \cap \overline D = \emptyset$. Write		 $D= \cup_{i=1}^N D_i$ where $D_i$, $1\le i\le N$, are connected components of $D$. For each $i$, 
		we either have $D_i \subset B(x_0, R_0+\wh R)$ or $D_i \cap B(x_0, R_0+\wh R) = \emptyset$. Since $ B'\ne D$, there exists at least one component $D_{i_0}$ such that $D_{i_0} \cap B(x_0, R_0+\wh R) = \emptyset$. Pick such an $i_0$ and let $z_0 \in \partial D_{i_0}$ be such that $|x_0-z_0|=$ dist$(x_0, D_{i_0})$.  Note that  $x_0$ belongs to a bounded component $D_{i_1}$, $i_1 \ne i_0$, in this case and the distance between $D_{i_0}$ and $D_{i_1}$ is bounded above by a positive constant  since $D$ has a finite number of connected components.	Hence,  there exists $c_1=c_1(D,R_0)>0$ independent of $x_0$ such that  $|x_0-z_0| <c_1$. Now, by repeating the arguments for Case 1 in Proposition \ref{p:upper-heatkernel-3}, we obtain the desired result. 
	\end{proof}

	For the last result in this subsection,
	we need a weaker form of  \hyperlink{UBS}{{\bf (UBS)}}:
	
	\medskip
	
	\noindent \hypertarget{IUBS}{{\bf (IUBS)}} There exists  $C>0$ such that  for  a.e. $x,y\in  D$,
	\begin{align*}
		\sB(x,y)\le \frac{C}{r^d}\int_{B(x,r)} \sB(z,y)dz \quad 
		\text{whenever}\;\, 0< r \le \frac12 (|x-y| \wedge \delta_D(x) \wedge \wh R).
	\end{align*}

	Condition \hyperlink{IUBS}{{\bf (IUBS)}} implies that \eqref{e:UJS} holds for a.e. $x,y\in  D$  and $0<r \le  2^{-1}(|x-y|\wedge \delta_D(x)\wedge \wh R)$. Using this, \eqref{e:stochastic-complete},  \eqref{e:VD},   the L\'evy system formula
	\eqref{e:Levysystem-Y-kappa} and Propositions \ref{p:upper-heatkernel-2},  \ref{p:ndl2} and \ref{p:E2}, one can repeat the arguments in the proof of \cite[Theorem 4.3]{CKSV23} and obtain 
	
	\begin{thm}\label{t:phi2}
		Suppose that $\sB$ 	satisfies \hyperlink{IUBS}{{\bf (IUBS)}}.  
		For every $R_0>0$,	there exist constants $\eps>0$ and $C,K\ge1$ depending on $R_0$
		such that for all  $x\in  D$, $0<r< \delta_D(x)\wedge R_0$, $t_0\ge 0$, and any non-negative function $u$ 
		on $(0,\infty)\times D$ which is caloric on $(t_0, t_0+4\eps  r^\alpha] \times B(x, r)$ with respect to $Y^\kappa$, we have	
		\begin{equation*}
			\sup_{(t_1,y_1)\in Q_{-}}u(t_1,y_1)\le C \inf_{(t_2, y_2)\in Q_{+}}u(t_2, y_2),	
		\end{equation*}	
		where $Q_-=[t_0+\eps r^\alpha, t_0+2\eps  r^\alpha]\times B(x, r/K)$ and 
		$Q_+=[t_0+3\eps  r^\alpha, t_0+4\eps r^\alpha]\times B(x, r/K)$.
	\end{thm}

	\subsection{Interior estimates of the Green function of $Y^\kappa$}\label{s-int-green}
	For an open set $A \subset D$, we define	\begin{align*}		G^{\kappa,A}(x,y)=\int_0^\infty p^{\kappa,A}(t,x,y)  dt, \quad x,y \in A.	\end{align*}
	When $G^{\kappa, A}(\cdot, \cdot)$ is not identically infinite, it is called the Green function of $Y^\kappa$ in $A$. 	Note that	by  Propositions \ref{p:upper-heatkernel-2} and \ref{p:largetime-1}, for any bounded open subset $A$ of $D$,  $G^A(x,y)<\infty$ for all $x,y \in A$, $x\ne y$. 
	We extend $G^{\kappa, A}$ to a function on $(D \cup \{\partial\})  \times (D \cup \{\partial\}) $ by letting $G^{\kappa,A}(x,y)=0$ if $x \in (D \cup \{\partial\})  \setminus A$ or $y \in (D \cup \{\partial\})  \setminus A$. We  denote  $G^{\kappa,D}(x,y)$ by  $G^\kappa(x,y)$.

	The proof of the next proposition uses the notion of  harmonic  and regular harmonic functions so we recall these definitions.
	
	\begin{defn}\label{df:harmonic}
	A Borel function $f:D\to [0,\infty]$ is said to be \emph{harmonic} in an open set $V\subset D$ with respect to the process $Y^{\kappa}$ if 
$f$ is finite on $V$	and, 
	for every open $U\subset \overline{U}\subset V$,	$$		f(x)=\E_x[f(Y^{\kappa}_{\tau_U})], \quad \text{for all } x\in U.	$$
		The function $f$ is said to be \emph{regular harmonic} in $V$ with respect to $Y^{\kappa}$ if
		 $f$ is finite on $V$	and, 	$$		f(x)=\E_x[f(Y^{\kappa}_{\tau_V})], \quad \text{for all } x\in V.	$$ \end{defn}
	
	It follows from the strong Markov property that a regular harmonic function is harmonic. Further, if $A\subset D$ is an open set, then  $G^{\kappa, A}(\cdot, y)$ is harmonic in $A\setminus \{y\}$ and regular harmonic in $A\setminus B(y, \delta)$ for any $\delta>0$, see e.g.~\cite[Section 2, Remark 2.3]{KSV20}.
	
	\begin{prop}\label{p:green-lower-bound}
		Let 
		$R_0>0$. For any $\eps \in (0,1)$,	there exists  $C=C(R_0,\eps)>0$ 
		such that for all $x_0 \in \overline D$, $R \in (0,R_0]$ and $x,y\in B_D(x_0, R/8)$ with $|x-y|\le \eps^{-1}(\delta_D(x)\wedge \delta_D(y))$,
		$$
		G^{\kappa, B_D(x_0,R)}(x,y)\ge C |x-y|^{-d+\alpha}.
		$$
	\end{prop}
	\begin{proof} 
		Let $x_0 \in \overline D$, $R \in (0,R_0]$ and
		$x,y\in B_D(x_0, R/8)$ with $|x-y|\le \eps^{-1}(\delta_D(x)\wedge \delta_D(y))$. Without loss of generality, we assume that $\delta_D(x) \le \delta_D(y)$. Write $B:=B_D(x_0,R)$. We consider two different cases separately.
		
		\smallskip

		\noindent Case 1: $|x-y| < \delta_D(y)/2$. Since $|x-y|<R_0/4$,  using Proposition \ref{p:ndl2} with $b=5/6$ (see Remark \ref{r:strong-Feller}), we obtain
		\begin{align*}
			&	G^{\kappa, B} (x,y) \ge 	G^{\kappa, B(y,2|x-y|)} (x,y) \\
			&\ge \int_{(6/5)^{\alpha}|x-y|^\alpha}^{(5/3)^{\alpha}|x-y|^\alpha} p^{\kappa, B(y,2|x-y|)} (t,x,y)dt\ge c_1|x-y|^{-d+\alpha}. 
		\end{align*}
		
		\noindent Case 2: $\delta_D(y)/2 \le |x-y| \le  \eps^{-1}\delta_D(x)$. Then  $y \notin B(x,\delta_D(x)/4)$ since $\delta_D(x) \le \delta_D(y)$. Hence  $G^{\kappa, B}(\cdot, y)$ is regular harmonic in $B(x, \delta_D(x)/4)$ and we get
		\begin{align}\label{e:interior-green-1}
			G^{\kappa, B} (x,y) &\ge \E_x \left[ 	G^{\kappa, B} (Y^\kappa_{
				\tau_{B(x,\delta_D(x)/4)}},y) : Y^\kappa_{
				\tau_{B(x,\delta_D(x)/4)}} \in B(y, \delta_D(y)/4) \right]\\
			& \ge \P_x \left(Y^\kappa_{
				\tau_{B(x,\delta_D(x)/4)}} \in B(y, \delta_D(y)/4) \right)   \inf_{w \in B(y, \delta_D(y)/4)} G^{\kappa, B}(w,y).\nn
		\end{align}
		By Case 1, we have
		\begin{align}\label{e:interior-green-2}
			\inf_{w \in B(y, \delta_D(y)/4)} G^{\kappa, B}(w,y) \ge c_1 (\delta_D(y)/4)^{-d+\alpha} \ge c_1 (|x-y|/2)^{-d+\alpha}.
		\end{align}
		On the other hand, note that for any $z\in B(x, \delta_D(x)/4)$ and $w \in B(y, \delta_D(y)/4)$, we have 
		$|z-w|< |x-y| + (\delta_D(x)+\delta_D(y))/4 \le 2|x-y|$ and $\delta_D(z) \wedge \delta_D(w) \ge 3\delta_D(x)/4 \ge 3\eps |x-y|/4$.
		Hence, by \hyperlink{B2-b}{{\bf (B2-b)}}, there exists $c_2>0$ depending only on $\eps$ such that  for all $z\in B(x, \delta_D(x)/4)$ and $w \in B(y, \delta_D(y)/4)$, 
		\begin{align}\label{e:interior-green-3}
			\sB(z, w) |z-w|^{-d-\alpha} \ge c_2 |x-y|^{-d-\alpha}.
		\end{align} 
		Using the L\'evy system formula \eqref{e:Levysystem-Y-kappa}, \eqref{e:interior-green-3} and Proposition \ref{p:E2}, 
		since $\delta_D(y) \ge \delta_D(x) \ge \eps |x-y|$, we obtain
		\begin{equation}\label{e:interior-green-4}
			\begin{split}
				&	\P_x \left(Y^\kappa_{
					\tau_{B(x,\delta_D(x)/4)}} \in B(y, \delta_D(y)/4) \right) \\
				& = \E_x \left[ \int_0^{
					\tau_{B(x,\delta_D(x)/4)}} \int_{B(y, \delta_D(y)/4)} \frac{\sB(Y^\kappa_s, w)}{|Y^\kappa_s - w|^{d+\alpha}} dw ds \right] \\
				&\ge c_2|x-y|^{-d-\alpha} m_d(B(y,\delta_D(y)/4)) \E_x\big[ 
				\tau_{B(x,\delta_D(x)/4)}\big]\\
				& \ge c_3\delta_D(x)^\alpha \delta_D(y)^d |x-y|^{-d-\alpha}\ge c_3\eps^{d+\alpha}.
			\end{split}
		\end{equation}
		Combining \eqref{e:interior-green-1} with \eqref{e:interior-green-2}  and \eqref{e:interior-green-4},
		we arrive at $	G^{\kappa, B} (x,y) \ge 2^{d-\alpha}c_1c_3\eps^{d+\alpha}|x-y|^{-d+\alpha}$. The proof is complete. 
	\end{proof}

	Using Propositions \ref{p:upper-heatkernel-2} and \ref{p:upper-heatkernel-3}, we get 
	
	\begin{prop}\label{p:green-upper-bound} 	In addition to 
 \hyperlink{B1}{{\bf (B1)}}, \hyperlink{B2-a}{{\bf (B2-b)}}, \hyperlink{B2-b}{{\bf (B2-b)}} and \hyperlink{K1}{{\bf (K1)}}, 
we assume that, when  $\alpha\le 1$, 
		\hyperlink{K2}{{\bf (K2)}}  holds.  For any $R_0>0$, there exists  $C=C(R_0)>0$ such that 		
		\begin{align*} G^{\kappa, B_D(x_0,R_0)}(x,y) \le C|x-y|^{-d+\alpha} \quad \text{for all} \;\, x_0 \in \overline D 			\text{ and }  x,y \in B_D(x_0,R_0).		\end{align*}\end{prop}

	When  $D$ has only finitely many components, we  obtain the following upper estimates for the Green function, with additional dependency on $D$, from Propositions \ref{p:upper-heatkernel-2} and \ref{p:upper-heatkernel-3+}. 
	\begin{prop}\label{p:green-upper-bound+}  Suppose that   $D$ has only finitely many components. For any $R_0>0$, there exists  $C=C(D,R_0)>0$ such that 	
	\begin{align*} 
	G^{\kappa, B_D(x_0,R_0)}(x,y) \le C|x-y|^{-d+\alpha} \quad \text{for all} \;\, x_0 \in \overline D 			\text{ and }  x,y \in B_D(x_0,R_0).		
	\end{align*}
	\end{prop}

	When $D$ is bounded, $D$ has only finitely many
	connected components. Hence, 	 by taking $R= R_0=9\,\text{diam}(D)$ in Propositions \ref{p:green-lower-bound},  \ref{p:green-upper-bound} and \ref{p:green-upper-bound+}, we obtain the following corollary.
	\begin{corollary}\label{c-int-green-bounded}
		Suppose that $D$ is bounded. Then  there exists $C=C(D)>0$ such that for all $x,y \in D$,
		\begin{align}\label{e:Green-bounded}
			G^\kappa(x,y) \le C|x-y|^{-d+\alpha}, 
		\end{align}
		and for any  $\eps \in (0,1)$, there exists a constant $C(\eps)>0$ such that for all $x,y\in D$ with $|x-y|\le \eps^{-1}(\delta_D(x)\wedge \delta_D(y))$,
		\begin{align*}
			G^\kappa(x,y) \ge C(\eps)|x-y|^{-d+\alpha}.
		\end{align*}
		 Moreover, if we assume that, in addition to 
 \hyperlink{B1}{{\bf (B1)}}, \hyperlink{B2-a}{{\bf (B2-b)}}, \hyperlink{B2-b}{{\bf (B2-b)}} and \hyperlink{K1}{{\bf (K1)}}, 
when  $\alpha\le 1$,		\hyperlink{K2}{{\bf (K2)}}  holds, then the constant $C$ in \eqref{e:Green-bounded} depends on $D$ only through $\Lambda_0, \wh R$ and \text{\rm diam}$(D)$.
	\end{corollary}


\section{Analysis of the operators $L^\sB_\alpha$ and $L^\kappa$}\label{ch:operator}

In this section we first analyze the operators $L^{\sB}_{\alpha}$ and $L^\kappa$, 
and prove a Dynkin-type formula. This analysis requires the assumption \hyperlink{B3}{{\bf (B3)}}. Then we introduce two new assumptions on the function $\sB$, construct a barrier function $\psi^{(r)}$ and establish an upper bound on 
$L^{\sB}_{\alpha} \psi^{(r)}$. For the set $D$ we keep assuming that it is 
a Lipschitz open set  with localization radius $\wh{R}$ and Lipschitz constant $\Lambda_0$. 
\smallskip

Consider a non-local operator $(L^\sB_\alpha, \sD(L^\sB_\alpha))$ of the form
	\begin{align}\label{e:def-L-alpha}
		L^\sB_\alpha f(x)&=\text{p.v.}\int_D (f(y)-f(x)) \frac{\sB(x,y)}{|x-y|^{d+\alpha}}dy, \quad x \in D,
	\end{align}
where  $\sD(L^\sB_\alpha)$ consists of  all functions $f:D \to \R$ for which the above principal value integral makes sense. 
Recall that $\kappa$ is a  non-negative Borel  function on $D$ satisfying \hyperlink{K1}{{\bf (K1)}}. 
We define an  operator $(L^\kappa, \sD(L^\sB_\alpha))$  by 
	\begin{align}\label{e:def-operator}
		L^\kappa f(x)=L^\sB_\alpha f(x) - \kappa(x)f(x), \quad  x \in D.
	\end{align}

\subsection{Dynkin-type formula}

In this subsection, in addition to \hyperlink{B1}{{\bf (B1)}}, \hyperlink{B2-a}{{\bf (B2-a)}}, \hyperlink{B2-b}{{\bf (B2-b)}}, 
we assume that $\sB$ satisfies the following assumption:

\medskip
\noindent\hypertarget{B3}{{\bf (B3)}} If $\alpha \ge 1$, then there exist constants $\theta_0>\alpha-1$ and $C_5>0$ such that
	\begin{align}\label{e:ass-B3}
		|\sB(x,x)-\sB(x,y)| \le C_5 \bigg(\frac{|x-y|}{\delta_D(x) \wedge \delta_D(y) \wedge  \wh R} \bigg)^{\theta_0} 
			\quad \text{for all }  x,y \in D.
	\end{align}

\medskip

For an open set $U \subset \R^d$, denote by $C^{1,1}(U)$ the family of all locally  $C^{1,1}$ functions on $U$, and by $C^{1,1}_c(U)$ the family of all functions in $C^{1,1}(U)$ with compact support in $U$. Then $C^{1,1}_c(U)$ is a normed space   equipped with the norm  
	$$
	\lVert u \rVert_{C_c^{1,1}(U)}:=\lVert u \rVert_{L^\infty(U)} + \lVert \nabla u \rVert_{L^\infty(U)} 
	+ \sup_{x,y \in U,\, x\ne y} \frac{|\nabla u(x)-\nabla u(y)|}{|x-y|}. 
	$$

For the closed set $\overline U \subset \R^d$, define
	\begin{align*}
		C^{1,1}(\overline U) :=\left\{u: \overline U \to \R \,: \begin{array}{ll}
		\text{There exist an open set $V$ with $ \overline U \subset V$}\\
		\text{and $f \in C^{1,1}(V)$ such that $u=f$ on $\overline U$}
		\end{array}  \right\}.
	\end{align*}
We also let
\begin{align*}
	C^{1,1}_c(D;\R^d) :=\big\{u: D \to \R \,:\, \text{There exists $f \in C^{1,1}_c(\R^d)$ such that $u=f$ on $D$}  \big\}.
\end{align*}

\begin{prop}\label{p:generator-C11}
	Let $(\sA^\kappa, \sD(\sA^\kappa))$ be the $L^2$-generator of $(\sE^\kappa,\sF^\kappa)$. 
	Then $C^{1,1}_c(D;\R^d) \subset \sD(\sA^\kappa) \cap \sD(L^\sB_\alpha)$, and
	 for all  $u \in C^{1,1}_c(D;\R^d)$,
	\begin{equation}\label{e:generator-C11}
		\lVert L^\kappa u \rVert_{L^\infty({\rm supp}(u))}<\infty
	\end{equation}
	and	
	$\sA^\kappa u = L^\kappa u$  a.e. in $D$.
\end{prop}
\begin{proof} 
For any $u \in C^{1,1}_c(\R^d)$ and $x,y \in \R^d$, by the mean value theorem, there exists $a \in [0,1]$ such that 
$u(y)-u(x)=\nabla u ( ax + (1-a)y) \cdot (y-x)$. Hence,  
	\begin{align}\label{e:generator-C11-1}
		|u(y)-u(x) - \nabla u(x) \cdot (y-x)| \le (1-a)\lVert  u \rVert_{C^{1,1}_c(\R^d)} |y-x|^2  
		\le \lVert  u \rVert_{C^{1,1}_c(\R^d)} |y-x|^2.
	\end{align}
By repeating the arguments of \cite[Proposition 4.2 and Corollary 4.4]{KSV22b}, 
 using \eqref{e:generator-C11-1} instead of  Taylor's theorem, we obtain the desired result. 
\end{proof}

\begin{prop}\label{p:Dynkin-martingale}
	Let $U$ be an open set with $\overline U \subset D$. For any $u \in C^{1,1}_c(D)$ and $x \in U$, 
	$$
		M^{[u]}_t:= u(Y^\kappa_{t \wedge \tau_U})  - u(Y^\kappa_0) - \int_0^{t \wedge \tau_U} L^\kappa u(Y^\kappa_s)ds
	$$
	is a $\P_x$-martingale with respect to the filtration of $Y^\kappa$.
\end{prop}
\begin{proof} 
Let $u \in C^{1,1}_c(D)$ and  $V$ be an open set with $\overline {U \cup \text{\rm supp}(u)} \subset V \subset \overline V \subset D$. By Proposition \ref{p:generator-C11}, one can  get (see  the proofs of \cite[Corollaries 4.4-4.5]{KSV22b}) that, if  $(\sA^{\kappa, V}, {\mathcal D}(\sA^{\kappa, V}))$ is the $L^2$ generator of the semigroup of $Y^{\kappa, V}$, then
\begin{align}\label{e:killed-generator-C11} 
	C^{1, 1}_c(V)\subset  {\mathcal D}(\sA^{\kappa, V}) \;\,\text{ and }\;\, \sA^{\kappa, V}f=L^\kappa f
	 \text{ a.e. in }V  
	 \,\text{ for all } f \in C^{1, 1}_c(V).  \end{align}
It follows from Remark \ref{r:strong-Feller} that $Y^{\kappa,U}$ is strongly Feller. Hence, since $u \in C^{1,1}_c(V)$, using \eqref{e:killed-generator-C11}, one can  follow the argument in the first paragraph of  the proof of \cite[Lemma 4.6]{KSV22b} and deduce that for any $x \in U$,
$$ u(Y^{\kappa, V}_t)-u(Y^{\kappa, V}_0)-\int^{t\wedge \tau_V}_0L^\kappa u(Y^{\kappa, V}_s)ds$$ is a $\P_x$-martingale with respect to the filtration of $Y^{\kappa}$.   Since $\tau_U\le \tau_V$ and $Y^{\kappa}_t=Y^{\kappa,V}_t$ for $t<\tau_V$, by the optional stopping theorem, the assertion of the proposition follows.	
\end{proof}

\begin{prop}\label{p:Dynkin-formula}
	Let $U$ be a  bounded open set with $\overline U \subset D$. For any bounded
function $u$ on $D$ such that $u|_{\overline U} \in C^{1,1}(\overline U)$, we have
	\begin{align}\label{e:Dynkin-formula}
		\E_x \big[ u(Y^\kappa_{\tau_U})\big] = u(x) + \E_x \bigg[\int_0^{\tau_U} L^\kappa u(Y^\kappa_s)ds \bigg] 
			\quad \text{for all} \;\, x \in U.
	\end{align}
\end{prop}
\begin{proof} 
	Choose an open set $V$ of $D$ with $\overline U \subset V$ and $f\in C^{1,1}_c(D)$ such that  $f=u$ on $V$. By Proposition \ref{p:Dynkin-martingale},  we see that for all $x \in U$ and $t>0$,
	\begin{equation}\label{e:Dynkin-formula-1}
		\begin{split} 
		\E_x \left[ f(Y^\kappa_{t\wedge \tau_U})  \right] &= f(x) + \E_x \bigg[ \int_0^{t \wedge \tau_U} L^\kappa f(Y^\kappa_s)ds\bigg]\\
		 &= u(x) + \E_x \bigg[ \int_0^{t \wedge \tau_U} \left( L^\sB_\alpha f(Y^\kappa_s)  - \kappa(Y^\kappa_s)  u(Y^\kappa_s)  \right)ds\bigg].
		\end{split}
	\end{equation}
	By \eqref{e:generator-C11}, we have $\lVert L^\kappa f \rVert_{L^\infty(\overline U)}<\infty$. Letting $t\to \infty$ in \eqref{e:Dynkin-formula-1} and applying  the dominated convergence theorem, we get that for any $x \in U$,
	\begin{align}\label{e:Dynkin-formula-2}
			\E_x \left[ f(Y^\kappa_{\tau_U})  \right]= u(x) +\E_x \bigg[ \int_0^{ \tau_U} \left( L^\sB_\alpha f(Y^\kappa_s)  - \kappa(Y^\kappa_s)  u(Y^\kappa_s)  \right)ds\bigg].
	\end{align}
Let $h:=u-f$. Then $h=0$ on $V$ and $h$ is bounded. In particular, by  \hyperlink{B2-a}{{\bf (B2-a)}},
\begin{align*}
\left	\lVert  \int_{D\setminus V} \frac{h(y)\sB(\cdot ,y)}{|\cdot -y|^{d+\alpha}} dy\right \rVert_{L^\infty(\overline U)} \le C_1 \lVert h \rVert_{L^\infty(D)} \int_{B(0, \text{\rm dist}(\overline U,V^c))}^c \frac{dy}{|y|^{d+\alpha}} <\infty.
\end{align*}
 Hence, using the L\'evy system formula \eqref{e:Levysystem-Y-kappa}, we have for any $x \in U$,
\begin{equation}\label{e:Dynkin-formula-3}
	\begin{split}
		&\E_x \left[ h(Y^\kappa_{\tau_U})  \right] = 	\E_x \left[ h(Y^\kappa_{\tau_U}) : Y^\kappa_{\tau_U} \in D\setminus V \right]\\
		&= \E_x\bigg[\int_0^{\tau_U} \int_{D\setminus V}  \frac{h(y)\sB(Y_s^\kappa,y)}{|Y^\kappa_s-y|^{d+\alpha}}dyds \bigg]	= \E_x\bigg[\int_0^{\tau_U} L^\sB_\alpha h (Y^\kappa_s)ds \bigg].
	\end{split}
\end{equation}
Adding \eqref{e:Dynkin-formula-2} and \eqref{e:Dynkin-formula-3},  we conclude that  \eqref {e:Dynkin-formula} holds.
\end{proof} 

\begin{corollary}\label{c:Dynkin-local}
Let $U \subset D$ be a bounded open set and $u$ be a bounded Borel  function on $D$ such that $u|_{U} \in C^{1,1}(U)$. Suppose that either  $L^\kappa u(y) \ge 0$ in $U$ or $L^\kappa 
u(y) \le  0$ in $U$. Then  \eqref{e:Dynkin-formula} holds.
\end{corollary}
\begin{proof} 
Let $x \in U$. For $j \ge 1$, define $A_j:=\{y \in U: \delta_D(y) >2^{-j} \}$. 
Clearly, $A_j \uparrow U$. 
Hence, there exists $j_0 \ge 1$ such that $x \in A_j$ for all $j \ge j_0$. Since  $u|_{\overline{A_j}} \in C^{1,1}(\overline{A_j})$ for all $j\ge 1$, by Proposition \ref{p:Dynkin-formula}, we get that for all $j \ge j_0$,	
	\begin{align}\label{e:Dynkin-local-2}
		\E_x \big[ u(Y^\kappa_{\tau_{A_j}})\big] 
			= u(x) + \E_x \bigg[\int_0^{\tau_{A_j}} L^\kappa u(Y^\kappa_s)ds \bigg].	
	\end{align}
By the dominated convergence theorem, the left-hand side of \eqref{e:Dynkin-local-2} converges to that of \eqref{e:Dynkin-formula} as $j \to \infty$. On the other hand, since $L^\kappa u$ 
is either positive in $U$  or negative in $U$, using the monotone convergence theorem, we see that  the right-hand side of \eqref{e:Dynkin-local-2} converges to that of \eqref{e:Dynkin-formula} as $j \to \infty$. Now we arrive at the result by letting $j \to \infty$ in \eqref{e:Dynkin-local-2}.
\end{proof}

\subsection{Construction of  barrier}\label{s-barrier}

In this subsection, we introduce two new assumptions on the function $\sB$, and construct a barrier $\psi^{(r)}$.

Let $\Phi_0$ be a Borel  function on $(0,\infty)$ such that $\Phi_0(r)=1$ for $r \ge 1$ and 
\begin{align}
	c_L \bigg( \frac{r}{s}\bigg)^{\lb_0}\le 	\frac{\Phi_0(r)}{\Phi_0(s)} 
	\le c_U \bigg( \frac{r}{s}\bigg)^{\ub_0} \quad \text{for all} \;\, 0<s\le r\le 1, \label{e:scale}
\end{align} 
for some constants $\ub_0\ge\lb_0\ge0$ and $c_L,c_U>0$.   Let $\beta_0$ be the lower Matuszewska index of $\Phi_0$, see \eqref{e:Matuszewska}. 

Consider the following conditions on $\sB$. 

\medskip

\noindent\hypertarget{B4-a}{{\bf (B4-a)}} 
There exists a constant $C_6>0$ such that 
\begin{align*}
	\sB(x,y) \le C_6 \Phi_0\bigg(\frac{\delta_D(x) \wedge \delta_D(y)}{|x-y|}\bigg) \quad \text{for all }  x,y \in D.
\end{align*}

\noindent\hypertarget{B4-b}{{\bf (B4-b)}}  There exists a constant $C_7>0$ such that
\begin{align*}
	\sB(x,y) \ge C_7 \Phi_0\left(\frac{\delta_D(x) \wedge \delta_D(y)}{|x-y|}\right) 
	\quad \text{for all }  x,y \in D \text{ with } \delta_D(x) \vee \delta_D(y) \ge  \frac{|x-y|}{2}.
\end{align*}

\medskip

From now until the end of Section \ref{ch:green},
we assume that
$$
\textit{$\sB$ satisfies \hyperlink{B1}{{\bf (B1)}}, \hyperlink{B2-b}{{\bf (B2-b)}}, \hyperlink{B3}{{\bf (B3)}},  \hyperlink{B4-a}{{\bf (B4-a)}} and  \hyperlink{B4-b}{{\bf (B4-b)}}. }
$$ 
Note that \hyperlink{B4-a}{{\bf (B4-a)}} implies \hyperlink{B2-a}{{\bf (B2-a)}}.  

\begin{remark}\label{r:A4}
	Since $\sB$ is bounded by \hyperlink{B2-a}{{\bf (B2-a)}},
	the inequality \eqref{e:ass-B3} automatically holds for all $x,y \in D$ with 
	$|x-y| \ge \delta_D(x) \wedge \delta_D(y) \wedge \wh R$. 
	Hence, since \hyperlink{B4-a}{{\bf (B4-a)}} is assumed,
	for \hyperlink{B3}{{\bf (B3)}} it suffices to require 
	that \eqref{e:ass-B3} holds for all  $x,y \in D$ with $|x-y| < \delta_D(x) \wedge \delta_D(y) \wedge \wh R$.
\end{remark}

We now define a barrier $\psi^{(r)}$ and give an upper bound on $L^{\sB}_{\alpha} \psi^{(r)}$. This upper bound will be used in Section \ref{ch:decay-rate}, leading eventually to the important Theorem \ref{t:Dynkin-improve}.

Fix a positive integer $N_0> \alpha + \ub_0+2$. 
Let  $\psi:\bH\to [0,\infty)$ be a $C^{N_0}$ function such that 
(i) $\psi(v)=|\wt v|^{2N_0} + v_d^{2N_0}$ for $v \in U_\bH(2)$; and (ii) $\psi(v)=0$ for $\bH\setminus U_\bH(3)$.  
For a multi-index $\rho=(\rho_1,...,\rho_{d})\in \N_0^{d}$, we define $|\rho|:=\sum_{i=1}^{d}\rho_i$ and  
$\rho!:=\Pi_{i=1}^{d} \rho_i !$. 
Let  further $v^\rho:=\Pi_{i=1}^{d}\,v_i^{\rho_i}$ for $v=(v_1,...v_d) \in 
 \R^{d}$
and
\begin{align*}
	\partial^\rho \psi(v) := \frac{\partial^{|\rho|}\psi(v)}{\partial v_1^{\rho_1}\cdots \partial v_{d}^{\rho_{d}} }, \quad v=(v_1,...v_d) \in 
 \bH
\end{align*}
Denote by $\mathfrak i(k)$ the family of all multi-indices $\rho=(\rho_1,...,\rho_d) \in \N_0^d$ with  $|\rho|=k$. Let $\mathfrak i_0(k):=\{ \rho \in \mathfrak i(k): \rho_d=0\}$. One sees  that  for any integer $1\le k \le N_0$, there exists a constant $c(k)>0$ depending only on $k$ such that for any $v=(\wt v, v_d) \in U_\bH(2)$,
\begin{align}\label{e:8.3-null}
	\sum_{\rho \in \mathfrak i_{0}(k)} \frac{\partial^\rho (\partial\psi/\partial v_d)(v)}{\rho!}=0,
\end{align}
\begin{align}\label{e:8.3}
	\bigg| \sum_{\rho \in \mathfrak i_{0}(k)} \frac{\partial^\rho \psi(v)}{\rho!} \bigg| \le c(k)|\wt v|^{2N_0-k} 
	\quad \text{and} \quad   \bigg|  \frac{\partial^k \psi(v)}{\partial v_d^k} \bigg| \le c(k) v_d^{2N_0-k}.
\end{align}

For  $Q \in \partial D$ and $0<r\le \wh R/(18+9\Lambda_0)$,  we define $\psi^{(r)}=\psi^{(r)}_Q:D\to [0,\infty)$ by
\begin{align}\label{e:def-psi-r}
	\psi^{(r)}(y)=\begin{cases}
		\psi((f_Q^{(r)})^{-1}(y)) & \mbox{ if }  y\in U^Q(3r),\\[3pt]
		0 & \mbox{ if } y\in D\setminus U^Q(3r),
	\end{cases}
\end{align}
where $f_Q^{(r)}$ is the function defined in \eqref{e:def-fr}. Then $\psi^{(r)}$ is a non-negative $C^{1,1}$ function 
with support in $U^Q(3r)$.  By Proposition \ref{p:generator-C11}, $L^\sB_\alpha \psi^{(r)}$ is well defined. 

In the remainder of this subsection, we will work with a fixed $Q\in \partial D$, and 
will write $U(r)$ for $U^Q(r)$  and $f^{(r)}$ for $f_Q^{(r)}.$ 

The goal of this subsection is to prove the following proposition.
\begin{prop}\label{p:compensator}
	Let $Q \in \partial D$. For any $\eps>0$,	there exists a constant $C(\eps)>0$ independent of  $Q$ such that for any $0<r\le \wh R/(18+9\Lambda_0)$ and any $y \in U(r)$,	
	\begin{align*}
		L_\alpha^\sB \psi^{(r)}(y)\le  \eps  \delta_D(y)^{-\alpha}  \psi^{(r)}(y)+ C(\eps)r^{-\alpha} \Phi_0(\delta_D(y)/r).
	\end{align*}
\end{prop}

\smallskip

We will prove Proposition \ref{p:compensator} by 
estimating some specific integrals within the half space  through a series of lemmas. 

Define for  $r \in (0, \wh R/(18+9\Lambda_0)]$ and $v \in U_\bH(1)$,
\begin{align}
	\sI^{(r)}_{1}(v)&:=\int_{U_\bH(3) \setminus B(v,v_d/2)} (\psi(w)-\psi(v)) 
	\frac{\sB(f^{(r)}(w),f^{(r)}(v)) }{|f^{(r)}(w)-f^{(r)}(v)|^{d+\alpha}} dw,\label{e:compensator1}
\end{align}
\begin{align}
	\sI^{(r)}_{2}(v)&:=\int_{B(v, v_d/2)}(\psi(w)-\psi(v) ) \frac{ (\sB(f^{(r)}(w),f^{(r)}(v)) - 
		\sB(f^{(r)}(v),f^{(r)}(v)) )}{|f^{(r)}(w)-f^{(r)}(v)|^{d+\alpha}}  dw, \label{e:compensator2}
\end{align}
\begin{align}
	\sI^{(r)}_{3}(v)&:= \sB(f^{(r)}(v),f^{(r)}(v))\int_{B(v, v_d/2)} \frac{\psi(w)-\psi(v) - 
		\nabla \psi(v) \cdot ( w- v)}{|f^{(r)}(w)-f^{(r)}(v)|^{d+\alpha}}   dw .\label{e:compensator3}
\end{align}

To get estimates for $\sI^{(r)}_{1}(v)$,  $\sI^{(r)}_{2}(v)$  and  $\sI^{(r)}_{3}(v)$, we use the following lemma.  
\begin{lemma}\label{l:1n}  
	\noindent	(i)	There exists  $C>0$ such that  for  any $v \in U_\bH(1)$,
	\begin{align*}
		\int_{U_\bH(2)\setminus B(v,v_d/2)} \frac{\Phi_0(v_d/|w-v|)}{|w-v|^{d+\alpha-N_0}}  dw \le   C\Phi_0(v_d).
	\end{align*}
	
	\noindent	(ii)	For any $\eps \in (0,1)$, 	there exists a constant  $C(\eps)>0$ such that  for any  $1\le k\le N_0-1$ and any $v \in U_\bH(1)$,
	\begin{align*}
		(|\wt v|^{2N_0-k} + v_d^{2N_0-k})	\int_{U_\bH(2)\setminus B(v,v_d/2)} \frac{\Phi_0(v_d/|w-v|)}{|w-v|^{d+\alpha-k}} dw
		\le \eps |\wt v|^{2N_0} v_d^{-\alpha}   +  C(\eps)\Phi_0(v_d).
	\end{align*}
\end{lemma}
\begin{proof} 
	(i) Using \eqref{e:scale} and $v_d/2<|v-w| <4$ for $w\in U_{\bH}(2)\setminus B(v,v_2/2)$, 
	since $N_0>\alpha + \ub_0$, we get
	\begin{align*}
		\int_{U_\bH(2)\setminus B(v,v_d/2)} \frac{\Phi_0(v_d/|w-v|)}{|w-v|^{d+\alpha-N_0}}  dw 
		\le c \Phi_0(v_d/4) \int_{v_d/2}^{4}  
		s^{N_0-1-\alpha-\ub_0}ds  
		\le c\Phi_0(v_d).
	\end{align*}
	
	(ii) Let $\eps \in (0,1)$ and $1\le k \le N_0-1$. Using \eqref{e:scale},  
	$|\wt v|^{2N_0-k}+v_d^{2N_0-k} \le  (|\wt v|+v_d)^{2N_0-k}$,
	$k-\alpha/2>0$, the boundedness of $\Phi$, and 
	$2N_0-\alpha-\ub_0>0$, we obtain
	\begin{align*}
		&	(|\wt v|^{2N_0-k} + v_d^{2N_0-k})	\int_{U_\bH(2)\setminus B(v,v_d/2)} \frac{\Phi_0(v_d/|w-v|)}{|w-v|^{d+\alpha-k}} dw\\
		&\le c_1 (|\wt v|^{2N_0-k} + v_d^{2N_0-k}) \int_{v_d/2}^{4} s^{k-1-\alpha} \Phi_0(v_d/s) ds\\
		& \le c_2  
		(|\wt{v}|+v_d)^{2N_0-k} 
		\bigg(  \int_{v_d/2}^{|\wt v|+v_d} s^{k-1-\alpha} ds 
		+ \Phi_0(v_d/4)\int_{|\wt v|+v_d}^{4} 
		s^{k-1-\alpha-\ub_0}  ds \bigg)\\
		&\le \frac{c_2 	(|\wt{v}|+v_d)^{2N_0-k}  }{ (v_d/2)^{\alpha/2}}
		\int_{v_d/2}^{|\wt v|+v_d} s^{k-1-\alpha/2} ds
		+  c_3 \Phi_0(v_d)\int_{|\wt v|+v_d}^{4} 
		s^{2N_0-1-\alpha-\ub_0}  ds \\
		&\le c_4  ( |\wt v|+v_d)^{2N_0-\alpha/2} v_d^{-\alpha/2} + c_5 \Phi_0(v_d).
	\end{align*}
	Hence, it remains to show that there exists $c_6>0$ independent of $v$ such that 
	\begin{align}\label{e:1n-1}
		c_4(|\wt v|+v_d)^{2N_0-\alpha/2} <\eps  |\wt v|^{2N_0} v_d^{-\alpha/2}  + c_6v_d^{\alpha/2} \Phi_0(v_d).
	\end{align}
	Indeed, for 	$c_7=c_7(\eps):= (2^{\alpha/2-2N_0}\eps/c_4)^{ -2/\alpha }+1$, 
	if $|\wt v|>  c_7v_d$, then 
	$ c_4 (|\wt v|+v_d)^{2N_0-\alpha/2} \le 2^{2N_0-\alpha/2}c_4 |\wt v|^{2N_0-\alpha/2} < \eps |\wt v|^{2N_0} v_d^{-\alpha/2}$. 
	If $|\wt v|\le c_7v_d$, then by \eqref{e:scale}, since  $2N_0-\alpha/2-\ub_0>0$, 
	\begin{align*} 
		&c_4(|\wt v|+v_d)^{2N_0-\alpha/2} \le c_4(1+c_7)^{2N_0-\alpha/2} v_d^{2N_0-\alpha/2} \\
		&<c_8 v_d^{2N_0-\alpha/2-\ub_0}
		\Phi(v_d)/\Phi(1)<c_8\Phi(v_d)/\Phi(1).
	\end{align*}
	The proof is complete.
\end{proof}

\begin{lemma}\label{l:comp-1}
  For any $\eps>0$, there exists a constant $C=C(\eps)>0$ such that for all  $r \in (0, \wh R/(18+9\Lambda_0)]$ and $v \in U_\bH(1)$,
	\begin{align*}
		\sI^{(r)}_{1}(v) \le 
		r^{-d-\alpha} \left( \eps |\wt v|^{2N_0} v_d^{-\alpha}  + C\Phi_0(v_d) \right),
	\end{align*}
	where  $	\sI^{(r)}_{1}(v)$ is defined by \eqref{e:compensator1}.
\end{lemma}
\begin{proof}  
	By \hyperlink{B4-a}{{\bf (B4-a)}}, \eqref{e:U-rho-C11-2}, the almost increasing property of $\Phi_0$
	and Lemma   \ref{l:diffeo}, there exist  $c_1,c_2>0$ independent of $Q$ and $r$ such that for any  $w,z \in U_\bH(3)$, 
	\begin{align}\label{e:B-comp}
		\sB(f^{(r)}(w),f^{(r)}(z)) \le c_1 \Phi_0 \bigg(\frac{\rho_D(f^{(r)}(z))}{|f^{(r)}(w)-f^{(r)}(z)|} \bigg) 
		\le c_2\Phi_0\bigg(\frac{z_d}{|w-z|} \bigg).
	\end{align}
	Observe that
	$$
	\sI^{(r)}_{1}(v):=I_{1}+I_{2}+I_{3},
	$$
	where
	\begin{align*}
		I_1&:=\int_{U_\bH(3) \setminus U_\bH(2)} 
		(\psi(w)-\psi(v))\frac{\sB(f^{(r)}(w),f^{(r)}(v)) }{|f^{(r)}(w)-f^{(r)}(v)|^{d+\alpha}}\,dw,\\
		I_2&:=\int_{U_\bH(2) \setminus B(v,v_d/2)} \bigg( \psi(w)-\psi( v)-\sum_{k=1}^{N_0-1}
		\sum_{\rho \in \mathfrak i(k)}\frac{\partial^{\rho}\psi(v)}{\rho!} (w-v)^{\rho}\bigg)\\
		&\qquad \qquad \qquad \qquad\qquad  \times \frac{\sB(f^{(r)}(w),f^{(r)}(v))}{|f^{(r)}(w)-f^{(r)}(v)|^{d+\alpha}} \,dw,\\
		I_3&:=\int_{U_\bH(2)\setminus B(v,v_d/2)}	\sum_{k=1}^{N_0-1}\sum_{\rho \in \mathfrak i(k)} 
		\frac{\partial^{\rho}\psi(v)}{\rho!} (w-v)^{\rho}\frac{\sB(f^{(r)}(w),f^{(r)}(v))}{|f^{(r)}(w)-f^{(r)}(v)|^{d+\alpha}} \, dw.
	\end{align*}
Using the mean value theorem,  \eqref{e:B-comp}, \eqref{e:scale}  and Lemma \ref{l:diffeo}, we obtain
	\begin{align*}
		&I_1 \le  c\int_{U_\bH(3) \setminus B(v,1)} 
		\frac{\sup_{ \xi \in \R^{d}} |\nabla \psi(\xi)| |w- v|}{|f^{(r)}(w)-f^{(r)}(v)|^{d+\alpha}}  
		\Phi_0\bigg( \frac{v_d}{|w-v|}\bigg)dw\\
		& \le cr^{-d-\alpha} \Phi_0(v_d)\int_{U_\bH(3) \setminus B(v,1)} \frac{dw}{|w-v|^{d+\alpha-1}}  \le cr^{-d-\alpha} \Phi_0(v_d). 
	\end{align*}
By using  Taylor's theorem, \eqref{e:B-comp} and  Lemmas \ref{l:diffeo} and \ref{l:1n}(i),  we have
	\begin{align*}
		&	I_2 \le  cr^{-d-\alpha} \int_{U_\bH(2) \setminus B(v,v_d/2)}
		\frac{|w-v|^{N_0}}{|w-v|^{d+\alpha}}
		\Phi_0\bigg( \frac{v_d}{|w-v|}\bigg) dw\le cr^{-\alpha-d}\Phi_0(v_d).
	\end{align*}
	Moreover, using \eqref{e:B-comp},  \eqref{e:8.3-null}, \eqref{e:8.3} and Lemmas \ref{l:diffeo} and \ref{l:1n}(ii), we obtain
	\begin{align*}
		I_3 &=
		\int_{U_\bH(2)\setminus B(v,v_d/2)}
		\sum_{k=1}^{N_0-1}\sum_{\rho \in \mathfrak i_0(k)} \frac{\partial^{\rho}\psi(v)}{\rho!} 
		(w-v)^{\rho}\frac{\sB(f^{(r)}(w),f^{(r)}(v))}{|f^{(r)}(w)-f^{(r)}(v)|^{d+\alpha}} \, dw\\
		&\quad + 	\int_{U_\bH(2)\setminus B(v,v_d/2)}
		\sum_{k=1}^{N_0-1} \frac{1}{k!}\frac{\partial^{k}\psi(v)}{\partial v_d^k} 
		(w_d-v_d)^{k}\frac{\sB(f^{(r)}(w),f^{(r)}(v))}{|f^{(r)}(w)-f^{(r)}(v)|^{d+\alpha}} \, dw\\
		&\le cr^{-d-\alpha} \sum_{k=1}^{N_0-1} (|\wt v|^{2N_0-k} + v_d^{2N_0-k}) 
		\int_{U_\bH(2)\setminus B(v,v_d/2)} \frac{1}{|w-v|^{d+\alpha-k}} \Phi_0 \bigg(\frac{v_d}{|w-v|}\bigg) dw\\
		&\le r^{-d-\alpha} \sum_{k=1}^{N_0-1} \left(  (\eps / N_0) \,|\wt v|^{2N_0} v_d^{-\alpha}  + c\Phi_0(v_d)\right)\\
		& \le  r^{-d-\alpha} \left( \eps |\wt v|^{2N_0} v_d^{-\alpha}  + c\Phi_0(v_d) \right).
	\end{align*}
	Combining the above estimates, we arrive at the desired result. 
\end{proof}

\begin{lemma}\label{l:comp-2}
 For any $\eps>0$, there exists a constant $C=C(\eps)>0$ such that for all  $r \in (0, \wh R/(18+9\Lambda_0)]$ and $v \in U_\bH(1)$,
\begin{align*}
	\sI^{(r)}_{2}(v)  + 	\sI^{(r)}_{3}(v) \le 
	r^{-d-\alpha} \left( \eps |\wt v|^{2N_0} v_d^{-\alpha}  + C\Phi_0(v_d) \right),
\end{align*}
where  $	\sI^{(r)}_{2}(v)$ and  $	\sI^{(r)}_{3}(v)$  are defined by \eqref{e:compensator2} and \eqref{e:compensator3}.
\end{lemma}
\begin{proof} 
	 Denote the Hessian matrix of $\psi$ at point $\xi$ by $D^2 \psi(\xi)$.
	Using  \eqref{e:B(x,x)}, Taylor's theorem, Lemma \ref{l:diffeo}, \eqref{e:8.3-null} and \eqref{e:8.3}, 
	we get
	\begin{align*}
		\sI^{(r)}_{3}(v) & \le c \int_{B(v,v_d/2)} \frac{\sup_{\xi \in B(v,v_d/2)} |D^2 \psi(\xi)| 
			| w- v|^2}{|f^{(r)}(w)-f^{(r)}(v)|^{d+\alpha}} dw\\
		&\le cr^{-d-\alpha}\sup_{\xi \in B(v,v_d/2)} ( |\wt \xi|^{2N_0-2} + \xi_d^{2N_0-2}) 
		\int_{B(v,v_d/2)} \frac{dw}{|w-v|^{d+\alpha-2}} \\
		&\le c_1r^{-d-\alpha} v_d^{2-\alpha} ( |\wt v| + v_d)^{2N_0-2}.
	\end{align*}
	When $\alpha<1$, by using Lemma \ref{l:diffeo},  the mean value theorem and \eqref{e:8.3}, we have
	\begin{align*}
		\sI^{(r)}_{2}(v)  & \le c r^{-d-\alpha}\sup_{ \xi \in B(v,v_d/2)} (|\wt \xi|^{2N_0-1} 
		+ \xi_d^{2N_0-1})\int_{B(v,v_d/2)} \frac{dw}{|w-v|^{d+\alpha-1}} \\
		&\le c_2r^{-d-\alpha} v_d^{1-\alpha}(|\wt v|+v_d)^{2N_0-1}.
	\end{align*}
	When $\alpha \ge 1$, using Lemma \ref{l:diffeo}, the mean value theorem, \hyperlink{B3}{{\bf (B3)}}, 
	\eqref{e:U-rho-C11-2} and  \eqref{e:8.3},  since $\theta_0>\alpha-1$, we obtain
	\begin{align*}
		&	\sI^{(r)}_{2}(v) \\
		&\le c\int_{B(v,v_d/2)} \!\!
		\frac{\sup_{ \xi \in B(v,v_d/2)} |\nabla \psi(\xi)| |w- v|}{|f^{(r)}(w)-f^{(r)}(v)|^{d+\alpha}} 
		\bigg( \frac{|f^{(r)}(w)-f^{(r)}(v)|}{\rho_D(f^{(r)}(w)) \wedge \rho_D(f^{(r)}(v))}\bigg)^{\theta_0} dw\\
		&\le cr^{-d-\alpha}\sup_{\xi \in B(v,v_d/2)} (|\wt \xi|^{2N_0-1} 
		+ \xi_d^{2N_0-1}) \int_{B(v,v_d/2)} \frac{1}{|w-v|^{d+\alpha-1}} \bigg( \frac{|w-v|}{w_d \wedge v_d}\bigg)^{\theta_0} dw \\
		& \le cr^{-d-\alpha}(v_d/2)^{-\theta_0} (|\wt v| + v_d)^{2N_0-1} \int_{B(v,v_d/2)} 
		\frac{dw}{|w-v|^{d+\alpha-1-\theta_0}}\\
		&= c_3r^{-d-\alpha} v_d^{1-\alpha}(|\wt v| + v_d)^{2N_0-1}.
	\end{align*}
	Therefore, it holds that 
	 $$	\sI^{(r)}_{2}(v) +	\sI^{(r)}_{3}(v)  \le (c_1+c_2+c_3)r^{-d-\alpha}	v_d^{1-\alpha}(|\wt v| + v_d)^{2N_0-1}.$$
	
	Let $\eps>0$.  To obtain the desired result, we need to show that there exists a constant  $c(\eps)>0$ independent of $Q,r$ and $v$ such that
	\begin{align}\label{e:comp-2-1}
		(c_1+c_2+c_3)	v_d^{1-\alpha}(|\wt v| + v_d)^{2N_0-1} \le \eps |\wt v|^{2N_0} v_d^{-\alpha} + c(\eps)\Phi_0(v_d).
	\end{align}
	Set $c_4=c_4(\eps):=2^{2N_0-1} (c_1+c_2+c_3)\eps^{-1}+1$.  If $|\wt v|>c_4v_d$, 
	then $(c_1+c_2+c_3)v_d^{1-\alpha}(|\wt v| + v_d)^{2N_0-1}< 2^{2N_0-1}(c_1+c_2+c_3)v_d^{1-\alpha} |\wt v|^{2N_0-1}\le \eps |\wt v|^{2N_0} v_d^{-\alpha}$. If $|\wt v|\le c_4v_d$, 
	then since $N_0>\alpha+\ub_0$ and $v_d < 1$, we get from \eqref{e:scale} that 
	\begin{align*}
		v_d^{1-\alpha}(|\wt v| + v_d)^{2N_0-1}  \le (1+c_4)^{2N_0-1}v_d^{2N_0-\alpha}  
		\le (1+c_4)^{2N_0-1} v_d^{\ub_0}
		\le  c\Phi_0(v_d)/\Phi_0(1).
	\end{align*}
	Therefore, \eqref{e:comp-2-1} holds. 
\end{proof}

\textsc{Proof of Proposition \ref{p:compensator}.} 
Let $Q \in \partial D$, $0<r\le\wh R/(18+9\Lambda_0)$ and  $y\in U(r)$. Denote $v=(f^{(r)})^{-1}(y)\in U_\bH(1)$.   Since $\psi^{(r)}(z)=0$ in $D\setminus U(3r)$, by using the change of variables $z=f^{(r)}(w)$ in the second line below,  we have
\begin{align}\label{e:comp-new1}
	\begin{split} 	&\! \! \! \! \! \! \! \! \! L^{\sB}_{\alpha}\psi^{(r)}(y)
		\le \lim_{\eps \to 0}\int_{D\cap  (U(3r)\setminus B(y,\eps))}(\psi^{(r)}(z)-\psi^{(r)}(y)) 
		\frac{\sB(
		 z,y  )}{|z-y|^{d+\alpha}} dz  \\
		&=\lim_{\eps \to 0} \,r^d\!\int_{\bH\cap (f^{(r)})^{-1} (  U(3r)\setminus B(y,\eps))} (\psi(w)-\psi( v)) 
		\frac{\sB(f^{(r)}(w),f^{(r)}(v)) }{|f^{(r)}(w)-f^{(r)}(v)|^{d+\alpha}}  dw\\
		&=r^d\!\int_{U_\bH(3) \setminus B(v,v_d/2)} (\psi(w)-\psi(v)) 
		\frac{\sB(f^{(r)}(w),f^{(r)}(v)) }{|f^{(r)}(w)-f^{(r)}(v)|^{d+\alpha}} dw\\
		&\;\, + r^d\!\int_{B(v, v_d/2)}(\psi(w)-\psi(v) ) \frac{ (\sB(f^{(r)}(w),f^{(r)}(v)) - 
			\sB(f^{(r)}(v),f^{(r)}(v)) )}{|f^{(r)}(w)-f^{(r)}(v)|^{d+\alpha}}  dw \!\! \\
		&\;\, +r^d \sB(f^{(r)}(v),f^{(r)}(v))\int_{B(v, v_d/2)} \frac{\psi(w)-\psi(v) - 
			\nabla \psi(v) \cdot ( w- v)}{|f^{(r)}(w)-f^{(r)}(v)|^{d+\alpha}}   dw \\
		&= r^d \big( 	\sI^{(r)}_{1}(v) + 	\sI^{(r)}_{2}(v) + 	\sI^{(r)}_{3}(v)\big).
	\end{split} 
\end{align}  Since $\wt v = \wt y/r$ and   $v_d = \rho_D(y)/r$, by  Lemmas \ref{l:comp-1} and \ref{l:comp-2},  for any $\eps>0$,   there exists a constant $c(\eps)>0$ such that
\begin{equation*}
	\begin{split} 	r^d \big( 	\sI^{(r)}_{1}(v) + 	\sI^{(r)}_{2}(v) + 	\sI^{(r)}_{3}(v)\big)&\le 
		r^{-\alpha} \left( \eps|\wt v|^{2N_0} v_d^{-\alpha}  + c(\eps)\Phi_0(v_d) \right)\\
		&=  \eps(|\wt y|/r)^{2N_0} \rho_D(y)^{-\alpha}  + c(\eps)r^{-\alpha}\Phi_0(\rho_D(y)/r).
	\end{split} 
\end{equation*} 
Using this, \eqref{e:U-rho-C11-2}, \eqref{e:scale} and  $\psi^{(r)}(y)= \psi(v) = |\wt v|^{2N_0} + v_d^{2N_0} \ge |\wt v|^{2N_0} = (|\wt y|/r)^{2N_0}$, we obtain 
\begin{equation}\label{e:comp-new2}
	\begin{split} 	r^d \big( 	\sI^{(r)}_{1}(v) + 	\sI^{(r)}_{2}(v) + 	\sI^{(r)}_{3}(v)\big)
		&\le   \eps  \delta_D(y)^{-\alpha} \psi^{(r)}(y)+ c_1 c(\eps) r^{-\alpha} \Phi_0(\delta_D(y)/r).
	\end{split} 
\end{equation} 
Combining \eqref{e:comp-new1} and \eqref{e:comp-new2}, we get the desired result.
\qed


\section{Key estimates on $C^{1,1}$ open sets}\label{ch:key-estimates}

	Starting from this section, we assume that $D \subset \R^d$ is a $C^{1,1}$ open set with characteristics $(\wh R,\Lambda)$. 
	See Definition \ref{df-c11}.
	Without loss of generality, we assume that   $\wh R\le 1 \wedge (1/(2\Lambda))$. Furthermore, in the remainder of this work, we assume that the killing  potential $\kappa$ satisfies the following:

	\medskip

	\noindent \hypertarget{B3}{{\bf (K3)}} There exist constants $\eta_0>0$ and 	$C_8, C_9\ge 0$ 	such that for all $x \in D$,
	\begin{align}\label{e:kappa-explicit}		\begin{cases}			|\kappa(x) - C_9 \sB(x,x) \delta_D(x)^{-\alpha}| \le C_8\delta_D(x)^{-\alpha+\eta_0} &\text{ if }  \delta_D(x) < 1,\\[4pt]			\kappa(x) \le C_8 &\text{ if } \delta_D(x) \ge 1.		\end{cases}	\end{align}	When $\alpha \le 1$, we further assume that $C_9>0$. 
	
	\medskip
	
Note that  \hyperlink{K3}{{\bf (K3)}} implies conditions \hyperlink{K1}{{\bf (K1)}} and  \hyperlink{K2}{{\bf (K2)}}.

	Now, we introduce  additional  assumptions on $\sB$ that will play a central role in obtaining Proposition \ref{p:barrier}. 
	Recall the definition of the set $E^Q_\nu(r)$  from \eqref{e:def-Enu}. We consider the cases when the killing 
	potential $\kappa(x)$
	 is critical (namely, the constant $C_9$ is \hyperlink{K3}{{\bf (K3)}} is strictly positive) and is subcritical (namely, $C_9=0$), separately.

	\medskip
	
	\textbf{Case $C_9>0$:} In this case we consider the following assumption.
	
	\smallskip
	
	\noindent \hypertarget{B5-I}{{\bf (B5-I)}} There  exist  constants $\nu \in (0,1]$, $\theta_1,\theta_2,C_{10}>0$, and a non-negative Borel function $\F_0$ on $\bH_{-1}$   such that  for any $Q \in \partial D$ and    $x,y \in E^Q_\nu(\wh R/8)$ with $x=(\wt x,x_d)$ in CS$_Q$,
	\begin{align*}
			& \big|\sB(x,y)- \sB(x,x)\F_0((y-x)/x_d) \big| + \big|\sB(x,y)- \sB(y,y)\F_0((y-x)/x_d) \big|\nn\\	&\le C_{10} \bigg(\frac{ \delta_D(x)\vee\delta_D(y) \vee |x-y|}{ \delta_D(x) \wedge \delta_D(y) \wedge |x-y|} \bigg)^{\theta_1}
						\big( \delta_D(x)\vee\delta_D(y) \vee|x-y|\big)^{\theta_2}.
	\end{align*}
	
	\medskip
	
	Under condition \hyperlink{B5-I}{{\bf (B5-I)}}, we  define a  function $\F$ on $\bH_{-1}$ by 
	\begin{align}\label{e:def-F0-transform}
		\F(y) =\frac{\F_0(y) + \F_0(-y/(1+y_d))}{2}, \quad \;\;y  = (\wt y,y_d) \in \bH_{-1}.
	\end{align}
	We will see in Lemma \ref{l:F-basic} that $\F$ is a bounded function. Moreover, we observe that
	\begin{align}\label{e:A5-F}
		\F(y)=\F(-y/(1+y_d)) \quad \text{for all} \;\, y \in \bH_{-1}.
	\end{align}
	This property is in a crucial way related to the symmetry of $\sB$ (see Lemma \ref{l:ass-1}(ii) below). 

For a function $f$ on $\bH_{-1}$, define  
\begin{align}\label{e:def-killing-constant}
		&	C(\alpha,q,f)\\
		&=\int_{\R^{d-1}}\frac{1}{(|\wt u|^2+1)^{(d+\alpha)/2}}\int_0^1  
		\frac{(s^q -1)(1-s^{\alpha-1-q})}{(1-s)^{1+\alpha}} f\big(((s-1)\wt u, s-1)\big) ds \, d \wt u,\nn
	\end{align}

	With the function $\F$ in \eqref{e:A5-F} and $q \in [(\alpha-1)_+, \alpha+\beta_0)$, 
	we associate a  constant $C(\alpha,q,\F)$ using the definition in \eqref{e:def-killing-constant}
and additionally assume that 
	\begin{align}\label{e:K4}	
		C_9 < \lim_{q \to \alpha+\beta_0} C(\alpha,q,\F).
	\end{align}
See Lemma \ref{l:C=infty} for a simple sufficient condition for  \eqref{e:K4}.
	
We will show in Lemma \ref{l:constant} that	$q\mapsto C(\alpha,q,\F)$ is a well-defined strictly increasing continuous function 
on $[(\alpha-1)_+, \alpha+\beta_0)$ 
and $C(\alpha,(\alpha-1)_+,\F)=0$.  Therefore, under \eqref{e:K4},  there exists a unique constant 
$p\in ((\alpha-1)_+, \alpha + \beta_0)$ such that 
	\begin{align}\label{e:C(alpha,p,F)}
		C_9=C(\alpha,p,\F).
	\end{align}

	\medskip
	
	\textbf{Case $C_9=0$:} In this case, instead of \hyperlink{B5-I}{{\bf (B5-I)}}, we introduce the following weaker condition:
	
	\medskip
	
	\noindent \hypertarget{B5-II}{{\bf (B5-II)}} There  exist  constants 
	$\nu \in (0,1]$, $\theta_1,\theta_2,C_{10}>0$, $C_{11}>1$, $i_0 \ge 1$, and  non-negative  Borel  functions 
	$\F_0^i:\bH_{-1} \to [0,\infty)$ and $\mu^i:D \to (0,\infty)$, $1\le i \le i_0$, such that 
	\begin{align}\label{e:mu-bound}
		C_{11}^{-1} \le	\mu^i(x) \le C_{11} \quad \text{for all} \;\, x \in D,
	\end{align}
and for any $Q \in \partial D$ and    $x,y \in E^Q_\nu(\wh R/8)$ with $x=(\wt x,x_d)$ in CS$_Q$,
	\begin{align}\label{e:ass-B5'}	
	\begin{split}
		&\bigg|\sB(x,y)- \sum_{i=1}^{i_0}\mu^i(x)\F_0^i((y-x)/x_d) \bigg| 
			+  \bigg|\sB(x,y)- \sum_{i=1}^{i_0}\mu^i(y)\F_0^i((y-x)/x_d) \bigg| \\
		&\le C_{10} \bigg(\frac{ \delta_D(x)\vee\delta_D(y) \vee |x-y|}{ \delta_D(x) \wedge \delta_D(y) \wedge |x-y|} \bigg)^{\theta_1}
			\big( \delta_D(x)\vee\delta_D(y) \vee|x-y|\big)^{\theta_2}.
	\end{split}
	\end{align}
Analogously to \eqref{e:def-F0-transform}, for each $1\le i\le i_0$, 
 we also define $$\F^i(y):=(\F_0^i(y)+\F_0^i(-y/(1+y_d)))/2$$ and 
  associate  constants
  $C(\alpha, q,\F^i)$ for 	$q \in [(\alpha-1)_+, \alpha+\beta_0)$ . 
	
	\medskip
	
Note that if \hyperlink{B5-I}{{\bf (B5-I)}} holds, then  \hyperlink{B5-II}{{\bf (B5-II)}} holds with $i_0=1$, $\F_0^1=\F_0$ and 
$\mu^1(x)=\sB(x,x)$.

\medskip

 The conditions \hyperlink{B5-I}{{\bf (B5-I)}} and \hyperlink{B5-II}{{\bf (B5-II)}} always come in connection with \hyperlink{K3}{{\bf (K3)}}. We now combine these two conditions into the condition \hyperlink{B5}{{\bf (B5)}}, and assume that \hyperlink{B5}{{\bf (B5)}} holds from here on until the end of Section \ref{ch:green}.

	\medskip
	
	\noindent \hypertarget{B5}{{\bf (B5)}}  If $C_9>0$, then	\hyperlink{B5-I}{{\bf (B5-I)}} and \eqref{e:K4} hold, and  if $C_9=0$, then \hyperlink{B5-II}{{\bf (B5-II)}} holds.

	\medskip
	
\noindent We will  let $p$ denote the constant satisfying \eqref{e:C(alpha,p,F)} if $C_9>0$ and let $p=\alpha-1$ if $C_9=0$. Recall that we assume $\alpha>1$ if $C_9=0$. Hence, we always have $p \in [(\alpha-1)_+,\alpha+\beta_0) \cap (0,\infty)$. 
Furthermore,  we  treat \hyperlink{B5-I}{{\bf (B5-I)}}  as a special case of \hyperlink{B5-II}{{\bf (B5-II)}} with $i_0=1$ in all instances. This means that whenever $C_9>0$, we set $i_0=1$, $\F_0=\F_0^1$ and $\mu^1(x)=\sB(x,x)$. Here $\mu^1(x)=\sB(x,x)$ satisfies \eqref{e:mu-bound} by \eqref{e:B(x,x)}.

\subsection{Properties of   $C(\alpha,q,\F)$ and   $C(\alpha,q,\F^i)$}

In this subsection we show that $q\mapsto C(\alpha,q, F^i)$ is well defined, continuous and strictly increasing, and give a sufficient condition for \eqref{e:K4} to hold true. We start with the symmetrized version of condition \hyperlink{B5-II}{{\bf (B5-II)}}.

\begin{lemma}\label{l:B5'}
For any $Q \in \partial D$ and     $x,y \in E^{Q}_\nu(\wh R/8)$ with $x=(\wt x, x_d)$ in CS$_Q$,
	\begin{align*}
		& \bigg|\sB(x,y)- \sum_{i=1}^{i_0}\mu^i(x)\F^i((y-x)/x_d) \bigg|\nn\\
		&\le C_{10} \bigg(\frac{ \delta_D(x)\vee\delta_D(y) \vee |x-y|}{ \delta_D(x) \wedge \delta_D(y) \wedge |x-y|} \bigg)^{\theta_1}	
			\big(  \delta_D(x)\vee\delta_D(y) \vee|x-y|\big)^{\theta_2}.	
	\end{align*}
\end{lemma}
\begin{proof}
	Let $Q \in \partial D$ and  $x=(\wt x,x_d)$, $y=(\wt y,y_d)\in E^{Q}_\nu(\wh R/8)$  in CS$_Q$.  Using \hyperlink{B1}{{\bf (B1)}}  and  applying \eqref{e:ass-B5'} to $\sB(x,y)$ and $\sB(y,x)$, we obtain
	\begin{align*}
		&2 \bigg|\sB(x,y)- \sum_{i=1}^{i_0}\mu^i(x)\F^i((y-x)/x_d) \bigg|\\
		&= \bigg|\sB(x,y) + \sB(y,x)- \sum_{i=1}^{i_0}\mu^i(x) \big(\F_0^i((y-x)/x_d) + \F_0^i ((x-y)/y_d) \big) \bigg|\\
		&\le  \bigg|\sB(x,y)- \sum_{i=1}^{i_0}\mu^i(x)\F_0^i((y-x)/x_d) \bigg| + \bigg|\sB(y,x)- \sum_{i=1}^{i_0}\mu^i(x)\F_0^i((x-y)/y_d) \bigg| \\
		&\le 2C_{10} \bigg(\frac{ \delta_D(x)\vee\delta_D(y) \vee |x-y|}{ \delta_D(x) \wedge \delta_D(y) \wedge |x-y|} \bigg)^{\theta_1}
		\big( \delta_D(x)\vee\delta_D(y) \vee|x-y|\big)^{\theta_2}.
	\end{align*}
\end{proof}

\begin{lemma}\label{l:F-basic}
There exists $C>1$ such that for all $y \in \bH_{-1}$, 
\begin{align*}
	C^{-1}\1_{|\wt y| \le (y_d+1) \vee 1 } \Phi_0\left(\frac{(y_d+1) \wedge 1}{|y|}\right)\le \sum_{i=1}^{i_0}	 \F^i(y)  \le C \Phi_0\left(\frac{(y_d+1) \wedge 1}{|y|}\right).
\end{align*}
In particular, $\F^i$ is bounded for all $1\le i \le i_0$. 
\end{lemma}
\begin{proof} 
Let $y=(\wt y,y_d)\in \bH_{-1}$. We fix  $Q\in \partial D$ and use the  coordinate system CS$_{Q}$.   For each $\eps \in (0,\wh R/8]$, 
define $y_\eps=\eps(y+\e_d) =(\wt y_\eps, (y_\eps)_d)$. Then we have
	\begin{align}\label{e:y_eps}
		\wt y_\eps = \eps \wt y, \quad (y_\eps)_d=\eps(y_d+1), \quad  
		\delta_D(\eps \e_{d})
		=\eps \quad \text{and} \quad |\eps \e_d - y_\eps| = \eps|y|.
	\end{align}

Fix a constant $\eps_0\in (0,\wh R/16)$ satisfying $\eps_0 |\wt y|<\wh R/32$, 
$\eps_0(y_d+1)<\wh R/16$ and  $\eps_0(y_d+1)>2^{2+3\nu}\eps_0^{1+\nu} \wh R^{-\nu} |\wt y|^{1+\nu}$. 
Then  $y_\eps \in E^Q_\nu(\wh R/8)$ for all $\eps \in (0,\eps_0)$. Moroever, for all $\eps \in (0,\eps_0)$, using Lemma \ref{l:B5'}, 
and \eqref{e:y_eps}, we get
	\begin{align}\label{e:F-B5-1}
		& \bigg|\sB(\eps \e_d,y_\eps)- \sum_{i=1}^{i_0}\mu^i(\eps \e_d)\F^i(y) \bigg|\\
		&= \bigg|\sB(\eps \e_d,y_\eps)- \sum_{i=1}^{i_0}\mu^i(\eps \e_d)\F^i((y_\eps - \eps \e_d)/\eps) \bigg|\nn\\
		& \le c\bigg(\frac{1 \vee (y_d+1) \vee |y|}{1 \wedge (y_d+1)  \wedge |y| }\bigg)^{\theta_1} 
		\big(\eps(1 \vee  (y_d+1) \vee |y|)\big)^{\theta_2}
		=:c(y) \eps^{\theta_2}.\nn
	\end{align}

If $|\wt y|\le (y_d+1) \vee 1$, then by Lemma \ref{l:C11}(iii), we get that for all $\eps<\eps_0$,
 \begin{align*}
 	& |\eps \e_d - y_\eps|^2= \eps^2(|\wt y|^2 + y_d^2) \le \begin{cases}	2\eps^2	&\mbox{ if } y_d\le 0,\\2\eps^2(y_d+1)^2	&\mbox{ if } y_d>0\end{cases}\\
 	& \le  2(\eps \vee (y_\eps)_d)^2 \le 4(\delta_D(\eps \e_d) \vee \delta_D(y_\eps))^2.\end{align*} 
Hence,  by \hyperlink{B4-a}{{\bf (B4-a)}}, \hyperlink{B4-b}{{\bf (B4-b)}}, Lemma \ref{l:C11}(iii),  \eqref{e:y_eps} and \eqref{e:scale},   there exists $c_1 \ge 1$ such that  for all $\eps <\eps_0$,
\begin{align}\label{e:F-B5-2}
 c_1^{-1} \1_{|\wt y| \le (y_d+1) \vee 1} \,\Phi_0\bigg(\frac{(y_d+1) \wedge 1}{|y|}\bigg)\le	\sB(\eps \e_d, y_\eps) \le c_1 \Phi_0\bigg(\frac{(y_d+1) \wedge 1}{|y|}\bigg).
\end{align}

Now, by choosing $\eps\in (0,\eps_0)$ small enough so that $c(y)\eps^{\theta_2}<2^{-1}c_1^{-1}\Phi_0(((y_d+1) \wedge 1)/|y|)$, we deduce from  \eqref{e:F-B5-1},  \eqref{e:F-B5-2}, the triangle inequality and \eqref{e:mu-bound} that 
	\begin{align*}
		\frac{\1_{|\wt y| \le (y_d+1) \vee 1 }}{2c_1C_{11}} \Phi_0\left(\frac{(y_d+1) \wedge 1}{|y|}\right)\le  \sum_{i=1}^{i_0}\F^i(y)  		\le 2c_1C_{11} \Phi_0\left(\frac{(y_d+1) \wedge 1}{|y|}\right).
	\end{align*}
The proof is complete.
\end{proof}

\begin{lemma}\label{l:constant}	
	For every $1\le i \le i_0$,	$q\mapsto C(\alpha,q,\F^i)$ is a well-defined strictly increasing continuous function on 
	$[(\alpha-1)_+, \alpha+\beta_0)$ and $C(\alpha,(\alpha-1)_+,\F^i)=0$.
\end{lemma}
\begin{proof}
Fix $1\le i \le i_0$ and an arbitrary  $\lb_0\in [0, \beta_0]$ such that the first inequality in \eqref{e:scale} holds. 
Let $q \in [(\alpha-1)_+, \alpha+\lb_0)$.
Since $\F^i$ is non-negative, $C(\alpha,q,\F^i)$ is non-negative. By Lemma \ref{l:F-basic}, $C(\alpha,q,\F^i)$ is bounded above by
	\begin{align*}
		&c\bigg(\int_{0}^{1/2}+\int_{1/2}^1\bigg)\int_{\R^{d-1}}\frac{1}{(|\wt u|^2+1)^{(d+\alpha)/2}} \\
		&\qquad \qquad \qquad \qquad \qquad \times \frac{(s^q -1)(1-s^{\alpha-1-q})}{(1-s)^{1+\alpha}} 
		\Phi_0\left(\frac{s}{(1-s)|(\wt u,1)|}\right)   d \wt u \,ds 	\\
		&=:c(I_1+I_2).
	\end{align*}
Since $\Phi_0$ is bounded, using the mean value theorem, we get
	\begin{align*}
		I_2 &\le c\int_{\R^{d-1}}\!\frac{d \wt u}{(|\wt u|^2+1)^{(d+\alpha)/2}} \!\int_{1/2}^1  
		\!\frac{(1-s^q)(s^{\alpha-1-q}-1)}{(1-s)^{1+\alpha}} ds   \le c\int_{1/2}^1  \frac{ds}{(1-s)^{\alpha-1}}<\infty.
	\end{align*}
Besides, using \eqref{e:scale}, since $\sup_{s\in (0,1/2)}( (1-s^q)/(1-s)^{1+\alpha})<\infty$ and $q<\alpha+\lb_0$, we get 
	\begin{align*}
		&I_1\le c\int_{\R^{d-1}}\frac{1}{(|\wt u|^2+1)^{(d+\alpha)/2}}\int_0^{1/2}  (s^{\alpha-1-q}-1) 
	 		\left(\frac{s}{(1-s)|(\wt u,1)|}\right)^{\lb_0}  ds \, d \wt u \\
		&\le c\int_{\R^{d-1}}\frac{d \wt u}{(|\wt u|^2+1)^{(d+\alpha+\lb_0)/2}}
			\int_0^{1/2}  (s^{\alpha-1-q+\lb_1}-s^{\lb_0})  ds <\infty.
\end{align*}
Therefore, $C(\alpha,q, \F^i)$ is well-defined. 
Continuity can be proved using the dominated convergence theorem.

For each fixed $s \in (0,1)$, the map  $f_s(q):= (s^q-1)(1-s^{\alpha-1-q})/(1-s)^{1+\alpha}$ is strictly increasing on 
$[(\alpha-1)_+,\infty)$ and satisfies  $f_s(\alpha-1)=0$. Thus, $q\mapsto C(\alpha,q,\F^i)$ is strictly increasing on 
$[(\alpha-1)_+,\alpha+\beta_0)$ and $C(\alpha,(\alpha-1)_+,\F^i)=0$.
\end{proof}

Under \hyperlink{B5-I}{{\bf (B5-I)}},  we give a sufficient condition for $\lim_{q \to \alpha+\beta_0} C(\alpha,q,\F)=\infty$ 
so that \eqref{e:K4} holds trivially.
\begin{lemma}\label{l:C=infty}
Assume \hyperlink{B5-I}{{\bf (B5-I)}}. Let $\ell_0:(0,1) \to (0,\infty)$ be  a non-decreasing function satisfying
	\begin{align}\label{e:sufficient-C=infty}
		\int_0^1 \frac{\ell_0(r)}{r} dr = \infty.
	\end{align}
Suppose that  
$\Phi_0(r) = r^{\beta_0} \ell_0(r)$ for $0<r \le 1$.
Then  $\lim_{q \to \alpha+\beta_0} C(\alpha,q,\F)=\infty$. 
\end{lemma}
\begin{proof} 
By Lemma \ref{l:F-basic}, for $\eps \in (0, (\alpha \wedge 1)/2)$, the constant $C(\alpha,\alpha+\beta_0-\eps,\F)$ is bounded below by
	\begin{align*}
		I(\eps)&:=c\int_{0}^{1/2}\int_{\R^{d-1}, \, |\wt u|<1}\frac{1}{(|\wt u|^2+1)^{(d+\alpha)/2}} \\
		&\qquad \qquad \qquad\qquad  \times  
		\frac{ (1-s^{\alpha+\beta_0-\eps} )(s^{-\beta_0-1+\eps}-1)}
		{(1-s)^{1+\alpha}} \Phi_0\left(\frac{s}{(1-s)|(\wt u,1)|}\right)   d \wt u \,ds.
	\end{align*}
Note that for all $s \in (0,1/2)$, we have  
$1-s^{\alpha+\beta_0-\eps} \ge 1- 2^{-\alpha/2}$, $1-s \le 1$ and
$s^{-\beta_0-1+\eps}\ge 2^{1/2}$ so that	 $s^{-\beta_0-1+\eps}-1 \ge (1-2^{-1/2})s^{-\beta_0-1+\eps}$.
Moreover, by \eqref{e:scale}, we see  that for all $s \in (0,1/2)$ and $\wt u \in \R^{d-1}$ with $|\wt u|<1$,
	$$
		\Phi_0\left(\frac{s}{(1-s)|(\wt u,1)|}\right)  \ge c\Phi_0(2^{-1/2}s) \ge c\Phi_0(s). 
	$$  
It follows that
	\begin{align*}
		I(\eps) &\ge \frac{c(1-2^{-\alpha/2})(1-2^{-1/2})}{2^{(d+\alpha)/2}} \int_{0}^{1/2}  s^{-\beta_0-1+\eps} 
		\Phi_0(s) ds \int_{\R^{d-1}, \, |\wt u|<1}   d \wt u\\
 		&  =c \int_{0}^{1/2} s^{-1+\eps} \ell_0(s) ds.
	\end{align*}
Hence, by \eqref{e:sufficient-C=infty}, we obtain 
$\lim_{q \to \alpha+\beta_0} C(\alpha,q,\F)\ge c\lim_{\eps \to 0} I(\eps)=\infty$. 
\end{proof}

\begin{remark}\label{r:C=infty}
Let $a \ge 0$. Note that $\ell_0(r)=\log^a(e/r)$ satisfies \eqref{e:sufficient-C=infty}. Hence, if 
$\Phi_0(r)=r^{\beta_0} \log^a(e/r)$ for $0< r \le 1$, then  $\lim_{q \to \alpha+\beta_0} C(\alpha,q,\F)=\infty$.
\end{remark}

\subsection{Estimates of some auxiliary integrals}

In this subsection, we present estimates of some integrals that will be used in the proof of Proposition \ref{p:barrier}.

\begin{lemma}\label{l:int-F}
Let $q \in [0, \alpha+\beta_0)$.
For all  $1\le i \le i_0$, $x=(\wt 0,x_d) \in \bH$ and $\eps>0$, 
	\begin{align*}
		\int_{\bH, \, |x-y|>\eps} \frac{y_d^q\, \F^i((y-x)/x_d)}{|x-y|^{d+\alpha}} dy<\infty.
	\end{align*}
\end{lemma}
\begin{proof} 
Choose $\lb_0\in [0, \beta_0]$ such that $q \in [0, \alpha+\lb_0)$ and the first inequality in \eqref{e:scale} holds.
Fix $1 \le i \le i_0$, $x=(\wt 0,x_d) \in \bH$ and $\eps>0$.
	Using Lemma \ref{l:F-basic} and the almost monotonicity of $\Phi_0$, we obtain
	\begin{align*}
		\int_{\bH, \, |x-y|>\eps} \frac{y_d^q\, \F^i((y-x)/x_d)}{|x-y|^{d+\alpha}} dy 
		&\le c \eps^{-d-\alpha}	\int_{\bH, \, |x-y|\le 2|x|} y_d^q \Phi_0(x_d/|x-y|) dy \\
		&\quad + c	\int_{\bH, \, |x-y|>2|x|} \frac{y_d^q \Phi_0(x_d/|x-y|)}{|x-y|^{d+\alpha}} dy\\
		&=:I_1+I_2.
	\end{align*}
Since $\Phi_0$ is bounded,  $I_1 \le 
c(\eps) |x|^q \int_{\R^d, \, |x-y|\le 2|x|} dy <\infty$.  
On the other hand, for any $y\in \bH$ with $|x-y|>2|x|$, we have $|y|\le |x-y|+|x|<2|x-y|$ and $|y| \ge |x-y|-|x| >|x|=x_d$. Thus, using \eqref{e:scale}, since $q<\alpha+\lb_0$, we obtain
	\begin{align*}
		I_2 \le c\int_{\bH, \, |y|>|x|} \frac{y_d^q }{(|y|/2)^{d+\alpha}} 
		\bigg(\frac{x_d}{|y|/2}\bigg)^{\lb_0} dy\le cx_d^{\lb_0} \int_{\R^d \setminus B(0,x_d)} |y|^{-d-\alpha+q-\lb_0}
		dy<\infty.
	\end{align*}
The proof  is complete. \end{proof}

\begin{lemma}\label{l:I2}
Let $q \in [(\alpha-1)_+, \alpha+\beta_0)$. 
There exists $C>0$  such that for all $1\le i \le i_0$, $x=(\wt 0, x_d) \in \bH$ and $\delta \in (0, (x_d \wedge 1)/2)$,
	\begin{align*}	
		\left|	\int_{\bH, \, |x-y|> \delta} (y_d^q-x_d^q)  
		\frac{\F^i((y-x)/x_d)}{|x-y|^{d+\alpha}}dy -	C(\alpha,q,\F^i)x_d^{q-\alpha} \right| 
		\le C(\delta/x_d)^{2-\alpha} x_d^{q-\alpha}.	\end{align*}
	\end{lemma}
\begin{proof} 
Let $1\le i \le i_0$, $x =(\wt 0,x_d) \in \bH$ and $\delta \in (0,(x_d \wedge 1)/2)$. We set
	$$
		I(q,\delta):=\int_{\bH, \, |x-y|>\delta}  (y_d^q - x_d^q)\frac{ \F^i((y-x)/x_d)}{|x-y|^{d+\alpha}}dy.
	$$
Using Lemma \ref{l:int-F} twice (with $q$ and $q=0$), one sees that the integrand in the above integral is absolutely integrable. Hence $I(q,\delta)$ is well-defined. Using the change of variables $z=y/x_d$ in the first equality below and  the change of  
variables $\wt u = \wt z/(z_{d}-1)$ in the second, we get
	\begin{align}\label{e:generator-F-0}
		&I(q,\delta)		=x_d^{q-\alpha} \int_{\bH,\, |\wt{z}|^2+(z_d-1)^2> (\delta/x_d)^2}
			\frac{(z_d^q-1)\,\F^i(z- \e_d)}{| (\wt{z}, z_d)-\mathbf{e}_d|^{d+\alpha}} d \wt{z}\,	dz_d \\
		&=x_d^{q-\alpha}\int_{\bH,\, (|\wt{u}|^2 +1)(z_d-1)^2> (\delta/x_d)^2}
			\frac{(z_d^q-1)\,\F^i(((z_d-1)\wt u, z_d-1))}{(|\wt u|^2+1)^{(d+\alpha)/2} |z_d-1|^{1+\alpha}} d \wt u\,	dz_d.\nn
	\end{align}
Set  $\epsilon(\delta, \wt u):=(\delta/x_d)(|\wt{u}|^2+1)^{-1/2}  \in (0,1/2)$.  
By Fubini's theorem, we obtain from \eqref{e:generator-F-0} that
	\begin{align}\label{e:generator-F-1}
		I(q,\delta)	=x_d^{q-\alpha}\int_{\R^{d-1}} ( I_1(q,\delta,\wt u)  + I_2(q,\delta,\wt u) )
			\frac{d\wt u}{(|\wt u|^2+1)^{(d+\alpha)/2}},
	\end{align}
where
	\begin{align*}
		I_1(q,\delta, \wt u)&:= \int_0^{1-	\epsilon(\delta, \wt u) } 
			\frac{(z_d^q -1)\, \F^i(((z_d-1)\wt u, z_d-1))}{|z_d-1|^{1+\alpha}} dz_d,\\
		I_2(q,\delta, \wt u)&:= \int_{1+	\epsilon(\delta, \wt u) }^\infty 
			\frac{(z_d^q -1)\, \F^i(((z_d-1)\wt u, z_d-1))}{|z_d-1|^{1+\alpha}}  dz_d.
	\end{align*}
Using the change of the variables $s=1/z_d$ and  \eqref{e:A5-F}, we see that
	\begin{align*}
		I_2(q,\delta, \wt u)&= \int_0^{(1+\epsilon(\delta, \wt u))^{-1}}  
			\frac{(1/s)^q -1}{|(1/s)-1|^{1+\alpha}} \F^i(((1/s-1)\wt u, 1/s-1))\, \frac{ds}{s^2}\\
		&=\bigg(\int_0^{1-\epsilon(\delta, \wt u) }+ \int_{1-\epsilon(\delta, \wt u) }^{(1+\epsilon(\delta, \wt u))^{-1}} \bigg) 
			\frac{s^{\alpha-1-q}(1-s^q)}{(1-s)^{1+\alpha}} \F^i(((s-1)\wt u, s-1)) ds\\
		&=:I_{2,1}(q,\delta,\wt u) + I_{2,2}(q,\delta,\wt u).
	\end{align*}
Note that 
$(1+\epsilon(\delta, \wt u))^{-1}-1+\epsilon(\delta, \wt u) 
= 	\epsilon(\delta, \wt u)^2 (1+\epsilon(\delta, \wt u))^{-1} \le \epsilon(\delta, \wt u)^2.$
Therefore, since $\F^i$ is bounded and  $\epsilon(\delta, \wt u) \le \delta/x_d<1/2$,
by using the mean value theorem we have
	\begin{align*}
		\left| I_{2,2}(q,\delta, \wt{u})\right|
		& \le c\int_{1-\epsilon(\delta, \wt u)}^{1-\epsilon(\delta, \wt u)+\epsilon(\delta, \wt u)^2} 
			\frac{1-s^q}{(1-s)^{1+\alpha}}ds\\
		& \le c(2^{q-1}\vee 1) \int_{1-\epsilon(\delta, \wt u)}^{1-\epsilon(\delta, \wt u)+\epsilon(\delta, \wt u)^2} 
			\frac{ds}{(1-s)^{\alpha}}\\
		&\le  \frac{c(2^{q-1} \vee 1)}{(\epsilon(\delta, \wt u)/2)^{\alpha}}  
			\int_{1-\epsilon(\delta, \wt u)}^{1-\epsilon(\delta, \wt u)+\epsilon(\delta, \wt u)^2} ds
			\le c\epsilon(\delta, \wt u)^{2-\alpha},
	\end{align*}
which implies that 
	\begin{align}\label{e:generator-F-2}
		\left| \int_{\R^{d-1}}  \frac{I_{2,2}(q,\delta, \wt u)}{(|\wt u|^2+1)^{(d+\alpha)/2}}d\wt u\right|
		\le c  (\delta/x_d)^{2-\alpha}\int_{\R^{d-1}}  \frac{ d\wt u}{(|\wt u|^2+1)^{(d+2)/2}}=c(\delta/x_d)^{2-\alpha}.
	\end{align}
On the other hand,  by the mean value theorem,  we have
	\begin{align*}
		\left| (s^q -1)(1-s^{\alpha-1-q})(1-s)^{-1-\alpha} \right| \le c(1-s)^{1-\alpha} \quad \text{for all} \;\, s \in (1/2,1).
	\end{align*} 
Thus, since $\F^i$ is bounded,  we get
	\begin{align}\label{e:generator-F-3}	
		& \left| \int_{\R^{d-1}}(I_1(q,\delta, \wt u) + I_{2,1}(q,\delta, \wt u)) \frac{d\wt u}{(|\wt u|^2+1)^{(d+\alpha)/2}} 
			- C(\alpha,q,\F^i)\right| \\	
		&= \int_{\R^{d-1}}\frac{1}{(|\wt u|^2+1)^{(d+\alpha)/2}}\int_{1-\epsilon(\delta, \wt u)}^1  
			\frac{(s^q -1)(1-s^{\alpha-1-q})}{(1-s)^{1+\alpha}} \F^i(((s-1)\wt u, s-1)) ds \, d \wt u\nn\\
		&\le c\int_{\R^{d-1}}\frac{1}{(|\wt u|^2+1)^{(d+\alpha)/2}}\int_{1-\epsilon(\delta, \wt u)}^1  
			\frac{ds}{(1-s)^{-1+\alpha}}  d \wt u \nn\\
		& \le c(\delta/x_d)^{2-\alpha} \int_{\R^{d-1}}  \frac{ d\wt u}{(|\wt u|^2+1)^{(d+2)/2}}=c(\delta/x_d)^{2-\alpha}.\nn 
	\end{align}
Combining \eqref{e:generator-F-1} with \eqref{e:generator-F-2} and \eqref{e:generator-F-3}, we arrive at the result.
\end{proof}

\begin{lemma}\label{l:tech-1} 
Let $q \in [0,\alpha+\beta_0)$, $\nu \in (0,1]$, $r\in (0, \wh R/8]$ and $x=(\wt 0,x_d) \in \bH$ with $x_d \le r/4$. 
Then the following hold.

\noindent (i) There exist constants  $C>0$ and $b_1>0$  independent of $r$ and $x$ such that
	\begin{align*}
		\int_{\R^d \setminus (E_\nu^0(r) \cup \wt E_\nu^0(r))}   
			\frac{|y_d|^q\,\Phi_0(x_d/|x-y|)}{|x-y|^{d+\alpha}}dy\le C (x_d/r)^{b_1} x_d^{q-\alpha}.
	\end{align*}

\noindent (ii) There exist constants  $C>0$ and $b_1'>0$  independent of $r$ and $x$ such that
	\begin{align*}
		\int_{B(0,\wh R) \setminus (E_\nu^0(r) \cup \wt E_\nu^0(r))}   
			\frac{(\wh R^{-1}|\wt y|^2)^q\,\Phi_0(x_d/|x-y|)}{|x-y|^{d+\alpha}}dy\le C (x_d/r)^{b_1'} x_d^{q-\alpha}.
	\end{align*}
\end{lemma}
\begin{proof} 
Choose $\lb_0\in [0, \beta_0]$ so that $q \in [0,\alpha+\lb_0)$ and \eqref{e:scale} holds.
Since $\Phi_0$ is almost increasing, for all $y =(\wt y, y_d)\in \bH \setminus   E_\nu^0(r)$, 
	\begin{align}\label{e:tech-reflect}
		\frac{\Phi_0(x_d/|x-(\wt y,- y_d)|)}{|x-(\wt y,-y_d)|^{d+\alpha}}\le c	\frac{\Phi_0(x_d/|x-y|)}{|x-y|^{d+\alpha}} .
	\end{align}
	
\noindent (i)  By \eqref{e:tech-reflect}, to get the desired result, it suffices to show that
	\begin{align*}
		I:= \int_{ \bH \setminus   E_\nu^0(r)}\frac{y_d^q\,\Phi_0(x_d/|x-y|)}{|x-y|^{d+\alpha}}dy\le c_1(x_d/r)^{b_1} x_d^{-\alpha}
	\end{align*}
for some constants $c_1,b_1>0$ independent of $r$ and $x$.
	
We now estimate $I$ from above by splitting the integral into five pieces over not necessarily disjoint subsets of $	\bH$ that may contain parts of $E^0_{\nu}(r)$.
Set $l_\nu:=(r^\nu x_d)^{1/(1+\nu)}/4 \in (0,r/4)$.  For any $y=(\wt y,y_d) \in \bH \setminus E^0_\nu(r)$ with $|\wt y|=s<l_\nu$, we have either $y_d \le 4r^{-\nu}s^{1+\nu}$ or $y_d \ge r/2$. Thus,  since $\Phi_0$ is almost increasing, we get
	\begin{align*}
		I&\le c\int_{0}^{x_d/8} s^{d-2} \int_0^{4r^{-\nu}s^{1+\nu}} 
			\frac{y_d^q\, \Phi_0(x_d/|x_d-y_d|) }{|x_d-y_d|^{d+\alpha}} dy_d \,  ds \\
		&\quad  +  c\int_{x_d/8}^{l_\nu} s^{d-2} \int_0^{4r^{-\nu}s^{1+\nu}} \frac{y_d^q \, \Phi_0(x_d/s)}{s^{d+\alpha}} dy_d \,  ds\\
		&\quad +   c\int_0^{l_\nu} s^{d-2} \int_{r/2}^\infty 
			\frac{y_d^q \, \Phi_0(x_d/|x_d-y_d|)}{|x_d-y_d|^{d+\alpha}} dy_d \,  ds  \\
		&\quad +   c\int_{l_\nu}^\infty s^{d-2} \int_0^{16s} \frac{y_d^q \, \Phi_0(x_d/s)}{s^{d+\alpha}} dy_d \,  ds  \\
		&\quad  +   c\int_{l_\nu}^\infty s^{d-2} \int_{16s}^\infty 
			\frac{y_d^q \, \Phi_0(x_d/|x_d-y_d|)}{|x_d-y_d|^{d+\alpha}} dy_d \,  ds  \\
		&=:cI_1+cI_2+cI_3+cI_4+ cI_5.
	\end{align*}
	
For $s\in (0,x_d/8)$ and $y_d \in (0,4r^{-\nu}s^{1+\nu})$, we have $y_d \le 4s \le  x_d/2$, and so $|x_d-y_d| \ge x_d/2$. 
Thus, since $\Phi_0$ is bounded, we get
	\begin{align*}
		I_1&\le c(x_d/2)^{-d-\alpha} \int_0^{x_d/8} s^{d-2} \int_0^{4r^{-\nu}s^{1+\nu}} y_d^q \,dy_d\,ds\\
		&= cr^{-\nu(q+1)} x_d^{-d-\alpha}\int_0^{x_d/8} s^{d-2 + (1+\nu)(q+1)}ds = c(x_d/r)^{\nu(q+1)} x_d^{q-\alpha}.
	\end{align*}
	
For $I_2$,  using \eqref{e:scale} and the fact that $q<\alpha+\lb_0$, we get
	\begin{align*}
		I_2 &\le c\Phi_0(1)x_d^{\lb_0}\int_{x_d/8}^{l_\nu} s^{-\alpha-\lb_0-2} 
			\int_0^{4r^{-\nu}s^{1+\nu}} y_d^q\, dy_d\,ds\\
		&\le cr^{-\nu(q+1)}x_d^{\lb_0}\int_{x_d/8}^{l_\nu} s^{-\alpha-\lb_0-2 + (1+\nu)(q+1)}ds \\
		&\le cl_\nu^{\,\nu(q+1)}r^{-\nu(q+1)}x_d^{\lb_0}\int_{x_d/8}^{l_\nu} s^{-\alpha-\lb_0+q-1}ds\\
		&\le c l_\nu^{\,\nu(q+1)} r^{-\nu(q+1)} x_d^{q-\alpha} = c(x_d/r)^{\nu (q+1)/(1+\nu)}x_d^{q-\alpha}.
	\end{align*}
	
For  $y_d>r/2$,  we have $|x_d-y_d|\ge  y_d-r/4 \ge y_d/2$. Thus, using \eqref{e:scale} and the fact that $q<\alpha+\lb_0$, we obtain
	\begin{align*}
		I_3 &\le  c\int_0^{l_\nu} s^{d-2} \int_{r/2}^\infty y_d^{-d-\alpha+q} \, \Phi_0(x_d/y_d) dy_d \,  ds\\
		&\le  c\Phi_0(1)x_d^{\lb_0}\int_0^{l_\nu} s^{d-2} \int_{r/2}^\infty  y_d^{-d-\alpha-\lb_0+q} dy_d \,  ds  \\
		&= cr^{-d-\alpha- \lb_0+q+1}x_d^{\lb_0}\int_0^{l_\nu}  s^{d-2}ds\\
		&= cl_\nu^{\,d-1}r^{-d-\alpha- \lb_0+q+1}x_d^{\lb_0}
			= c(x_d/r)^{\alpha+\lb_0-q+(d-1)/(1+\nu)} \, x_d^{q-\alpha}.
	\end{align*} 
	
For $I_4$, by \eqref{e:scale},  we see that
	\begin{align*}
		I_4&\le c \Phi_0(1) x_d^{\lb_0} \int_{l_\nu}^\infty s^{-\alpha-\lb_0-2} \int_0^{16s} y_d^q\, dy_d \, ds 
		\le cx_d^{\lb_0}\int_{l_\nu}^\infty s^{-\alpha-\lb_0+q-1}  ds\\
		&= c l_\nu^{-\alpha-\lb_0+q} x_d^{\lb_0}=c(x_d/r)^{\nu(\alpha+\lb_0-q)/(1+\nu)} x_d^{q-\alpha}.
	\end{align*}
		
For 	$s>  l_{\nu} >  (r^\nu x_d)^{1/(1+\nu)}/8$ and $y_d>16s$, we have $y_d>2x_d$, so that $y_d-x_d>y_d/2$. 
Thus, using \eqref{e:scale}, we obtain
	\begin{align*}
		I_5 &\le c\int_{l_\nu}^\infty s^{d-2} \int_{16s}^\infty y_d^{q-d-\alpha} \, \Phi_0(x_d/y_d)dy_d \,  ds\\
		&\le  c\Phi_0(1)x_d^{\lb_0}\int_{l_\nu}^\infty  s^{d-2} \int_{16s}^\infty  y_d^{q-d-\alpha-\lb_0} dy_d \,  ds \\
		&= c x_d^{\lb_0}\int_{l_\nu}^\infty s^{-\alpha- \lb_0+q-1}ds= c(x_d/r)^{\nu(\alpha+\lb_0-q)/(1+\nu)} x_d^{q-\alpha}.
	\end{align*} 
The proof of (i) is complete.

(ii) By \eqref{e:tech-reflect}, to get the desired result, it suffices to show that
	\begin{align*}
		II:= 	\int_{(B(0,\wh R) \cap \bH) \setminus E_\nu^0(r)}   
			\frac{(\wh R^{-1}|\wt y|^2)^q\,\Phi_0(x_d/|x-y|)}{|x-y|^{d+\alpha}}dy\le c_1(x_d/r)^{b_1'} x_d^{-\alpha}
	\end{align*}
for some constants $c_1,b_1'>0$ independent of $r$ and $x$. Since $\Phi_0$ is almost increasing, 
 	\begin{align*}
		II		&\le c\int_0^{x_d/8} s^{d-2}   \int_{0}^{4r^{-\nu}s^{1+\nu}} 
			\frac{(r^{-1} s^2)^q\, \Phi_0(x_d/(x_d-y_d))}{(x_d-y_d)^{d+\alpha}} dy_d  \, ds\\
		&\quad + c\int_0^{x_d/8} s^{d-2}   \int_{r/2}^{\wh R} 
			\frac{(r^{-1} s^2)^q \, \Phi_0(x_d/(y_d-x_d))}{(y_d-x_d)^{d+\alpha}} dy_d \,  ds  \\
		&\quad +   c\int_{x_d/8}^{r/4} s^{d-2} \int_{0}^{4r^{-\nu}s^{1+\nu}} 
			\frac{(r^{-1} s^2)^q\, \Phi_0(x_d/s)}{s^{d+\alpha }} dy_d \, ds   \\
		&\quad +   c\int_{x_d/8}^{r/4} s^{d-2}  \int_{r/2}^{\infty} 
			\frac{(r^{-1} s^2)^q \, \Phi_0(x_d/(y_d-x_d))}{(y_d-x_d)^{d+\alpha }} dy_d\,  ds   \\
		& \quad + c\int_{r/4}^{\wh R}  s^{d-2} \int_{0}^{2s} \frac{(\wh R^{-1} s^2)^q\, \Phi_0(x_d/s)}{s^{d+\alpha }} dy_d \, ds \\
		& \quad + c \int_{r/4}^{\wh R}  s^{d-2}  \int_{2s}^{\wh R} 
			\frac{ (\wh R^{-1} s^2)^q \, \Phi_0(x_d/(y_d-x_d))}{|y_d-x_d|^{d+\alpha}} dy_d \, ds \\
		&=:II_{1}+II_{2}+II_{3}+II_{4}+II_{5}+II_{6}.
	\end{align*}	
Note that since $x_d\le r/4$, for all  $s \in (0,x_d/8)$ and $y_d \in (0, 4r^{-\nu}s^{1+\nu})$,
	\begin{align*}
		x_d-y_d \ge x_d - 4r^{-\nu}(x_d/8)^{1+\nu} \ge x_d/2 .
	\end{align*}
Using this, since $\Phi_0$ is bounded, we get
	\begin{align*}
		II_1&\le  cr^{-q} x_d^{-d-\alpha}\int_0^{x_d/8} s^{d-2+2q}   \int_{0}^{4r^{-\nu}s^{1+\nu}}  dy_d  \, ds\\
			&\le  cr^{-q-\nu} x_d^{-d-\alpha}\int_0^{x_d/8} s^{d-1+2q+\nu}    \, ds = c (x_d/r)^{q+\nu}x_d^{q-\alpha}.
	\end{align*}
We also note that $y_d-x_d \ge y_d-r/4 \ge y_d/2$ for all $y_d>r/2$. Using this and \eqref{e:scale},  since $q<\alpha+\lb_0$, we obtain
	\begin{align*}
		II_2&\le c\Phi_0(1)r^{-q}x_d^{\lb_0}\int_0^{x_d/8} s^{d-2+2q} ds \int_{r/2}^{\wh R} y_d^{-(d+\alpha +\lb_0)} dy_d \\
		&\le c(x_d/r)^{d+\alpha+\lb_0+q-1}x_d^{q-\alpha},
	\end{align*}
	\begin{align*}
	II_4&\le  c\Phi_0(1)r^{-q}x_d^{\lb_0}\int_{x_d/8}^{r/4} s^{d-2+2q}  ds\int_{r/2}^{\infty} y_d^{-(d+\alpha+\lb_0)} dy_d
	\le c(x_d/r)^{\alpha+\lb_0-q} x_d^{q-\alpha},
	\end{align*}
and
	\begin{align}
		II_6&\le c\Phi_0(1)\wh R^{-q} x_d^{\lb_0} \int_{r/4}^{\wh R}  s^{d-2+2q}\int_{2s}^{\wh R}y_d^{-(d+\alpha+\lb_0)} dy_d\, ds\nn\\
		&\le  c\wh R^{-q} x_d^{\lb_0} \int_{r/4}^{\wh R}  s^{-\alpha-\lb_0-1+2q} ds\nn\\
		&\le  c x_d^{\lb_0} \int_{r/4}^{\wh R}  s^{-\alpha-\lb_0-1+q} ds 
		\le c(x_d/r)^{\alpha+\lb_0-q} x_d^{q-\alpha}.\label{e:tech-large}
	\end{align}
For $II_5$, using \eqref{e:scale} and \eqref{e:tech-large}, we see that
	\begin{align*}
		II_5&\le c\Phi_0(1) \wh R^{-q} x_d^{\lb_0} \int_{r/4}^{\wh R}  s^{-\alpha-\lb_0-2+2q} \int_{0}^{2s} dy_d \, ds\\
		&\le c \wh R^{-q} x_d^{\lb_0} \int_{r/4}^{\wh R}  s^{-\alpha-\lb_0+1+2q}  ds 
		\le c(x_d/r)^{\alpha+\lb_0-q} x_d^{q-\alpha}.
	\end{align*}
For $II_3$, we fix a constant $\eps\in (0, (\alpha+\lb_1-q) \wedge (q+\nu))$. Using \eqref{e:scale}, we get
	\begin{align*}
		II_3&\le c\Phi_0(1)r^{-q} x_d^{\lb_0}\int_{x_d/8}^{r/4} s^{-\alpha-\lb_0-2+2q } \int_{0}^{4r^{-\nu}s^{1+\nu}} dy_d \, ds\\
		&\le cr^{-q-\nu} x_d^{\lb_0}\int_{x_d/8}^{r/4} s^{-\alpha-\lb_0-1+2q+\nu }  \, ds \\
		&\le cr^{-\eps} x_d^{\lb_0}\int_{x_d/8}^{r/4} s^{-\alpha-\lb_0-1+q+\eps}  \, ds
		\le c(x_d/r)^\eps x_d^{q-\alpha}.
	\end{align*}	
The proof is complete. 
\end{proof}

\subsection{Key estimates on cutoff distance functions}

For a Borel set $V \subset D$ and $q \in  [(\alpha-1)_+,\alpha+\beta_0) \cap (0,\infty)$,  let 
	\begin{align}\label{e:def-hqV}
		h_{q,V}(y)=\1_V(y) \delta_D(y)^q
	\end{align} 
be the $q$-th power of the cutoff distance function. 
 
The next proposition is the main result of this subsection  and one of the key estimates of this work.

\begin{prop}\label{p:barrier}
Let  $q\in [(\alpha-1)_+,\alpha+\beta_0) \cap (0,\infty)$, 
$Q \in \partial D$ and $r \in (0,\wh R/8]$. There exist constants $C>0$ and 
$\eta_1>0$ independent of $Q$ and $r$ such that for any Borel set $V$ satisfying 
$U^Q(3r)  \subset V \subset B_D(Q, \wh R)$ and   any $x \in U^Q(r/4)$,
	\begin{align*}
		|L_\alpha^\sB h_{q,V}(x)-\sum_{i=1}^{i_0} \mu^i(x) C(\alpha,q,\F^i)  \delta_D(x)^{q-\alpha}| 
		\le C (\delta_D(x)/r)^{\eta_1} \delta_D(x)^{q-\alpha}.
	\end{align*}
\end{prop}
\begin{proof} 
Let $x \in U^Q(r/4)$ and $Q_x \in \partial D$ be the point such that $\delta_D(x)=|x-Q_x|$. We use the coordinate system CS$_{Q_x}$ and denote $E^{Q_x}_\nu(r)$ and $\wt E^{Q_x}_\nu(r)$  by $E_\nu$ and $\wt E_\nu$  respectively. Note that $x=(\wt 0,x_d)=(\wt 0, \delta_D(x))\in E_\nu$ and
	\begin{align}\label{e:delta-Enu-x}
		\delta_{E_\nu}(x) \ge \delta_{\{(\wt y,y_d):y_d>4|\wt y|\}} ((\wt 0,x_d))  = x_d/\sqrt{17}=\delta_D(x)/\sqrt{17}.
	\end{align}
Using \eqref{e:U-rho-C11-1} twice, we see that $B_D(x,r) \subset B_D(Q, r + |x-Q|) \subset B_D(Q,3r/2) \subset U^Q(3r)$. Also, for any $y=(\wt y,y_d) \in E_\nu$, we have $|x-y| \le |\wt y| + |y_d-x_d| < r/4+r/2$. Hence,
 	\begin{align}\label{e:EsubsetV}
		E_\nu \subset B_D(x, r) \subset U^Q(3r)\subset V.
 	\end{align}

Let 
	$$
		\sO:=B(x,5^{-1}r^{-\theta_2/(2\alpha+2\theta_1)} x_d^{1+\theta_2/(2\alpha+2\theta_1)}),
	$$
where $\theta_1,\theta_2>0$ are the constants in \eqref{e:ass-B5'}.
Since $x_d<r/4$, we have by \eqref{e:delta-Enu-x}, 
	\begin{equation}\label{e:delta-Enu-x-2}
		5^{-1}r^{-\theta_2/(2\alpha+2\theta_1)} x_d^{1+\theta_2/(2\alpha+2\theta_1)}<5^{-1} x_d\le \delta_{E_\nu}(x)
	\end{equation}
so that $\sO \subset E_\nu$.  Thus, since   $h_{q,V}(x)=x_d^q$,  we get that

	\begin{align*}
		L_\alpha^{\sB} h_{q,V}(x)&=\text{p.v.}\int_{D}  \frac{(h_{q,V}(y)-h_{q,V}(x))\sB(x,y)}{|x-y|^{d+\alpha}}dy\\
		& =\text{p.v.} \int_{E_\nu}   \frac{(h_{q,V}(y)-y_d^q)\sB(x,y)}{|x-y|^{d+\alpha}}dy 
			+ \text{p.v.} \int_{E_\nu}   \frac{(y_d^q-x_d^q)\sB(x,y)}{|x-y|^{d+\alpha}}dy\\
		&\quad +\int_{D \setminus E_\nu}  \!\!\!\frac{(h_{q,V}(y)-h_{q,V}(x)) \sB(x,y)}{|x-y|^{d+\alpha}}dy \\
		&=: I_1+J_2+J_3.
	\end{align*}
We further split $J_2$ and $J_3$ as follows:
	\begin{align*}
		J_2&=\text{p.v.} \int_{\sO} \frac{ qx_d^{q-1}(y_d-x_d)\sB(x,x)}{|x-y|^{d+\alpha}}dy\\
		&\quad + \int_{\sO}  \frac{(y_d^q-x_d^q-qx_d^{q-1}(y_d-x_d)) \sB(x,x) }{|x-y|^{d+\alpha}}dy 
			+\int_{\sO}(y_d^q-x_d^q) \frac{\sB(x,y)-\sB(x,x)}{|x-y|^{d+\alpha}}dy \\
		& \quad  +\int_{E_\nu\setminus \sO}   \frac{(y_d^q-x_d^q) (\sB(x,y)- 
			\sum_{i=1}^{i_0}\mu^i(x)\F^i((y-x)/x_d)))}{|x-y|^{d+\alpha}}dy\\
		&\quad + \sum_{i=1}^{i_0}\int_{\bH\setminus \sO}  \frac{(y_d^q-x_d^q) \mu^i(x)\F^i((y-x)/x_d)}{|x-y|^{d+\alpha}}dy\\
		&\quad   - \sum_{i=1}^{i_0} \int_{\bH \setminus E_\nu} \frac{(y_d^q-x_d^q) \mu^i(x)\F^i((y-x)/x_d)}{|x-y|^{d+\alpha}}dy \\
		& =:I_2+I_3+I_4+I_5+I_6-I_7,
	\end{align*}
	\begin{align*}
		&J_3=\int_{V \setminus E_\nu}  \frac{h_{q,V}(y)\sB(x,y)}{|x-y|^{d+\alpha}}dy  
			-  \int_{D \setminus E_\nu}  \frac{h_{q,V}(x)\sB(x,y)}{|x-y|^{d+\alpha}}dy =:I_8-I_9.
	\end{align*}

Estimates of the integrals $I_1$ and $I_5$ are the most delicate and are postponed to, respectively,  Lemmas \ref{l:I1} and  \ref{l:I5} below which together give that
	\begin{equation}\label{e:I1-I5}
		|I_1|+|I_5|  \le c(x_d/r)^{\eta} x_d^{q-\alpha}
	\end{equation}
for some constants $c,\eta>0$ independent of $Q,r,x$ and $V$. In the rest of the proof we estimate the remaining seven integrals.

By Lemma \ref{l:I2} and  \eqref{e:mu-bound},  we get 
	\begin{align*}
		&\bigg| I_6-\sum_{i=1}^{i_0}\mu^i(x)C(\alpha, q, \F^i)x_d^{q-\alpha} \bigg|\\
		& \le \sum_{i=1}^{i_0} \mu^i(x) \bigg| \int_{\bH\setminus \sO} 
			 \frac{(y_d^q-x_d^q) \F^i((y-x)/x_d)}{|x-y|^{d+\alpha}}dy -C(\alpha, q, \F^i)x_d^{q-\alpha} \bigg| \\
		&\le  c\sum_{i=1}^{i_0} ( r^{-\theta_2/(2\alpha+2\theta_1)} x_d^{\theta_2/(2\alpha+2\theta_1)})^{2-\alpha} x_d^{q-\alpha} \\
		&\le c  (x_d/r)^{(2-\alpha)\theta_2/(2\alpha+2\theta_1)} x_d^{q-\alpha}.
	\end{align*}

We have $I_2=0$ by symmetry. Note that for any $y \in \sO$,  by the triangle inequality and \eqref{e:delta-Enu-x-2}, 
$(4/5)x_d \le \delta_D(y) \le (6/5)x_d$. Hence, 	since $\sO \subset E_\nu$, using  Lemma \ref{l:C11}(iii), we see that
	\begin{align}\label{e:sO}
		&	y_d \asymp \delta_D(y)  \asymp x_d \quad \text{for} \;\, y \in \sO.
	\end{align}
Using \eqref{e:B(x,x)},  Taylor's theorem and \eqref{e:sO}, we obtain
	\begin{align*}
		|I_3| &\le  c \int_{\sO} \frac{x_d^{q-2}|x_d-y_d|^2}{|x-y|^{d+\alpha}} dy \le cx_d^{q-2} 
			\int_0^{5^{-1}r^{-\theta_2/(2\alpha+2\theta_1)} x_d^{1+\theta_2/(2\alpha+2\theta_1)}} l^{1-\alpha}dl\\
		&\le  c x_d^{q-2} (x_d^{1+\theta_2/(2\alpha+2\theta_1)}/r^{\theta_2/(2\alpha+2\theta_1)})^{2-\alpha} 
			= c(x_d/r)^{(2-\alpha)\theta_2/(2\alpha+2\theta_1)} x_d^{q-\alpha}.
	\end{align*}

When $\alpha\ge 1$, by the mean value theorem, \eqref{e:sO} and  \hyperlink{B3}{{\bf (B3)}}, 
	\begin{align}\label{e:I4}
 		|I_4| &\le c \int_\sO \frac{x_d^{q-1-\theta_0}}{|x-y|^{d+\alpha-1-\theta_0}}dy \le c x_d^{q-1-\theta_0} 
 			\Big(\frac{ x_d^{1+\theta_2/(2\alpha+2\theta_1)}}{r^{\theta_2/(2\alpha+2\theta_1)}}\Big)^{1+\theta_0-\alpha}\nn\\	
 		& = c \left(\frac{x_d}{r}\right)^{(1+\theta_0-\alpha)\theta_2/(2\alpha+2\theta_1)}x_d^{q-\alpha}.
	\end{align}
When $\alpha<1$, since $\sB$ is bounded, \eqref{e:I4} holds with $\theta_0=0$. 

For $I_7$, using    Lemma \ref{l:F-basic}, \eqref{e:mu-bound} and Lemma \ref{l:tech-1}(i) twice, we get that
	\begin{align*}	
		|I_7| &\le  c \int_{\bH \setminus E_\nu }  \frac{y_d^q\Phi_0(x_d/|x-y|)}{|x-y|^{d+\alpha}}dy + cx_d^q 
			\int_{\bH \setminus E_\nu }  \frac{\Phi_0(x_d/|x-y|)}{|x-y|^{d+\alpha}}dy \\
			&\le c(x_d/r)^{b_1}x_d^{q-\alpha} +c (x_d/r)^{b_2}r^{q-\alpha} ,
	\end{align*}
for some constants $b_1,b_2>0$ independent of $Q,r,x$ and $V$.

For $I_9$, using Lemma \ref{l:C11}(ii), \hyperlink{B4-a}{{\bf (B4-a)}} and Lemma \ref{l:tech-1}(i),  we obtain
	\begin{align*}     
		&    I_9	\le cx_d^q \int_{\R^d \setminus (E^0_\nu \cup \wt E^0_\nu)} 
			\frac{\Phi_0(x_d/|x-y|)}{|x-y|^{d+\alpha}} dy \le c (x_d/r)^{b_2}r^{q-\alpha}.
	\end{align*}

For $I_8$, using the fact that $V \subset B_D(Q,\wh R)$, Lemma \ref{l:C11}(i)-(ii),   \hyperlink{B4-a}{{\bf (B4-a)}}
and Lemma \ref{l:tech-1}(i)-(ii), we get
	\begin{align*}
		I_8 &\le c\int_{B(0,\wh R) \setminus (E^0_\nu \cup \wt E^0_\nu)}  
			\frac{(|y_d| + \wh R^{-1}|\wt y|^2)^q\, \Phi_0(x_d/|x-y|)}{|x-y|^{d+\alpha}}dy\\
		&\le c\int_{\R^d \setminus (E^0_\nu \cup \wt E^0_\nu)} \frac{|y_d|^q \, \Phi_0(x_d/|x-y|)}{|x-y|^{d+\alpha}}dy\\
		&\quad +  c\int_{B(0,\wh R) \setminus (E^0_\nu \cup \wt E^0_\nu)} (\wh R^{-1}|\wt y|^2)^q 
			\frac{\Phi_0(x_d/|x-y|)}{|x-y|^{d+\alpha}}dy\\
		&\le c(x_d/r)^{b_1} x_d^{q-\alpha} + c(x_d/r)^{b_3} x_d^{q-\alpha},
	\end{align*}
for some constant $b_3>0$ independent of $Q,r,x$ and $V$.

Together with \eqref{e:I1-I5} this completes the proof. 
\end{proof}

In the remainder of this subsection, we fix $q, Q, r$ and $V$ as in Proposition \ref{p:barrier},  
let $\lb_0\in [0, \beta_0]$ be such that $q\in [(\alpha-1)_+,\alpha+\lb_0) \cap (0,\infty)$ and that the first inequality in \eqref{e:scale} holds, let $x\in U^Q(r/4)$ and let $Q_x\in \partial D$ be such that $\delta_D(x)=|x-Q_x|$,  use the coordinate system CS$_{Q_x}$, and denote 	$E^{Q_x}_\nu(r)$ and $\wt E^{Q_x}_\nu(r)$  by $E_\nu$ and $\wt E_\nu$  respectively.

\begin{lemma}\label{l:I1}
There exist constants  $C, b>0$ independent of $Q, r,x$ and $V$  such that
	\begin{align*}
		|I_1| \le C(x_d/r)^{b} x_d^{q-\alpha}.
	\end{align*}
\end{lemma} 
\begin{proof}  
Recall from the proof of Proposition \ref{p:bound-for-integral-new} that
	$$
		I_1=\text{p.v.} \int_{E_\nu}   \frac{(h_{q,V}(y)-y_d^q)\sB(x,y)}{|x-y|^{d+\alpha}}dy.
	$$
We will show that the integral above is actually absolutely convergent and will establish the required estimate.

Recall from \eqref{e:EsubsetV} that $E_\nu \subset V$. By Lemma \ref{l:C11}(i), (iii) and the mean value theorem, we see that for any $y \in E_\nu$,
	\begin{align*}
		&|h_{q,V}(y)-y_d^q| \le \big( (y_d+r^{-1}|\wt y|^2)^q - y_d^q\big) \vee \big(y_d^q-(y_d-r^{-1}|\wt y|^2)^q \big) \\
		&\le \big( 2q(2^{q-1} \vee 1)y_d^{q-1} r^{-1}|\wt y|^2 \big) \vee \big( q(2^{1-q} \vee 1)y_d^{q-1} r^{-1}|\wt y|^2 \big) 
			\le q2^{q+1} r^{-1} y_d^{q-1}|\wt y|^2 .
	\end{align*}
Hence, using \hyperlink{B4-a}{{\bf (B4-a)}}and \eqref{e:scale}, since  $\Phi_0$ is bounded and almost increasing,  we obtain
	\begin{align*}
		&|I_{1}| \le c \int_{0}^{8^{-1/(1+\nu)}r}  s^{d-2}\int_{4r^{-\nu}s^{1+\nu}}^{r/2} 
			\frac{r^{-1}y_d^{q-1}s^2}{(s+|y_d-x_d|)^{d+\alpha}} \Phi_0\bigg(  \frac{x_d}{s+|y_d-x_d|} \bigg) dy_d\, ds \\
		&\le c\int_0^{x_d/2} s^{d}\int_{0}^{x_d/2} \frac{r^{-1}y_d^{q-1}}{(x_d-y_d)^{d+\alpha}} dy_d\,ds  
			+ c \int_0^{x_d/2} s^{d}\int_{x_d/2}^{x_d-s}\frac{ r^{-1}y_d^{q-1}}{(x_d-y_d)^{d+\alpha}}dy_d\,ds   \\
		&\quad+ c \int_0^{x_d/2} s^{d}\int_{x_d-s}^{x_d+s}\frac{ r^{-1}y_d^{q-1}}{s^{d+\alpha}}dy_d\,ds  
			+ c  \int_0^{x_d/2} s^{d}\int_{x_d+s}^{2x_d} \frac{r^{-1}y_d^{q-1}}{(y_d-x_d)^{d+\alpha}}dy_d\,ds \\
		&\quad  + c\int_0^{x_d/2} s^{d}\int_{2x_d}^{r/2} \frac{r^{-1}y_d^{q-1} \Phi_0(x_d/(y_d-x_d))}{(y_d-x_d)^{d+\alpha}}dy_d\,ds  \\
		&\quad+ c \int_{x_d/2}^{r/4} s^{d} \int_{0}^{4s} \frac{r^{-1} y_d^{q-1} \Phi_0(x_d/s)}{s^{d+\alpha}}dy_d \, ds \\
		&\quad+  c \int_{x_d/2}^{r/4} s^{d} \int_{4s}^{r} \frac{r^{-1} y_d^{q-1} \Phi_0(x_d/y_d)}{(y_d-x_d)^{d+\alpha}}dy_d \, ds\\
		&=:I_{1,1}+I_{1,2}+I_{1,3}+I_{1,4}+I_{1,5}+I_{1,6}+I_{1,7}.
	\end{align*}
For $I_{1,1}$, we have 
$I_{1,1} \le  cr^{-1}x_d^{-d-\alpha}\int_0^{x_d/2} s^{d}ds\int_{0}^{x_d/2} y_d^{q-1} dy_d  = cr^{-1}x_d^{q-\alpha+1}$. Note that
	\begin{align}\label{e:I_12}
		y_d^{q-1} \le (2^{1-q}+2^{q-1})x_d^{q-1} \quad \text{for} \;\, y_d \in (x_d/2, 2x_d).
	\end{align}
Using \eqref{e:I_12}, we get
	\begin{align*}
		I_{1,2} &\le cr^{-1}x_d^{q-1} \int_0^{x_d/2} s^{d}\int_{x_d/2}^{x_d-s} \frac{dy_d}{(x_d-y_d)^{d+\alpha}}ds \\
		& \le c r^{-1}x_d^{q-1} \int_0^{x_d/2} s^{1-\alpha}ds =cr^{-1}x_d^{q-\alpha+1},
	\end{align*}
	\begin{align*}
		I_{1,3} \le cr^{-1} x_d^{q-1} \int_0^{x_d/2}s^{-\alpha} \int_{x_d-s}^{x_d+s} dy_d\,ds 
			=  cr^{-1} x_d^{q-1} \int_0^{x_d/2}s^{1-\alpha} ds=cr^{-1}x_d^{q-\alpha+1}
	\end{align*}
and
	\begin{align*}
		I_{1,4}& \le cr^{-1}x_d^{q-1} \int_0^{x_d/2} s^d \int_{x_d+s}^{2x_d} \frac{dy_d}{(y_d-x_d)^{d+\alpha}}ds\\
		&\le    cr^{-1} x_d^{q-1} \int_0^{x_d/2}s^{1-\alpha} ds=cr^{-1}x_d^{q-\alpha+1}.
	\end{align*}
Besides, by using \eqref{e:scale}, since $q<\alpha+\lb_0$, we obtain
	\begin{align*}
		&	I_{1,5}\le cr^{-1}\int_0^{x_d/2} s^{d}\int_{2x_d}^{r/2} y_d^{-d-\alpha+q-1} \Phi_0(x_d/y_d)dy_d\,ds\\
		& \le c\Phi_0(1)r^{-1}x_d^{\lb_0}\int_0^{x_d/2} s^{d}ds\int_{2x_d}^{r/2} y_d^{-d-\alpha-\lb_0+q-1} dy_d 
			\le c r^{-1}x_d^{q-\alpha+1},
	\end{align*}
	\begin{align*}
		&	I_{1,6} \le c\Phi_0(1)r^{-1} x_d^{\lb_0} \int_{x_d/2}^{r/4} s^{-\alpha-\lb_0} \int_0^{4s} y_d^{q-1}dy_d \,ds
		=   cr^{-1} x_d^{\lb_0}  \int_{x_d/2}^{r/4}s^{-\alpha-\lb_0+q} ds
	\end{align*}
and
	\begin{align*}
		I_{1,7} &\le c \Phi_0(1)r^{-1}x_d^{\lb_0}  \int_{x_d/2}^{r/4} s^{d} \int_{4s}^{r} y_d^{-d-\alpha-\lb_0+q-1}dy_d \, ds \\
		&\le c r^{-1}x_d^{\lb_0}  \int_{x_d/2}^{r/4} s^{-\alpha-\lb_0+q}  ds.
	\end{align*}
Since
	\begin{align*}
		& r^{-1}x_d^{\lb_0}  \int_{x_d/2}^{r/4} s^{-\alpha-\lb_0+q}  ds 
			\le 	 c x_d^{q-\alpha} 
			\begin{cases}
				(x_d/r)^{\alpha+\lb_0-q} &\mbox{ if } -\alpha-\lb_0+q>-1,\\
				x_d/r &\mbox{ if } -\alpha-\lb_0+q<-1,\\ 	
				(x_d/r) \log(r/x_d) &\mbox{ if } -\alpha-\lb_0+q=-1
	 		\end{cases}
	\end{align*}
and $(x_d/r) \log(r/x_d) \le (x_d/r)^{1/2}$, we arrive at the result. 
\end{proof} 

\begin{lemma}\label{l:I5}
There exist constants $C,b>0$  independent of $Q,r,x$ and $V$ such that
	\begin{align*}
|		I_5| \le C (x_d/r)^{b} x_d^{q-\alpha}.
	\end{align*}	
\end{lemma} 
\begin{proof} 
We first recall that 
	$$
		I_5=\int_{E_\nu\setminus \sO}   \frac{(y_d^q-x_d^q) (\sB(x,y)- \sum_{i=1}^{i_0}\mu^i(x)\F^i((y-x)/x_d)))}{|x-y|^{d+\alpha}}dy.
	$$
Let  $\lambda \in (0,1/(\theta_1 \vee \theta_2 \vee 1))$  be  such that $q<\alpha + (1-\lambda)\lb_0 -  \lambda(\theta_1+ 2\theta_2)$, 
where $\theta_1,\theta_2>0$ are the constants in \eqref{e:ass-B5'}.  Define 
	\begin{align*}
		A_1&=\{y\in E_\nu\setminus \sO: 3|x-y|<y_d \vee (2x_d) \},\\
		A_2&=\{y \in E_\nu\setminus \sO: 3|x-y| \ge 2x_d \ge y_d\},\\
		A_3&=\{y \in E_\nu \setminus \sO: 3|x-y| \ge y_d > 2x_d\}.
	\end{align*}
Using Lemma \ref{l:B5'},   Lemma   \ref{l:C11}(iii), \hyperlink{B4-a}{{\bf (B4-a)}}, Lemma \ref{l:F-basic} and \eqref{e:scale}, we obtain
	\begin{align*}
		|I_5| &\le   \int_{A_1}  \bigg|\sB(x,y)- \sum_{i=1}^{i_0}\mu^i(x)\F^i((y-x)/x_d))) \bigg|\, 
			\frac{|x_d^q-y_d^q|}{|x-y|^{d+\alpha}}\,dy\\
			&\quad +  \int_{A_2 \cup A_3} \bigg(\sB(x,y)+\sum_{i=1}^{i_0}\mu^i(x)\F^i((y-x)/x_d))) \bigg)^{1-\lambda}\\
			&\qquad \qquad \qquad \times  \bigg|\sB(x,y)- \sum_{i=1}^{i_0}\mu^i(x)\F^i((y-x)/x_d))) \bigg|^{\lambda} 
				\frac{|x_d^q-y_d^q|}{|x-y|^{d+\alpha}}\,dy \\
			& \le cr^{-\theta_2}  \int_{A_1}   
				\frac{(x_d \vee y_d \vee  |x-y| )^{\theta_1+\theta_2} }{(x_d \wedge y_d \wedge |x-y|)^{\theta_1}} 
				\frac{|x_d^q-y_d^q|}{|x-y|^{d+\alpha}}\,dy\\
			& \quad +  cr^{-\lambda \theta_2} \int_{A_2}   \Phi_0(x_d/|x-y|) ^{1-\lambda} 
				\left(\frac{(x_d \vee y_d \vee  |x-y| )^{\theta_1+\theta_2} }{(x_d \wedge y_d \wedge |x-y|)^{\theta_1}}\right)^\lambda
				 \frac{|x_d^q-y_d^q|}{|x-y|^{d+\alpha}}\,dy\\
 			&\quad + cr^{-\lambda \theta_2}  \int_{A_3} \Phi_0(x_d/|x-y|) ^{1-\lambda} 
 			\left(  \frac{(x_d \vee y_d \vee  |x-y| )^{\theta_1+\theta_2} }{(x_d \wedge y_d \wedge |x-y|)^{\theta_1}}\right)^\lambda
 			 	\frac{|x_d^q-y_d^q|}{|x-y|^{d+\alpha}}\,dy\\
		&=:I_{5,1}+I_{5,2}+I_{5,3}.
	\end{align*}	
By the triangle inequality,  for every $y \in A_1$, if $3|x-y|<y_d$, then $2|x-y|<(2/3)y_d<x_d<(4/3)y_d$, and 
if $3|x-y|<2x_d$, then  $(1/2)|x-y|<(1/3)x_d <y_d <(5/3)x_d$.  
Hence, for every $y \in A_1$, we have 
	\begin{align*}
		(1/2)|x-y|<x_d \wedge y_d\le x_d \vee y_d \le (5/3)(x_d \wedge y_d).
	\end{align*}
 Using this, since $|x_d^q-y_d^q|\le x_d^q+y_d^q$ for all $y\in A_1$, we obtain
 	\begin{align*}
 		I_{5,1}&\le cr^{-\theta_2} x_d^{q+\theta_1+\theta_2}\int_{A_1} \frac{dy}{|x-y|^{d+\alpha +\theta_1}} \\
 		&\le   cr^{-\theta_2} x_d^{q+\theta_1+\theta_2} 
 			\int_{\R^d \setminus B(x,5^{-1}r^{-\theta_2/(2\alpha+2\theta_1)} x_d^{1+\theta_2/(2\alpha+2\theta_1)})} 
 			\frac{dy}{|x-y|^{d+\alpha+\theta_1}}\\
 		& \le c (x_d/r)^{\theta_2/2} x_d^{q-\alpha}.
 	\end{align*}
For $I_{5,2}$, since $\Phi_0$ is bounded, $|x_d^q-y_d^q| \le x_d^q$ for all $y \in A_2$ and  $\lambda (\theta_1\vee \theta_2)<1$,  we have
	\begin{align*}
		I_{5,2} &\le cr^{-\lambda\theta_2} x_d^q\int_{A_2} 
			\frac{(|x-y|^{\theta_1+\theta_2} /y_d^{\theta_1})^\lambda }{ |x-y|^{d+\alpha}}dy\\
		&\le cr^{-\lambda\theta_2}x_d^q \int_{A_2} 
			\frac{dy}{    y_d^{\lambda \theta_1} (2x_d/3)^{\alpha+1-\lambda \theta_1} |x-y|^{d-1-\lambda\theta_2} }\\ 
		& \le cr^{-\lambda\theta_2}x_d^{q-\alpha-1+\lambda \theta_1} 
			\int_{\wt y \in \R^{d-1}, \, 4r^{-\nu}|\wt y|^{1+\nu}<2x_d}  \int_{4r^{-\nu}|\wt y|^{1+\nu}}^{2x_d}   
			\frac{1}{y_d^{\lambda \theta_1} |\wt y|^{d-1-\lambda \theta_2}}dy_d d\wt y \\
		& \le cr^{-\lambda\theta_2}x_d^{q-\alpha-1+\lambda \theta_1} \int_{0}^{(r^\nu x_d)^{1/(1+\nu)}}   
			\frac{s^{d-2}}{s^{d-1-\lambda\theta_2}} \int_{4r^{-\nu}s^{1+\nu}}^{4x_d} y_d^{-\lambda \theta_1} dy_d  \, ds  \\
		&\le cr^{-\lambda\theta_2}x_d^{q-\alpha} 
			\int_{0}^{(r^\nu x_d)^{1/(1+\nu)}}   s^{-1+\lambda\theta_2}ds = c (x_d/r)^{\lambda\theta_2/(1+\nu)}  x_d^{q-\alpha}.
	\end{align*}
For all $y \in A_3$, we have $|x-y| \ge y_d-x_d \ge y_d/2>x_d$. Using this,   \hyperlink{B4-a}{{\bf (B4-a)}}, Lemma \ref{l:F-basic} and \eqref{e:scale}, since $q-d-\alpha-(1-\lambda)\lb_0+\lambda(\theta_1+\theta_2)<-1$ 
and $q-\alpha-1-(1-\lambda)\lb_0+\lambda(\theta_1+2\theta_2)<-1$, we get
	\begin{align*}
		I_{5,3}&\le c\Phi_0(1)^{1-\lambda} r^{-\lambda \theta_2} \int_{A_3}   
				\bigg(\frac{x_d}{|x-y|}\bigg)^{(1-\lambda)\lb_0}
				\bigg(  \frac{|x-y|^{\theta_1+\theta_2} }{x_d^{\theta_1}}\bigg)^\lambda \frac{y_d^q}{|x-y|^{d+\alpha}}\, dy\\
                 &\le  cr^{-\lambda \theta_2} x_d^{ (1-\lambda)\lb_0-\lambda \theta_1} 
                 \int_{0}^{x_d/4} s^{d-2} \int_{2x_d}^{r/2} 
                 y_d^{q-d-\alpha - (1-\lambda)\lb_0 + \lambda (\theta_1+\theta_2)}dy_d\,  ds\\
                	&\quad + cr^{-\lambda \theta_2}x_d^{ (1-\lambda)\lb_0-\lambda \theta_1}  
				\int_{x_d/4}^{r/4 } s^{d-2} \int_{2x_d}^{r/2} 
				\frac{y_d^{q-\alpha-1-(1-\lambda)\lb_0+\lambda(\theta_1+2\theta_2)}}{s^{d-1+\lambda\theta_2}}dy_d\,  ds\\
				&\le  cr^{-\lambda \theta_2}x_d^{q-d-\alpha+\lambda \theta_2+1} \int_{0}^{x_d/4} s^{d-2} ds 
				+ cr^{-\lambda \theta_2}x_d^{q-\alpha+2\lambda \theta_2} 
				\int_{x_d/4}^{r/4}  s^{-1-\lambda \theta_2} ds\\
				&\le c(x_d/r)^{\lambda \theta_2} x_d^{q-\alpha}.
	\end{align*}
The proof is complete.
\end{proof} 


\section{Explicit decay rates}\label{ch:decay-rate}

In this section we establish the explicit decay rate of some particular harmonic functions, namely, exit probabilities from small boxes based at a boundary point. We will show that these functions decay as the $p$-th power of the distance to the boundary. The first step towards this goal is to combine the already constructed barrier $\psi^{(r)}$ (see Subsection \ref{s-barrier}) with cutoff functions of the type $h_{q,U(r)}$ to obtain more refined barriers. This is done in the next subsection. Subsection \ref{s:decay-spec-har} is devoted to the proof of Theorem \ref{t:Dynkin-improve} -- sharp two-sided estimates of some exit probabilities.

Throughout this section, we work with $Q\in \partial D$ and will
denote $U^Q(a, b)$ by $U(a, b)$, and $U^Q(a)$ by $U(a)$.

\subsection{Barriers revisited}\label{s-refined-barriers}

Recall the definition of  $h_{q,V}(y)$ from \eqref{e:def-hqV}. Since $D$ is a $C^{1,1}$ open set, it is known that for any 
$q\in [(\alpha-1)_+,\alpha+\beta_0) \cap (0,\infty)$, $Q \in \partial D$ and $r \in (0,\wh R/8]$,
	\begin{align}\label{e:hq-C11}
		h_{q,U(r)} \in C^{1,1}(U(r/2)).
	\end{align}
See, e.g., \cite[Theorem 7.8.4]{DZ11}. 

We also recall that $p \in  [(\alpha-1)_+,\alpha+\beta_0) \cap (0,\infty)$
denotes the constant satisfying \eqref{e:C(alpha,p,F)} if \hyperlink{K3}{{\bf (K3)}} holds with $C_9>0$, and $p=\alpha-1$ if $C_9=0$,  and  the   operators $L^\sB_\alpha$ and $L^\kappa$ are defined by
	\begin{align*}
		L^\sB_\alpha f(x)&=\text{p.v.}\int_D (f(y)-f(x)) \frac{\sB(x,y)}{|x-y|^{d+\alpha}}dy \quad \text{and} 
			\quad 	L^\kappa f(x)=L^\sB_\alpha f(x) - \kappa(x)f(x).
	\end{align*}

\begin{lemma}\label{l:Dynkin1}
Let $Q \in \partial D$, $0<r \le \wh R/8$, and 
	$$
		q_0=p +  2^{-1} \left(  (\alpha+\beta_0-p) \wedge \eta_0 \wedge \eta_1 \right),
	$$
where  $\eta_0,\eta_1>0$ are the constants in \eqref{e:kappa-explicit} and Proposition \ref{p:barrier} respectively.
Define functions $\phi_p$ and $\varphi_p$ on $D$ by
	\begin{align}\label{e:def-phi-varphi}
		\phi_p(y)&=2 h_{p, U(r)}(y) -r^{p-q_0} h_{q_0, U(r)}(y),\nn\\
	 	\varphi_p(y)&=  h_{p, U(r)}(y) +r^{p-q_0} h_{q_0, U(r)}(y).
	\end{align}
Then there  exists a constant $\epsilon_1 \in (0,1/12)$ independent of $Q$ and $r$ such that  
	
	\smallskip
	
	\noindent (i) $\phi_p$ satisfies the following properties:
	
	\smallskip

	(a) $\phi_p \in C^{1,1}(U(\epsilon_1 r))$ and $\phi_p(y)=0$ for all $y \in D \setminus U(r)$;
	
	(b) $\delta_D(y)^p\le \phi_p(y) \le 2\delta_D(y)^p$ for all $y \in U( r)$; 
	
	(c) $L^\kappa \phi_p(y) \le -(\delta_D(y)/r)^{\eta_0 \wedge \eta_1}\delta_D(y)^{p-\alpha}$ for all $y \in  U(\epsilon_1r)$.

	\smallskip

	\noindent (ii)  $\varphi_p$ satisfies the following properties:
	
	\smallskip
	
	(a)  $\varphi_p \in C^{1,1}(U(\epsilon_1 r))$ and $\varphi_p(y)=0$ for all $y \in D \setminus U(r)$;
	
	(b) $\delta_D(y)^p\le \varphi_p(y) \le 2\delta_D(y)^p$ for all $y \in U( r)$;

	(c)  $L^\kappa\varphi_p(y) \ge (\delta_D(y)/r)^{\eta_0 \wedge \eta_1}\delta_D(y)^{p-\alpha}$ for all $y \in  U(\epsilon_1r)$.
\end{lemma}
\begin{proof}  
By \eqref{e:hq-C11}, we have that $\phi_{p}, \varphi_p \in C^{1,1}(U(\varepsilon r))$ for any $\varepsilon\in (0,1/12)$. 
Clearly, $\phi_p(y)=\varphi_p(y)=0$ for  $y \in D \setminus U(r)$. Hence, $\phi_p$ and $\varphi_p$ satisfy property (a) in (i) and (ii), respectively. Moreover, for all $y \in U(r)$, since $\delta_D(y)<r$ and $q_0>p$, we have 
	\begin{align*}
		r^{p-q_0}h_{q_0,U(r)}(y) =(\delta_D(y)/r)^{q_0-p} h_{p,U(r)}(y) \le h_{p,U(r)}(y) = \delta_D(y)^p.
	\end{align*}
Using this, we see that  $\phi_p$ and $\varphi_p$ satisfy property (b) in (i) and (ii), respectively.

Now, we show that   $\phi_p$ and $\varphi_p$ satisfy property (c) in (i) and (ii) respectively. When $C_9>0$, by  Proposition \ref{p:barrier}, \eqref{e:kappa-explicit} and  \eqref{e:C(alpha,p,F)},  for all $y \in U(r/12)$,
	\begin{align*}
	\left|	L^\kappa h_{p, U(r)}(y) \right|&= \left|	L^\sB_\alpha h_{p, U(r)}(y) - \kappa(y) h_{p, U(r)}(y)\right| \nn\\
		& \le \left|  L_\alpha^\sB h_{p, U(r)}(y) - C(\alpha,p, \F) \sB(y,y)\delta_D(y)^{-\alpha} h_{p, U(r)}(y) \right| \nn\\
		&\quad +    \left| \kappa(y) - C_9 \sB(y,y)\delta_D(y)^{-\alpha}  \right| h_{p, U(r)}(y)\nn\\
		&\le c (\delta_D(y)/r)^{\eta_1} \delta_D(y)^{p-\alpha}  + c \delta_D(y)^{p-\alpha+\eta_0}\nn\\
		&\le  c(\delta_D(y)/r)^{\eta_0 \wedge \eta_1} \delta_D(y)^{p-\alpha}.
	\end{align*}
Moreover, using Proposition \ref{p:barrier}, \eqref{e:kappa-explicit},  \eqref{e:C(alpha,p,F)} and   \eqref{e:B(x,x)},   since $C(\alpha,q_0,\F)>C(\alpha,p,\F)$ by Lemma \ref{l:constant}, we also get that for all $y \in U(r/12)$,
	\begin{align*}
		L^\kappa h_{q_0, U(r)}(y) &=	L^\sB_\alpha h_{q_0, U(r)}(y) - \kappa(y) h_{q_0, U(r)}(y) \nn\\
		&=    L_\alpha^\sB h_{q_0, U(r)}(y) - C(\alpha,q_0, \F)\sB(y,y)\delta_D(y)^{-\alpha} h_{q_0, U(r)}(y)   \nn\\
		&\quad  +  (C(\alpha,q_0,\F) - C(\alpha,p,\F))  \sB(y,y) \delta_D(y)^{-\alpha} h_{q_0, U(r)}(y)  \nn\\
		&\quad -   (\kappa(y) - C(\alpha,p,\F) \sB(y,y)\delta_D(y)^{-\alpha})h_{q_0, U(r)}(y) \nn\\
		&\ge C_2(C(\alpha,q_0,\F) - C(\alpha,p,\F))\delta_D(y)^{q_0-\alpha}    \nn\\
		&\quad - c(\delta_D(y)/r)^{\eta_1} \delta_D(y)^{q_0-\alpha}-  c \delta_D(y)^{q_0-\alpha+\eta_0}\nn\\
		&\ge c\delta_D(y)^{q_0-\alpha}-c(\delta_D(y)/r)^{\eta_0 \wedge \eta_1} \delta_D(y)^{q_0-\alpha}.
	\end{align*}
When $C_9=0$,  using \eqref{e:kappa-explicit}, Proposition \ref{p:barrier} and \eqref{e:mu-bound}, since 
$C(\alpha, \alpha-1, \F^i)=0<C(\alpha, q_0, \F^i)$ for all $1\le i \le i_0$, we get that for all $y \in U(r/12)$,
	\begin{align*}
		&\left|	L^\kappa h_{p, U(r)}(y) \right|	 \le \left|  L_\alpha^\sB h_{p, U(r)}(y) \right|  +   \kappa(y)  h_{p,U(r)}(y)\nn\\
		&\le c (\delta_D(y)/r)^{\eta_1} \delta_D(y)^{p-\alpha}  + c \delta_D(y)^{p-\alpha+\eta_0}
			\le  c(\delta_D(y)/r)^{\eta_0 \wedge \eta_1} \delta_D(y)^{p-\alpha}
	\end{align*}
and
	\begin{align*}
		L^\kappa h_{q_0, U(r)}(y) &=	L^\sB_\alpha h_{q_0, U(r)}(y) - \kappa(y) h_{q_0, U(r)}(y) \nn\\
		&=    L_\alpha^\sB h_{q_0, U(r)}(y) 
			- \sum_{i=1}^{i_0} \mu^i(x) C(\alpha, q_0, \F^i)\delta_D(y)^{-\alpha} h_{q_0, U(r)}(y)   \nn\\
		&\quad +  \sum_{i=1}^{i_0} \mu^i(x) C(\alpha, q_0, \F^i)\delta_D(y)^{-\alpha} h_{q_0, U(r)}(y)   
			 -   \kappa(y)h_{q_0, U(r)}(y) \nn\\
		&\ge C_{11}^{-1} \sum_{i=1}^{i_0}  C(\alpha, q_0, \F^i)\delta_D(y)^{q_0-\alpha}    \nn\\
		&\quad  - c(\delta_D(y)/r)^{\eta_1} \delta_D(y)^{q_0-\alpha}-  c \delta_D(y)^{q_0-\alpha+\eta_0}\nn\\
		&\ge c\delta_D(y)^{q_0-\alpha}-c(\delta_D(y)/r)^{\eta_0 \wedge \eta_1} \delta_D(y)^{q_0-\alpha}.
	\end{align*}
Therefore, in both cases,  we have
	\begin{align}\label{e:generator-hp}
		\left|	 L^\kappa h_{p, U(r)}(y) \right|	\le  c_1 (\delta_D(y)/r)^{\eta_0 \wedge \eta_1} \delta_D(y)^{p-\alpha}
	\end{align}
and
	\begin{align}\label{e:generator-hq}
	 	L^\kappa h_{q_0, U(r)}(y)	\ge  c_2\delta_D(y)^{q_0-\alpha}  
	 		-  c_3 (\delta_D(y)/r)^{\eta_0 \wedge \eta_1} \delta_D(y)^{q_0-\alpha}.
	\end{align}
Since $p<q_0<p+\eta_0 \wedge \eta_1$, there exists $\epsilon_1 \in (0,1/12)$ such that
	\begin{align}\label{e:property-eps1}
		c_2s^{q_0-p-\eta_0 \wedge \eta_1} \ge  c_3s^{q_0-p}+2c_1 +1\quad \text{for all} \;\, s \in (0,\epsilon_1].
	\end{align}

(i) Using \eqref{e:generator-hp}, \eqref{e:generator-hq} and \eqref{e:property-eps1}, we get that  for all $y \in U(\epsilon_1 r)$,
	\begin{align*}
		&L^\kappa \phi_p(y) =  2 	L^\kappa h_{p, U(r)}(y)  -r^{p-q_0} L^\kappa h_{q_0, U(r)}(y)  \nn\\
		& \le - \left(  c_2  (\delta_D(y)/r)^{q_0-p-\eta_0 \wedge \eta_1 }  - c_3 (\delta_D(y)/r)^{q_0-p}  
			- 2c_1  \right) (\delta_D(y)/r)^{\eta_0 \wedge \eta_1}  \delta_D(y)^{p-\alpha}\nn\\
		&\le - (\delta_D(y)/r)^{\eta_0 \wedge \eta_1}  \delta_D(y)^{p-\alpha}.
	\end{align*}

(ii)  Using \eqref{e:generator-hp}, \eqref{e:generator-hq} and \eqref{e:property-eps1}, we see that  for all $y \in U(\epsilon_1 r)$,
	\begin{align*}
		&L^\kappa \varphi_p(y) = r^{p-q_0} L^\kappa h_{q_0, U(r)}(y)  + 	L^\kappa h_{p, U(r)}(y) \nn\\
		& \ge  \left(  c_2  (\delta_D(y)/r)^{q_0-p-\eta_0 \wedge \eta_1 }  - c_3 (\delta_D(y)/r)^{q-p}
			- 2c_1 \right) (\delta_D(y)/r)^{\eta_0 \wedge \eta_1}  \delta_D(y)^{p-\alpha}\nn\\
		&\ge  (\delta_D(y)/r)^{\eta_0 \wedge \eta_1}  \delta_D(y)^{p-\alpha}.
	\end{align*}
The proof is complete. 
\end{proof}

Recall that  we  treat \hyperlink{B5-I}{{\bf (B5-I)}}  as a special case of \hyperlink{B5-II}{{\bf (B5-II)}} with $i_0=1$. 
	
	To control certain exit 
 probabilities  from below (see Lemma \ref{l:Dynkin-pre-lower}), we need to introduce the following non-local operators with appropriate  additional critical killings.

For $q \in (p, \alpha+\beta_0)$, we define
	\begin{align}\label{e:def-Lq}
		\wt L_q f(y) = L^\kappa f(y) - \sum_{i=1}^{i_0}  \mu^i (y)( C(\alpha, q, \F^i)- C(\alpha, p, \F^i)) \delta_D(y)^{-\alpha}f(y).
	\end{align}

In the following two lemmas, we let $\psi^{(r)}$ denote the function defined by \eqref{e:def-psi-r} and 
let $N_0>\alpha +\ub_0+2$ be the constant appearing in the construction of $\psi^{(r)}$.

\begin{lemma}\label{l:compensator} 
Let $Q \in \partial D$ and $q \in (p, \alpha+\beta_0)$.	
There exists  $C>0$ independent of  $Q$ such that for all $0<r\le \wh R/24$ and $y \in U(r)$,	
	\begin{align*}	
		\wt	L_q \psi^{(r)}(y)\le C r^{-\alpha} \Phi_0(\delta_D(y)/r).	
	\end{align*} 
\end{lemma}
\begin{proof} 
Let $r\in (0, \wh R/24]$ and $y \in U(r)$. By Lemma \ref{l:constant},  $C(\alpha,q,\F^i)>C(\alpha,p,\F^i)$ for all $1\le i\le i_0$. Thus, by \eqref{e:mu-bound}, we have
	 \begin{align}\label{e:compensator-1}
	 	\wt	L_q \psi^{(r)}(y) \le L^\sB_\alpha  \psi^{(r)}(y) - C_{11}^{-1} 
	 		\sum_{i=1}^{i_0}  ( C(\alpha, q, \F^i)- C(\alpha, p, \F^i)) \delta_D(y)^{-\alpha}\psi^{(r)}(y) .
	 \end{align}
Note that $r\le \wh R/24< \wh R/(18+9\Lambda_0)$ by \eqref{e:Lipschitz-constant}. Applying Proposition \ref{p:compensator} with   
$\eps = C_{11}^{-1}\sum_{i=1}^{i_0}(C(\alpha,q,\F^i)$ $-C(\alpha,p,\F^i))$, we deduce the result from \eqref{e:compensator-1}.
\end{proof}

\begin{lemma}\label{l:Dynkin2}
Let $q \in (p,\alpha+
\beta_0)$, $q_1 \in (p,q)$, $Q \in \partial D$ and $0<r \le \wh R/24$.  There exist constants 
$  \epsilon_2 =\epsilon_2 (q, q_1)   \in (0,\epsilon_1]$ and $C>0$ independent of $Q$ and $r$ such that the function  $\chi_q$  defined by
	$$
		\chi_{q}(y)=h_{q,U(r)}(y) -\epsilon_2^{q-2N_0}r^q\psi^{(r)}(y), \quad y \in D,
	$$
satisfies the following properties:

	(a) $\chi_q \in C^{1,1}(U(\epsilon_2 r))$ and $\chi_q(y) \le 0$ for all $y \in D \setminus U(\epsilon_2 r)$;
	
	(b) $2^{-1}\delta_D(y)^q \le \chi_q(y) \le \delta_D(y)^q$ for all $y=(\wt 0, y_d) \in U(\epsilon_2 r/2)$;
	
	(c)   $\wt L_{q_1}\chi_q(y) \ge - C r^{q-\alpha} \Phi_0(\delta_D(y)/r)$ for all $y \in U(\epsilon_2 r)$.
	
	\smallskip
	
\noindent Here, $\epsilon_1\in (0,1/12)$ is the constant in Lemma \ref{l:Dynkin1}.
\end{lemma}
\begin{proof} Let $\epsilon_2 \in (0,\epsilon_1]$ be  a constant to be  determined later. 
		
By \eqref{e:hq-C11}, since $\psi^{(r)} \in C^{1,1}(D)$, we have $\chi_q  \in C^{1,1}(U(\epsilon_2 r))$. 
Since $\psi^{(r)}$ is non-negative,  $\chi_q(y)=-\epsilon_2^{q-2N_0}r^q\psi^{(r)}(y) \le 0$ for all $y \in D \setminus U(r)$.	
Let $y \in U(r) \setminus U(\epsilon_2r)$ and denote $v=(f^{(r)})^{-1}(y)$ where $f^{(r)}$ is the function defined in \eqref{e:def-fr}. Then $v \in U_\bH(1) \setminus U_\bH(\epsilon_2)$, $\rho_D(y)=rv_d$ and $\psi^{(r)}(y)=|\wt v|^{2N_0} + v_d^{2N_0}$.  
If $v_d \ge \epsilon_2$, then since $\delta_D(y) \le \rho_D(y)=rv_d$ and $N_0>q$, we get
	$$
		\chi_q(y) \le r^q v_d^q  - \epsilon_2^{q-2N_0} r^q v_d^{2N_0} = r^q v_d^q ( 1 - \epsilon_2^{q-2N_0} v_d^{2N_0-q} ) \le 0.
	$$
	Assume  $v_d<\epsilon_2$. Since $v \notin U_\bH(\epsilon_2)$, it follows that $|\wt v| \ge \epsilon_2$. Hence, we obtain
	$$
		\chi_q(y) \le r^q v_d^q  - \epsilon_2^{q-2N_0} r^q |\wt v|^{2N_0}  \le \epsilon_2^q r^q - \epsilon_2^{q} r^q=0.
	$$
	Therefore, $\chi_q$ satisfies (a).
	
Since $\psi^{(r)}$ is non-negative,   $\chi_q(y) \le \delta_D(y)^q$ for all $y \in U(\epsilon_2 r/2)$. Moreover, for any 
$y = (\wt 0, y_d) \in U(r)$,   we have $\psi^{(r)}(y) =(y_d/r)^{2N_0}= (\delta_D(y)/r)^{2N_0} $. Thus, since $N_0>q+2$,  
for all $y = (\wt 0,y_d) \in U(\epsilon_2 r/2)$, we have
	\begin{align*}
		&	\chi_q(y) \ge \delta_D(y)^q - \epsilon_2^{q-2N_0}r^q (\epsilon_2/2)^{2N_0-q}  (\delta_D(y)/r)^q \\
		&= (1-2^{q-2N_0}) \delta_D(y)^q \ge 2^{-1}\delta_D(y)^q.
	\end{align*}
Hence, $\chi_q$ satisfies (b).

For (c), using  \eqref{e:mu-bound}, Proposition \ref{p:barrier}, \eqref{e:kappa-explicit}  and  \eqref{e:C(alpha,p,F)},  we see that for all $y \in U(r/12)$,
	\begin{align}\label{e:wtL-upper}
	\begin{split}	
		\wt L_{q_1} h_{q,U(r)}(y) &=   L_\alpha^\sB h_{q,U(r)} (y) - \sum_{i=1}^{i_0}  \mu^i (y) 
			C(\alpha, q, \F^i)\delta_D(y)^{-\alpha}h_{q,U(r)}(y)  \\
		&\quad  +    \sum_{i=1}^{i_0}  \mu^i (y)( C(\alpha, q, \F^i)- C(\alpha, q_1, \F^i)) \delta_D(y)^{-\alpha} h_{q,U(r)}(y)   \\
		&\quad  - \bigg( \kappa(y)-\sum_{i=1}^{i_0}  \mu^i (y) C(\alpha, p, \F^i)\bigg) h_{q,U(r)}(y)\\
		&\ge  C_{11}^{-1} \sum_{i=1}^{i_0}( C(\alpha, q, \F^i)- C(\alpha, q_1, \F^i))  \delta_D(y)^{-\alpha}h_{q,U(r)}(y)\\
		&\quad -(c_1 (\delta_D(y)/r)^{\eta_1} +  C_8\delta_D(y)^{\eta_0})\delta_D(y)^{-\alpha}h_{q,U(r)}(y).
	\end{split} 
	\end{align}
Set $c_2:= C_{11}^{-1} \sum_{i=1}^{i_0}( C(\alpha, q, \F^i)- C(\alpha, q_1, \F^i)) $. Since $q_1<q$, by Lemma \ref{l:constant}, $c_2$ is a positive constant. Now we choose $\epsilon_2\in (0,\epsilon_1]$ to satisfy 
$c_2- c_1 \epsilon_2^{\eta_1} - C_8 \epsilon_2^{\eta_0} \ge 0$. By \eqref{e:wtL-upper} and  Lemma \ref{l:compensator},
we get that for all $y \in U(\epsilon_2r)$,
	\begin{align*}
		\wt L_{q_1} \chi_q \ge- \epsilon_2^{q-2N_0}r^q \wt L_{q_1} \psi^{(r)}(y) \ge - c_3 r^{q-\alpha} \Phi_0(\delta_D(y)/r).
	\end{align*}
The proof is complete.
\end{proof}

\subsection{Explicit decay rate of some special harmonic functions}\label{s:decay-spec-har}

In this subsection, we establish some estimates of exit probabilities from small boxes based at a boundary point. These exit probabilities are non-negative harmonic functions vanishing continuously at the boundary. 
Recall that the definition of  harmonic and regular harmonic functions is given in Definition \ref{df:harmonic}. 

\medskip
\emph{For the remainder of this work,  we suppress the superscript $\kappa$ from $Y^\kappa$ and related objects.}
\medskip

We also recall the following well-known fact: If $f:D\to [0,\infty)$ is harmonic in $D\cap B(Q,r)$, $Q\in \partial D$, and vanishes continuously on $\partial D\cap B(Q,r)$, then $f$ is regular harmonic in $D\cap B(Q,r/2)$ 
(see, for example, \cite[Lemma 5.1]{KSV18-a} and its proof). 

Throughout this subsection, 
 we let $\epsilon_1$  be the constant in Lemmas \ref{l:Dynkin1}. We also fix  
 	$$
	 q=\frac{p+2\alpha+2\beta_0}{3}
	\quad \text{and} \quad
	q_1= \frac{2p+\alpha+\beta_0}{3}
	$$ 
	and let $\epsilon_2$ be the constant in Lemma  \ref{l:Dynkin2} 
with these fixed $q$ and $q_1$.

 The goal of this subsection is to prove the following theorem.

\begin{thm}\label{t:Dynkin-improve}	Let $Q \in \partial D$ and $0<r \le \wh R/24$. There are comparison constants independent of $Q$ and $r$ such that for all $x \in U(\epsilon_2 r/4)$,
	\begin{align*} 	
		\P_x (Y_{\tau_{U(\epsilon_2 r)}} \in U(r) \setminus U(r, r/2) )  
			\asymp  	\P_x (Y_{\tau_{U(\epsilon_2 r)}}\in D ) \asymp (\delta_D(x)/r)^p.
	\end{align*}
\end{thm}

Before giving the proof of Theorem \ref{t:Dynkin-improve}, we record one of its consequences. 

\begin{corollary}\label{c:Carleson-1}
		There exists a constant $K_0>4$ such that for all  $x \in D$ with $\delta_D(x)\le \epsilon_2\wh R/(24K_0)$, 
		it holds that $\P_x(\tau_{B_D(x,(2K_0+1)\delta_D(x))}= \zeta)\ge 1/2$.
\end{corollary}
\begin{proof} Let $K_0>4$ be a constant to be  chosen later, $x \in D$ with $\delta_D(x)\le \epsilon_2\wh R/(24K_0)$ 
and $Q_x \in \partial D$ be such that $|x-Q_x|=\delta_D(x)$. Note that 
$B_D(x,(2K_0+1)\delta_D(x)) \supset B_D(Q_x, 2K_0\delta_D(x)) \supset U^{Q_x}(K_0\delta_D(x))$ by \eqref{e:U-rho-C11-1}. Hence, by Theorem \ref{t:Dynkin-improve} (with $r=K_0\delta_D(x)/\epsilon_2$), 
there exists $c_1>0$ independent of $x$ and $K_0$ such that
	\begin{align*}
		&	\P_x (Y_{\tau_{B_D(x,(2K_0+1)\delta_D(x))}}\in D ) \le 	\P_x (Y_{\tau_{U^{Q_x}(K_0\delta_D(x))}}\in D ) \le c_1K_0^{-p}.
	\end{align*}
Set $K_0:=(2c_1)^{1/p} \vee 5$. Then we arrive at
	\begin{align*}
		\P_x(\tau_{B_D(x,(2K_0+1)\delta_D(x))}= \zeta) =1- \P_x (Y_{\tau_{B_D(x,(2K_0+1)\delta_D(x))}}\in D )  \ge 1/2.
	\end{align*}
\end{proof}

Now we turn to the proof of Theorem \ref{t:Dynkin-improve}. To do so, we first establish several results that will be used in the proof. The proof of Theorem \ref{t:Dynkin-improve} will be presented  at the end of this subsection.

\begin{lemma}\label{l:exit-jump}
Let	$Q \in \partial D$, $0< r\le \wh R/8$ and $\eps \in (0,1/12)$. There exist  comparison constants independent of $Q$ and $r$ such that for all $x \in U(\eps r)$, 
	\begin{align*}
		\P_x \big(Y_{\tau_{U(\eps r)}} \in U(\eps r, r) \setminus U(\eps r, (3/4)r) \big) \asymp r^{-\alpha }\,
			\E_x 	\left[\int_0^{\tau_{U(\epsilon  r)}} \Phi_0(\delta_D(Y_s)/r)  ds \right].
	\end{align*}
\end{lemma}
\begin{proof}   Let $z \in U(\eps r)$ and  $y \in U(\eps r,r) \setminus U(\eps r, (3/4)r)$. Then  $|y-z| \ge (3/4-\eps)r \ge (2/3)r$ and  $|y-z| \le ((|\wt y | + |\wt z|)^2 + y_d^2)^{1/2} < ((2\eps)^2 + 1)^{1/2}r< \sqrt{37/36}\,r$. Moreover, by  \eqref{e:U-rho-C11-2}, we have $\delta_D(y) \ge \sqrt{4/5}\,\rho_D(y) \ge \sqrt{9/20} \, r >2^{-1}|y-z| \vee \delta_D(z)$. Thus, using \hyperlink{B4-a}{{\bf (B4-a)}}, \hyperlink{B4-b}{{\bf (B4-b)}} and \eqref{e:scale}, we get that
	\begin{align}\label{e:4.3.1}
		\frac{\sB(z,y)}{|z-y|^{d+\alpha}} \asymp  \frac{\Phi_0 (\delta_D(z)/|y-z|)}{|z-y|^{d+\alpha}} 
			\asymp  \frac{\Phi_0(\delta_D(z)/r)}{r^{d+\alpha}}.
	\end{align}
By using the L\'evy system formula \eqref{e:Levysystem-Y-kappa} and \eqref{e:4.3.1}, we deduce that for all $x \in U(\eps r)$,
	\begin{align*}
		& \P_x \big(Y_{\tau_{U(\eps r)}} \in U(\eps r, r) \setminus U(\eps r, (3/4)r) \big)\\
		&= \E_x \left[\int_0^{\tau_{U(\eps r)}}\int_{ U(\eps r, r) \setminus U(\eps r, (3/4)r)} 
			\frac{\sB(Y_s, y)}{|Y_s-y|^{d+\alpha}} dy ds \right] \nn\\
		& \asymp  r^{-d-\alpha} m_d(U(\eps r, r) \setminus U(\eps r, (3/4)r) )\, 
			\E_x \left[\int_0^{\tau_{U(\eps r)}} \Phi_0(\delta_D(Y_s)/r)  ds \right] \nn\\
		& \asymp  r^{-\alpha }\,\E_x \left[\int_0^{\tau_{U(\eps r)}} \Phi_0(\delta_D(Y_s)/r)  ds \right].
	\end{align*}
\end{proof}

\begin{lemma}\label{l:4.2}		
Let $Q \in \partial D$, $0< r\le \wh R/8$ and $\eps \in (0,\epsilon_1)$. For all $x \in U(\eps r)$, we have		
	 \begin{align}		
		2^{p+1}(\sqrt 5/3)^{p}	(\delta_D(x)/r)^p 
		&\ge 	\P_x \big(Y_{\tau_{U(\eps r)}} \in U(\eps r, r) \setminus U(\eps r, (3/4)r) \big)\label{e:4.2.1}	
	\end{align}
and
	\begin{align}		
		2^{-1}	(\delta_D(x)/r)^p & \le  \P_x \big(Y_{\tau_{U(\eps r)}} \in  U(r)\big).\label{e:4.2.2}	
	\end{align}
\end{lemma}
\begin{proof}  By Lemma \ref{l:Dynkin1}(i)-(ii), the functions $\phi_p$ and $\varphi_p$ defined in \eqref{e:def-phi-varphi} satisfy all the assumptions of Corollary \ref{c:Dynkin-local} with $U=U(\epsilon_1 r)$.
Hence, using Corollary \ref{c:Dynkin-local},  Lemma \ref{l:Dynkin1}(i)-(ii) and \eqref{e:U-rho-C11-2}, we get that for
all $x \in U(\eps r)$,	
	  \begin{align*} 
	  	2\delta_D(x)^p &\ge \phi_p(x) \ge \E_x \big[ \phi_p (Y_{\tau_{U(\eps r)}})\big]
	  		\ge	\E_x \big[ \delta_D (Y_{\tau_{U(\eps r)}})^p:Y_{\tau_{U(\eps r)}} \in U(r) \big]\nn\\	
	  	&\ge  (3/2\sqrt 5)^{p} r^p\,	\P_x \big(Y_{\tau_{U(\eps r)}}\in U(\eps r, r) \setminus U(\eps r, (3/4)r) \big)	
  	\end{align*}
and	
	\begin{align*}		
		\delta_D(x)^p &\le \varphi_p(x)\le\E_x \big[ \varphi_p (Y_{\tau_{U(\eps r)}})\big] 
			\le 2	\E_x \big[ \delta_D (Y_{\tau_{U(\eps r)}})^p:Y_{\tau_{U(\eps r)}} \in U(r) \big] \nn\\		
		& \le 2r^p	\, \P_x \big(Y_{\tau_{U(\eps r)}} \in  U(r)\big).	
	\end{align*}
\end{proof}

Combining \eqref{e:4.2.1} with Lemma \ref{l:exit-jump}, we arrive at

\begin{corollary}\label{c:Dynkin-upper} 	
Let $Q \in \partial D$ and  $0<r \le \wh R/8$. There exists $C>0$ independent of $Q$ and $r$ such that for all $x \in U(\epsilon_1 r)$,		\begin{align*} 	
		\E_x \left[\int_0^{\tau_{U(\epsilon_1 r)}} \Phi_0(\delta_D(Y_s)/r)  ds \right] \le C (\delta_D(x)/r)^p  r^\alpha.	
	\end{align*}
\end{corollary}

\begin{lemma}\label{l:6.1}
Let $\lb_0\in [0, \beta_0]$ be such that $p<\alpha+\lb_0$ and that the first inequality in \eqref{e:scale} holds.
Let $Q \in \partial D$, $0<r \le \wh R/8$ and $a,b \in (0,1)$ be such that $a<\epsilon_1 b/5$.	 
There exists $C>0$ independent of $Q, r, a$ and $b$ such that for all $x \in U(a r)$,
	\begin{align*} 
		\P_x \Big( Y_{\tau_{B_D(x,ar)}} \in D \setminus B(x,br) \Big) 
			\le C (a/b)^{\alpha+\lb_0} ( \delta_D(x)/(ar))^p.
	\end{align*}
\end{lemma}
\begin{proof} Let $y\in B_D(x,ar)$. Then $\delta_D(y) \le \delta_D(x)+ar <2ar$ and $B(y,4br/5) \subset B(x,br)$ 
since $a<\epsilon_1 b/5$. For all $z \in D\setminus B(y,4br/5)$
we have that $|y-z|\ge 4\epsilon^{-1}ar$, hence by \hyperlink{B4-a}{{\bf (B4-a)}} and \eqref{e:scale}, it holds that
	\begin{align*}
		\sB(y,z) \le c_1\Phi_0\bigg(\frac{\delta_D(y)}{|y-z| }\bigg) 
			\le c_2 \bigg(\frac{4\epsilon_1^{-1}ar}{|y-z|}\bigg)^{\lb_0}
			\Phi_0\bigg(\frac{\delta_D(y)}{4\epsilon_1^{-1}ar}\bigg).
	\end{align*}
Therefore, we have
	\begin{align}\label{e:jump-annulus}
	 	\int_{ D \setminus B(x,br)} \frac{\sB(y,z)}{|y-z|^{d+\alpha}}dz 
	 	&\le c_2 (4\epsilon_1^{-1}ar)^{\lb_0}
	 		\,\Phi_0\bigg(\frac{\delta_D(y)}{4\epsilon_1^{-1}ar}\bigg)\int_{ D \setminus B(y,4br/5) } 
	 		\frac{dz}{|y-z|^{d+\alpha+\lb_0}}\nn\\
		&\le  \frac{c_3 (4\epsilon_1^{-1}ar)^{\lb_0}}{(4br/5)^{\alpha+\lb_0}}\Phi_0\bigg(\frac{\delta_D(y)}{4\epsilon_1^{-1}ar}\bigg) 
			=\frac{c_4a^{\lb_0}}{b^{\alpha+\lb_0}r^{\alpha} }\Phi_0\bigg(\frac{\delta_D(y)}{4\epsilon_1^{-1}ar}\bigg).
	\end{align}
Note that  $B_D(x,ar) \subset  B_D(Q, 2ar) \subset U(4ar)$ by \eqref{e:U-rho-C11-1}. Thus, using 
the L\'evy system formula \eqref{e:Levysystem-Y-kappa}, \eqref{e:jump-annulus} and Corollary \ref{c:Dynkin-upper},  we arrive at
	\begin{align*}
		\P_x \Big( Y_{\tau_{B_D(x,a r)}} \in D \setminus B(x,br) \Big) &= \E_x \left[ \int_0^{\tau_{B_D(x,a r)}} 
			\int_{ D \setminus B(x,br) } \frac{\sB(Y_s,z)}{|Y_s-z|^{d+\alpha}}dz ds\right]\\
		&\le \frac{c_4a^{\lb_0}}{b^{\alpha+\lb_0}r^{\alpha}} 
			\E_x \left[ \int_0^{\tau_{ B_D(x,a r)}}\Phi_0\bigg(\frac{\delta_D(Y_s)}{4\epsilon_1^{-1}ar}\bigg) ds \right]\\
		& \le \frac{c_4a^{\lb_0}}{b^{\alpha+\lb_0}r^{\alpha}}
			\E_x \left[ \int_0^{\tau_{U(4ar)}}\Phi_0\bigg(\frac{\delta_D(Y_s)}{4\epsilon_1^{-1}ar}\bigg) ds \right] \\
		&\le c_5(a/b)^{\alpha+\lb_0} ( \delta_D(x)/(ar))^p.
	\end{align*}
\end{proof}

Let  $Q \in \partial D$ and $0< r \le \wh R/24$. Since $q \in (p,\alpha+\lb_1)$ and $q_1 \in (p,q)$, by using  Corollary \ref{c:Dynkin-local} and  Lemma \ref{l:Dynkin2}, we see that for all $x \in U(\epsilon_2r/2)$,
	\begin{align*}
		&2^{-1}\delta_D(x)^q \le \chi_q(x) \le \chi_q(x) -  \E_x \big[ \chi_{q} (Y_{\tau_{U(\epsilon_2r)})}\big] 
			=- \E_x \bigg[ \int_0^{\tau_{U(\epsilon_2r)}} L \chi_q(Y_s)ds \bigg] \\
		&\le  - \E_x \bigg[ \int_0^{\tau_{U(\epsilon_2r)}} \wt L_{q_1} \chi_q(Y_s)ds \bigg] \le cr^{q-\alpha} 
			\E_x \bigg[ \int_0^{\tau_{U(\epsilon_2r)}} \Phi_0(\delta_D(Y_s)/r)ds \bigg],
	\end{align*}
where the function $\chi_q$ is defined  in Lemma \ref{l:Dynkin2} and the operator $\wt L_{q_1}$ is defined by \eqref{e:def-Lq}.
Combining the above with Lemma \ref{l:exit-jump},  we obtain

\begin{lemma}\label{l:Dynkin-pre-lower}
Let 
$Q \in \partial D$ and $0<r \le \wh R/24$. There exists $C>0$ independent of $Q$ and $r$ such that for any $x \in U(\epsilon_2r/2)$,
	\begin{align*} 
		\P_x \big(Y_{\tau_{U(\epsilon_2 r)}} \in U(\epsilon_2 r, r) \setminus U(\epsilon_2 r, (3/4)r) \big) 
			\ge C (\delta_D(x)/r)^q.
	\end{align*}
\end{lemma}

The next proposition is the most demanding part of the proof of Theorem \ref{t:Dynkin-improve}.

\begin{prop}\label{p:Dynkin-improve}	Let $Q \in \partial D$ and $0<r \le \wh R/24$. There exists $C>0$ independent of $Q$ and $r$ such that for any $x \in 	U(\epsilon_2r/4)$,
	\begin{align*} 
		\P_x \big( Y_{\tau_{U(\epsilon_2 r)}} \in U(\epsilon_2 r, r) \setminus U(\epsilon_2 r, (3/4)r) \big) 
			\ge C 	\P_x \big(Y_{\tau_{U(\epsilon_2 r)}} \in U(r) \big).
	\end{align*}
\end{prop}

To prove Proposition \ref{p:Dynkin-improve}, we follow the proof of  \cite[Lemma 6.2]{KSV18-a} (see also \cite[Lemma 6.2]{KSV} and \cite[Lemma 5.2]{KSV21}). 
The main challenge is to prove Lemma \ref{l:one-step} below. Unlike in \cite{KSV18-a} and \cite{KSV},
since we allow the killing potential $\kappa$ to be zero,  highly non-trivial modifications are needed. In \cite{KSV21}, where the case of no killing potential was studied, the step corresponding to Lemma \ref{l:one-step} 
was a consequence of the scaling property of the underlying process (see \cite[Corollary 3.4(b)]{KSV21}),  which is not applicable in the current setting.

Before giving the proof of Proposition \ref{p:Dynkin-improve}, we introduce some notation which is 
used in the proof.  Let $Q \in \partial D$, $0<r \le \wh R/24$,
	\begin{align*}
		H_1:=\left\{Y_{\tau_{U(\epsilon_2 r)}} \in U(\epsilon_2 r, r) \setminus U(\epsilon_2 r, (3/4)r)\right\} 
			\quad \text{and} \quad 	H_2:=\left\{Y_{\tau_{U(\epsilon_2 r)}} \in U(r)\right\}.
	\end{align*}
For $i \ge 1$, we set
	\begin{align*}
  		&s_i:= \frac{5\epsilon_2 r}{8} \bigg(\frac12 - \frac{1}{50} \sum_{j=1}^i \frac{1}{j^2}\bigg), \\
  		& U_i^-:=U(s_i, 2^{-i-1}\epsilon_2 r)	 \quad \text{ and } \quad U_i^+:=U(s_i, 2^{-i} \epsilon_2   r) \setminus U_i^-.
	\end{align*}
Note that for all $i \ge 1$, we have  $\epsilon_2 r/4<s_i<5\epsilon_2r/16$, $U_{i+1}^+ \subset U_i^- \subset U(\epsilon_2 r)$ for all $i\ge 1$ and 
	\begin{align}\label{e:Ui+}
		2^{-i-2}\epsilon_2 r\le \delta_D(z) \le 2^{-i}\epsilon_2 r, \quad z \in U_i^+
	\end{align}  
by  \eqref{e:U-rho-C11-2}. This implies that $U(\epsilon_2 r/4)\subset \cup_{i\ge 1}U_i^+$. 
\begin{figure}[!h]
	\centering
	
	\includegraphics[width=0.7\textwidth]{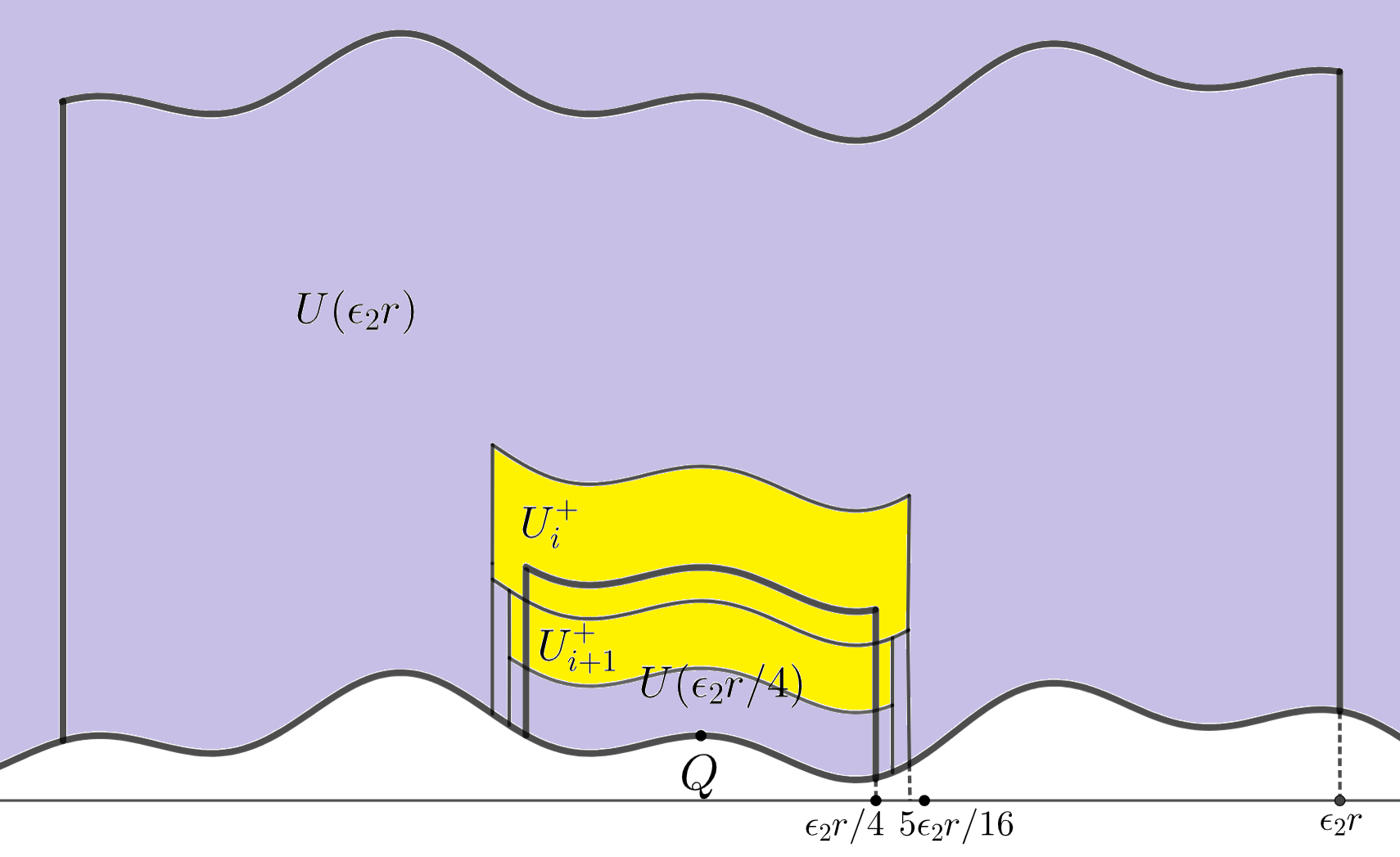}
	
	\vspace{-3mm}
	
	\caption{The sets $U^+_i$ and $U^+_{i+1}$}\label{Boxes}
\end{figure} 
Moreover, by Lemma \ref{l:Dynkin-pre-lower}, \eqref{e:Ui+},
there exists a constant $c>0$ such that  
	\begin{align}\label{e:H1-prebound}
		\P_z(H_1) \ge c(2^{-i-2}\epsilon_2)^q \quad \text{for  all $i \ge 1$ and $z \in U_{i}^+$}.
	\end{align}
Define for $i \ge 1$,
	\begin{align*}
		a_i= \sup_{z \in U_i^+} \big(\P_z(H_2)/\P_z(H_1)\big) \quad \text{and} \quad \tau_i = \tau_{U_i^-}.
	\end{align*}
For every $i \ge 1$, the constant $a_i$ is finite  by \eqref{e:H1-prebound}. Our goal is to show that 
	$$
		\sup_{z\in U(\epsilon_2 r/4)}\big(\P_z(H_2)/\P_z(H_1)\big)
			\le \sup_{z\in \cup_{i\ge 1}U_i^+}\big(\P_z(H_2)/\P_z(H_1)\big)=\sup_{i\ge 1}a_i <\infty,
	$$
which proves the proposition. This will be done through a series of lemmas.

\begin{lemma}\label{l:6.29}
For all $i \ge 1$,
	\begin{align}\label{e:6.29}
		a_{i+1} \le \sup_{1 \le j \le i}a_j + \sup_{z \in U_{i+1}^+}\frac{\P_z\big( Y_{\tau_{i}} 
			\in U(r) \setminus  \cup_{k=1}^{i} U_k^+\big)}{\P_z(H_1)}.
	\end{align}
\end{lemma}
\begin{proof}
Let $i\ge 1$ and $z\in U_{i+1}^+$. Since $\tau_i\le \tau_{U(\epsilon_2 r)}$, we have by the strong Markov property that
	\begin{align*}
		& \P_z(H_2, Y_{\tau_i}\in \cup_{k=1}^i U_k^+)=\sum_{k=1}^i \P_z(H_2, Y_{\tau_i}\in U_k^+)\\
		& = \sum_{k=1}^i \E_z(\P_{Y_{\tau_i}}(H_2), Y_{\tau_i}\in U_k^+) \le \sum_{k=1}^i 
			\E_z(a_k \P_{Y_{\tau_i}}(H_2), Y_{\tau_i}\in U_k^+)\\
		&\le (\sup_{1\le j \le i} a_j )\P_z(H_1, Y_{\tau_i}\in \cup_{k=1}^i U_k^+)\le (\sup_{1\le j \le i} a_j ) \P_z(H_1).
	\end{align*}
This implies that
	\begin{align*}
		\P_z(H_2)&=\P_z(H_2, Y_{\tau_i}\in \cup_{k=1}^i U_k^+)+\P_z(Y_{\tau_i}\in U(r)\setminus \cup_{k=1}^i U_k^+)\\
		&\le (\sup_{1\le j \le i} a_j ) \P_z(H_1)+ \P_z(Y_{\tau_i}\in U(r)\setminus \cup_{k=1}^i U_k^+),
	\end{align*}
which implies \eqref{e:6.29}.
\end{proof}

For $i \ge 1$, we define $\sigma_{i,0}=0$, $ \sigma_{i,1}=\inf\{t>0:|Y_t-Y_0| \ge 2^{-i-1}\epsilon_2 r\}$ 
and $\sigma_{i,k+1}=\sigma_{i,k} + \sigma_{i,1}\circ \uptheta_{\sigma_{i,k}}$ for $k \ge 1$. Here $\uptheta_t$ denotes the shift operator for $Y$.
\begin{lemma}\label{l:one-step}
	There exists a constant  $b_1 \in (0,1)$ independent of $Q$ and $r$ such that for all $i \ge 1$ and $w \in U_i^-$,
	\begin{align*}
		\P_w (\tau_{i} > \sigma_{i, 1} ) \le b_1.
	\end{align*}
\end{lemma}
\begin{proof} Let $i \ge 1$ and  $w \in U_{i}^-$. By the L\'evy system formula\eqref{e:Levysystem-Y-kappa}, 
since  $Y$ can be regarded as the part process of $\overline Y$ killed at $\zeta$,  we have
	\begin{equation}\label{e:one-step-1}
	\begin{split}
		\P_w (\tau_{i} \le \sigma_{i,1} ) &\ge  \P_w \big( Y_{\tau_{B_D(w, 2^{-i-2} \epsilon_2r)}} \in (D \setminus U_i^-) 
			\cup \{ \partial\}\big)\\
		& \ge \P_w\left(Y_{\tau_{B_{ D}(w, 2^{-i-2} \epsilon_2r)}} 
			\in D \setminus U_i^-, \, \tau_{ B_D(w,2^{-i-2} \epsilon_2 r)}<\zeta \right)   \\
		&\quad + \P_w\left(\tau_{ B_D(w,2^{-i-2} \epsilon_2 r)}=\zeta \right)\\
		& = \E_w \bigg[\1_{\{\tau_{ B_D(w,2^{-i-2} \epsilon_2 r)} <\zeta\}}
			\!\int_0^{\bar\tau_{B_{\overline D}(w,2^{-i-2} \epsilon_2 r)}} \!\! \int_{D \setminus U_i^-}  
			\frac{\sB(\overline Y_s,y)}{|\overline Y_s-y|^{d+\alpha}} dyds  \bigg] \\
		&\quad + \E_w\Big[\1_{\{\tau_{ B_D(w,2^{-i-2} \epsilon_2 r)} =\zeta\}}\Big]  \\
		&\ge  \E_w \bigg[1 \wedge \int_0^{\bar\tau_{B_{\overline D}(w,2^{-i-2} \epsilon_2 r)}} \int_{D \setminus U_i^-} 
			\frac{\sB(\overline Y_s,y)}{|\overline Y_s-y|^{d+\alpha}}dyds   \bigg]. 
	\end{split}
	\end{equation}
Thus, to obtain the desired result, it suffices to show that there exist constants $c_0,c_1 \in (0,1)$ independent of $Q,r,i$ and 
 $w\in U_i^-$ such that
	\begin{align}\label{e:one-step-claim}
		 \P_w \bigg( \int_0^{\bar\tau_{B_{\overline D}(w,2^{-i-2} \epsilon_2 r)}} 
	 		\int_{D \setminus U_i^-} \frac{\sB(\overline Y_s,y)}{|\overline Y_s-y|^{d+\alpha}}dyds   \ge c_0 \bigg) \ge c_1.
	\end{align}
Indeed, if \eqref{e:one-step-claim} holds, then we deduce from \eqref{e:one-step-1} that
	\begin{align*}
		&\P_w(\tau_i> \sigma_{i,1}) =  1 - 	\P_w(\tau_i \le \sigma_{i,1}) \\
		&\le 1 - c_0  \P_w \bigg( \int_0^{\bar\tau_{B_{\overline D}(w,2^{-i-2} \epsilon_2 r)}} 
			\int_{D \setminus U_i^-} \frac{\sB(\overline Y_s,y)}{|\overline Y_s-y|^{d+\alpha}}dyds   \ge c_0 \bigg) \le 1- c_0c_1,
	\end{align*}
which yields the result.

Now we prove \eqref{e:one-step-claim}. We will use the coordinate system CS$_Q$. For $i \ge 1$, define
	$$ 
		A_i=\left\{ z \in B_{\overline D}(w,\,2^{-i-3} \epsilon_2 r) : \delta_D(z) > 2^{-i-5} \epsilon_2 r\right\}.
	$$
Then $m_d(A_i) \ge c_2(2^{-i-3}\epsilon_2 r)^d$ for a constant $c_2>0$ independent of $Q,r,i$ and $w$.
Let $K_i$ be any compact subset of $A_i$ such that $m_d(K_i) \ge 2^{-1}m_d(A_i)$. Then by Lemma \ref{l:ckw-3-7}
(with $b=1/2$, $R_0=\wh{R}/24$, and $r$ replaced by $2^{-i-2}\epsilon_2 r$), 
there exists $c_3 \in (0,1)$ independent of $Q,r,i$ and $w$  such that 
	\begin{align}\label{e:killing-technical-0}
		\P_w\big(\bar\sigma_{K_i}< \bar\tau_{B_{\overline D}(w,2^{-i-2} \epsilon_2 r)} \big) \ge c_3.
	\end{align}
Choose any $z \in K_i$,  $v\in B(z,2^{-i-6} \epsilon_2 r)$ and $y \in B(z+2^{-i}\epsilon_2 r \e_d , 2^{-i-6} \epsilon_2 r)$. Then we have $\delta_D(v) \ge \delta_D(z) - |z-v| \ge 2^{-i-6}\epsilon_2 r$, 
$\delta_D(v) \le \delta_D(w)  + |w-z| + |z-v|  \le  2^{-i}\epsilon_2r$ and
	\begin{align*}
		|v-y| \le |v-z| +2^{-i}\epsilon_2 r  + |z+2^{-i}\epsilon_2 r \e_d -y| < (2^{-i}+2^{-i-5})\epsilon_2 r.
	\end{align*}
Moreover, using the mean value theorem and \eqref{e:local-map},  since $z_d \ge \Psi(\wt z)$, 
$|\wt y| \vee |\wt z| \le |w| + 2^{1-i}\epsilon_2 r < r$ by \eqref{e:U-rho-C11-1} 
$\Lambda r\le \Lambda \wh R/24 \le 1/48$ and $\epsilon_2\le 1/12$, we see that
	\begin{align*}
		\rho_D(y) &\ge z_d+2^{-i}\epsilon_2 r - 2^{-i-6}\epsilon_2r - \Psi(\wt y)\\
		& \ge z_d +2^{-i}\epsilon_2 r - 2^{-i-6}\epsilon_2r - \Psi(\wt z) - \Lambda (|\wt y| \vee |\wt z|)|\wt y- \wt z|\\
		& \ge 2^{-i}\epsilon_2 r - 2^{-i-6}\epsilon_2r - (2^{-i}\epsilon_2 r + 2^{-i-6}\epsilon_2r)/48 \ge (23/24)2^{-i} \epsilon_2 r.
	\end{align*}
Thus $y\in D\setminus U_i^-$ showing that $B(z+2^{-i}\epsilon_2 r \e_d, 2^{-i-6}\epsilon_2 r)\subset D\setminus U_i^-$. 
Further, by \eqref{e:U-rho-C11-2}, 
$\delta_D(y) \ge  (2/\sqrt 5) \rho_D(y) \ge  (2^{-i-1}+2^{-i-6}) \epsilon_2 r \ge 2^{-1} (|v-y| \vee \delta_D(v))$.
By \hyperlink{B4-b}{{\bf (B4-b)}}, \eqref{e:scale} and \eqref{e:VD}, it follows that
	\begin{align*}
		&\int_{D \setminus U_i^-} \frac{\sB(v,y)}{| v-y|^{d+\alpha}}dy 
			\ge  C_7	\int_{B(z + 2^{-i}\epsilon_2 r \e_d , \,2^{-i-6} \epsilon_2 r)} 
			\frac{\Phi_0((\delta_D(v) \wedge \delta_D(y)) /|v-y|)}{| v-y|^{d+\alpha}}dy\nn\\
		&\quad  \ge \frac{c_4\Phi_0(2^{-6})}{(2^{1-i}\epsilon_2 r) ^{d+\alpha}}
			\int_{B(z + 2^{-i}\epsilon_2 r \e_d ,\, 2^{-i-6} \epsilon_2 r)} dy \ge c_5  (2^{-i}\epsilon_2 r)^{-\alpha},
	\end{align*}
where $c_5 \in (0,1)$ is a constant independent of $Q,r,i$ and $w$.

On the other hand, by Lemma \ref{l:EP} (with $T=\wh R^\alpha$, see also Remark \ref{r:conti}), 
there exists $c_6\in (0,1)$ such that 
	\begin{align}\label{e:killing-technical-1}
		\P_w \big( \bar\tau_{B(Y_0, 2^{-i-6}\epsilon_2 r)} \circ \uptheta_{\bar \sigma_{K_i}} 
			\ge c_6 ( 2^{-i-6}\epsilon_2 r)^{\alpha}\big) \ge 2^{-1}.
	\end{align}
On the event $\{\tau_{B(Y_0, 2^{-i-6}\epsilon_2 r)} \circ \uptheta_{\bar \sigma_{K_i}} \ge c_6 ( 2^{-i-6}\epsilon_2 r)^{\alpha}, \,\bar\sigma_{K_i}< \bar\tau_{B_{\overline D}(w,2^{-i-2} \epsilon_2 r)}\}$, we have 
	\begin{equation}\label{e:killing-technical-2}
		\begin{split}
		& \int_0^{\bar\tau_{B_{\overline D}(w,2^{-i-2} \epsilon_2 r)}} 
			\int_{D \setminus U_i^-} \frac{\sB(\overline Y_s,y)}{|\overline Y_s-y|^{d+\alpha}}dyds \\
	 	&\ge \int_{\bar\sigma_{K_i}}^{\bar \sigma_{K_i}+\bar\tau_{B(Y_0, 2^{-i-6}\epsilon_2 r)} 
	 		\circ \uptheta_{\bar \sigma_{K_i}}}	 
	  		\int_{D \setminus U_i^-} \frac{\sB(\overline Y_s,y)}{|\overline Y_s-y|^{d+\alpha}}dyds\\
		&\ge \int_{\bar\sigma_{K_i}}^{ \bar \sigma_{K_i}+\bar\tau_{B(Y_0, 2^{-i-6}\epsilon_2 r)} 
			\circ \uptheta_{\bar \sigma_{K_i}} } 	  
	   		\inf_{z \in K_i,\,v \in B(z, 2^{-i-6}\epsilon_2r)}\int_{D \setminus U_i^-} \frac{\sB(v,y)}{|v-y|^{d+\alpha}}dyds\\
	  	&\ge c_5 (2^{-i}\epsilon_2 r)^{-\alpha}
	 		 \int_{\bar\sigma_{K_i}}^{\bar \sigma_{K_i}+\bar\tau_{B(Y_0, 2^{-i-6}\epsilon_2 r)} 
	 		 \circ \uptheta_{\bar \sigma_{K_i}} }ds 
	  	\ge 2^{-6\alpha}c_5c_6.
	 \end{split}
	\end{equation}
By the strong Markov property,  \eqref{e:killing-technical-0},  \eqref{e:killing-technical-1}  and  \eqref{e:killing-technical-2}, we arrive at
	\begin{align*}
		 &\P_w \bigg( \int_0^{\bar\tau_{B_{\overline D}(w,2^{-i-2} \epsilon_2 r)}} \int_{D \setminus U_i^-} 
		 	\frac{\sB(\overline Y_s,y)}{|\overline Y_s-y|^{d+\alpha}}dyds   \ge 2^{-6\alpha}c_5c_6 \bigg) \\
		 &\ge \P_w \Big(\bar\tau_{B(Y_0, 2^{-i-6}\epsilon_2 r)} \circ \uptheta_{\bar \sigma_{K_i}} 
		 	\ge c_6 ( 2^{-i-6}\epsilon_2 r)^{\alpha}, \,\bar\sigma_{K_i}< \bar\tau_{B_{\overline D}(w,2^{-i-2} \epsilon_2 r)}   \Big)\\
		 &\ge 2^{-1}\P_w \big(\bar\sigma_{K_i}< \bar\tau_{B_{\overline D}(w,2^{-i-2} \epsilon_2 r)}   \big) \ge 2^{-1}c_3,
	\end{align*}
proving that  \eqref{e:one-step-claim} holds with  $c_0:=2^{-6\alpha}c_5c_6$. The proof is complete.
\end{proof}

\begin{lemma}\label{l:m-step}
For all $i, m \ge 1$ and $z \in U_{i+1}^+$, we have
	\begin{align*}
		\P_z (\tau_{i} > \sigma_{i, mi} ) \le b_1^{mi},
	\end{align*}
where $b_1 \in (0,1)$ is the constant in Lemma \ref{l:one-step}.
\end{lemma}
\begin{proof} Using the strong Markov property and Lemma \ref{l:one-step}, since $U_{i+1}^+ \subset U_i^-$, we obtain
	\begin{align*}
		&\P_z (\tau_{i} > \sigma_{i, mi} ) 
		= \E_z \left[ \P_{Y_{\sigma_{i, mi-1}}}  (\tau_i > \sigma_{i,1}) ; Y_{\sigma_{i, k}} \in U_i^-, \, 1 \le k \le mi-1 \right]\\
		&\le 	\sup_{w \in U_i^-}	\P_w (\tau_{i} > \sigma_{i, 1} ) \P_z( \tau_i > \sigma_{i,mi-1}) 
		\le \cdots \le \Big( \sup_{w \in U_i^-}	\P_w (\tau_{i} > \sigma_{i, 1} )\Big)^{mi} \le b_1^{mi}.
	\end{align*}
\end{proof}

Let $\lb_0\in [0, \beta_0]$ be such that $p<\alpha+\lb_0$ and that the first inequality in \eqref{e:scale} holds.
We now choose $m_0 \in \N$ such that $b_1^{m_0} < 2^{-(\alpha+\lb_0)}$, 
where $b_1 \in (0,1)$ is the constant in Lemma \ref{l:one-step}. Then we  choose $i_0 \in \N$ such that 
$400m_0(i+1)^3 < \epsilon_1 2^{i+1}$ for all $i \ge i_0$, where $\epsilon_1 \in (0,1/12)$ is the constant in Lemma \ref{l:Dynkin1}.

\begin{lemma}\label{l:jump-horizon}
	There exists $C>0$ independent of $Q$ and $r$ such that for any $i \ge i_0$ and $z \in U_{i+1}^+$,
	\begin{align*}
		\P_z\big( Y_{\tau_{i}} \in U(r) \setminus  \cup_{k=1}^{i} U_k^+, \, \tau_i \le \sigma_{i,m_0i}\big) 
			\le C  i^{3(\alpha+\lb_0)+1} \, 2^{-i(\alpha+\lb_0)}.
	\end{align*}
\end{lemma}
\begin{proof}  Let $z \in U^+_{i+1}$.  Note that for any $y=(\wt y,y_d) \in U(r) \setminus (U_i^- \cup  \cup_{k=1}^{i} U_k^+)$ in CS$_Q$, if $|\wt y|<s_i$, then using the mean value theorem and \eqref{e:local-map},  since $\rho_D(y) \ge 2^{-1}\epsilon_2r$, 
$\rho_D(z) \le 2^{-i-1}\epsilon_2 r$, $|\wt z| \le s_{i+1}< s_i<5\epsilon_2 r/16$ and  $\Lambda r\le \Lambda \wh R/24 \le 1/48$, we see that
	\begin{align*}
		y_d-z_d &\ge \rho_D(y) - \rho_D(z) - (\Psi(\wt y)-\Psi(\wt z))\\
		& \ge 2^{-1}\epsilon_2 r - 2^{-i-1}\epsilon_2 r  -   \sup\{\Lambda | w| : w \in \R^{d-1}, \, |w|<s_i\}\, |\wt y- \wt z|\\
		& \ge 2^{-2}\epsilon_2 r- 2\Lambda s_i^2 \ge 2^{-3}\epsilon_2 r > s_i-s_{i+1} = (\epsilon_2r)/(80(i+1)^2).
	\end{align*}
Thus, it holds that
	$$
		|y-z| \ge s_i-s_{i+1} =(\epsilon_2r)/(80(i+1)^2) 
		\quad \text{for all} \;\, y \in U(r) \setminus (U_i^- \cup  \cup_{k=1}^{i} U_k^+).
	$$
Hence, on the event $\{Y_{\tau_{i}} \in U(r) \setminus  \cup_{k=1}^{i} U_k^+, \, \tau_i \le \sigma_{i,m_0i} \}$, we have
	\begin{align*}
		\frac{\epsilon_2rm_0i}{40 m_0(i+1)^3}<	\frac{\epsilon_2r}{40(i+1)^2} 
			\le \sum_{1\le k \le m_0i, \, \sigma_{i,k-1}<\tau_i} |Y_{\sigma_{i,k}}- Y_{\sigma_{i,{k-1}}}|,
	\end{align*}
which yields that
	\begin{align*}
		&\left\{Y_{\tau_{i}} \in U(r) \setminus  \cup_{k=1}^{i} U_k^+, \, \tau_i \le \sigma_{i,m_0i} \right\}\\
		& \subset \cup_{k=1}^{m_0i} \left\{ |Y_{\sigma_{i,k}}- Y_{\sigma_{i,{k-1}}}| > \epsilon_2r / (40m_0(i+1)^3), \,
			 Y_{\sigma_{i,k-1}} \in U_i^-, \, Y_{\sigma_{i,k}} \in U(r)\right\}.
	\end{align*}
Now, using the strong Markov property, subadditivity and Lemma \ref{l:6.1}  (with $a=2^{-i-1}\epsilon_2$ and 
$b= \epsilon_2/(80m_0(i+1)^3)$), we obtain
	\begin{align*}
		&	\P_z\big( Y_{\tau_{i}} \in U(r) \setminus  \cup_{k=1}^{i} U_k^+, \, \tau_i \le \sigma_{i,m_0i}\big)\\
		&\le m_0i \sup_{ w \in U_i^-} \P_w \Big( |Y_{\sigma_{i,1}}-w| >\epsilon_2r/(80m_0(i+1)^3) \Big)\le c_1m_0i 
			\bigg( \frac{80m_0(i+1)^3}{2^{i+1}} \bigg)^{\alpha+\lb_0}.
	\end{align*}
The proof is complete. 
\end{proof}

\textsc{Proof of Proposition \ref{p:Dynkin-improve}.}  
By Lemmas \ref{l:m-step} and \ref{l:jump-horizon} and the definition of $m_0$, 
we have for all $i \ge i_0$ and $z \in U_{i+1}^+$,
	\begin{align*}
		&	\P_z\big( Y_{\tau_{i}} \in U(r) \setminus  \cup_{k=1}^{i} U_k^+\big) \\
		&\le 	\P_z\big( Y_{\tau_{i}} \in U(r) \setminus  \cup_{k=1}^{i} U_k^+, \, \tau_i \le \sigma_{i,m_0i}\big) 
			+ \P_z\big(\tau_i > \sigma_{i,m_0i}\big)\\
		&\le  c_1  i^{3(\alpha+\lb_0)+1} \, 2^{-i(\alpha+\lb_0)} + b_1^{m_0i} 
			\le  (c_1+1) i^{3(\alpha+\lb_0)+1} \, 2^{-i(\alpha+\lb_0)}.
	\end{align*}
Therefore, we deduce from \eqref{e:H1-prebound} and \eqref{e:6.29} that for all $i \ge i_0$,
	\begin{align*}
		\sup_{1 \le j \le i+1}	a_{j} \le \sup_{1 \le j \le i}a_j + c_2 i^{3(\alpha+\lb_0)+1} \, 2^{-i(\alpha+\lb_0-q)},
	\end{align*}
which implies that 
	$$	
		\sup_{j \ge 1}	a_{j} \le \sup_{1 \le j \le i_0}a_j 
			+ c_2\sum_{i=i_0}^\infty i^{3(\alpha+\lb_0)+1} \, 2^{-i(\alpha+\lb_0-q)}<\infty.
	$$
This proves the proposition. 
\qed

Finally, we are ready to give the proof of Theorem \ref{t:Dynkin-improve}.

\medskip

\textsc{Proof of Theorem \ref{t:Dynkin-improve}.}  
Using Proposition \ref{p:Dynkin-improve}, \eqref{e:4.2.1} and \eqref{e:4.2.2}, we get
	\begin{align*}
		\P_x (Y_{\tau_{U(\epsilon_2 r)}} \in U(r) ) 
			\le c 	\P_x (Y_{\tau_{U(\epsilon_2 r)}}\in U(\epsilon_2r, r) \setminus U(\epsilon_2 r, (3/4)r) ) \le c(\delta_D(x)/r)^p
	\end{align*}
and
	\begin{align*} 	
		&\P_x (Y_{\tau_{U(\epsilon_2 r)}} \in U(r) \setminus U(r, r/2) )   \\
		&\ge 	\P_x (Y_{\tau_{U(\epsilon_2 r)}} \in U(\epsilon_2r, r) \setminus U(\epsilon_2 r, (3/4)r) ) 
			\ge c	\P_x (Y_{\tau_{U(\epsilon_2 r)}} \in U(r) ) \ge c(\delta_D(x)/r)^p.
	\end{align*} 
Thus, it remains to show that $\P_x(Y_{\tau_{U(\epsilon_2 r)}}\in D\setminus U(r))\le c_1(\delta_D(x)/r)^p$ for some $c_1>0$. 

Let  $z \in U(\epsilon_2 r)$ and $w \in D \setminus U(r)$.  By \eqref{e:U-rho-C11-1}, we have $|z-Q|<2\epsilon_2 r$ and $|w-Q|>2r/3$. Hence, by  \hyperlink{B4-a}{{\bf (B4-a)}}, since  $\epsilon_2 \le 1/12$ and $\Phi_0$ is almost increasing, we see that 
	\begin{align}\label{e:new-lemma-1}
		&|z-w| \ge |w-Q| - |z-Q| \ge |w-Q|/2 \ge r/3 \quad \text{and} \quad	\sB(z,w) \le c \Phi_0(\delta_D(w)/r).
	\end{align}
Using the L\'evy system formula \eqref{e:Levysystem-Y-kappa} in the first line,  \eqref{e:new-lemma-1} and \eqref{e:scale} in the second, \eqref{e:U-rho-C11-1} in the third, and Corollary \ref{c:Dynkin-upper}  in the last, we arrive at
	\begin{align*}
		\P_x (Y_{\tau_{U(\epsilon_2 r)}} \in D \setminus U(r) ) &= \E_x \int_0^{\tau_{U(\epsilon_2 r)}} 
			\int_{D\setminus U(r)}\frac{\sB(Y_s, w)}{|Y_s-w|^{d+\alpha}}\, dw\, ds\\
		&\le  c_2 \E_x \int_0^{\tau_{U(\epsilon_2 r)}} \Phi_0(\delta_D(Y_s)/r)ds\,\int_{D\setminus U(r)}\frac{dw}{|w-Q|^{d+\alpha}}\\
		&\le  c_2 \E_x \int_0^{\tau_{U(
		\ \epsilon_1  r)}} \Phi_0(\delta_D(Y_s/r))ds\int_{B(Q,2r/3)^c}\frac{dw}{|w-Q|^{d+\alpha}}\\
		&\le c_3 (\delta_D(x)/r)^p r^{\alpha} (2r/3)^{-\alpha}  =c_4(\delta_D(x)/r)^p.
	\end{align*}
The proof is complete. 
\qed


\section{Estimates of Green potentials}\label{ch-killed-potentials}

In this section we establish some upper and lower bounds of the Green function $G^{B_D(x_0,R_0)}(x,y)$, $x_0\in \overline{D}$, $R_0>0$, that incorporate the decay rate at the boundary. Based on these estimates and using the technical Lemma \ref{l:key-curved}, we obtain sharp two-sided estimates of (killed) Green potentials of powers of distance to the boundary.

 We let $\epsilon_2 \in (0,1/12)$ be the constant in Theorem \ref{t:Dynkin-improve}   for the remainder of this work.  

We first deal with  the  upper  bound. 
\begin{prop}\label{p:prelub} 
	 Let $x_0\in \overline{D}$ and $R_0>0$. There exists $C=
	 C(R_0)>0$ independent of $x_0$ such that 
	\begin{align*}
		G^{B_D(x_0,R_0)}(x, y) \le C
		\left(\frac{\delta_D(x) \wedge \delta_D(y) }{|x-y|} \wedge 1\right)^p
		\frac{1}{|x-y|^{d-\alpha}}, \quad x,y \in B_D(x_0, R_0).
	\end{align*}
\end{prop}
\begin{proof}  Without loss of generality, we assume $R_0>4$.  Set $B:=B_D(x_0,R_0)$ and $\wt B  :=B_D(x_0,R_0+1)$. By  the symmetry and Proposition \ref{p:green-upper-bound},  we only need to show 
that  there exists a constant $c_1=c_1(R_0)>0$ such that for any $x,y \in B$ with 
$\delta_D(x)=\delta_D(x) \wedge \delta_D(y) <2^{-8}(\epsilon_2\wh R/R_0)|x-y|$, 
	$$
		G^{B}(x, y) \le c_1 \delta_D(x)^p |x-y|^{-d+\alpha-p}.
	$$ 

Let $x,y \in B$ with $\delta_D(x)=\delta_D(x) \wedge \delta_D(y) < 2^{-8}(\epsilon_2\wh R/R_0)|x-y|$ and  
$Q_x\in \partial D$ be such that $|x-Q_x|=\delta_D(x)$. In the following, we  use the coordinate system CS$_{Q_x}$,
 and write $U(r)$ for $U^{Q_x}(r)$. 
Set $r:=2^{-6}(\wh R/R_0)|x-y|$.  Then  $r<(\wh R/32) \wedge (|x-y|/8)$. Moreover, by \eqref{e:U-rho-C11-1}, we see that
	\begin{align}\label{e:prelub-1}
		U(\epsilon_2 r) \subset 	U(r) \subset 	B_D(Q_x, 2r) \subset B_D(x,3r) \subset \wt B\setminus B_D(y,5r).
	\end{align}
In particular, $G^{\wt B}(\cdot, y)$ is regular harmonic in $U(\epsilon_2 r)$. Thus, we have
	\begin{align*}
		&G^{B}(x, y) \le G^{\wt B}(x, y)\\
		&=\E_x\left[G^{\wt B}(Y_{\tau_{U(\epsilon_2 r)}}, y); Y_{\tau_{U(\epsilon_2 r)}}\in U(r)\right] 
			+ \E_x\left[G^{\wt B}(Y_{\tau_{U(\epsilon_2 r)}}, y); Y_{\tau_{U(\epsilon_2 r)}} \in \wt B \setminus U(r)	\right]\\
		&=:I_1+I_2.
	\end{align*}

For $I_1$, using Proposition \ref{p:green-upper-bound}, \eqref{e:prelub-1} and Theorem \ref{t:Dynkin-improve}, we get
	\begin{align*}
		&I_1\le  c_2 (5r)^{\alpha-d}\,\P_x(Y_{\tau_{U(\epsilon_2 r)}}\in U(r))\le c_3   \delta_D(x)^p r^{-d+\alpha-p}.
	\end{align*}

For $w\in U(\epsilon_2 r)$ and $z\in D\setminus U(r)$, we have  
$|w| \le 2\epsilon_2 r$ and $|z| \ge r/2$ by \eqref{e:U-rho-C11-1}, 
so that  $|z| \asymp |z-w|\ge r/3$. Thus, by using  Proposition \ref{p:green-upper-bound} and \hyperlink{B4-a}{{\bf (B4-a)}}, 
since $\Phi_0$ is almost increasing,  we see that for all $w \in U(\epsilon_2 r)$,
	\begin{align*}
		&\int_{D\setminus U(r)}G^{\wt B}(z,y) \frac{\sB(w,z)}{|w-z|^{d+\alpha}} dz
			\le c_4 \Phi_0(\delta_{D}(w)/r)\int_{D\setminus U(r)}\frac{dz}{|y-z|^{d-\alpha}|z|^{d+\alpha}}.
	\end{align*}
Hence, by using the L\'evy system formula \eqref{e:Levysystem-Y-kappa} and Corollary \ref{c:Dynkin-upper}, since $\epsilon_2$ is less  than or equal to the constant $\epsilon_1$ in Corollary \ref{c:Dynkin-upper}, we obtain
	\begin{align*}
		I_2 &   = \E_x \int_{0}^{\tau_{U(\epsilon_2 r)}} \int_{D\setminus U(r)}G^{\wt B}(z,y) 
			\frac{\sB(Y_s,z)}{|Y_s-z|^{d+\alpha}} dz \,ds\\
		&\le c_4 \E_x\int^{\tau_{U(\epsilon_2r)}}_0 \Phi_0(\delta_D(Y_s)/r)ds
			\int_{D\setminus U(r)}\frac{dz}{|y-z|^{d-\alpha}|z|^{d+\alpha}} \\
		&\le c_5 \delta_D(x)^p r^{\alpha-p} \int_{D\setminus U(r)}\frac{dz}{|y-z|^{d-\alpha}|z|^{d+\alpha}}.
	\end{align*}
Since $D\setminus U(r) \subset \R^d \setminus B(0,r/2)$ by \eqref{e:U-rho-C11-1}, we have
	\begin{align*}
		\int_{D\setminus (U(r) \cup B(y, r))} \frac{dz}{|y-z|^{d-\alpha}|z|^{d+\alpha}} 
			\le  r^{-d+\alpha} \int_{\R^d \setminus B(0,r/2)}\frac{dz}{|z|^{d+\alpha}}   \le c_6 r^{-d}
	\end{align*}
and
	\begin{align*}
		\int_{B(y,r)\setminus U(r)} \frac{dz}{|y-z|^{d-\alpha}|z|^{d+\alpha}} 
			\le     r^{-d-\alpha} \int_{B(y, r)}\frac{dz}{|y-z|^{d-\alpha}} \le c_7 r^{-d}.
	\end{align*}
Thus, $I_2 \le c_8 \delta_D(x)^p r^{-d+\alpha-p}$ and the proof is complete.
\end{proof}

We now deal with  the \ lower  bound. 
\begin{thm}\label{t:GB}
Let $x_0 \in \overline D$ and $R_0>0$.	There exists  $C
= C(R_0)>0$ independent of $x_0$ such that for all $R \in (0,R_0]$ 
and  $x,y\in B_D(x_0,R/10)$, 
	$$
		G^{B_D(x_0,R)}(x,y)\ge C \left(\frac{\delta_D(x)}{|x-y|}  \wedge 1 \right)^p\left(\frac{\delta_D(y)}{|x-y|} 
			 \wedge 1 \right)^p \frac{1}{|x-y|^{d-\alpha}}.
	$$
\end{thm}

We will prove Theorem \ref{t:GB} using the following two lemmas.  

\begin{lemma}\label{l:GB_1}
Let $x_0 \in \overline D$ and $R_0>0$.	For every $k>1$, there exists $C=C(R_0,k)>0$ such that for all $R\in (0,R_0]$ a
nd  $x,y\in B_D(x_0,R/9)$  with $\delta_D(x) \le |x-y|\le k \delta_D(y) $, it holds that
	$$
		G^{B_D(x_0,R)}(x,y)\ge C \delta_D(x)^p|x-y|^{-d+\alpha-p}.
	$$
\end{lemma}
\begin{proof} Fix $x,y\in B_D(x_0,R/9)$  with $\delta_D(x) \le |x-y|\le k \delta_D(y)$.  Let $r:=\wh R|x-y|/(100+48R_0)$ and  
$Q_x\in \partial D$ be such that $|x-Q_x|=\delta_D(x)$. If $\delta_D(x) \ge 2^{-2}\epsilon_2r$, then  
$|x-y| \le (k \vee (2^2\epsilon_2^{-1} (100+48R_0)/\wh R))(\delta_D(x) \wedge \delta_D(y))$ so that  the result holds by Proposition \ref{p:green-lower-bound}.  

Suppose now that $\delta_D(x) \le 2^{-2}\epsilon_2r$. Set $U:=U^{Q_x}(\epsilon_2 r)$. 
By \eqref{e:U-rho-C11-1},  since $r<|x-y|/100< R/400$, we have
	\begin{align}\label{e:green-lower-region-1}
		U \subset  U^{Q_x}(r) \subset B_D(Q_x, 2r) \subset B_D(x, 33r/16) \subset B_D(x_0, R/8) \setminus \{y\}.
	\end{align}
Besides, by \eqref{e:U-rho-C11-1} and \eqref{e:U-rho-C11-2},  for all $z\in U^{Q_x}(r) \setminus U^{Q_x}(r,r/2)$, 
we have 
	\begin{align}\label{e:green-lower-region-2}
		|z-y|\le |z-Q_x| + |x-Q_x|+  |x-y|\le (3+ (100+48R_0)/\wh R )\,r
	\end{align}
and $\delta_D(z)\ge (2/\sqrt 5)\rho_D(z)\ge r/\sqrt 5$.
Thus, there exists  $c_1=c_1(R_0,k)>0$ such that
	\begin{align}\label{e:green-lower-region-3}
		|z-y| \le c_1(\delta_D(y) \wedge \delta_D(z)) \quad \text{for all} \;\, z\in 	U^{Q_x}(r) \setminus U^{Q_x}(r,r/2).
	\end{align}
By \eqref{e:green-lower-region-1} and \eqref{e:green-lower-region-3}, we get from Proposition \ref{p:green-lower-bound} and \eqref{e:green-lower-region-2}  that there exists $c_2=c_2(R_0,k)>0$ such that for all $z\in U^{Q_x}(r) \setminus U^{Q_x}(r,r/2)$,
	\begin{align*}
		G^{B_D(x_0,R)}(z,y) \ge c|z-y|^{-d+\alpha} \ge  c_2r^{-d+\alpha}.
	\end{align*}
Using this and Theorem \ref{t:Dynkin-improve}, since $G^{B_D(x_0,R)}(\cdot, y)$ is regular harmonic in $U$ 
by \eqref{e:green-lower-region-1}, we arrive at
	\begin{align*}
		&G^{B_D(x_0,R)}(x, y) \ge \E_x\left [G^{B_D(x_0,R)}(Y_{\tau_{U}}, y): Y_{\tau_{U}} 
			\in U^{Q_x}(r) \setminus U^{Q_x}(r,r/2)\right ]\\
		&\ge c_2r^{-d+\alpha} \P_x\left(Y_{\tau_{U}} 
			\in	U^{Q_x}(r) \setminus U^{Q_x}(r, r/2)\right) 
		\ge c_3 \delta_D(x)^p r^{-d+\alpha-p}.
	\end{align*}
\end{proof}

\begin{lemma}\label{l:GB_4}
	 Let $x_0 \in \overline D$ and $R_0>0$. There exists $C=C(R_0)>0$ such that for all $R \in (0,R_0]$ and  $x,y\in B_D(x_0,R/10)$ with $|x-y|\ge 4(\delta_D(x) \vee \delta_D(y))$, it holds that
	$$
		G^{B_D(x_0,R)}(x, y)\ge C\delta_D(x)^p\delta_D(y)^p|x-y|^{-d+\alpha-2p}.
	$$
\end{lemma}
\begin{proof}  
Let $x,y \in B_D(x_0,R/10)$ and  $r:=\wh R|x-y|/(100+48R_0)$.  By symmetry and Lemma \ref{l:GB_1},  it suffices to consider 
the case $\delta_D(x)\le \delta_D(y) \le 2^{-2}\epsilon_2 r$ only. Let  
$Q_x\in \partial D$ be such that $|x-Q_x|=\delta_D(x)$ and $U:=U^{Q_x}(\epsilon_2 r)$.  
By \eqref{e:U-rho-C11-1},  since $r<|x-y|/100< R/500$,
	\begin{align*}
		U \subset  U^{Q_x}(r) \subset B_D(Q_x, 2r) \subset B_D(x, 3r) \subset B_D(x_0, R/9) \setminus \{y\}.
	\end{align*}
Moreover, by \eqref{e:U-rho-C11-1}, we see that for all $z \in U^{Q_x}(r) \setminus U^{Q_x}(r,r/2)$, 
	$$
		|z-y| \ge |x-y| - |x-Q_x| - |z-Q_x| \ge 100r-3r=97r \ge \delta_D(y).
	$$ 
Note that \eqref{e:green-lower-region-2} and \eqref{e:green-lower-region-3} are still valid. By \eqref{e:green-lower-region-3}, there exists $c_1=c_1(R_0,k)>0$ such that  $\delta_D(z) \ge c_1 |z-y| \ge 97c_1r$ for all $z \in U^{Q_x}(r) \setminus U^{Q_x}(r,r/2)$. 
Thus, we can use Lemma \ref{l:GB_1} to conclude that 
	\begin{equation}\label{e:GB_4}
		G^{B_D(x_0,R)}(z,y)\ge c_2\delta_D(y)^p\, r^{-d+\alpha-p} \quad \text{for all } z\in  U^{Q_x}(r) \setminus U^{Q_x}(r, r/2).
	\end{equation}
Since $G^{B_D(x_0,R)}(\cdot , y)$ is regular harmonic in $U$,
by \eqref{e:GB_4}  and  Theorem \ref{t:Dynkin-improve}, we arrive at
	\begin{align*}
		&G^{B_D(x_0,R)}(x,y) \ge \E_x\left[	G^{B_D(x_0,R)}(Y_{\tau_{U}}, y): Y_{\tau_{U}} 
			\in 	U^{Q_x}(r) \setminus U^{Q_x}(r, r/2)\right]\\
			&\quad \ge c_2\delta_D(y)^pr^{-d+\alpha-p}\, \P_x\left(Y_{\tau_{U}} 
			\in U^{Q_x}(r) \setminus U^{Q_x}(r, r/2)\right)\\
		&\quad \ge c_3 \delta_D(x)^p \delta_D(y)^p r^{-d+\alpha-2p}.
	\end{align*}
\end{proof}

 Now we are in the position to give the proof of Theorem \ref{t:GB}. 

\textsc{Proof of Theorem \ref{t:GB}:} 
Let $x,y\in B_D(x_0,R/10)$. Without loss of generality, we assume that $\delta_D(x)\le \delta_D(y)$.  We have three cases:

\smallskip

\noindent
Case 1: $\delta_D(x)\le \delta_D(y)\le |x-y|/4$. Then we conclude from Lemma \ref{l:GB_4} that
	$$
		G^{B_D(x_0,R)}(x,y)\ge \frac{c_1 \delta_D(x)^p\,\delta_D(y)^p}{ |x-y|^{d-\alpha+2p}}\asymp  
		\left(\frac{\delta_D(x)}{|x-y|}\wedge 1\right)^p \left(\frac{\delta_D(y)}{|x-y|}\wedge 1\right)^p \frac{1}{|x-y|^{d-\alpha}}.
	$$

\noindent
Case 2: $\delta_D(x) \le |x-y|/4 \le \delta_D(y)$. Then we conclude from Lemma \ref{l:GB_1} that
	$$
		G^{B_D(x_0,R)}(x,y)\ge \frac{c_2 \delta_D(x)^p}{ |x-y|^{d-\alpha+p}}\asymp  
		\left(\frac{\delta_D(x)}{|x-y|}\wedge 1\right)^p \left(\frac{\delta_D(y)}{|x-y|}\wedge 1\right)^p \frac{1}{|x-y|^{d-\alpha}}.
	$$

\noindent
Case 3: $|x-y|/4\le \delta_D(x)\le \delta_D(y)$. Then  we conclude from Proposition \ref{p:green-lower-bound} that 
	$$
		G^{B_D(x_0,R)}(x,y)\ge \frac{c_3}{|x-y|^{d-\alpha}}\asymp  \left(\frac{\delta_D(x)}{|x-y|}\wedge 1\right)^p 
			\left(\frac{\delta_D(y)}{|x-y|}\wedge 1\right)^p \frac{1}{|x-y|^{d-\alpha}}.
	$$
The proof is complete.
\qed

\bigskip

In the next lemma, we let $\Phi$ be a positive Borel  function on $(0,1]$ such that 
	\begin{align}\label{e:scale4}
		c_l \bigg( \frac{r}{s}\bigg)^{\lb}\le 	\frac{\Phi(r)}{\Phi(s)} 
			\le c_u \bigg( \frac{r}{s}\bigg)^{\ub} \quad \text{for all} \;\, 0<s\le r \le 1
	\end{align}
for some constants $\lb,\ub \in \R$ with $\lb \le \ub$ and $c_l,c_u>0$.  Observe that for any $\gamma>-1-\lb$, by \eqref{e:scale4}, there exists $c_1=c_1(\gamma)>0$ such that
	\begin{align}\label{e:intscale}
		\int_0^s u^\gamma \Phi(u) du \le c_l^{-1} s^{-\lb}\, \Phi(s) 	\int_0^s u^{\gamma+\lb}  du 
			= c_1 s^{\gamma+1}\Phi(s) \quad \text{for all} \;\, s \in (0,1].
	\end{align}

The next technical lemma is a generalization of \cite[Lemma 6.1]{KSV20}, which was inspired by \cite[Lemma 3.3]{AGCV22}.
A simple version of the next lemma is used in Proposition \ref{p:bound-for-integral-new} below, while its full power will be used in Section \ref{ch:green}.

\begin{lemma}\label{l:key-curved}
Let $\Phi$ be as above, $\gamma>-1-\lb$,   $q>\alpha-1$,  $r \in (0, \wh R/8]$,  $x\in D$ with $\delta_D(x)\le r/2$, 
and let $Q_x\in \partial D$ be such that  $|x-Q_x|=\delta_D(x)$. For $a \in (0,r]$, define
	\begin{align*}
		\sI^{q,\gamma}(r,a)= \int_{U^{Q_x}(r,a)}\left(\frac{\delta_D(x)}{|x-w|}\wedge 1\right)^q  
			\frac{\rho_{D}^{Q_x}(w)^{\gamma} \,\Phi(\rho_{D}^{Q_x}(w)/r)}{|x-w|^{d-\alpha}} dw.
	\end{align*}
The following statements hold.
	
	\smallskip	
	
\noindent (i)  There exists $C>0$ independent of $r$ and  $x$ such that  for any $2\delta_D(x) \le a \le b\le r$,
	\begin{align*}
		&	\sI^{q,\gamma}(r,b)-\sI^{q,\gamma}(r,a)\le Cr^{\alpha+\gamma-q}\delta_D(x)^q 
			\int_{a/r}^{b/r} s^{\alpha+\gamma-q-1}\Phi(s)ds.
	\end{align*}
	
\noindent (ii) There exists $C>0$ independent of $r$ and  $x$ such that  for any $a \in (0,2\delta_D(x)]$,
	\begin{align*}
		&	\sI^{q,\gamma}(r,a)\le C \delta_D(x)^{\alpha-1}\, a^{\gamma+1}   \Phi(a/r).  
	\end{align*}  
	
\noindent (iii)  Assume that  $q < \alpha+\gamma+\lb$. Then there exists $C>0$ independent of $r$ 
and  $x$ such that  for any $a \in (0,r]$,
	\begin{align*}
		&	\sI^{q,\gamma}(r,a)\le C \delta_D(x)^{\alpha-1}\, a^{\gamma+1} 
			\bigg( \frac{\delta_D(x)}{a} \wedge 1 \bigg)^{q-\alpha+1}  \Phi(a/r).  
	\end{align*}  
\end{lemma}
\begin{proof} Let  $f^{(r)}=f^{(r)}_{Q_x}:U_\bH(3)\to U^{Q_x}(3r)$ be the function defined by \eqref{e:def-fr}. 
Set $v:=(f^{(r)})^{-1}(x)=(\wt 0, \delta_D(x)/r)$. 

\smallskip

(i) Let $2\delta_D(x) \le a \le b\le r$. Using the change of the variables $w=f^{(r)}(\xi)$ and Lemma \ref{l:diffeo}, we obtain
	\begin{align*}
		&	\sI^{q,\gamma}(r,b)-\sI^{q,\gamma}(r,a)	\le cr^d\int_{U_{\bH}(1, b/r) \setminus U_{\bH}(1,a/r)}
			\left(\frac{rv_d}{r|v-\xi|}\right)^q \frac{(r\xi_d)^{\gamma}\,\Phi(\xi_d)}{(r|v-\xi|)^{d-\alpha}} \,d\xi  \\
		&\qquad = cr^{\alpha+\gamma}  \int_{a/r}^{b/r} \bigg( \int_{|\wt \xi| <\xi_d}  
			+ \int_{\xi_d\le |\wt \xi| <1}  \bigg) \left(\frac{v_d}{|v-\xi|}\right)^q 
			\frac{\xi_d^{\gamma}\,\Phi(\xi_d)}{|v-\xi|^{d-\alpha}}\, d\wt\xi \, d\xi_d \\
		&\qquad =:cr^{\alpha+\gamma}(I_1+I_2).
	\end{align*}
Since $a/r\ge 2\delta_D(x)/r = 2v_d$, we get
	\begin{align*}
		I_1 &\le cv_d^q \int_{a/r}^{b/r}  \int_0^{\xi_d} \frac{\xi_d^{\gamma}\,
			\Phi(\xi_d)}{(\xi_d-v_d)^{q+d-\alpha}}\, s^{d-2}ds \, d\xi_d \\
		&\le  2^{q+d-\alpha}cv_d^q \int_{a/r}^{b/r} \xi_d^{\alpha+\gamma-q-1} \Phi(\xi_d) d\xi_d.
	\end{align*}
Besides, since $q>\alpha-1$, we have
	\begin{align*}
		I_2 \le cv_d^q \int_{a/r}^{b/r} \int_{\xi_d}^1 \frac{\xi_d^{\gamma}\,\Phi(\xi_d)}{s^{q+d-\alpha}}\, s^{d-2}ds \, d\xi_d  
			\le c v_d^q \int_{a/r}^{b/r} \xi_d^{\alpha+\gamma-q-1}\Phi(\xi_d)d\xi_d.
	\end{align*}
Since $v_d= \delta_D(x)/r$, we arrive at the desired result.

(ii) By the change of the variables $w=f^{(r)}(\xi)$ and Lemma \ref{l:diffeo}, for all $a \in (0,r]$,
	\begin{align}\label{e:key-curved-1}
		\sI^{q,\gamma}(r,a)	&\le c\int_{U_{\bH}(1,a/r)}\left(\frac{rv_d}{r|v-\xi|}\wedge 1\right)^q 
			\frac{(r\xi_d)^{\gamma}\,\Phi(\xi_d)}{(r|v-\xi|)^{d-\alpha}}\, r^d d\xi \\
		& =cr^{\alpha+\gamma}\int_{U_{\bH}(1,a/r)}\left(\frac{v_d}{|v-\xi|}\wedge 1\right)^q 
			\frac{\xi_d^{\gamma}\,\Phi(\xi_d)}{|v-\xi|^{d-\alpha}} d\xi. \nn
	\end{align}

Let $a \in (0, 2\delta_D(x)]$. Using \eqref{e:intscale}, since $q>\alpha-1$ and $\gamma>-1-\lb$, we obtain
	\begin{align*}
		& \int_{U_{\bH}(1,a/r) \setminus U_{\bH}(v_d, a/r)}\left(\frac{v_d}{|v-\xi|}\wedge 1\right)^q 
			\frac{\xi_d^{\gamma}\,\Phi(\xi_d)}{|v-\xi|^{d-\alpha}} d\xi\\ 
		&\le  cv_d^q \int_{v_d}^1 \int_{0}^{a/r}   \frac{\xi_d^{\gamma}\,\Phi(\xi_d)}{s^{q+d-\alpha}} s^{d-2} \,d\xi_d\, ds 
			\le c v_d^{\alpha-1} (a/r)^{\gamma+1}\Phi(a/r).
	\end{align*}
Since $v_d = \delta_D(x)/r \ge a/(2r)$, using \eqref{e:intscale} again, we also get
	\begin{align*}
		&  \int_{U_{\bH}(v_d,a/(4r))}\left(\frac{v_d}{|v-\xi|}\wedge 1\right)^q 
			\frac{\xi_d^{\gamma}\,\Phi(\xi_d)}{|v-\xi|^{d-\alpha}} d\xi	\\
		&\le  c \int_{0}^{v_d} \int_{0}^{a/(4r)}  \frac{\xi_d^{\gamma}\,\Phi(\xi_d)}{(v_d-\xi_d)^{d-\alpha}} s^{d-2} \,d\xi_d\, ds \\
		&	\le  c(v_d/2)^{-d+\alpha} \int_{0}^{v_d} s^{d-2} ds  \int_{0}^{a/r} \xi_d^\gamma \Phi(\xi_d)\,   d\xi_d\\
		&\le c v_d^{\alpha-1} (a/r)^{\gamma+1}\Phi(a/r).
	\end{align*}
Further, by using \eqref{e:scale4},  we see that
	\begin{align*}
		&I:=\int_{U_{\bH}(v_d,a/r) \setminus U_{\bH}(v_d,a/(4r))} 
			\left(\frac{v_d}{|v-\xi|}\wedge 1\right)^q \frac{\xi_d^{\gamma}\,\Phi(\xi_d)}{|v-\xi|^{d-\alpha}} d\xi\\
		&\le c(a/r)^\gamma \Phi(a/r) \int_{U_{\bH}(v_d,a/r) \setminus U_{\bH}(v_d,a/(4r))} \frac{d\xi}{|v-\xi|^{d-\alpha}}  \\
		&\le c(a/r)^\gamma \Phi(a/r) \int_{a/(4r)}^{a/r} \int_0^{v_d} 
			\frac{s^{d-2} }{s^{d-1-\alpha/2} \, |\xi_d-v_d|^{1-\alpha/2}}ds  \, d\xi_d\\
		&= cv_d^{\alpha/2} (a/r)^\gamma \Phi(a/r)  \int_{a/(4r)}^{a/r} \frac{d\xi_d}{ |\xi_d-v_d|^{1-\alpha/2}}.
	\end{align*}
If $v_d \ge 2a/r$, then 
	$$ 
		\int_{a/(4r)}^{a/r} \frac{d\xi_d}{ |\xi_d-v_d|^{1-\alpha/2}} \le (v_d/2)^{\alpha/2-1} 
			\int_{a/(4r)}^{a/r}  d\xi_d \le (v_d/2)^{\alpha/2-1}(a/r) 
	$$
and if $a/(2r)\le v_d < 2a/r$, then 
	$$ 
		\int_{a/(4r)}^{a/r} \frac{d\xi_d}{ |\xi_d-v_d|^{1-\alpha/2}} \le  \int_{v_d/8}^{2v_d}  
			\frac{d\xi_d}{ |\xi_d-v_d|^{1-\alpha/2}} \le c v_d^{\alpha/2} \asymp v_d^{\alpha/2-1}(a/r).
	$$
Thus, in any case, we get $I \le cv_d^{\alpha-1}(a/r)^{\gamma+1}\Phi(a/r)$.

Combining the above estimates with \eqref{e:key-curved-1}, since $v_d=\delta_D(x)/r$,  we arrive at
	$$
		\sI^{q,\gamma}(r,a)\le cr^{\alpha+\gamma} (\delta_D(x)/r)^{\alpha-1}(a/r)^{\gamma+1}\Phi(a/r)
			= c \delta_D(x)^{\alpha-1} a^{\gamma+1}\Phi(a/r).
	$$

(iii) By (ii), it remains to prove the claim for  $a\in [2\delta_D(x),r]$. Let $a\in [2\delta_D(x),r]$. Using (i) and (ii) in the second line, and \eqref{e:intscale} and \eqref{e:scale4} in the third, since $\alpha+\gamma-q-1>-1-\lb$, we get
	\begin{align*}
		\sI^{q,\gamma}(r,a) &= 	\sI^{q,\gamma}(r,a)- 	\sI^{q,\gamma}(r,2\delta_D(x)) + \sI^{q,\gamma}(r,2\delta_D(x))\\
		&\le cr^{\alpha+\gamma-q}\delta_D(x)^q \int_{2\delta_D(x)/r}^{a/r} s^{\alpha+\gamma-q-1}\Phi(s)ds 
			+ c \delta_D(x)^{\alpha+\gamma}  \Phi(2\delta_D(x)/r)\\
		&\le c r^{\alpha+\gamma-q}\delta_D(x)^q (a/r)^{\alpha+\gamma-q}\Phi(a/r) 
			+ c \delta_D(x)^{\alpha+\gamma} (2\delta_D(x)/a)^{\lb} \,\Phi(a/r)\\
		& 	\le c\delta_D(x)^q a^{\alpha+\gamma-q}\Phi(a/r) 
			+ c \delta_D(x)^{\alpha+\gamma} (2\delta_D(x)/a)^{-\alpha-\gamma+q} \,\Phi(a/r)\\
		&=c\delta_D(x)^q a^{\alpha+\gamma-q}\Phi(a/r) .
	\end{align*}
The proof is complete.   
\end{proof}
Now we are ready  to give the sharp two sided estimates on the Green potentials. 

\begin{prop}\label{p:bound-for-integral-new} 
Let $Q \in \partial D$ and $\gamma>-p-1$.  Then for any $R \in (0, \wh R/24]$,  any Borel set $A$ satisfying 
$B_D(Q,R/4) \subset A \subset B_D(Q,R)$  and any $x \in B_D(Q,R/8)$,
	\begin{equation*} 
		\E_x \int_0^{\tau_A}\delta_D(Y_t)^{\gamma }\, dt  = \int_A G^A(x,y) \delta_D(y)^{ \gamma }\, dy  \asymp  
		\begin{cases} 
			R^{\alpha+  \gamma   -p} \delta_D(x)^p, &  \gamma >p-\alpha,\\[2pt]
			\delta_D(x)^p\log(R/\delta_D(x)), &  \gamma =p-\alpha, \\[2pt]
			\delta_D(x)^{\alpha+ \gamma }, &\gamma <p-\alpha,
		\end{cases}
	\end{equation*}
where the comparison constants are independent of $Q$, $R$, $A$ and $x$.
\end{prop}
\begin{proof} Let  $R \in (0, \wh R/24]$, $A$ be a Borel set satisfying $B_D(Q,R/4) \subset A \subset B_D(Q,R)$ and $x \in B_D(Q,R/8)$. Note that $\delta_D(x)<R/8$.

\smallskip

{\bf Upper bound:} Let $Q_x\in \partial D$ be such that $|x-Q_x|=\delta_D(x)$. Since $|Q-Q_x| \le |Q-x|+\delta_D(x)<R/4$, using  \eqref{e:U-rho-C11-1}, we see that	$A \subset B_D(Q_x, 2R) \subset U^{Q_x}(3R)$. Thus, by  Proposition \ref{p:prelub} 
(with $x_0=Q_x$ and $R_0=\wh{R}$), we have
	\begin{align*}
		&\int_A G^A(x,y) \delta_D(y)^{\gamma}\, dy \le \int_{B_D(Q_x,2R)} G^{B_D(Q_x,\wh R)}(x,y)\delta_D(y)^{\gamma}\, dy\\
		&\le c\bigg( \int_{U^{Q_x}(3R) \setminus U^{Q_x}(3R, 2\delta_D(x) )} 
			+  \int_{U^{Q_x}(3R,2\delta_D(x))}\bigg)  \left(\frac{\delta_D(x)}{|x-y|}\wedge 1\right)^p  
			\frac{\rho_D(y)^\gamma}{|x-y|^{d-\alpha}} dy\\
		&=:I_1+I_2.
	\end{align*}
Applying Lemma \ref{l:key-curved}(i)-(ii) with $\Phi=1$  and $q=p$, we get that
	\begin{align*}
		I_1+I_2\le c R^{\alpha+\gamma-p}\delta_D(x)^p \int_{2\delta_D(x)/(3R)}^{1} s^{\alpha+\gamma-p-1}ds 
			+ c \delta_D(x)^{\alpha+\gamma}.
	\end{align*}
By considering each case separately, we deduce that the upper bound holds.

\smallskip

{\bf Lower bound:} By  Theorem \ref{t:GB}, we have
	\begin{align*}
		&\int_A G^A(x,y) \delta_D(y)^{\gamma}\, dy   \ge \int_{B_D(x,R/80)} 
			G^{B_D(x,R/8)}(x, y) \delta_D(y)^\gamma\, dy \nn\\
		&\quad \ge c\delta_D(x)^p \int_{B_D(x, R/80) \setminus B_D(x, \delta_D(x)/80)}    
			\frac{\delta_D(y)^{p+\gamma} } {|x-y|^{d-\alpha+2p}} dy =: c\delta_D(x)^p II.
	\end{align*}
Note that there exist $z_1\in D$ and a constant $c_1\in (0,1)$ depending only on $\Lambda$ 
such that $R/320<|z_1-x|<R/160$ and $\delta_D(z_1) \ge c_1 R/320$. Let $z_2 \in D$ be such that  
$\delta_D(x)/40<|z_2-x|<\delta_D(x)/20$. Now if $\gamma>p-\alpha$, then 
	$$
		II \ge (R/80)^{-d+\alpha-2p}\int_{B(z_1, c_1R/640)} \delta_D(y)^{p+\gamma}  dy \ge c R^{\alpha+\gamma-p}
	$$
and  if $\gamma<p-\alpha$, then 
	$$
		II \ge (\delta_D(x)/20 + \delta_D(x)/80)^{-d+\alpha-2p}\int_{B(z_2, \delta_D(x)/80)} \delta_D(y)^{p+\gamma}  dy 
			\ge c \delta_D(x)^{\alpha+\gamma-p}.
	$$

Now suppose that $\gamma=p-\alpha$. Define
	$$
	 V		=\left\{w=(\wt w,w_d) \text{ in CS}_{Q_x}: \delta_D(x)/80<w_d-x_d<R/160,\;  w_d-x_d>|\wt w| \right\}.
	$$
For any $w=(\wt w,w_d) \in V$, we have $|w-x|<2(w_d-x_d)<R/80$ and $\Psi(\wt w) \le |\wt w|<w_d$ by \eqref{e:Psi-bound}. 
Thus, $V \subset B_D(x, R/80) \setminus B_D(x, \delta_D(x)/80)$. It follows that
	\begin{align*}
		II&\ge \int_{V}    \frac{dy}{|x-y|^{d}}  
			\ge c \int_{\delta_D(x)/80}^{R/160} \int_{0}^{w_d} \frac{1}{(2w_d)^d} s^{d-2}ds \, dw_d\\
		& =  c \int_{\delta_D(x)/80}^{R/160}  \frac{dw_d}{w_d}  \asymp \log (R/\delta_D(x)).
	\end{align*}
The proof is complete. \end{proof}


\section{Carleson's estimate and  the  boundary Harnack principle}\label{ch:bhp}

So far our assumptions on the function $\sB(x,y)$ do not provide a full description of its behavior near the boundary -- the lower bound in \hyperlink{B4-b}{{\bf (B4-b)}} need not hold when both $x$ and $y$ are close to the boundary. Moreover, compared with \eqref{e:intro-skst} the factor containing $\delta_D(x)\vee \delta_D(y)$ is missing. This will be rectified in our final assumption on $\sB$. 
This final assumption will imply  \hyperlink{B4-a}{{\bf (B4-a)}} and  \hyperlink{B4-b}{{\bf (B4-b)}} with a specific $\Phi_0$.
With this assumption we will first prove Carleson's estimate and then also the boundary Harnack principle.

Let  $\Phi_1$ and $\Phi_2$ be Borel  functions on $(0,\infty)$ such that $\Phi_1(r)= \Phi_2(r)=1$ for $r \ge 1$ and that
	\begin{align*}
		c_L' \bigg( \frac{r}{s}\bigg)^{\lb_1}\le 	\frac{\Phi_1(r)}{\Phi_1(s)} \le c_U' \bigg( \frac{r}{s}\bigg)^{\ub_1} 
			\quad \text{for all} \;\, 0<s\le r\le 1,
	\end{align*}
and
	\begin{align*}
		c_L''\bigg( \frac{r}{s}\bigg)^{\lb_2}\le 	\frac{\Phi_2(r)}{\Phi_2(s)} \le c_U'' \bigg( \frac{r}{s}\bigg)^{\ub_2} 
			\quad \text{for all} \;\, 0<s\le r\le 1
	\end{align*}
for some $\ub_1\ge \lb_1\ge 0$, $\ub_2\ge \lb_2\ge 0$ and $c_L', c_U', c_L'', c_U''>0$. Note that $\Phi_1$ and $\Phi_2$ are almost increasing.
Let $\beta_1$ and $\beta_2$ be the lower Matuszewska indices of $\Phi_1$ and $\Phi_2$, defined by \eqref{e:Matuszewska} with $\Phi_1$ and $\Phi_2$ instead of $\Phi_0$, respectively.  Then by the definition of  the lower Matuszewska index, since $\Phi_1$ and $\Phi_2$ are almost increasing, we see that for any $\eps>0$, there exist constants  $c_L'(\eps)>0$ and $c_L''(\eps)>0$ such that 
	\begin{align}
		c_L'(\eps)\bigg( \frac{r}{s}\bigg)^{\beta_1-\eps  \wedge \beta_1}
			\le 	\frac{\Phi_1(r)}{\Phi_1(s)} \le c_U'\bigg( \frac{r}{s}\bigg)^{\ub_1} \quad \text{for all} \;\, 0<s\le r\le 1 
				\label{e:scale1new}
	\end{align}
		and
	\begin{align}
		c_L''(\eps)\bigg( \frac{r}{s}\bigg)^{\beta_2-\eps \wedge \beta_2}
			\le 	\frac{\Phi_2(r)}{\Phi_2(s)} \le c_U'' \bigg( \frac{r}{s}\bigg)^{\ub_2} \quad \text{for all} \;\, 
			0<s\le r\le 1.\label{e:scale2}
	\end{align} 

Let $\ell$ be a Borel  function on $(0,\infty)$ with the following properties: (i)   $\ell(r)=1$ for $r \ge 1$,  
and (ii) for every $\eps>0$, there exists a constant  $c(\eps)>1$ such that 
	\begin{align}	
		c(\eps)^{-1} \bigg( \frac{r}{s}\bigg)^{ -\eps \wedge    \beta_1}
			\le \frac{\ell(r)}{\ell(s)} 
			\le c(\eps) \bigg( \frac{r}{s}\bigg)^{  \eps  \wedge \beta_2} \quad \text{for all} \;\, 0<s\le r\le 1.
			\label{e:scale3}
	\end{align}
Note that $\ell$  is almost increasing if $\beta_1=0$, and $\ell$ is almost decreasing if $\beta_2=0$.

We consider the following condition.

\medskip

\noindent\hypertarget{B4-c}{{\bf (B4-c)}} There exist  comparison constants such that for all $x,y \in D$,
	\begin{align*}
		\sB(x,y) \asymp  \Phi_1\bigg(\frac{\delta_D(x) \wedge \delta_D(y)}{|x-y|}\bigg)\,
		\Phi_2\bigg(\frac{\delta_D(x) \vee \delta_D(y)}{|x-y|}\bigg)\,
		\ell\bigg(\frac{\delta_D(x)\wedge \delta_D(y)}{(\delta_D(x)\vee\delta_D(y))\wedge |x-y|}\bigg).
	\end{align*}

\medskip

\begin{remark}\label{r:compare-KSV}
Let $\beta_1,\beta_2,\beta_3,\beta_4\ge 0$ be such that $\beta_1>0$ if $\beta_3>0$, and $\beta_2>0$ if $\beta_4>0$. By letting  $\Phi_1(r)=(r\wedge 1)^{\beta_1}$, 
	$$
		\Phi_2(r)= \frac{1}{(\log 2)^{\beta_4}}(r\wedge 1)^{\beta_2}(\log (1+1/(r\wedge 1)))^{\beta_4}
	$$ 
	and 
	$$
		\ell(r) = \frac{1}{(\log 2)^{\beta_3}} (\log(1+1/(r\wedge 1)))^{\beta_3},
	$$ 
\hyperlink{B4-c}{{\bf (B4-c)}}  covers the assumption {\bf (A3)} in \cite{KSV20}. Moreover, by Remark \ref{r:C=infty}, 
 we have  $\lim_{q \to \alpha+\beta_1} C(\alpha,q,\F)=\infty$.
\end{remark}

\ We now explain how  \hyperlink{B4-c}{{\bf (B4-c)}} is related to the previous assumptions. 

For given $\Phi_1$ and $\ell$ satisfying \eqref{e:scale1new} and \eqref{e:scale3} respectively, 
we define a function $\Phi_0$ on $(0, \infty)$ by
	\begin{align}\label{e:phi0new}
		\Phi_0(r):=\Phi_1(r)\ell(r), \qquad r>0.
	\end{align}
Then $\Phi_0(r)=1$ for $r \ge 1$. Further, for any $\eps>0$,  
by \eqref{e:scale1new} and \eqref{e:scale3}, there exists a constant $\wt c(\eps)>1$ such that 
	\begin{align}\label{e:scalenew}
		\wt c(\eps)^{-1}\bigg( \frac{r}{s}\bigg)^{\beta_1-\eps \wedge \beta_1}
			\le \frac{\Phi_0(r)}{\Phi_0(s)}\le \wt c(\eps) 
			\bigg( \frac{r}{s}\bigg)^{\ub_1+\eps\wedge\beta_2}\quad \mbox{for all } 0<s\le r\le1.
	\end{align}
Indeed, the second inequality in \eqref{e:scalenew} directly follows from \eqref{e:scale1new} and \eqref{e:scale3}. When $\beta_1=0$, the first inequality in \eqref{e:scalenew} holds since both $\Phi_1$ and $\ell$ are almost increasing in this case. When $\beta_1>0$, using \eqref{e:scale1new} and \eqref{e:scale3} with $\eps$ replaced by $(\eps \wedge \beta_1)/2$, we see that  the first inequality in \eqref{e:scalenew} holds. 
Therefore, the function $\Phi_0$ defined in \eqref{e:phi0new}
satisfies \eqref{e:scale} and is thus almost increasing.
It is clear from \eqref{e:scale3} that the lower Matuszewska index of $\Phi_0$ is equal to  $\beta_1$.
 
In this section and the next, 
we assume that \hyperlink{B1}{{\bf (B1)}}, \hyperlink{B3}{{\bf (B3)}}, \hyperlink{B4-c}{{\bf (B4-c)}},  \hyperlink{K3}{{\bf (K3)}} and \hyperlink{B5}{{\bf (B5)}} hold.
In the next lemma, we will show that, under these assumptions, \hyperlink{B2-a}{{\bf (B2-a)}}, \hyperlink{B2-b}{{\bf (B2-b)}}, 
\hyperlink{UBS}{{\bf (UBS)}} (hence \hyperlink{IUBS}{{\bf (IUBS)}}), and \hyperlink{B4-a}{{\bf (B4-a)}}--\hyperlink{B4-b}{{\bf (B4-b)}} (with the $\Phi_0$ defined in \eqref{e:phi0new}) holds. 

\emph{
In this section and the next, 
we will always take $\Phi_0$ to be the function defined in \eqref{e:phi0new} and thus}
	$$
 		\emph{the constant $\beta_0$ in Sections \ref{ch:operator}--\ref{ch-killed-potentials} is equal to $\beta_1$.}
	$$
	
\begin{lemma}\label{l:A3-c}
The following statements hold under \hyperlink{B4-c}{{\bf (B4-c)}}.
	
\noindent (i) 
	\hyperlink{B2-a}{{\bf (B2-a)}}, \hyperlink{B2-b}{{\bf (B2-b)}}, \hyperlink{B4-a}{{\bf (B4-a)}} and  \hyperlink{B4-b}{{\bf (B4-b)}} (with the $\Phi_0$ defined in \eqref{e:phi0new}) 	 hold. 
	
\noindent (ii) 
For every $\eps \in (0,1)$, there exists  $C=C(\eps)>1$ such that for every $x_0 \in D$ and $0<r<\delta_D(x_0)/(1+\eps)$, we have
	\begin{align*}
	 	C^{-1}\sB(z,y)\le 	\sB(x,y) \le C\sB(z,y) \;\; \text{for all} \;\, x,z \in B(x_0, r) 
	 		\text{ and } y \in D  \setminus B(x_0, (1+\eps)r).
	\end{align*}
	
\noindent (iii) For every $k\ge 1$, there exists $C=C(k)>0$ such that for all $x,y,z \in D$ satisfying $\delta_D(x) \le  k\delta_D(z)$ and $|y-z| \le M|y-x|$ with $M\ge 1$,
	\begin{align}\label{e:B6}
		\sB(x,y) \le CM^{\beta_1+\ub_1+\beta_2+\ub_2} \sB(z,y).
	\end{align}
	
\noindent (iv)  \hyperlink{UBS}{{\bf (UBS)}} holds.
\end{lemma}
\begin{proof} 
For $x,y\in D$, we define 
	\begin{align*}	
		&r_1^{x,y}=\frac{\delta_D(x) \wedge \delta_D(y)}{|x-y|}, \qquad r_2^{x,y}=\frac{\delta_D(x) \vee \delta_D(y) }{|x-y|}\\	
		&  \text{and} \quad  r_3^{x,y}=\frac{\delta_D(x) \wedge \delta_D(y)}{(\delta_D(x)\vee\delta_D(y)) \wedge |x-y|}.
	\end{align*}
Note that 
	\begin{align}\label{e:r3/r1=r2}
		r_3^{x,y}=r_1^{x,y}/ (r_2^{x,y} \wedge 1), \quad x,y \in D.
	\end{align}

(i) Let $x,y \in D$.  Using \hyperlink{B4-c}{{\bf (B4-c)}} in the first line below,  \eqref{e:scale3} in the second, \eqref{e:r3/r1=r2} in the third, and \eqref{e:scale2} in the last, we get 
	\begin{align}\label{e:B4-c-bounded}
	\begin{split}
		\sB(x,y) 	&\asymp \Phi_1(r_1^{x,y})\Phi_2(r_2^{x,y})	 \ell(r_3^{x,y}) \\
		& \le c_1 \Phi_0(r_1^{x,y})\Phi_2(r_2^{x,y}) (r_3^{x,y}/r_1^{x,y})^{\beta_2/2} \\
		&= c_1\begin{cases}
			\Phi_0(r_1^{x,y}) \Phi_2(1)  &\mbox{ if } r_2^{x,y} \ge 1,\\
			\Phi_0(r_1^{x,y}) \Phi_2(r_2^{x,y})(r_2^{x,y})^{-\beta_2/2}&\mbox{ if } r_2^{x,y} < 1
		\end{cases}\\
		& \le c_2 \Phi_0(r_1^{x,y}) \Phi_2(1).
	\end{split}
	\end{align}
Hence, \hyperlink{B4-a}{{\bf (B4-a)}} holds. 

For any $a \in (0,1)$, if $r_2^{x,y} \ge a$, then using the almost monotonicity of $\Phi_2$, \eqref{e:scale3} 
and  \eqref{e:r3/r1=r2}, we get
	\begin{align}\label{e:check-A3-b}
	\begin{split}
		&	\Phi_1(r_1^{x,y})\Phi_2(r_2^{x,y})\ell(r_3^{x,y}) 
			\ge c_3 \Phi_0(r_1^{x,y})\Phi_2(a) (r_3^{x,y}/r_1^{x,y})^{ -\beta_1}\\
		&\ge c_4 \Phi_2(a) (r_2^{x,y}\wedge 1)^{\beta_1} \Phi_0(r_1^{x,y}) \ge c_4 a^{\beta_1} \Phi_2(a) \Phi_0(r_1^{x,y}).
	\end{split}
	\end{align} 
Thus, \hyperlink{B4-b}{{\bf (B4-b)}} holds. Further, if $r_2^{x,y} \ge r_1^{x,y} \ge a$, then  since $\Phi_0$ is almost increasing, we get from \eqref{e:check-A3-b} that 
	$$
		\Phi_1(r_1^{x,y}) \Phi_2(r_2^{x,y})\ell(r_3^{x,y}) \ge c_5 a^{\beta_1} \Phi_0(a) \Phi_2(a),
	$$ 
which yields that  \hyperlink{B2-b}{{\bf (B2-b)}} holds. 
\hyperlink{B2-a}{{\bf (B2-a)}} holds since $\Phi_0$ is almost increasing.

(ii) Let $x_0 \in D$, $0<r<\delta_D(x_0)/(1+\eps)$ and $x,z \in B(x_0, r)$. We have 
$\delta_D(x) \vee \delta_D(z)\le  \delta_D(x_0) + r < 2\delta_D(x_0)$ and  
$\delta_D(x) \wedge \delta_D(z)\ge  \delta_D(x_0) - r>\eps\delta_D(x_0)/(1+\eps)$ so that $\delta_D(x) \asymp \delta_D(z)$ with  comparison constants depending only on $\eps$. Moreover,  $|x-y| \asymp |z-y|$ for $y \in D  \setminus B(x_0, (1+\eps)r)$. 
Hence, using \eqref{e:scale1new}, \eqref{e:scale2} and \eqref{e:scale3},  we get the result  from \hyperlink{B4-c}{{\bf (B4-c)}}.

(iii) Fix $k \ge 1$. Let  $x,y,z \in D$  be such that $\delta_D(x) \le k\delta_D(z)$ and $|y-z| \le M|y-x|$ with $M\ge 1$.  	Observe that
	\begin{align}
		\frac{r_1^{x,y}}{r_1^{z,y}}  &   \le M \left( \frac{\delta_D(x) \wedge \delta_D(y)}{\delta_D(z) \wedge \delta_D(y)} \right)
			\le kM \label{e:B6-check-r1}
	\end{align}
and
	\begin{equation}\label{e:B6-check-r2}
		\frac{r_2^{x,y}}{r_2^{z,y}}  \le M \left( \frac{\delta_D(x) \vee \delta_D(y)}{\delta_D(z) \vee \delta_D(y)} \right) \le kM.
	\end{equation}
We consider the following two cases separately.
	
	\smallskip
	
Case 1: $\delta_D(z)\vee \delta_D(y) \ge |y-z|/(kM)$. Recall that,  by (i),  \hyperlink{B4-c}{{\bf (B4-c)}} implies  \hyperlink{B4-a}{{\bf (B4-a)}}. Using   \hyperlink{B4-a}{{\bf (B4-a)}},  \eqref{e:scalenew} and \eqref{e:B6-check-r1}, we get
	\begin{align*}
		\sB(x,y) \le  c\Phi_0(r_1^{x,y})\le  c ( 1 \vee (r_1^{x,y}/ r_1^{z,y}))^{\ub_1+\beta_2} \Phi_0(r_1^{z,y}) 
			\le  c(k)M^{\ub_1+\beta_2}\Phi_0(r_1^{z,y}).
	\end{align*}
On the other hand, note that we have $r_2^{z,y} \ge  (kM)^{-1}$ in this case.  Using this,  \hyperlink{B4-c}{{\bf (B4-c)}},  \eqref{e:scale2},   \eqref{e:scale3} with the fact that $r_3^{z,y}\ge r_1^{z,y}$, and  \eqref{e:r3/r1=r2}, we  obtain
	\begin{align*}
		\sB(z,y)& \ge c\Phi_1(r_1^{z,y})\Phi_2(r_2^{z,y})	 \ell(r_3^{z,y}) 
			= 	c \frac{	\Phi_2(r_2^{z,y}) \ell(r_3^{z,y})}{\Phi_2(1)\ell(r_1^{z,y})} \Phi_0(r_1^{z,y})  \\
		&\ge c(r_2^{z,y} \wedge 1)^{\ub_2}(r_3^{z,y}/r_1^{z,y} )^{-\beta_1}\Phi_0(r_1^{z,y})\\
		&= c(r_2^{z,y} \wedge 1)^{\beta_1+\ub_2}\Phi_0(r_1^{z,y})\\
		& \ge c(k)M^{-\beta_1-\ub_2}\Phi_0(r_1^{z,y}).
	\end{align*}
	Hence, we conclude that  \eqref{e:B6} holds in this case. 
	
	\smallskip
	
Case 2: $\delta_D(z) \vee \delta_D(y)<|y-z|/(kM)$. 
In this case, we have $r_2^{z,y} <(kM)^{-1}$ and 
$\delta_D(x) \vee \delta_D(y) \le k (	\delta_D(z) \vee \delta_D(y)) <|y-z|/M\le |y-x|$. Hence,   $r_2^{z,y}\wedge1=r_2^{z,y}$ and  $r_2^{x,y}\wedge1=r_2^{x,y}$.
Thus, by  \eqref{e:scale3}, \eqref{e:r3/r1=r2}, \eqref{e:B6-check-r1} and \eqref{e:B6-check-r2}, if $r_3^{x,y} \ge r_3^{z,y}$, then
	\begin{align*}
		\frac{\ell(r_3^{x,y})}{\ell(r_3^{z,y})} 
		&\le c\bigg(\frac{r_3^{x,y}}{r_3^{z,y}}\bigg)^{\beta_2/2} 
			= c\bigg(\frac{r_1^{x,y}}{r_1^{z,y}}\bigg)^{\beta_2/2} \bigg(\frac{r_2^{z,y}}{r_2^{x,y}}\bigg)^{\beta_2/2} 
			\le c(k)M^{\beta_2/2}  \bigg(\frac{r_2^{z,y}}{r_2^{x,y}}\bigg)^{\beta_2/2}
	\end{align*}
and if $r_3^{x,y} < r_3^{z,y}$, then
	\begin{align*}
		\frac{\ell(r_3^{x,y})}{\ell(r_3^{z,y})} 
		&\le c\bigg(\frac{r_3^{x,y}}{r_3^{z,y}}\bigg)^{-\beta_1/2} = c\bigg(\frac{r_1^{z,y}}{r_1^{x,y}}\bigg)^{\beta_1/2}
		 \bigg(\frac{r_2^{x,y}}{r_2^{z,y}}\bigg)^{\beta_1/2} \le c(k)M^{\beta_1/2}\bigg(\frac{r_1^{z,y}}{r_1^{x,y}}\bigg)^{\beta_1/2} .
	\end{align*}
Therefore, whether $r_3^{x,y} \ge r_3^{z,y}$ or not, it holds that
	\begin{align}\label{e:checkB6-1}
		\frac{\ell(r_3^{x,y})}{\ell(r_3^{z,y})} 
			&\le c(k)M^{(\beta_1 +\beta_2)/2 }\bigg(1 \vee \frac{r_1^{z,y}}{r_1^{x,y}}\bigg)^{\beta_1/2}  
			\bigg(1\vee \frac{r_2^{z,y}}{r_2^{x,y}}\bigg)^{\beta_2/2}.
	\end{align}
By \eqref{e:scale1new} and \eqref{e:B6-check-r1}, we have
	\begin{align}\label{e:B6-check-Phi1}
		\begin{split}
		\frac{\Phi_1(r_1^{x,y})}{\Phi_1(r_1^{z,y})}\bigg(1 \vee \frac{r_1^{z,y}}{r_1^{x,y}}\bigg)^{\beta_1/2}& \le c\begin{cases}
			(r_1^{x,y}/r_1^{z,y})^{ \ub_1}  &\mbox{if } r_1^{z,y}\le r_1^{x,y},\\
			(r_1^{x,y}/r_1^{z,y})^{\beta_1/2 - \beta_1/2} 
			&\mbox{if } r_1^{z,y}> r_1^{x,y}
		\end{cases}\\
		& \le c(k) M^{ \ub_1} .
		\end{split}
	\end{align}
Similarly, using \eqref{e:scale2} and \eqref{e:B6-check-r2}, we get 
	\begin{align}\label{e:B6-check-Phi2}
		\begin{split}
		\frac{\Phi_2(r_2^{x,y})}{\Phi_2(r_2^{z,y})}\bigg(1 \vee \frac{r_2^{z,y}}{r_2^{x,y}}\bigg)^{\beta_2/2}& \le c\begin{cases}
			(r_2^{x,y}/r_2^{z,y})^{\ub_2}  &\mbox{if } r_2^{z,y}\le r_2^{x,y},\\
			(r_2^{x,y}/r_2^{z,y})^{\beta_2/2-\beta_2/2}  
			&\mbox{if } r_2^{z,y}> r_2^{x,y}
		\end{cases}\\
		& \le c(k) M^{\ub_2} .
		\end{split} 
	\end{align}
Using \hyperlink{B4-c}{{\bf (B4-c)}} in the first line below, \eqref{e:checkB6-1} in the second,  and \eqref{e:B6-check-Phi1} and  \eqref{e:B6-check-Phi2}  in the  last, we arrive at
	\begin{align*}
		\frac{\sB(x,y)}{\sB(z,y)}& \asymp 	\frac{\Phi_1(r_1^{x,y})
			\Phi_2(r_2^{x,y})	 \ell(r_3^{x,y})}{\Phi_1(r_1^{z,y})
			\Phi_2(r_2^{z,y})	 \ell(r_3^{z,y})} \\
		& \le c(k)M^{(\beta_1 +\beta_2)/2 }\bigg(1 \vee \frac{r_1^{z,y}}{r_1^{x,y}}\bigg)^{\beta_1/2}  
			\bigg(1\vee \frac{r_2^{z,y}}{r_2^{x,y}}\bigg)^{\beta_2/2}	\frac{\Phi_1(r_1^{x,y})
			\Phi_2(r_2^{x,y})}{\Phi_1(r_1^{z,y})
			\Phi_2(r_2^{z,y})} \\
		&\le c(k) M^{\beta_1 + \ub_1+\beta_2 + \ub_2} .
	\end{align*}
The proof of (iii) is complete.
	
(iv)  Let $x,y \in D$ and   $0<r \le  (|x-y| \wedge  \wh R)/2$. Let
	$$
		V:=\left\{z \in B_{\overline D}(x,r) : \delta_D(z) \ge \delta_D(x)/4 \right\}.
	$$
Since $D$ is a Lipschitz open set, we have
	\begin{align}\label{e:measure-density-2}
		m_d(V)  \ge c_1r^d,
	\end{align}
where $c_1>0$ is a constant independent of $x$ and $r$. Besides, note that for all $z \in V$,  
	$$
		|y-z| \le |y-x| + |x-z| \le |y-x| + r <(3/2) |y-x|
	$$
by the triangle inequality. Hence, by (iii), we get $\sB(z,y) \ge c_2 \sB(x,y)$ for all $z \in V$. Using this 
and \eqref{e:measure-density-2}, we arrive at
	\begin{align*}
		\frac{1}{r^d}\int_{	B_{\overline D}(x,r)} \sB(z,y)dz \ge 	\frac{1}{r^d}\int_{V} \sB(z,y)dz \ge c_1c_2\sB(x,y).
	\end{align*} 
\end{proof}

The next result is Carleson’s estimate for $Y$, which is a usual step 
 in proving the boundary Harnack principle.

\begin{thm}\label{t:Carleson}\!{\rm(Carleson's estimate)}	 
Suppose that \hyperlink{B1}{{\bf (B1)}},  \hyperlink{B3}{{\bf (B3)}}, \hyperlink{B4-c}{{\bf (B4-c)}}, \hyperlink{K3}{{\bf (K3)}} and  \hyperlink{B5}{{\bf (B5)}} hold.
Let  $p \in [(\alpha-1)_+, \alpha+\beta_1) \cap (0,\infty)$ denote the constant satisfying \eqref{e:C(alpha,p,F)} if $C_9>0$ and let $p=\alpha-1$ if $C_9=0$ where $C_9$ is the contant in \hyperlink{K3}{{\bf (K3)}}.
Then there exists  $C\ge1$ such that for any $Q \in \partial D$, $r \in (0, \wh R]$ and any non-negative Borel function $f$ on $D$ which is harmonic in $B_D(Q,r)$ with respect to $Y$ and vanishes continuously on $\partial D \cap B(Q,r)$, we have
	\begin{align}\label{e:Carleson}
		f(x) \le C f(z_0) \quad \text{for all} \;\, x \in B_D(Q,r/2),
	\end{align}
where $z_0 \in B_D(Q,8r/9)$ is any point with $\delta_D(z_0) \ge r/8$.
\end{thm}
\begin{proof} 
By using  Lemma \ref{l:A3-c},  Theorem \ref{t:phi2} and  Corollary \ref{c:Carleson-1}, the assertion
can be proved by arguments similar to that for \cite[Theorem 1.2]{KSV}. We give the details for completeness. 

Let $Q \in \partial D$, $r \in (0, \wh R]$, $z_0 \in B_D(Q,8r/9)$ with $\delta_D(z_0) \ge r/8$ 
and let $f$ be a non-negative Borel function on $D$ which is harmonic in $B_D(Q,r)$ and vanishes continuously 
on $\partial D \cap B(Q,r)$.  
 Recall that $\epsilon_2 \in (0,1/12)$ 
 is  the constant in Theorem \ref{t:Dynkin-improve}.  
Note that  \hyperlink{IUBS}{{\bf (IUBS)}} holds  by Lemma \ref{l:A3-c}(iv). Hence, by Theorem \ref{t:phi2} and a standard chain argument, it suffices to prove \eqref{e:Carleson} for 
 $x\in B_D(Q, \epsilon_2r / (48K_0))$ where 
   $K_0 >4$ is the constant in 
 Corollary \ref{c:Carleson-1}. 
  Moreover, we also deduce from Theorem \ref{t:phi2} and a standard chain argument that there exist  constants $c_1,\gamma>0$ independent of $Q,r,f$ and $z_0$ such that 
	\begin{align}\label{e:Carleson-chain}
 		f(x) \le  c_1 (\delta_D(x)/r)^{-\gamma} f(z_0) \quad \text{for all} \;\, x\in B_D(Q, \epsilon_2r / (24K_0)).
	\end{align} 
In the following, the constants  $c_i$ are always independent of $Q,r, f$ and $z_0$.

Set $\theta:= \beta_1+\ub_1+\beta_2+\ub_2$ and $\lambda:= \alpha/(d+\alpha+\theta)$. Define
	\begin{align*}
		U_1:=B(z_0, \delta_D(z_0)/8), \qquad U_2:=B(z_0, \delta_D(z_0)/4)
	\end{align*}
and for $x\in B_D(Q, \epsilon_2r / (12K_0))$,
	\begin{align*}
		V_1(x):=B_D(x, (2K_0+1)\delta_D(x)), \qquad V_2(x):= B_D(x, (4K_0+2)r^{1-\lambda} \delta_D(x)^\lambda).
	\end{align*}
First note that,  since for all $w \in U_1$, 
	$$
		|w-Q| \le |w-z_0|+|z_0-Q| < \delta_D(z_0)/8+ |z_0-Q| \le 9 |z_0-Q| /8,
	$$
we have $U_1 \subset B_D(Q,r)$.
For all  $x\in B_D(Q, \epsilon_2r / (12K_0))$, since $\delta_D(x) <\epsilon_2 r/(12K_0)$, 
we have $V_1(x) \subset V_2(x) \cap B_D(Q,r)$.
Further, by Corollary \ref{c:Carleson-1}, it holds that
	\begin{align}\label{e:Carleson-lifetime}
		\P_x (\tau_{V_1(x)}=\zeta) \ge 1/2 \quad \text{for all} \;\, x\in B_D(Q, \epsilon_2r / (24K_0)).
	\end{align}

Pick any $x \in B_D(Q, \epsilon_2r / (24K_0))$.  Since $V_1(x) \subset B_D(Q,r)$, by the harmonicity of $f$,  we have
	\begin{align*}
		f(x) &= \E_x \left[ f(Y_{\tau_{V_1(x)}}) ; Y_{\tau_{V_1(x)}} \in V_2(x)\right] 
			+ \E_x \left[ f(Y_{\tau_{V_1(x)}}) ; Y_{\tau_{V_1(x)}} \in D\setminus V_2(x)	\right]\\
		& =:I_1+I_2.
	\end{align*}
By \eqref{e:Carleson-lifetime}, 
	\begin{align}\label{e:Carleson-I1}
		I_1 \le  \Big( \sup_{y \in V_2(x)} f(y)\Big)\,\P_x ( Y_{\tau_{V_1(x)}} \in V_2(x)) \le 2^{-1} \sup_{y \in V_2(x)} f(y).
	\end{align}
Observe that for all $w \in V_1(x)$ and $y \in D \setminus V_2(x)$, 
	\begin{align}\label{e:Carleson-dist}
		\delta_D(w) \le (2K_0+2) \delta_D(x) \quad \text{ and } \quad |w-y| \ge |x-y| - |x-w| \ge |x-y|/2.
	\end{align}
Thus, by Lemma \ref{l:A3-c}(iii), $\sB(w,y) \le c_2 \sB(x,y)$ for all $w \in V_1(x)$ and $y \in D\setminus V_2(x)$. Using this and the second inequality in \eqref{e:Carleson-dist} in the second line below, and  Proposition \ref{p:E2} in the third,  we obtain
	\begin{align}\label{e:Carleson-I2-0}
	\begin{split}
		I_2 &= \E_x \bigg[ \int_0^{\tau_{V_1(x)}} \int_{D\setminus V_2(x)} 
			\frac{f(y)\,\sB(Y_s,y)}{|Y_s-y|^{d+\alpha}} dy\,  ds \bigg]\\
		&	\le c_3\E_x[\tau_{V_1(x)}] \, \int_{D\setminus V_2(x)} \frac{f(y)\,\sB(x,y)}{|x-y|^{d+\alpha}} dy\\
		&\le c_4  \delta_D(x)^\alpha \bigg[\int_{(D\setminus V_2(x)) \cap U_2} \frac{f(y)\,\sB(x,y)}{|x-y|^{d+\alpha}} dy 
			+ \int_{(D\setminus V_2(x)) \cap U_2^c} \frac{f(y)\,\sB(x,y)}{|x-y|^{d+\alpha}} dy \bigg]\\
		&=:c_4\delta_D(x)^\alpha(I_{2,1}+I_{2,2}).
	\end{split}
	\end{align}
Here, we used the L\'evy system formula \eqref{e:Levysystem-Y-kappa} in the first line. Using the triangle inequality, we see that for all $y \in U_2$,
	$$
		|x-y| \ge |z_0-Q| - |Q-x| - |z_0-y| \ge 3\delta_D(z_0)/4 - \epsilon_2 r/(24K_0) \ge r/16.
	$$
Further, by Theorem \ref{t:phi2}, we get $f(y) \le c_5 f(z_0)$ for all $y \in U_2$. Thus, since $\sB$ is bounded, we obtain
	\begin{align}\label{e:Carleson-I2-1}
		I_{2,1} \le c_6 f(z_0) \int_{U_2} \frac{dy}{|x-y|^{d+\alpha}} \le c_6 f(z_0) \int_{B(x, r/16)^c} \frac{dy}{|x-y|^{d+\alpha}}
			 \le c_7r^{-\alpha} f(z_0).
	\end{align}
For $I_{2,2}$, we observe that for all $y \in D \setminus V_2(x)$,
	\begin{align*}
		&|z_0-y| \le |x-y|+ |z_0-Q|+|x-Q| \\
		&\le |x-y|  + 2r \le  (1+r^{\lambda} \delta_D(x)^{-\lambda})|x-y| \le 2r^{\lambda} \delta_D(x)^{-\lambda}|x-y|.
	\end{align*}
Thus, since $\delta_D(x)<\delta_D(z_0)$,  by Lemma \ref{l:A3-c}(iii), we have
	$$
		\sB(x,y) \le c_8  (2r^{\lambda} \delta_D(x)^{-\lambda})^{\theta} \sB(z_0, y) \quad \text{for all} \;\, 
			y \in D \setminus V_2(x).
	$$
Using the  two displays above, since $f$ is non-negative, we get
	\begin{align}\label{e:Carleson-I2-2}
		I_{2,2} \le  c_9(r^\lambda \delta_D(x)^{-\lambda})^{d+\alpha+\theta} \int_{D \setminus U_2} 
			\frac{f(y)\,\sB(z_0,y)}{|z_0-y|^{d+\alpha}} dy .
	\end{align}
Besides, using the harmonicity of $f$ on $U_1 \subset B_D(Q,r)$, and the fact $f \ge 0$ in the first line below,  
the L\'evy system formula \eqref{e:Levysystem-Y-kappa} in the second,  Lemma \ref{l:A3-c}(ii) and the fact that 
$|w-y| \le |z_0-w| + |z_0-y|\le  2|z_0-y|$ for all $ w \in U_1$ and $y \in D\setminus U_2$ 
in the third, and Proposition \ref{p:E2} in the last, we get
	\begin{align*}
		f(z_0) &\ge \E_{z_0} \left[ f(Y_{\tau_{U_1}}) ; 
				Y_{\tau_{U_1}} \in D\setminus U_2 \right]\\
			&= \E_{z_0} \bigg[ \int_0^{\tau_{U_1}} \int_{D\setminus U_2} \frac{f(y)\,\sB(Y_s,y)}{|Y_s-y|^{d+\alpha}} dy\,  ds \bigg]\\
			&\ge c_{10} \E_{z_0}[\tau_{U_1}] \, \int_0^{\tau_{U_1}} \int_{D\setminus U_2} 
				\frac{f(y)\,\sB(z_0,y)}{|z_0-y|^{d+\alpha}} dy\\
			&\ge c_{11} r^{\alpha} \int_0^{\tau_{U_1}} \int_{D\setminus U_2} \frac{f(y)\,\sB(z_0,y)}{|z_0-y|^{d+\alpha}} dy.
	\end{align*} 
Hence, we deduce from \eqref{e:Carleson-I2-2} that 
	\begin{align}\label{e:Carleson-I2-3}
		I_{2,2} \le c_{12} r^{-\alpha }   (r^\lambda \delta_D(x)^{-\lambda})^{d+\alpha+\theta}  f(z_0) 
			\le c_{13} \delta_D(x)^{-\alpha}  f(z_0).
	\end{align}
Combining \eqref{e:Carleson-I1}, \eqref{e:Carleson-I2-1} and \eqref{e:Carleson-I2-3}, since $\delta_D(x)< \epsilon_2r/(24K_0)$, we arrive at
	\begin{align}\label{e:Carleson-induction}
		f(x) \le 2^{-1} \sup_{y \in V_2(x)} f(y) + c_{14} f(z_0)  \quad \text{for all} \;\, x\in B_D(Q, \epsilon_2r / (24K_0)).
	\end{align}

Now we prove that \eqref{e:Carleson} holds for all $x\in B_D(Q, \epsilon_2r / (48K_0))$ with 
	$$
		C=M:=3c_{14} + c_1 \bigg(\frac{24K_0(4K_0+2)}{a_0\epsilon_2}\bigg)^{\gamma/\lambda},
	$$
where
 	$$
 		a_0:=2^{-1}\bigg(\sum_{n=0}^\infty \bigg(\frac{3}{4}\bigg)^{n\lambda/ \gamma} \bigg)^{-1}.
 	$$
Suppose this fails. Then there exists $x_1 \in  B_D(Q, \epsilon_2r / (48K_0))$  such that  $f(x_1) >Mf(z_0)$.  In the following, we  construct a sequence  $(x_n)_{n \ge 2}$ in  $B_D(Q, \epsilon_2r / (24K_0))$  such that for all $n \ge 2$, 
	\begin{align}\label{e:Carleson-claim}
		& |x_{n}-x_{n-1}|<\frac{a_0\epsilon_2 r}{24K_0}\bigg(\frac{3}{4}\bigg)^{(n-2)\lambda/\gamma} \;\;\text{ and } \;\; f(x_{n}) 
			\ge f(x_1)\bigg(\frac{4}{3}\bigg)^{n-1}.
 	\end{align}
This leads to a contradiction since $f$ is bounded. 

By \eqref{e:Carleson-induction},  since $M>3c_{14}$, there exists $x_2 \in V_2(x_1)$ such that 
 	$$
 		f(x_2) \ge 2(f(x_1) - c_{14} f(z_0)) \ge (4/3) f(x_1).
 	$$
Note that 
$\delta_D(x_1) \le c_1^{1/\gamma} r (f(z_0)/f(x_1))^{1/\gamma} <  (c_1/M)^{1/\gamma}r$ 
by \eqref{e:Carleson-chain}. Thus, we have
 	\begin{align*}
 		|x_2-x_1| <(4K_0+2) r^{1-\lambda} \delta_D(x_1)^{\lambda} < (4K_0+2)(c_1/M)^{\lambda/\gamma}
 			r <a_0\epsilon_2 r/(24K_0)
 	\end{align*}
so that
	$$
 		|x_2-Q|\le |x_1-Q| + |x_1-x_2| < \epsilon_2 r/ (24K_0)
 	$$
by the triangle inequality. Hence, $x_2 \in B_D(Q,\epsilon_2 r/ (24K_0))$ and \eqref{e:Carleson-claim}  holds for $n=2$. Next, assume that $x_n \in B_D(Q,\epsilon_2 r/ (24K_0))$, $1\le n \le k$, are chosen to satisfy \eqref{e:Carleson-claim} for all $1\le n \le k$, for some $k \ge 2$. By \eqref{e:Carleson-induction}, since $f(x_k)\ge f(x_1) >Mf(z_0)$, there exists $x_{k+1} \in V_2(x_k)$ such that 
 	$$
 		f(x_{k+1}) \ge 2(f(x_k) - c_{14} f(z_0)) \ge (4/3) f(x_k).
 	$$
Since 
	$$
		\frac{\delta_D(x_k)}{r} \le c_1^{1/\gamma}  \left(\frac{f(z_0)}{f(x_k)}\right)^{1/\gamma} 
		\le  c_1^{1/\gamma}  \left(\frac{f(z_0)}{f(x_1) (4/3)^{k-1} }\right)^{1/\gamma}   
		<  (c_1/M)^{1/\gamma} \bigg(\frac{3}{4}\bigg)^{(k-1)\lambda/\gamma}
	$$ 
by \eqref{e:Carleson-chain} and the induction hypothesis, we have
 	\begin{align*}
 		&|x_{k+1}-x_k|<(4K_0+2)r\bigg(\frac{\delta_D(x_k)}{r}\bigg)^{\lambda}\\
 		&< (4K_0+2) (c_1/M)^{\lambda/\gamma} r  \bigg(\frac{3}{4}\bigg)^{(k-1)\lambda/\gamma}
 			< \frac{a_0\epsilon_2 r}{24K_0}\bigg(\frac{3}{4}\bigg)^{(k-1)\lambda/\gamma}.
	 \end{align*}
Using this and the induction hypothesis, we get
	\begin{align*}
		&	|x_{k+1}-Q| \le
		 |x_1-Q| + \sum_{n=2}^{k+1} |x_{n}-x_{n-1}|\\ &< \frac{\epsilon_2 r}{24K_0}\bigg( \frac{1}{2} 
			+ a_0\sum_{n=2}^{k+1} \bigg(\frac{3}{4}\bigg)^{(n-2)\lambda/\gamma} \bigg)  <  \frac{\epsilon_2 r}{24K_0}.
	\end{align*}
Therefore, $x_{k+1} \in B_D(Q, \epsilon_2r/(24K_0))$ and we deduce that \eqref{e:Carleson-claim} holds for all $n$ by the induction. The proof is complete. 
\end{proof}

The above Theorem \ref{t:Carleson} will be used 
in the proof of the next theorem which is  our first main result -- the boundary Harnack principle.

\begin{thm}\label{t:BHPnew} \!{\rm(Boundary Harnack principle)} 
Suppose that \hyperlink{B1}{{\bf (B1)}}, \hyperlink{B3}{{\bf (B3)}}, \hyperlink{B4-c}{{\bf (B4-c)}},  \hyperlink{K3}{{\bf (K3)}} and  \hyperlink{B5}{{\bf (B5)}} hold. Suppose also that $p<\alpha+(\beta_1\wedge \beta_2)$. 
Here   $p \in [(\alpha-1)_+, \alpha+\beta_1) \cap (0,\infty)$ denotes the constant satisfying \eqref{e:C(alpha,p,F)} if $C_9>0$ and  $p=\alpha-1$ if $C_9=0$ where $C_9$ is the contant in \hyperlink{K3}{{\bf (K3)}}. 
Then  for any $Q \in \partial D$, $0<r \le \wh R$, and any non-negative Borel function $f$ in $D$ which is harmonic in $B_D(Q,r)$ with respect to $Y$ and vanishes continuously on $\partial D\cap B(Q,r)$, 	we have 
	\begin{equation}\label{e:TAMSe1.8new}
		\frac{f(x)}{\delta_D(x)^p}\asymp 	\frac{f(y)}{\delta_D(y)^p} \quad \text{for} \;\, x,y\in B_D(Q,r/2),
	\end{equation}
where the comparison constants are independent of $Q,r$ and $f$,  and depend on $D$ only through $\wh R$ and $\Lambda_0$. 
\end{thm}

\smallskip

It is worth mentioning that given any (large) $p>(\alpha-1)_+$, 
 there exist $\sB(x,y)$ and $\kappa(x)$ 
such that the BHP holds with  decay rate $\delta_D(x)^p$ for the operator $L$ in \eqref{e:def-operator}.

For $Q \in \partial D$, $0<r\le \wh R/8$ and $y \in D$, define 
	\begin{align}\label{e:k}
		k_r(y)=\frac{1}{|y-Q|^{d+\alpha} }
		\Phi_1\bigg(\frac{r \wedge \delta_D(y)}{|y-Q|}\bigg)\,
		\Phi_2\bigg(\frac{r \vee \delta_D(y)}{|y-Q|}\bigg)\,
		\ell\bigg(\frac{r\wedge \delta_D(y)}{r\vee\delta_D(y)}\bigg).
	\end{align}

	In the following  lemma,  we compare the above function $k_r$ with the jump kernel in certain regions 
	that appear
	in the proof of Theorem \ref{t:BHPnew}.

\begin{lemma}\label{l:estimates-of-J-for-BHP}
Let $Q \in \partial D$ and $0<r\le \wh R/8$.
	
\noindent (i) There exists $C>0$ independent of $Q$ and $r$ such that for all $z \in U(2^{-1}r) \setminus U(2^{-1}r, 2^{-3}r)$ 
and   $y \in  D\setminus U(r)$, 
	\begin{align*}
		\sB(z,y)|z-y|^{-d-\alpha} \ge  Ck_r(y).
	\end{align*}
	
\noindent (ii) Let $\eps \in ( (\beta_1  -\beta_2)_+,\infty)$. There exists $C=C(\eps)>0$ independent of $Q$ and $r$ such that  for all $z\in U(2^{-1}r)$ and $y \in D\setminus U(r)$, 
	\begin{equation*}
		\sB(z,y)|z-y|^{-d-\alpha}  \le  C (\delta_D(z)/r)^{\beta_1-\eps} k_r(y).
	\end{equation*}
\end{lemma}
\begin{proof} 
In this proof, we use the coordinate system CS$_Q$,  and write $U(r)$ for $U^Q(r)$.  By  \eqref{e:U-rho-C11-1},  
$U(2^{-1}r) \subset B(0,r)$ and $B(0,2^{-1}r) \subset U(r)$. Thus, for $z\in U(2^{-1}r)$ and $y\in D\setminus U(r)$, we have  
$ \delta_D(z) \vee \delta_D(y) \le (r/2) \vee |y|  =|y|$,
$|z-y|\le |y|+|z|<2|y|$ and
	$$ 
		|z-y| \ge (1/3)(|y|-|z|)  + (2/3)|z-y|  > (1/3)(|y|-r) + (1/3)r = |y|/3.  
	$$
Therefore,  by \hyperlink{B4-c}{{\bf (B4-c)}} and the scaling properties of $\Phi_1,\Phi_2$ and $\ell$, it holds that  
for any $z\in U(2^{-1}r)$ and $y\in D\setminus U(r)$,
	\begin{align}\label{e:flb}
		\frac{\sB(z,y)}{|z-y|^{d+\alpha}}  \asymp \frac{1}{|y|^{d+\alpha}}	\Phi_1\bigg(\frac{\delta_D(z) \wedge \delta_D(y)}{|y|}\bigg) 			\Phi_2\bigg(\frac{\delta_D(z) \vee  \delta_D(y)}{|y|}\bigg)	
			\ell\bigg(\frac{\delta_D(z) \wedge \delta_D(y)}{\delta_D(z) \vee  \delta_D(y)}\bigg).
	\end{align}
(i)  Observe that 
	\begin{align}\label{e:flb-1}
		r/\sqrt{80}\le 	\delta_D(z) \le r/2 \quad \text{for all} \;\,z \in U(2^{-1}r) \setminus U(2^{-1}r,2^{-3}r). 
	\end{align}
Indeed, the second inequality in \eqref{e:flb-1} is clear. Besides, using \eqref{e:U-rho-C11-2}, we get  
$\delta_D(z) \ge (2/\sqrt 5) \rho_D(z) \ge r/\sqrt{80}$. Hence, \eqref{e:flb-1} holds.  Now the result follows from \eqref{e:flb}, \eqref{e:flb-1} and the scaling properties of $\Phi_1,\Phi_2$ and $\ell$.

\smallskip
\noindent
(ii) Let $\eps \in ( (\beta_1-\beta_2)_+,\infty)$, $z \in U(2^{-1}r)$, $y \in D \setminus U(r)$ and
	$$
 		 I:=\sB(z,y)|z-y|^{-d-\alpha}/k_r(y).
	$$
Choose a constant $\lambda \in (0,1/2)$ such that $(1-2\lambda) \eps \ge \beta_1-\beta_2$. Note that $\delta_D(z)\le r/2$.  There are four cases. 

\smallskip

\noindent Case 1: $\delta_D(y)<\delta_D(z)$. Since  $\beta_2- 2 \lambda\eps \ge\beta_1-\eps$, by \eqref{e:flb}, 
\eqref{e:scale2} and the upper scaling property of $\ell$  in \eqref{e:scale3} (with $\eps$ replaced by  $\lambda\eps$), we have
	\begin{align*}
		 I &\le c \bigg(\frac{ \delta_D(z) \vee \delta_D(y)}{r \vee \delta_D(y) }\bigg)^{\beta_2- \lambda \eps  } 
		 	\bigg(\frac{\delta_D(y)/\delta_D(z)}{\delta_D(y)/ r}\bigg)^{\lambda\eps}\\
		&=  c(\delta_D(z)/r)^{\beta_2-2\lambda\eps} \le  c(\delta_D(z)/r)^{\beta_1-\eps}.
	\end{align*}

\noindent Case 2:  $\delta_D(z) \le \delta_D(y) <(r\delta_D(z))^{1/2}$. Since $-\beta_1+\beta_2+(1-2\lambda)\eps\ge 0$, using \eqref{e:flb}, \eqref{e:scale1new}, \eqref{e:scale2} and the upper scaling property of $\ell$  in \eqref{e:scale3} (with $\eps$ replaced by  $\lambda\eps$), we get
	\begin{align*}
		I	& \le c \bigg(\frac{\delta_D(z) \wedge \delta_D(y)}{ r \wedge \delta_D(y)}\bigg)^{\beta_1-\lambda \eps} 
				\bigg(\frac{ \delta_D(z) \vee \delta_D(y)}{r \vee \delta_D(y) }\bigg)^{\beta_2-\lambda \eps } 
				\bigg(\frac{\delta_D(z)/\delta_D(y)}{\delta_D(y)/ r}\bigg)^{\lambda\eps}\\
			&= c (\delta_D(z)/r)^{\beta_1} (\delta_D(y)/r)^{-\beta_1+\beta_2-2\lambda \eps } \\
			& \le  c (\delta_D(z)/r)^{\beta_1-\eps} (\delta_D(y)/r)^{-\beta_1+\beta_2+(1-2\lambda) \eps } \\
			&\le  c(\delta_D(z)/r)^{\beta_1-\eps}.
	\end{align*}

\noindent Case 3: $(r\delta_D(z))^{1/2}\le \delta_D(y)<r$.  Since   $-\beta_1+\beta_2+\eps \ge 0$, we get from \eqref{e:flb}, \eqref{e:scale1new}, \eqref{e:scale2} and the lower scaling property of $\ell$ in \eqref{e:scale3} (with $\eps$ replaced by  
$\eps/2$) that
	\begin{align*}
		I	&\le c \bigg(\frac{\delta_D(z)\wedge \delta_D(y)}{r \wedge \delta_D(y)}\bigg)^{\beta_1-\eps/2} 
			\bigg(\frac{\delta_D(z)\vee \delta_D(y)}{r \vee \delta_D(y)}\bigg)^{\beta_2-\eps/2} 
			\bigg(\frac{\delta_D(z)/\delta_D(y)}{\delta_D(y)/ r}\bigg)^{-\eps/2}\\
			&= c(\delta_D(z)/r)^{\beta_1-\eps}( \delta_D(y)/r)^{-\beta_1 + \beta_2 + \eps} \\
			& \le c(\delta_D(z)/r)^{\beta_1-\eps}.
	\end{align*}

\noindent Case 4: $\delta_D(y) \ge r$. 
By  \eqref{e:flb}, \eqref{e:scale1new}  and the lower scaling property of $\ell$ in  \eqref{e:scale3}  (with $\eps$ replaced by  
$\eps/2$), we obtain
	\begin{align*}
		I& \le c  \bigg(\frac{\delta_D(z)\wedge \delta_D(y)}{r\wedge \delta_D(y)}\bigg)^{\beta_1-\eps/2}
			\bigg(\frac{\delta_D(z)/\delta_D(y)}{r/\delta_D(y)}\bigg)^{-\eps/2} = c (\delta_D(z)/r)^{\beta_1-\eps} .
	\end{align*}

The proof is complete.
\end{proof}

We now give the proof of Theorem \ref{t:BHPnew}.
Proposition \ref{p:bound-for-integral-new} will play an important role in the proof.

\textsc{Proof of Theorem \ref{t:BHPnew}}. 
We use the coordinate system CS$_Q$ in this proof,  and write $U(r)$ for $U^Q(r)$. 
 Recall that $\epsilon_2 \in (0,1/12)$ 
 is  the constant in Theorem \ref{t:Dynkin-improve}.   
Recall from Lemma \ref{l:A3-c}(iv) that \hyperlink{IUBS}{{\bf (IUBS)}} holds under \hyperlink{B4-c}{{\bf (B4-c)}}. Hence, by Theorem \ref{t:phi2} and a standard chain argument, it suffices to prove \eqref{e:TAMSe1.8new}
for  $x, y\in B_D(Q,2^{-10} \epsilon_2r)$. 

Let $x\in B_D(Q,2^{-10} \epsilon_2r)$ and set $z_0:=(\widetilde{0}, 2^{-5}r)$.  Using Theorem \ref{t:phi2} and a chain argument, 
we see  that there exists   $c_1>0$ independent of $Q,r$ and $f$ such that
	\begin{align}\label{e:BHP-chain}
		f(z) \ge c_1f(z_0) \quad\text{for all} \;\, z \in B(z_0, (2^{-10} -  2^{-15})^{1/2}r ).
	\end{align} 
Note that for all $w=(\wt w,w_d) \in U(2^{-7}r) \setminus U(2^{-7}r, 2^{-8}r)$, we have $|\wt w|<2^{-7}r$, so by \eqref{e:Psi-bound},
$w_d = \rho_D(w) +  \Psi(\wt w)  <(2^{-7}+2^{-14})r$ and  $w_d > (2^{-8} - 2^{-14})r$. Thus, 
	$$
	|z_0-w|^2 = |\wt w|^2  + (2^{-5}r-w_d)^2< (2^{-14} + (2^{-5} - 2^{-9})^2)r^2<(2^{-10}-2^{-15})r^2.
	$$
Hence, $U(2^{-7}r) \setminus U(2^{-7}r, 2^{-8}r) \subset B(z_0, (2^{-10} - 2^{-15})^{1/2}r )$.
Using this,  \eqref{e:BHP-chain} and Theorem \ref{t:Dynkin-improve}, since  $f$ is harmonic in $B_D(Q,r)$,  we obtain
	\begin{align}\label{e:TAMSe6.37-new}
		f(x)&=\E_{x}\big[f(Y_{\tau_{U(2^{-7}\epsilon_2 r)}})\big]\\
			&\ge \E_{x}\big[f(Y_{\tau_{U(2^{-7}\epsilon_2 r)}}); Y_{\tau_{U(2^{-7}\epsilon_2 r)}}\in  U(2^{-7}r) 
				\setminus U(2^{-7}r, 2^{-8}r)\big]\nonumber\\
		&\ge c_1f(z_0)\P_{x}\big( Y_{\tau_{U(2^{-7}\epsilon_2 r)}}\in  U(2^{-7}r) \setminus U(2^{-7}r, 2^{-8}r)\big)\nn\\
		& \ge c_2 (\delta_D(x)/r)^p f(z_0).\nn
	\end{align}

On the other hand, using the harmonicity of $f$ again, we see that
	\begin{align*}
		f(x)&= \E_x\big[f( Y_{\tau_{
				U(2^{-7}\epsilon_2 r)}});  Y_{\tau_{
				U(2^{-7}\epsilon_2 r)}}\in  U(2^{-3}r)\big]  \\
			&\quad + \E_x\big[f( Y_{\tau_{
				U(2^{-7}\epsilon_2 r)}});  Y_{\tau_{
				U(2^{-7}\epsilon_2 r)}}  \in D\setminus U(2^{-3}r)\big]\\
			&=:I_1+I_2.
	\end{align*}
Since $U(2^{-3}r) \subset B(0,2^{-2}r)$, by Theorem \ref{t:Carleson}  (with $r$ replaced by $2^{-2}r$), we have $f(z)\le c_3f(z_0)$  for all $z \in U(2^{-3}r)$. Thus, using Theorem \ref{t:Dynkin-improve}, we get that
	\begin{align}\label{e:BHP-1}
		I_1&\le  c_3f(z_0)\P_{x}( Y_{\tau_{U(2^{-7}\epsilon_2 r)}}\in D) \le c_4 (\delta_D(x)/r)^p f(z_0).
	\end{align}
Now we estimate $I_2$. Let $k_{2^{-3}r}$ be the function defined in \eqref{e:k}. Note that for all 
$w=(\wt w,w_d) \in B(z_0, 2^{-7}r)$, $|\wt w|<2^{-7}r$, $\rho_D(w)<(2^{-5}+2^{-7})r$ and 
$\rho_D(w)>w_d  -\wh R^{-1}|\wt w|^2>(2^{-5}-2^{-7}- 2^{-14})r$ by \eqref{e:Psi-bound}. 
Hence, $B(z_0, 2^{-7}r) \subset U(2^{-4}r) \setminus U(2^{-4}r, 2^{-6}r)$. Using this,  the harmonicity  of $f$, 
the L\'evy system formula \eqref{e:Levysystem-Y-kappa}, Lemma \ref{l:estimates-of-J-for-BHP}(i)  (with $r$ replaced by $2^{-3}r$) and Proposition \ref{p:E2},  we get
	\begin{align}\label{e:POTAe7.27}
	\begin{split}
		f(z_0)&\ge \E_{z_0} \big[f(Y_{\tau_{U(2^{-4}r) \setminus U(2^{-4}r, 2^{-6}r)}}); 
			Y_{\tau_{U(2^{-4}r) \setminus U(2^{-4}r, 2^{-6}r)}} \in D \setminus U(2^{-3}r)	\big]\\
		&= \E_{z_0}\int^{\tau_{U(2^{-4}r) \setminus U(2^{-4}r, 2^{-6}r)}	}_0\int_{D\setminus U(2^{-3}r)}	
			\frac{\sB(Y_t,w)}{|Y_t-w|^{d+\alpha}} f(w)dw\,dt\\
		&\ge c_5 \E_{z_0}\tau_{B(z_0, 2^{-7}r)} \int_{D\setminus U(2^{-3}r)}k_{2^{-3}r}(w)f(w)dw\\
		&\ge c_6r^{\alpha}	\int_{D\setminus U(2^{-3}r)}k_{2^{-3}r}(w)f(w)dw.
	\end{split}
	\end{align} 
Using the assumption $p<\alpha+(\beta_1\wedge \beta_2)$,  we see that $(\beta_1-\beta_2)_+< \beta_1-p+\alpha$. 
We now choose a positive constant $\eps \in ((\beta_1-\beta_2)_+, \beta_1-p+\alpha)$ so that  $p-\alpha<\beta_1-\eps$. 
By Lemma \ref{l:estimates-of-J-for-BHP}(ii) (with $r$ replaced by $2^{-3}r$),  Proposition \ref{p:bound-for-integral-new} and \eqref{e:POTAe7.27}, we have
	\begin{align}\label{e:BHP-2}
		I_2	&= \E_{x}\int^{\tau_{U(2^{-7}\epsilon_2r)}	}_0\int_{D\setminus U(2^{-3}r)}	
			\frac{\sB(Y_t,w)}{|Y_t-w|^{d+\alpha}} f(w)dw\,dt\\	
		&\le c_{7} r^{-(\beta_1-\eps)}\E_{x}\int^{\tau_{U(2^{-7}\epsilon_2r)}	}_0\delta_D(Y_t)^{\beta_1-\eps}dt 
			\int_{D\setminus U(2^{-3}r)}k_{2^{-3}r}(w)f(w)dw\nn\\
		& \le c_8 r^{-(\beta_1-\eps)} r^{\alpha+\beta_1-\eps-p} \delta_D(x)^p  r^{-\alpha}f(z_0) = c_8(\delta_D(x)/r)^p f(z_0).\nn
	\end{align}

Combining \eqref{e:TAMSe6.37-new} with \eqref{e:BHP-1} and \eqref{e:BHP-2}, we arrive at $f(x) \asymp (\delta_D(x)/r)^p f(z_0)$ which implies the conclusion of the theorem.  
\qed

\bigskip

	The result of  Theorem \ref{t:BHPnew} implies the following statement: There exists $C>0$ such that for any $Q \in \partial D$ and $0<r \le \wh R$, whenever two Borel functions $f, g$ in $D$ are harmonic in $B_D(Q,r)$ with respect to $Y$ and vanish continuously on $\partial D \cap B(Q,r)$,
	\begin{align}\label{e:def-BHP}
		\frac{f(x)}{f(y)}\,\le C\,\frac{g(x)}{g(y)}\quad \text{for all} \;\, x,y\in B_D(Q,r/2).
	\end{align}
	 The inequality \eqref{e:def-BHP} is referred to as the \textit{scale-invariant boundary Harnack principle} for $Y$.

	We say that the  \textit{inhomogeneous non-scale-invariant boundary Harnack principle}  holds for $Y$, if there is a constant $r_0\in (0,\wh R]$  such that for any $Q \in \partial D$ and 
	$0<r\le r_0$, there exists  a constant $C=C(Q,r)\ge 1$ such that \eqref{e:def-BHP} holds for any two Borel functions $f, g$ in $D$ which are harmonic in $B_D(Q,r)$ with respect to $Y$ and vanish continuously on $\partial D \cap B(Q,r)$.
	
	We will show that without the extra condition $p < \alpha + (\beta_1 \wedge \beta_2)$ in Theorem \ref{t:BHPnew}, even inhomogeneous non-scale-invariant BHP may not hold for $Y$. In the remainder of this section, we assume that  \hyperlink{B1}{{\bf (B1)}}, \hyperlink{B3}{{\bf (B3)}}, \hyperlink{B4-c}{{\bf (B4-c)}},  \hyperlink{K3}{{\bf (K3)}} and \hyperlink{B5}{{\bf (B5)}} hold.  Consider the following condition:

	\medskip
	
	\noindent\hypertarget{F}{{\bf (F)}} For any   $0<r\le \wh R$, there exists a constant $C=C(r)$ such that
	\begin{align*}
		\liminf_{s\to 0} \frac{\Phi_2(b/r) \ell(s/b)}{\ell(s)} \ge C b^{p-\alpha} \quad \text{for all} \;\, 0<b \le r.
	\end{align*}

	\begin{thm}\label{t:BHPfail-general}
		Suppose that \hyperlink{B1}{{\bf (B1)}}, \hyperlink{B3}{{\bf (B3)}}, \hyperlink{B4-c}{{\bf (B4-c)}},  \hyperlink{K3}{{\bf (K3)}} and \hyperlink{B5}{{\bf (B5)}} hold.  Suppose also that   \hyperlink{F}{{\bf (F)}} holds.
		Then the inhomogeneous non-scale-invariant boundary Harnack principle fails for $Y$.
	\end{thm}

\begin{remark}\label{r:condition-F}
	If $p<\alpha+\beta_2$, then \hyperlink{F}{{\bf (F)}} fails. Indeed, suppose that  $\eps:=\alpha+\beta_2 - p>0$. Then using  the definition of the lower Matuszewska index and \eqref{e:scale3-2} (with $\eps$ replaced by $\eps/3$),  we get that for any $0<s<b \le r \le \wh R$,
	\begin{align*}
		\frac{\Phi_2(b/r) \ell(s/b)}{\ell(s)}  \le c_1 \Phi_2(1)  (b/r)^{\beta_2 -\eps/3}  (1/b)^{\eps/3} = c_2(r) b^{p-\alpha+\eps/3}.
	\end{align*}
	Hence, \eqref{e:condition-F-2} can not hold for all $0<b\le r$.
\end{remark}

A measurable function $f:(0,1] \to (0,\infty)$ is said to be slowly varying at zero if 
\begin{align*}
	\lim_{s \to 0} \frac{f(\lambda s)}{f(s)} = 1  \quad\;\; \text{for all $\lambda>1$}.
\end{align*}

We present two sufficient conditions for condition  \hyperlink{F}{{\bf (F)}}.

\begin{lemma}\label{l:condition-F}
	(i) If $p>\alpha+\ub_2$, then \hyperlink{F}{{\bf (F)}} holds. 
	
	\noindent (ii) If $p=\alpha+\beta_2$,  $\ell$ is  slowly varying at zero, and there exists $c_0>0$ such that  $\Phi_2(r) \ge c_0 
r^{\beta_2}$	 for all $0<r\le 1$,  then \hyperlink{F}{{\bf (F)}} holds. 
\end{lemma}
\begin{proof}
	(i) Assume that $\eps:=p-\alpha-\ub_2>0$. Using \eqref{e:scale2-2} and \eqref{e:scale3-2},  we get that for any $0<b \le r \le \wh R$,
	\begin{align*}
		\liminf_{s\to 0}	\frac{\Phi_2(b/r) \ell(s/b)}{\ell(s)}  \ge c_1 \Phi_2(1)  (b/r)^{\ub_2}  b^{\eps} = c_2(r) b^{p-\alpha}.
	\end{align*}

	(ii) By the  assumptions,  we get that for any $0<b \le r \le \wh R$,
	\begin{align*}
		\liminf_{s \to0}\frac{\Phi_2(b/r) \ell(s/b)}{\ell(s)}  \ge c_0 (b/r)^{\beta_2}  	\liminf_{s \to0}	\frac{ \ell(s/b)}{\ell(s)}   =c_0 r^{-\beta_2} b^{p-\alpha}.
	\end{align*}
\end{proof} 
In particular, if  $\Phi_1$, $\Phi_2$ and $\ell$ are 
 the functions from  Remark \ref{r:compare-KSV}, then \hyperlink{F}{{\bf (F)}} holds if $\beta_1 > \beta_2$ and $p \in [\alpha+\beta_2, \alpha+\beta_1)$.
	
	To prove Theorem \ref{t:BHPfail-general}, we first establish the following lemma. 
	\begin{lemma}\label{l:BHPfail}
		Suppose that \hyperlink{B1}{{\bf (B1)}},  \hyperlink{B3}{{\bf (B3)}},  \hyperlink{B4-c}{{\bf (B4-c)}}, \hyperlink{K3}{{\bf (K3)}} and \hyperlink{B5}{{\bf (B5)}} hold. If the inhomogeneous non-scale-invariant boundary Harnack principle holds  for $Y$ with $r_0 \in (0, \wh R]$, then the following  is true:  For any $Q \in \partial D$ and $0<r\le r_0 \wedge (\epsilon_2\wh R/288)$, there exists $C=C(Q,r)\ge 1$ such that for any non-negative Borel function $f$ in $D$ which is harmonic  in $ B_D(Q, r)$ and vanishes continuously on $ \partial D \cap B(Q, r)$, 
		\begin{equation*}
			\frac{f(x)}{f(y)} \le C\bigg( \frac{\delta_D(x)}{\delta_D(y)}\bigg)^p\quad  \hbox{for all } x, y\in   B_D(Q,r/2) \, \text{ with }\,  \delta_D(x) \vee \delta_D(y) \le \epsilon_2 r/8,
		\end{equation*}
		where	$\epsilon_2 \in (0,1/12)$ is the constant in Lemma \ref{l:Dynkin2}.
		
	\end{lemma}
	\begin{proof}
		Let $Q \in \partial D$ and $r \in (0, r_0 \wedge (\epsilon_2\wh R/288)]$. We use the coordinate system CS$_Q$ in this proof.

		Define  $g(x)= \P_x(Y_{\tau_{U(3r)}}\in D)$.  By the strong Markov property,  since $B_D(Q, r) \subset U(3r/2)$ by \eqref{e:U-rho-C11-1}, the function $g$ is harmonic in $B_D(Q,r)$. We claim that  there exists $c_1>1$ such that for all $x \in B_D(Q,r)$ with $\delta_D(x)\le \epsilon_2 r/8$,
		\begin{align}\label{e:BHP-failure-claim}
			c_1^{-1}(\delta_D(x)/r)^p\le 	g(x) \le  c_1(\delta_D(x)/r)^p.
		\end{align}
		To establish this claim, choose any  $x \in B_D(Q,r)$ with  $\delta_D(x)\le \epsilon_2 r/8$, and  let $Q_x \in \partial D$ be such that $|x-Q_x|=\delta_D(x)$. Since $\epsilon_2<1/12$, by the triangle inequality, it holds that  $|Q-Q_x|\le |Q-x|  +\delta_D(x)<(2-2\epsilon_2)r$. Hence, by   \eqref{e:U-rho-C11-1}, we have
		$$
		U^{Q_x}(\epsilon_2 r) \subset B_D(Q_x, 2\epsilon_2r) \subset B_D(Q, 2 r) \subset U(3r)
		$$
		and
		$$
		U(3r) \subset B_D(Q, 6r) \subset B_D(Q_x, 8r) \subset U^{Q_x}(12r).
		$$
		Using these and Theorem \ref{t:Dynkin-improve}, since $12r\le \epsilon_2 \wh R/24$ and  $x \in U^{Q_x}(\epsilon_2 r/4)$, we obtain
		\begin{align*}
			g(x) \le  \P_x(Y_{\tau_{U^{Q_x}(\epsilon_2 r)}}\in D) \le c (\delta_D(x)/ r)^p
		\end{align*}
		and 
		\begin{align*}
			g(x) \ge  \P_x(Y_{\tau_{U^{Q_x}(12 r)}}\in D) \ge c (\epsilon_2\delta_D(x)/ (12r))^p.
		\end{align*}
		Therefore, \eqref{e:BHP-failure-claim} holds. 	Note that \eqref{e:BHP-failure-claim} particularly implies that   $g$ vanishes continuously on $\partial D \cap B(Q,r)$.  Now, the desired result  follows  from \eqref{e:def-BHP} and \eqref{e:BHP-failure-claim}.	\end{proof}
	
	\smallskip

	\textsc{Proof of Theorem \ref{t:BHPfail-general}}. 
We suppose that  the inhomogeneous non-scale-invariant BHP holds for $Y$ with $r_0 \in (0,\wh R]$
and derive a contradiction. 
	Let $Q \in \partial D$ and $r \in (0, r_0 \wedge (\epsilon_2\wh R/288)]$. We use the coordinate system CS$_Q$ in this proof.

	Let  $P\in \partial D$ be such that $10r<|P-Q|<12r$. Using \eqref{e:U-rho-C11-1}, we see that 
	\begin{align}\label{e:BHPfail-1}
		3r<|z- y| < 20r \quad \text{for all} \;\, z \in  U(3r), \; y \in B_D(P,r).
	\end{align}
	Since $D$ is a $C^{1,1}$ open set, by  \eqref{e:scale3} (with $\eps=1/2$), we see that
	\begin{align*}
		\int_{B_D(P,r/n)} \ell (\delta_D(y))dy \le c 	\int_{B_D(P,r/n)}  \delta_D(y)^{-1/2}dy <\infty.
	\end{align*}
	For $n \ge 1$, define
	$$
	K_n:= \int_{B_D(P,r/n)} \ell (\delta_D(y))dy, \qquad 	\Xi_n(y):= \frac{r^{d+\alpha}\,{\bf 1}_{B_D(P,r/n)}(y)}{K_n\Phi_1(\delta_D(y)/(3r))} 
	$$
	and
	$$ F_n(x):=\,\E_x\big[\Xi_n(Y_{\tau_{U(3r)}} )\big].
	$$
Since $B_D(Q,r) \subset U(3r)$ by \eqref{e:U-rho-C11-1}, it 
follows
 by the strong Markov property that $F_n$ is harmonic 
in $B_D(Q,r)$ for any $n\ge 1$.
	
We first show that $F_n$ vanishes continuously on $\partial D \cap B(Q,r)$. 
Using
	 the L\'evy system formula \eqref{e:Levysystem-Y-kappa} in the first line below, and \hyperlink{B4-c}{{\bf (B4-c)}},   \eqref{e:BHPfail-1} and the scaling properties of $\Phi_1$, $\Phi_2$ and $\ell$ in the second,  we get that for all $x \in B_D(Q,r)$,
	\begin{align*}
		&F_n(x)=\frac{r^{d+\alpha}}{K_n}\E_x \bigg[ \int_{0}^{\tau_{U(3r)}}  \int_{B_D(P, r/n)}  \frac{\sB(Y_t,y)}{ \Phi_1(\delta_D(y)/(3r))\,|Y_t-y|^{d+\alpha}} dy\,dt\bigg]\\
		&\asymp \frac{1}{K_n} \,\E_x \bigg[\int_{0}^{\tau_{U(3r)}}   \int_{ B_D(P, r/n)}\,  \Phi_1\bigg(\frac{\delta_D(Y_t) \wedge \delta_D(y)}{3r}\bigg) \Phi_1\bigg(\frac{\delta_D(y)}{3r}\bigg)^{-1} \\
		&\qquad\qquad\qquad   \times \Phi_2\bigg(\frac{\delta_D(Y_t) \vee \delta_D(y)}{3r}\bigg)\ell \bigg( \frac{\delta_D(Y_t) \wedge \delta_D(y)}{\delta_D(Y_t) \vee \delta_D(y)}\bigg)\,dy\,dt \bigg]\\
		&= \frac{1}{K_n} \int_{B_D(P,r/n)}  \int_{z\in U(3r):  \delta_D(z) \le \delta_D(y)} G^{U(3r)}(x,z) \Phi_1\bigg(\frac{\delta_D(z)}{3r}\bigg)\Phi_1\bigg(\frac{\delta_D(y)}{3r}\bigg)^{-1} \\
		&\qquad\qquad\qquad  \qquad\qquad\qquad \qquad\qquad\quad  \times \Phi_2\bigg(\frac{ \delta_D(y)}{3r}\bigg)\ell \bigg( \frac{\delta_D(z)}{ \delta_D(y)}\bigg)\,dzdy\\
		&\quad +\frac{1}{K_n} \int_{B_D(P,r/n)}  \int_{z\in U(3r):  \delta_D(z) > \delta_D(y)} G^{U(3r)}(x,z)  \Phi_2\bigg(\frac{ \delta_D(z)}{3r}\bigg)\ell \bigg( \frac{\delta_D(y)}{ \delta_D(z)}\bigg)\,dzdy\\
		&=:f_{n,1}(x) + f_{n,2}(x).
	\end{align*}
	Note that $B_D(Q,2r) \subset U(3r) \subset B_D(Q,8r)$ by \eqref{e:U-rho-C11-1} and $p- \alpha  \ge \beta_2 \ge 0$ by Remark \ref{r:condition-F}.
	Using the almost monotonicity of $\Phi_1$ and the boundedness of $\Phi_2$ in the first line below,  \eqref{e:scale3} (with $\eps=\alpha/2$) in the second and third, and  Proposition \ref{p:bound-for-integral-new} in the last,  we get  that for all $x \in B_D(Q,r)$,
	\begin{align*}
		&f_{n,1}(x) \le \frac{c_1}{K_n} \int_{B_D(P,r/n)}  \int_{z\in U(3r):  \delta_D(z) \le \delta_D(y)} G^{U(3r)}(x,z)\ell \bigg( \frac{\delta_D(z)}{ \delta_D(y)}\bigg)\,dzdy\\
		&  \le \frac{c_2}{K_n} \int_{B_D(P,r/n)}  \int_{z\in U(3r):  \delta_D(z) \le \delta_D(y)} G^{U(3r)}(x,z) \frac{\ell  (\delta_D(z))}{\delta_D(y)^{\alpha/2}}\,dzdy\\
		& \le \frac{c_3}{K_n}  \int_{B_D(P,r/n)}  \int_{z\in U(3r):  \delta_D(z) \le \delta_D(y)} G^{U(3r)}(x,z) \frac{ (\delta_D(y)/\delta_D(z))^{\alpha/2}}{\delta_D(y)^{\alpha/2}} \ell  (\delta_D(y))\,dzdy\\
		& \le c_3 \int_{U(3r)} G^{U(3r)}(x,z) \delta_D(z)^{-\alpha/2}\,dz \\
		&\le c_4 \delta_D(x)^{\alpha/2}.
	\end{align*}
	Further,  for all $x \in B_D(Q,r)$, using the boundedness of $\Phi_2$ and \eqref{e:scale3} (with $\eps=\alpha/2)$ in the first inequality below, and Proposition \ref{p:bound-for-integral-new} in the third, we also get
	\begin{align*}
		f_{n,2}(x) &\le \frac{ c_5}{K_n} \int_{B_D(P,r/n)}  \int_{z\in U(3r):  \delta_D(z) > \delta_D(y)} G^{U(3r)}(x,z) \frac{\ell (\delta_D(y))}{\delta_D(z)^{\alpha/2}}\,dzdy\\
		&\le  c_5 \int_{U(3r)} G^{U(3r)}(x,z) \delta_D(z)^{-\alpha/2}\,dz\\
		& \le c_6 \delta_D(x)^{\alpha/2}.
	\end{align*}
	Therefore,  there exists $c_7>0$ such that for all $x \in B_D(Q,r)$,
	\begin{align}\label{e:BHP-failure-claim-1}
		f_{n,1}(x)+f_{n,2}(x) \le c_7 \delta_D(x)^{\alpha/2}.
	\end{align}
	In particular, the above estimate shows that for any $n\ge 1$, 
the function $F_n$ vanishes continuously on $\partial D \cap B(Q,r)$.

	We claim that 	 there exists $c_8=c_8(r)>0$ such that the following statement holds: For every $u \in (0,\epsilon_2r/8)$, there exists $N(u) \in \N$ such that
	\begin{align}\label{e:BHP-failure-claim-2}
		f_{N(u),2}( u\e_d) \ge c_8 u^{p} \log (3r/u).
	\end{align}
	Assume  for the moment that \eqref{e:BHP-failure-claim-2} holds. Then  for all $ u\in (0,\epsilon_2r/8)$, by  \eqref{e:BHP-failure-claim-1} and  \eqref{e:BHP-failure-claim-2}, it holds that 
	\begin{align*}
		\frac{F_{N(u)}(u\e_d)}{F_{N(u)}( (\epsilon_2r/8)\e_d)}& \ge 	\frac{c_9 f_{N(u),2}(u\e_d)}{f_{N(u),1}( (\epsilon_2r/8)\e_d) + f_{N(u),2}( (\epsilon_2r/8)\e_d)}\\
		& \ge \frac{c_8c_9u^p \log(3r/u)}{c_7 (\epsilon_2 r/8)^{\alpha/2}},
	\end{align*}
	while by  Lemma \ref{l:BHPfail}, there exists $c_{10}>0$ independent of $u$ such that 
	\begin{align*}
		\frac{F_{N(u)}(u\e_d)}{F_{N(u)}( (\epsilon_2r/8)\e_d)} \le  \frac{c_{10} u^p}{(\epsilon_2r /8)^p}.
	\end{align*}
	Since $\lim_{u \to 0}\log(r/u)  = \infty$, this gives a contradiction, thereby concluding the proof.
	
	Now, we show that \eqref{e:BHP-failure-claim-2} holds. Let $u \in (0,\epsilon_2r/8)$. Observe that for all $n \ge 1$,
	\begin{align*}
		f_{n,2}(u\e_d)
		&\ge\frac{1}{K_n} \int_{B_D(P,r/n)}  \int_{z\in U(3r):  \delta_D(z) > r/n} G^{U(3r)}(u\e_d,z)  \Phi_2\bigg(\frac{ \delta_D(z)}{3r}\bigg)\\
		&\qquad \qquad\qquad \qquad \qquad \qquad \qquad \qquad  \times  \frac{\ell(\delta_D(y)/\delta_D(z))}{\ell(\delta_D(y))}  \ell(\delta_D(y))\,dzdy\\
		&\ge \int_{z\in U(3r):  \delta_D(z) > r/n} G^{U(3r)}(u\e_d,z)  \Phi_2\bigg(\frac{ \delta_D(z)}{3r}\bigg)  \inf_{0<s<r/n}\frac{\ell(s/\delta_D(z))}{\ell(s)} dz.
	\end{align*}
	Thus, using Fatou's lemma and  \hyperlink{F}{{\bf (F)}} in the first inequality below, and Proposition \ref{p:bound-for-integral-new} in the second, we obtain
	\begin{align*}
		\liminf_{n \to \infty} f_{n,2}(u\e_d) &\ge  c_{11}\int_{ U(3r)} G^{U(3r)}(u\e_d,z) \delta_D(z)^{p-\alpha} dz \ge c_{12} u^p \log(3r/u).
	\end{align*}
	This implies \eqref{e:BHP-failure-claim-2}. The proof is complete. \qed


\section{Sharp  estimates of Green function}\label{ch:green}

In this section, we establish sharp two-sided Green function estimates when $D$ is bounded.
With the functions $\Phi_1$ and $\Phi_2$ in \hyperlink{B4-c}{{\bf (B4-c)}}, we define a positive function $\Upsilon$ on $(0,\infty)$ by
	\begin{equation}\label{e:def-of-Theta}
		\Upsilon(t)	:= \int_{t\, \wedge\, 1}^2 u^{2\alpha-2p-1}\Phi_1(u)\Phi_2(u)\, du.
	\end{equation}
Since $\Phi_1(u)=\Phi_2(u)=1$ for $u\ge 1$, it holds that  for all $t>0$,
	\begin{equation}\label{e:lower-bound-Theta-const}
		\Upsilon(t)\ge \int_{1}^2 u^{2\alpha-2p-1}du= c_1.
	\end{equation} 
Moreover, by \eqref{e:scale1new} and \eqref{e:scale2}, we see that for all $t\in (0,1]$,
	\begin{equation}\label{e:lower-bound-Theta}
		\Upsilon(t)\ge \int_t^{2t} u^{2\alpha-2p-1}\Phi_1(u)\Phi_2(u)\, du \ge c_2 t^{2\alpha-2p}\Phi_1(t)\Phi_2(t).
	\end{equation}
Further, given $a\in (0,1)$, there exists $c=c(a)>0$ such that for all  for all $t>0$,
	\begin{equation}\label{e:scaling-Theta}
 		\Upsilon(at)\ge	\Upsilon(t)\ge c\, \Upsilon(at).
	\end{equation}
Indeed, the first inequality in \eqref{e:scaling-Theta} is obvious.   Next,  if $at \ge 1$, then $\Upsilon(at)=\Upsilon(t)$ 
and if $at <1$, then by   \eqref{e:scale1new},  \eqref{e:scale2}, \eqref{e:lower-bound-Theta-const} and \eqref{e:lower-bound-Theta}, 
	\begin{align*}
		\Upsilon(at)&=\Upsilon(t) + \int_{at}^{t \wedge 1} u^{2\alpha-2p-1}\Phi_1(u)\Phi_2(u)\, du \\
		&\le \Upsilon(t) +  c_1 \Phi_1(t \wedge 1)\Phi_2(t \wedge 1)\int_{at}^{t} u^{2\alpha-2p-1}\, du \\
		&\le \Upsilon(t) + c_2
			\begin{cases}
			1  &\mbox{ if } t\ge 1,\\
			t^{2\alpha-2p}	\Phi_1(t)\Phi_2(t) &\mbox{ if } t<1
			\end{cases} \\
		& \le (1+c_3)\, \Upsilon(t),
	\end{align*}
proving the claim. 

The goal of this section is to get the following  two-sided estimates on the Green function. 

\begin{thm}\label{t:Green}
Suppose that $D$ is a bounded $C^{1,1}$ open set and \hyperlink{B1}{{\bf (B1)}}, \hyperlink{B3}{{\bf (B3)}}, \hyperlink{B4-c}{{\bf (B4-c)}}, 
 \hyperlink{K3}{{\bf (K3)}} and  \hyperlink{B5}{{\bf (B5)}} hold.
Let  $p \in [(\alpha-1)_+, \alpha+\beta_1) \cap (0,\infty)$ 
denote the constant satisfying \eqref{e:C(alpha,p,F)} if $C_9>0$ and let $p=\alpha-1$ if $C_9=0$ where $C_9$ is the contant in \hyperlink{K3}{{\bf (K3)}}.  Then for all $x,y \in D$,
	\begin{align}\label{e:Green-full}
	\begin{split}
		G(x, y)& \asymp  \left(\frac{\delta_D(x)\wedge \delta_D(y)}{|x-y|} \wedge 1\right)^p 
			\left(\frac{\delta_D(x)\vee \delta_D(y)}{|x-y|} \wedge 1\right)^p \\
		&\qquad \times \Upsilon\left(\frac{\delta_D(x)\vee \delta_D(y)}{|x-y|}\right)\frac1{|x-y|^{d-\alpha}},
		\end{split}
	\end{align}
 	where  the comparison constants depend on $D$ only through $\wh R, \Lambda$ 
and \text{\rm diam}$(D)$.
\end{thm}

Note that the function $\ell$  does not play a role in  the Green function estimates in \eqref{e:Green-full}, while it appears in the estimate of $\sB$ in  \hyperlink{B4-c}{{\bf (B4-c)}}. 

By \eqref{e:scale1new} and \eqref{e:scale2},  since $\Phi_1(r)=\Phi_2(r)=1$ for $r\ge 1$, we see that for every $\eps>0$, there exists a constant $c(\eps)>1$ such that for all $0<s\le r\le 2$,
	\begin{align*}
			c(\eps)^{-1}\bigg( \frac{r}{s}\bigg)^{\beta_1 + \beta_2-\eps \wedge  (\beta_1+\beta_2)}
				\le 	\frac{\Phi_1(r)\Phi_2(r)}{\Phi_1(s)\Phi_2(s)} \le c(\eps) 
			 \bigg( \frac{r}{s}\bigg)^{\ub_1+\ub_2 }.
	\end{align*}
Therefore, if $p<\alpha+(\beta_1+\beta_2)/2$ or $p>\alpha+(\ub_1+\ub_2)/2$,  
then we obtain the following explicit estimates for $\Upsilon$ from  \cite[Lemma 5.1]{CKSV22}: For all $t>0$,
	\begin{align}\label{e:Theta-explicit}
		\Upsilon(t)\asymp 
			\begin{cases}1 
				&\mbox{ if }\,p<\alpha+(\beta_1+\beta_2)/2,\\
				(t\wedge 1)^{2\alpha-2p}\Phi_1(t \wedge 1)\Phi_2(t \wedge1)  &\mbox{ if } \,  p>\alpha+(\ub_1+\ub_2)/2.
			\end{cases}
	\end{align}
Observe that $\Phi_1(t \wedge 1)\Phi_2(t \wedge1)= \Phi_1(t)\Phi_2(t)$ for all $t>0$. Hence, by \eqref{e:Theta-explicit},  we obtain the next corollary from Theorem \ref{t:Green}.

\begin{corollary}\label{c:Green}
Under the setting of Theorem \ref{t:Green}, the following statements hold true.
	
\noindent (i) Suppose that $p<\alpha + (\beta_1+\beta_2)/2$. Then  for all $x,y \in D$,
	\begin{align*}
		&G(x, y) \asymp
		 \left(\frac{\delta_D(x)\wedge \delta_D(y)}{|x-y|} \wedge 1\right)^p 
		 \left(\frac{\delta_D(x)\vee \delta_D(y)}{|x-y|} \wedge1\right)^p\frac1{|x-y|^{d-\alpha}}.
	\end{align*}
\noindent (ii)  
Suppose that $\alpha + (\ub_1+\ub_2)/2<p<\alpha+\beta_1$. 
	Then for all $x,y \in D$,
	\begin{align}\label{e:Green-special-2}
		\begin{split}
		G(x, y)
		& \asymp
			\left(\frac{\delta_D(x)\wedge \delta_D(y)}{|x-y|} \wedge 1\right)^p 
			\left(\frac{\delta_D(x)\vee \delta_D(y)}{|x-y|} \wedge1\right)^{2\alpha-p} \\
		&\qquad \times  \Phi_1\left(\frac{\delta_D(x)\vee \delta_D(y)}{|x-y|}\right) \Phi_2
			\left(\frac{\delta_D(x)\vee \delta_D(y)}{|x-y|}\right)\frac1{|x-y|^{d-\alpha}}.
		\end{split}
	\end{align}
\end{corollary}
\begin{remark}
Let $\beta_1^*$ be the  upper Matuszewska index of $\Phi_1$, namely,
	\begin{equation*}
		\beta_1^*:=\inf\left \{\beta:\exists \, a>0 \,\mbox{ s.\ t. } \Phi_1(r)/\Phi_1(s)\le a(r/s)^\beta \,
		 \mbox{ for } \,0<s\le r\le 1 \right\},
	\end{equation*}
and $\beta_2^*$ be the upper Matuszewska index of $\Phi_2$. Then  \eqref{e:Green-special-2} continues to hold true if   
$\alpha + (\beta^*_1+\beta^*_2)/2<p<\alpha+\beta_1$.
\end{remark}

Suppose that 
	\begin{align}\label{e:Green-special-case}
		\Phi_1(u)=(u \wedge 1)^{\beta_1}\ell_1(u), \quad\;\; \Phi_2(u)=(u \wedge 1)^{\beta_2}\ell_2(u)
	\end{align}
and $p=\alpha + (\beta_1+\beta_2)/2$ (so $2\alpha-2p-1 = -\beta_1-\beta_2-1$), where $\ell_1$ and $\ell_2$ are slowly varying at zero. Here we note that since $\Phi_1$ and $\Phi_2$ are assumed to be almost increasing, $\ell_1$ (resp. $\ell_2$) should be almost increasing if $\beta_1=0$ (resp. $\beta_2=0$). Then $\Upsilon(t) \asymp  \mathfrak L(t)$ where 
	\begin{align}\label{e:def-of-Theta0}
		\mathfrak L(t) =1+\left(\int_{t}^1 \frac{\ell_1(u)\ell_2(u)}{u}du \right)_+.
	\end{align} 
Consequently, we obtain the next corollary from Theorem \ref{t:Green}.

\begin{corollary}\label{c:Greenc}
Under the setting of Theorem \ref{t:Green}, suppose also that \eqref{e:Green-special-case} holds with $\ell_1$ and $\ell_2$ that are slowly varying at zero, and $p=\alpha + (\beta_1+\beta_2)/2$. Then  for all $x,y \in D$, 
	\begin{align*}
		G(x, y)& \asymp  \left(\frac{\delta_D(x)}{|x-y|} \wedge 1\right)^{p} \left(\frac{\delta_D(y)}{|x-y|} \wedge 1\right)^{p} \\
			&\qquad \times \mathfrak L\left(\frac{\delta_D(x)\vee \delta_D(y)}{|x-y|}\right)\frac1{|x-y|^{d-\alpha}}
	\end{align*}
where $\mathfrak L$ is defined in \eqref{e:def-of-Theta0}.
\end{corollary}

The following lemma will play an important role in obtaining the sharp upper estimates of the Green function.

\begin{lemma}\label{l:key-estimate}
Let $r\in (0, \wh{R}/8]$, $x, y\in D$ with $\delta_D(x) \vee \delta_D(y) \le r/2$,  and $Q_x, Q_y \in \partial D$ be such 
that $|x-Q_x|=\delta_D(x)$ and $|y-Q_y|=\delta_D(y)$.	There exists $C>0$ independent of $r$, $x$ and $y$ such that
	\begin{align*}
		&\int_{U^{Q_x}(r)}dw\int_{U^{Q_y}(r)}dz\left(  \frac{\delta_D(x)\wedge \delta_D(w)}{|x-w|} \wedge 1\right)^p	
			\left(  \frac{\delta_D(y)\wedge \delta_D(z)}{|y-z|} \wedge 1\right)^p |x-w|^{\alpha-d}\,\nn\\
		&\qquad \quad  \;\; \times   |y-z|^{\alpha-d} \,
			\Phi_1\bigg(\frac{\delta_D(w) \wedge \delta_D(z)}{r}\bigg)\Phi_2\bigg(\frac{\delta_D(w) \vee \delta_D(z)}{r}\bigg)
				\ell\bigg(\frac{\delta_D(w)\wedge \delta_D(z)}{\delta_D(w)\vee \delta_D(z)}\bigg)\nn\\
		& \le  C r^{2\alpha}\left(\frac{\delta_D(x)\wedge \delta_D(y)}{r}\right)^p 
			\left(\frac{\delta_D(x)\vee \delta_D(y)}{r}\right)^p \Upsilon\left(\frac{\delta_D(x)\vee \delta_D(y)}{r}\right).
	\end{align*}
\end{lemma}
\begin{proof} By symmetry, without loss of generality, we can assume that $\delta_D(x)\le \delta_D(y)$. For convenience, 
we  use $\rho_D(w)$ to denote  $\rho_D^{Q_x}(w)$ for $w\in U^{Q_x}(r)$,
and use $\rho_D(z)$ to denote $\rho_D^{Q_y}(z)$ for $z\in U^{Q_y}(r)$. 
Choose $\lb_1\in [0, \beta_1]$, $\lb_2\in [0, \beta_1]$ and $\varepsilon \in (0,1/2)$ such that 
$p<\alpha+\lb_1-\varepsilon\wedge \lb_2-\varepsilon$ and the first inequalities in \eqref{e:scale1new}-\eqref{e:scale2} hold.
Define
	$$
		\wt{\Phi}_1(t)=\Phi_1(t)(t\wedge 1)^{-\varepsilon} \quad \text{and} 
		\quad \wt{\Phi}_2(t)=\Phi_2(t)(t \wedge 1)^{\varepsilon}, \quad t>0.
	$$ 
We also define 
	$$
		\Phi_3(t)={\Phi}_1(t){\Phi}_2(t), \quad t>0.
	$$
Clearly, $\Phi_3(t)=\wt{\Phi}_1(t)\wt{\Phi}_2(t)$ for all $t>0$. By  \eqref{e:U-rho-C11-2}, \eqref{e:scale3},  \eqref{e:scale1new} and \eqref{e:scale2}, we see that for all $w \in U^{Q_x}(r)$ and $z \in U^{Q_y}(r)$,
	\begin{align*}
		&\Phi_1\left(\frac{\delta_D(w) \wedge\delta_D(z) }{r}\right)\Phi_2\left(\frac{\delta_D(w)  \vee \delta_D(z) }{r}\right)
			\ell\left(\frac{\delta_D(w) \wedge\delta_D(z)}{\delta_D(w)  \vee \delta_D(z)}\right)\nn\\
		&\le c_1\Phi_1\left(\frac{\rho_D(w) \wedge\rho_D(z) }{r}\right)\Phi_2\left(\frac{\rho_D(w)  \vee \rho_D(z) }{r}
			\right)\ell\left(\frac{\rho_D(w) \wedge\rho_D(z)}{\rho_D(w)  \vee \rho_D(z)}\right) \nn\\
		&\le c_2 \Phi_1\left(\frac{\rho_D(w) \wedge\rho_D(z) }{r}\right)\Phi_2\left(\frac{\rho_D(w)  \vee \rho_D(z) }{r}\right) 
			\ell(1)\left(\frac{\rho_D(w) \wedge\rho_D(z)}{\rho_D(w)  \vee \rho_D(z)}\right)^{-\varepsilon} \nn\\
		& = c_2\wt{\Phi}_1\left(\frac{\rho_D(w) \wedge\rho_D(z) }{r}\right)\wt{\Phi}_2\left(\frac{\rho_D(w)  
			\vee \rho_D(z) }{r}\right).
	\end{align*}
Hence, we have
	\begin{align*}
		&\int_{U^{Q_x}(r)}dw\int_{U^{Q_y}(r)}dz\left(  \frac{\delta_D(x)\wedge \delta_D(w)}{|x-w|} \wedge 1\right)^p	
			\left(  \frac{\delta_D(y)\wedge \delta_D(z)}{|y-z|} \wedge 1\right)^p |x-w|^{\alpha-d}\nn\\
		&\qquad \quad  \times  |y-z|^{\alpha-d} \,
			\Phi_1\bigg(\frac{\delta_D(w) \wedge \delta_D(z)}{r}\bigg)\Phi_2\bigg(\frac{\delta_D(w) \vee \delta_D(z)}{r}\bigg)
				\ell\bigg(\frac{\delta_D(w)\wedge \delta_D(z)}{\delta_D(w)\vee \delta_D(z)}\bigg)\nn\\
		&\le  c_2 \bigg(\int_{U^{Q_x}(r)} dw \int_{U^{Q_y}(r), \, \rho_D(z)<\rho_D(w)}dz + \int_{U^{Q_x}(r)} dw \int_{U^{Q_y}(r), 
			\, 	\rho_D(z)\ge \rho_D(w)}dz \bigg)\\
 		&\quad\; \left(\frac{\delta_D(x)}{|x-w|} \wedge 1\right)^p 	
 			\left(  \frac{\delta_D(y)}{|y-z|} \wedge 1\right)^p\frac{\wt{\Phi}_1\big((\rho_D(w)\wedge 
 				\rho_D(z))/r\big)\wt{\Phi}_2\big((\rho_D(w)\vee \rho_D(z))/r\big)}{|x-w|^{d-\alpha}\,|y-z|^{d-\alpha}}  \\
		&\le c_2(I_1+I_2),
	\end{align*}
where
	\begin{align*}
		I_1&:=  \int_{U^{Q_x}(r)}   dw	\left(  \frac{\delta_D(x)}{|x-w|} \wedge 1\right)^p 
			\frac{\wt{\Phi}_2(\rho_D(w)/r) }{|x-w|^{d-\alpha}}\\
			&\qquad  \times 	\int_{U^{Q_y}(r,\rho_D(w))} 	
				\left(  \frac{\delta_D(y)}{|y-z|} \wedge 1\right)^p \frac{\wt{\Phi}_1(\rho_D(z)/r)}{|y-z|^{d-\alpha}}  dz
	\end{align*}
and
	\begin{align*}
		I_2&:=\int_{U^{Q_y}(r)} dz	\left(  \frac{\delta_D(y)}{|y-z|} \wedge 1\right)^p
			\frac{\wt{\Phi}_2(\rho_D(z)/r)}{|y-z|^{d-\alpha}}  \\
		&\qquad  \times	\int_{U^{Q_x}(r, \rho_D(z))}\left(  \frac{\delta_D(x)}{|x-w|} \wedge 1\right)^p 
			\frac{\wt{\Phi}_1(\rho_D(w)/r)}{|x-w|^{d-\alpha}}  dw.
	\end{align*}
To estimate $I_1$ and $I_2$, we use Lemma \ref{l:key-curved} several times. Note that by \eqref{e:scale1new} and  \eqref{e:scale2},  
$\wt \Phi_1$ satisfies \eqref{e:scale4} with $\lb=\wt \lb_1:=\lb_1- \eps \wedge \lb_2-\eps$ and   $\Phi_3$ satisfies \eqref{e:scale4} with $\lb= \lb_3:=\lb_1+\lb_2- \eps \wedge \lb_2$. Clearly, $\lb_3 \ge \wt \lb_1$. By the choice of $\eps$, we see that
	$$
		\wt \lb_1 >-2\eps >-1 \quad \text{and} \quad p<\alpha+\wt \lb_1.
	$$

We first estimate $I_1$. By applying Lemma  \ref{l:key-curved}(iii) with $\Phi=\wt \Phi_1$, $\gamma=0$ and $q=p$, we get that
	\begin{align*}
		I_1&\le c \delta_D(y)^{\alpha-1} 	\int_{U^{Q_x}(r)}   
			\bigg(  \frac{\delta_D(x)}{|x-w|} \wedge 1\bigg)^p  \rho_D(w)  
			\bigg(  \frac{\delta_D(y)}{\rho_D(w)} \wedge 1\bigg)^{p-\alpha+1}  \frac{\Phi_3 (\rho_D(w)/r) }{|x-w|^{d-\alpha}} dw \\
			&\le c \delta_D(y)^{\alpha-1}	\int_{U^{Q_x}(r,2\rho_D(y))}  
			\left(  \frac{\delta_D(x)}{|x-w|} \wedge 1\right)^p \frac{ \rho_D(w)\,\Phi_3 (\rho_D(w)/r) }{|x-w|^{d-\alpha}} dw \\
			&\quad + c \delta_D(y)^p	\int_{U^{Q_x}(r) \setminus U^{Q_x}(r,2\rho_D(y))}  
				\left(  \frac{\delta_D(x)}{|x-w|} \wedge 1\right)^p 
				\frac{ \rho_D(w)^{\alpha-p}\,\Phi_3 (\rho_D(w)/r)}{|x-w|^{d-\alpha}} dw \\
			&=:c(I_{1,1}+I_{1,2}).
	\end{align*}
Applying Lemma  \ref{l:key-curved}(iii) with $\Phi=\Phi_3$, $\gamma=1$ and $q=p$, and using the scaling property of $\Phi_3$  and \eqref{e:lower-bound-Theta}, since $\rho_D(y)=\delta_D(y)$, we obtain
	\begin{align*}
		I_{1,1} &\le c \delta_D(x)^{\alpha-1} \delta_D(y)^{\alpha-1} (2\rho_D(y))^2	
			\left(  \frac{\delta_D(x)}{2\rho_D(y)} \right)^{p-\alpha+1}  \Phi_3(2\rho_D(y)/r) \\
			&\le c  \delta_D(x)^{p}\delta_D(y)^{2\alpha-p} \,   \Phi_3(\rho_D(y)/r) \\
			&\le c r^{2\alpha-2p}\delta_D(x)^p\delta_D(y)^p \, \Upsilon( \delta_D(y)/r).
	\end{align*}
For $I_{1,2}$, since $2\rho_D(y) = 2\delta_D(y) \ge \delta_D(x)$, applying Lemma \ref{l:key-curved}(i) 
with $\Phi=\Phi_3$, $\gamma=\alpha-p$ and $q=p$, we obtain
	\begin{align*}
		I_{1,2}& \le cr^{2\alpha-2p}\delta_D(y)^p \delta_D(x)^p \int_{2\rho_D(y)/r}^1 s^{2\alpha-2p-1}\Phi_3(s)ds\\
			& \le cr^{2\alpha-2p}\delta_D(y)^p \delta_D(x)^p \,\Upsilon(\delta_D(y)/r).
	\end{align*}

For $I_2$, by applying Lemma  \ref{l:key-curved}(iii) with $\Phi=\wt \Phi_1$, $\gamma=0$ and $q=p$, we see that
	\begin{align*}
		I_2 &\le 	c \delta_D(x)^{\alpha-1} \int_{U^{Q_y}(r)} 
				\bigg(  \frac{\delta_D(y)}{|y-z|} \wedge 1\bigg)^p \rho_D(z) 
				\bigg(\frac{\delta_D(x)}{\rho_D(z)}  \wedge 1\bigg)^{p-\alpha+1}\frac{\Phi_3(\rho_D(z)/r)}{|y-z|^{d-\alpha}}  dz\\
			&\le c \delta_D(x)^{p} \int_{ U^{Q_y}(r, 2\rho_D(y))} 
				\left(  \frac{\delta_D(y)}{|y-z|} \wedge 1\right)^p 
				\frac{\rho_D(z)^{\alpha-p}\,\Phi_3(\rho_D(z)/r)}{|y-z|^{d-\alpha}}  dz\\
			&\quad + c \delta_D(x)^{p} \int_{U^{Q_y}(r) \setminus U^{Q_y}(r, 2\rho_D(y))} 
				\left(  \frac{\delta_D(y)}{|y-z|} \wedge 1\right)^p 
				\frac{\rho_D(z)^{\alpha-p}\,\Phi_3(\rho_D(z)/r)}{|y-z|^{d-\alpha}}  dz\\
			&=:c(I_{2,1}+I_{2,2}).
	\end{align*}
Applying Lemma \ref{l:key-curved}(ii) with $\Phi=\Phi_3$, $\gamma=\alpha-p$ and $q=p$,  and using the scaling propery of $\Phi_3$ and  \eqref{e:lower-bound-Theta},  we obtain
	\begin{align*}
		I_{2,1} &\le c \delta_D(x)^p\delta_D(y)^{\alpha-1} (2\rho_D(y))^{\alpha-p+1} \Phi_3(2\rho_D(y)/r) \\
			&\le c  \delta_D(x)^{p}\delta_D(y)^{2\alpha-p} \,   \Phi_3(\rho_D(y)/r)\\
			& \le c r^{2\alpha-2p}\delta_D(x)^p\delta_D(y)^p \, \Upsilon( \delta_D(y)/r).
	\end{align*}
Moreover, applying Lemma \ref{l:key-curved}(i) with $\Phi=\Phi_3$, $\gamma=\alpha-p$ and $q=p$,  we get that
	\begin{align*}
		I_{2,2} &\le c r^{2\alpha-2p} \delta_D(x)^p\delta_D(y)^{p} \int_{2\rho_D(y)/r}^1 s^{2\alpha-2p-1}\Phi_3(s) ds \\
			&\le cr^{2\alpha-2p}\delta_D(y)^p \delta_D(x)^p \,\Upsilon(\delta_D(y)/r).
	\end{align*}
The proof is complete. 
\end{proof}

\textsc{Proof of Theorem \ref{t:Green}.} 
Let $x,y \in D$ and set $r:=\wh R\,|x-y|/(30+24\, \text{diam}(D))$.    Note that $ r<   (|x-y|/30)  \wedge (\wh R/24)$. Without loss of generality, we assume that  $\delta_D(x)\le \delta_D(y)$.  

\smallskip

{\bf Upper bound:}
 Recall that $\epsilon_2 \in (0,1/12)$ 
 is the constant in Theorem \ref{t:Dynkin-improve}. 
 If $\delta_D(y) \ge 2^{-4}\epsilon_2r$, then by using \eqref{e:lower-bound-Theta-const}, we get the result from Proposition \ref{p:prelub} by taking $R_0>2\mathrm{diam} (D)$.

Suppose now that $\delta_D(y) <2^{-4}\epsilon_2 r$. 
Denote by $Q_x, Q_y \in \partial D$ the points satisfying $\delta_D(x)=|x-Q_x|$ and $\delta_D(y)=|y-Q_y|$. 
Set $U:=U^{Q_x}(\epsilon_2 r)$ and $V:= U^{Q_y}(\epsilon_2 r)$.  By \eqref{e:U-rho-C11-1}, we see that  
$U \subset B_D(Q_x, 2\epsilon_2 r) \subset B_D(x, r) \subset D \setminus B(y, 29r)$. Hence,  $G(\cdot, y)$ is regular harmonic in $U$. Thus, we get
	\begin{align*}
		&G(x, y)=\E_x\big[G(Y_{\tau_{U}}, y); Y_{\tau_{U}}\in V\big]+ \E_x\big[G(Y_{\tau_{U}}, y); 
			Y_{\tau_{U}} \in D\setminus V\big]=:I_1+I_2.
	\end{align*}

Observe that $|w-z| \asymp r$ for $w \in U$ and $z\in V$. Thus, by \hyperlink{B4-c}{{\bf (B4-c)}} and the scaling properties 
of $\Phi_1,\Phi_2$ and $\ell$, we see that  for $w \in U$ and $z \in V$,
	\begin{align}\label{e:dist-U-V}
		\sB(w,z)  \asymp \Phi_1\left(\frac{\delta_D(w) \wedge \delta_D(z)}{r}\right)
			\Phi_2\left(\frac{\delta_D(w) \vee \delta_D(z)}{r}\right) 
			\ell\left(\frac{\delta_D(w) \wedge \delta_D(z)}{\delta_D(w) \vee \delta_D(z)}\right).
	\end{align} 
By using the L\'evy system formula \eqref{e:Levysystem-Y-kappa} in the equality, \eqref{e:dist-U-V} and Proposition \ref{p:prelub}  (with $R_0=2\,\text{diam}(D)$) in the first inequality, Lemma \ref{l:key-estimate} in the second, 
and  \eqref{e:scaling-Theta} (with $a=\epsilon_2 \wh R/(30+24\, \text{diam}(D))$) in the last, we obtain
	\begin{align*}
		I_1	& = \E_x \left[ \int_0^{\tau_U}  \int_{V} \frac{\sB( Y_s, z)\,G( z, y) }{| Y_s-z|^{d+\alpha}} dz  \,ds  \right] \\
			& =\int_U G^U (x,w) \int_V \frac{\sB(w,z)}{|w-z|^{d+\alpha}} \,G(z,y)dz\,dw\\
			& \le \frac{c_1}{r^{d+\alpha}}\int_{U}dw\int_{V}dz
				\left(  \frac{\delta_D(x)\wedge \delta_D(w)}{|x-w|} \wedge 1\right)^p
				\left(  \frac{\delta_D(z) \wedge \delta_D(y)}{|y-z|} \wedge 1\right)^p  |x-w|^{\alpha-d} \\
			&\qquad  \times  |y-z|^{\alpha-d} \, \Phi_1\left(\frac{\delta_D(w) \wedge \delta_D(z)}{r}\right)
				\Phi_2\left(\frac{\delta_D(w) \vee \delta_D(z)}{r}\right) 		
				\ell\left(\frac{\delta_D(w) \wedge \delta_D(z)}{\delta_D(w) \vee \delta_D(z)}\right) \\
			&\le  \frac{c_2}{r^{d+\alpha}} (\epsilon_2 r)^{2\alpha}\left(\frac{\delta_D(x)\wedge \delta_D(y)}{\epsilon_2 r}\right)^p 
				\left(\frac{\delta_D(x)\vee \delta_D(y)}{\epsilon_2 r}\right)^p 
				\Upsilon\left(\frac{\delta_D(x)\vee \delta_D(y)}{\epsilon_2 r}\right)\\
			& \le  \frac{c_3}{r^{d-\alpha}} \left(\frac{\delta_D(x)\wedge \delta_D(y)}{|x-y|}\right)^p 
				\left(\frac{\delta_D(x)\vee \delta_D(y)}{|x-y|}\right)^p 
				\Upsilon\left(\frac{\delta_D(x)\vee \delta_D(y)}{|x-y|}\right).
	\end{align*}

For $I_2$, we note that $|y-z| \ge \epsilon_2r/2$ for all $z \in D\setminus  V$  by \eqref{e:U-rho-C11-1}.   Thus, by Proposition \ref{p:prelub},  it holds that   for all $z \in D \setminus  V$,
	$$
		G(z, y) \le c_4 \left(\frac{\delta_D(y)}{|y-z|} \wedge 1\right)^p \frac{1}{|y-z|^{d-\alpha}}
			 \le c_5 (\delta_D(y)/r)^p \,r^{-d+\alpha}.
	$$
Using this in the first inequality below and  Theorem \ref{t:Dynkin-improve} in the second,  we obtain
	\begin{align*}
		& I_2 \le c_5 (\delta_D(y)/r)^p \, r^{-d+\alpha} \, \P_x(Y_{\tau_{U}}  \in  D ) 
			\le c_6(\delta_D(x)/r)^p  (\delta_D(y)/r)^p \,r^{-d+\alpha}.
	\end{align*}
By \eqref{e:lower-bound-Theta-const}, we deduce that  the desired upper bound holds.

\smallskip

{\bf Lower bound:}   If $\delta_D(y) \ge r$, then  $\Upsilon(\delta_D(y)/|x-y|) \le \Upsilon(\wh R/(30+24\, \text{diam}(D)))$ and the result follows from Theorem \ref{t:GB} (with $R_0= 20 \,\mathrm{diam}(D)$).
Hence, we assume $\delta_D(x)\le \delta_D(y)<r$.
Again we denote by $Q_x, Q_y \in \partial D$ the points satisfying $\delta_D(x)=|x-Q_x|$ and $\delta_D(y)=|y-Q_y|$. 

Let $n_0 \ge 1$ be such that $ 2^{-n_0}r \le \delta_D(y) < 2^{-n_0+1}r$.  For $1\le n \le n_0$, we define 
	\begin{align*}
		V_x(n) &= \left\{ w = (\wt w, w_d) \text{ in CS$_{Q_x}$} : |\wt w|<2^{-n}r \le w_d-\delta_D(x) <2^{-n+1}r\right\},\\
		V_y(n) &=\left\{z = (\wt z, z_d) \text{ in CS$_{Q_y}$} : |\wt z|<2^{-n}r \le z_d-\delta_D(y) <2^{-n+1}r\right\}.
	\end{align*}
Then  there exists $c_1>0$ such that for all $1 \le n \le n_0$,
	\begin{align}\label{e:green-lower-1}
		m_d(V_x(n))\wedge  m_d(V_y(n)) \ge  c_1(2^{-n}r)^d.
	\end{align}
Moreover, we see that for all $1\le n \le n_0$,  $w \in V_x(n)$ and $z \in V_y(n)$,
	\begin{align}\label{e:green-lower-2}
		2^{-n}r \le |w-x| < 2^{-n+2}r, \quad \;\;2^{-n}r \le |z-y| < 2^{-n+2}r
	\end{align}
so that
	\begin{align}\label{e:green-lower-2+}
		|w-z| \le |x-y|  + 2^{-n+3}r \le ( (30+24\, \text{diam}(D))/\wh R + 4)\, r.
	\end{align}
Let $1\le n \le n_0$ and $w \in V_x(n)$. Then $	\delta_D(w) \le  |w-Q_x| \le \delta_D(x) + 2^{-n+2}r < 2^{-n+3}r$. On the other hand, by \eqref{e:U-rho-C11-2} and \eqref{e:Psi-bound},
	\begin{align*}
		\delta_D(w) \ge (2/ \sqrt 5)(w_d - (10r)^{-1}|\wt w|^2) \ge 2^{-n-1}r.
	\end{align*}
Therefore, by repeating the same argument for $1 \le n \le n_0$ and $z \in V_y(n)$, we get that
	\begin{align}  
		&2^{-n-1}r \le \delta_D(w) \wedge \delta_D(z)  \le \delta_D(w) \vee\delta_D(z)  \le 2^{-n+3}r.\label{e:green-lower-3}
	\end{align}
By \eqref{e:green-lower-2} and  \eqref{e:green-lower-3}, we get from Theorem \ref{t:GB} that for all  
$1\le n \le n_0$ and $w \in V_x(n)$,
	\begin{align}\label{e:green-lower-4}
		G^{B_D(x,20r)}(x,w) \ge c \bigg(\frac{\delta_D(x)}{|x-w|} \bigg)^p 
			\frac{1}{|x-w|^{d-\alpha}} \ge c (2^{-n}r)^{-d+\alpha-p}\, \delta_D(x)^p
	\end{align}
and for all   $1\le n \le n_0$ and  $z \in V_y(n)$,
	\begin{align}\label{e:green-lower-5}
		G^{B_D(y,20r)}(z,y) \ge c \bigg(\frac{\delta_D(y)}{|z-y|} \bigg)^p 
			\frac{1}{|z-y|^{d-\alpha}} \ge c (2^{-n}r)^{-d+\alpha-p}\, \delta_D(y)^p.
	\end{align}
Further, by \eqref{e:green-lower-2+}, \eqref{e:green-lower-3},  \hyperlink{B4-c}{{\bf (B4-c)}} and the  scaling properties 
of $\Phi_1,\Phi_2$ and $\ell$, we see  that for all  $1\le n \le n_0$, $w \in V_x(n)$ and $z \in V_y(n)$,
	\begin{align}\label{e:green-lower-6}
		\frac{\sB(w,z)}{|w-z|^{d+\alpha}} &\ge cr^{-d-\alpha} \Phi_1(2^{-n-1}) \,
			\Phi_2(2^{-n-1}) \,\ell(2^{-3})\ge cr^{-d-\alpha} \Phi_1(2^{-n+1}) \,\Phi_2(2^{-n+1}).
	\end{align}

Now using  the regular harmonicity of $G(\cdot, y)$ on $B_D(x,20r)$ in the first inequality below, the L\'evy system formula
\eqref{e:Levysystem-Y-kappa} in the second, \eqref{e:green-lower-5} and \eqref{e:green-lower-6} in the fourth, \eqref{e:green-lower-1} in the fifth, the scaling properties of $\Phi_1$ and $\Phi_2$ in the sixth and \eqref{e:scaling-Theta} in the last, we arrive at
	\begin{align*}
		&G(x, y)  \ge \E_x \left[ G( Y_{\tau_{B_D(x,20r)}},y ):Y_{\tau_{B_D(x,20r)}} \in \cup_{n=1}^{n_0} V_y(n)\right]\\
		&\ge \sum_{n=1}^{n_0} \int_{B_D(x,20r)} \int_{V_y(n)} G^{B_D(x,20r)}(x, w)\, \frac{\sB(w,z)}{|w-z|^{d+\alpha}} G(z,y) dz\,dw\\
		&\ge \sum_{n=1}^{n_0} \int_{V_x(n)} \int_{V_y(n)} G^{B_D(x,20r)}(x, w)\, 
			\frac{\sB(w,z)}{|w-z|^{d+\alpha}} G^{B_D(y,20r)}(z,y) dz\,dw\\
		&\ge \frac{c\delta_D(x)^p \delta_D(y)^p}{r^{d+\alpha}}\sum_{n=1}^{n_0} (2^{-n}r)^{2(-d+\alpha-p)} \,
			\Phi_1(2^{-n+1})\Phi_2(2^{-n+1}) \int_{W_x(n)}dz \int_{W_y(n)}  dw\\
		&\ge  \frac{c\delta_D(x)^p \delta_D(y)^p}{r^{d-\alpha+2p}}    
			\sum_{n=1}^{n_0} 2^{-2(\alpha-p)n} \Phi_1(2^{-n+1})\Phi_2(2^{-n+1}) \\
		& \ge  \frac{c\delta_D(x)^p \delta_D(y)^p}{r^{d-\alpha+2p}}    
			 \sum_{n=1}^{n_0} \int_{2^{-n+1}}^{2^{-n+2}} u^{2\alpha-2p-1}\Phi_1(u)\Phi_2(u) du\\
		& = cr^{-d+\alpha-2p} \delta_D(x)^p \delta_D(y)^p  \,\Upsilon( 2^{-n_0+1}) \\
		&\ge cr^{-d+\alpha-2p} \delta_D(x)^p \delta_D(y)^p \,\Upsilon( \delta_D(y)/r).
	\end{align*}
This finishes the proof. 
\qed


\section{Examples}\label{ch:examples}

This section is devoted to two types of examples. The main representatives of the first type are subordinate killed stable processes and their modifications. In Subsection \ref{s:subordinate-killed} we show that they satisfy all of the introduced assumptions. Note that these processes are defined through probabilistic transformations (killing and subordination), or, analytically, through their infinitesimal generators. This is different from the second type of examples where the processes are defined via their jump kernel 
$\sB^a(x,y)|x-y|^{-d-\alpha}$ with the function $\sB^a$ being equal to some function $a(x,y)$ multiplied by 
the quantity on the right-hand side of the display in assumption \hyperlink{(B4-c)}{{\bf (B4-c)}}, 
see \eqref{e:B-expression}. Such kernels are studied in Subsection \ref{s-example-general} where we give sufficient conditions on the function $a(x,y)$ so that all assumptions {\bf (B)} are satisfied. The last example of the subsection extends the setting of \cite{Gu07}.

We begin with a general lemma, inspired by the half-space 
setting in  \cite{KSV20},
that will be used several times in the section.

\begin{lemma}\label{l:ass-1}
	Let $K:\bH \times \bH \to [0,\infty)$ be  such that for all $x,y \in \bH$, $a>0$ and $\wt z \in \R^{d-1}$,  
	\begin{align}\label{e:ass-1-cond}
	 	K(x,y)=K(ax,ay)=K(x+(\wt z,0), y+(\wt z,0)).
	\end{align}
Define $F_K:\bH_{-1} \to [0,\infty)$ by 
	\begin{align*}
		F_K(z)=K(\e_d,\e_d+z).
	\end{align*}
Then the following statements hold.
 
\noindent (i) $K(x,y)=F_K((y-x)/x_d)$ for all $x,y \in \bH$.
 
\noindent (ii) If $K$ is also assumed to be symmetric in $x$ and $y$,  then 
	 \begin{align*}
 		F_K(z)=F_K(-z/(1+z_d)) \quad \text{for all} \;\, z \in \bH_{-1}.
 	\end{align*}
\end{lemma}
\begin{proof} 
(i) Using \eqref{e:ass-1-cond}, we get that for all $x,y \in \bH$,
	\begin{align*}
		K(x,y)=K((\wt 0, x_d),(\wt y-\wt x, y_d))= K((\wt 0,x_d),(\wt 0,x_d)+ y-x)) = F_K((y-x)/x_d).
	\end{align*}

(ii) Using \eqref{e:ass-1-cond} and symmetry, we obtain that for any $z \in \bH_{-1}$,
	\begin{align*}	
		&F_K(-z/(1+z_d))=K(\e_d, (-\wt z/(1+z_d),1/(1+z_d))) = K((1+z_d)\e_d, (-\wt z,1))\\
		&=K((-\wt z,1),(1+z_d)\e_d) =K(\e_d,(\wt z, 1+z_d)) = F_K(z).
	\end{align*}
\end{proof}

Throughout the next two subsections, we let $D \subset \R^{d}$,  $d \ge 2$, be a $C^{1,1}$ open set with 
characteristics $(\wh R, \Lambda)$, assume that $\wh R\le 1 \wedge (1/(2\Lambda))$ without loss of generality, and set  $R:=\wh R/8$.

\subsection{Subordinate killed stable processes}\label{s:subordinate-killed}

Let $\gamma \in (0,2]$. In this subsection, we  assume that $D$ is either (1) bounded or (2) the domain above the graph of a bounded $C^{1,1}$ function in $\R^{d-1}$. When $\gamma=2$, we additionally assume that $D$ is connected. 

Let $Z^\gamma$ be an isotropic  $\gamma$-stable process in $\R^d$, that is, a rotationally symmetric L\'evy process with  L\'evy  exponent $|\xi|^\gamma$. 
Denote by $q^{\gamma}(t,|x-y|)$ the transition density  of $Z^{\gamma}$.
For a $C^{1,1}$ open set $U\subset \R^d$, denote by  $Z^{\gamma,U}$ the part process of $Z^\gamma$ killed upon exiting $U$.  Denote by $q^{\gamma,U}(t,x,y)$ the transition density  of $Z^{\gamma,U}$. We extend the domain of $q^{\gamma,U}$ to $(0,\infty) \times \R^d \times \R^d$ by letting $q^{\gamma,U}(t,x,y)=0$ if $x \in \R^d \setminus U$ or $y \in \R^d \setminus U$. 
	
\smallskip
	
A non-negative function $\phi$  on $(0,\infty)$ is called a \textit{Bernstein function} if $\phi$ is infinitely differentiable and $(-1)^{n-1} \phi^{(n)} (\lambda) \ge 0$ for all $n \in \N$ and $\lambda>0$. It is known that every Bernstein function  $\phi$ has the following representation:
	\begin{align*}
		\phi(\lambda) = a + b\lambda +  \int_0^\infty (1-e^{-\lambda t}) \Pi(dt),
	\end{align*}
where $a, b\ge 0$  and $\Pi$ is a 	measure on $(0,\infty)$ satisfying $\int_0^\infty(1 \wedge t)\Pi(dt)<\infty$. 
The triplet $(a,b,\Pi)$ is called \textit{the L\'evy triplet} of the Bernstein function $\phi$. See \cite[Theorem 3.2]{SSV12}.

A process $T=(T_t)_{t \ge 0}$ is called a \textit{subordinator}, if it is a non-decreasing L\'evy process with 	$T_0=0$.
For a given subordinator $T$, there exists a unique Bernstein function $\phi$ such that 
	\begin{align}\label{e:Laplace-exponent}
		\E[e^{-\lambda T_t}] = e^{-t\phi(\lambda)} \quad \text{for all} \;\, \lambda, t>0.
	\end{align}
In this sense, the Bernstein function $\phi$ is called \textit{the Laplace exponent} of $T$. Conversely, 
given a Bernstein function $\phi$ with $\phi(0+)=0$, 
there exists a unique subordinator $T^\phi$ (up to equivalence) such that \eqref{e:Laplace-exponent} holds. 
See \cite[Theorem 5.2]{SSV12}.

\smallskip
	
Let $\beta \in (0,1)$ and $T=(T_t)_{t \ge 0}$ be a  $\beta$-stable subordinator  with  Laplace exponent $\lambda^\beta$,  independent of $Z^\gamma$.  Define a time-changed process $Y^{U}=Y^{\gamma,U,\beta}$ by 
 	\begin{align}\label{e:def-subordinate-killed}
 		Y^{U}_t=Z^{\gamma,U}_{T_t}, \quad t\ge 0.
 	\end{align} 
The generator of $Y^{U}$ is equal to 
$-((-\Delta)^{\gamma/2}|_U)^{\beta}$. When $\gamma=2$, it is the negative of the spectral fractional Laplacian.

 By \cite[(2.8)-(2.9)]{Oku}, the jump kernel $J^U(dx,dy)$ and the killing measure $\kappa^U(dx)$ of $Y^{U}$ have densities $J^U(x,y)$ and  $\kappa^U(x)$ given by 
	 \begin{align}
 		J^U(x,y)&= J^{\gamma,U,\beta}(x,y)  =c_\beta\int_0^\infty q^{\gamma,U}(t,x,y) t^{-1-\beta}dt,
 			\label{e:subordinate-killed-jump} \\
 	 	\kappa^U(x)&= \kappa^{\gamma,U,\beta}(x) =c_\beta\int_0^\infty \bigg(1- \int_U q^{\gamma,U}(t,x,y)dy \bigg) t^{-1-\beta}dt,
 	 		\label{e:subordinate-killed-killing} 
 	 \end{align}
where $c_\beta t^{-1-\beta}$ is the L\'evy density of the  subordinator  $T$. 

Let  $\alpha:=\gamma\beta$. Note that $J^{\R^d}(x,y)$  equals, 
$c_{d,-\alpha} |x-y|^{-d-\alpha}$, which is the jump kernel of isotropic $\alpha$-stable process.
Define 
	\begin{align}\label{e:def-BU-subordinate}
		\sB^U(x,y)=\sB^{\gamma,U,\beta}(x,y)=
		\begin{cases}
			|x-y|^{d+\alpha}J^U(x,y) &\mbox{ if } x\ne y,\\
			c_{d,-\alpha}
 		&\mbox{ if } x= y.
		\end{cases}
	\end{align}
By the scaling and translation invariance properties of $Z^\gamma$,  
we see that the kernel
$c_{d,-\alpha}^{-1}\sB^\bH(x,y)$ satisfies \eqref{e:ass-1-cond} and is symmetric in $x$ and $y$. Define a function $F_0^{\gamma,\beta}$ by 
	\begin{align}\label{e:subordinate-killed-F-1}
 		F_0^{\gamma,\beta}(z):=c_{d,-\alpha}^{-1} \sB^{\bH}(\e_d, \e_d+z), \quad\;\; z \in \bH_{-1}.
	\end{align}
By Lemma \ref{l:ass-1}(i)-(ii),  we have
	\begin{align}\label{e:subordinate-killed-F}		
		J^\bH(x,y)= c_{d,-\alpha}F_0^{\gamma,\beta}((y-x)/x_d)|x-y|^{-d-\alpha}  \quad \text{ for all} \;\, x,y \in \bH	
	\end{align}
and $F_0^{\gamma,\beta}(z)=F_0^{\gamma,\beta}(-z/(1+z_d))$ for all $z \in \bH_{-1}$. It follows that
	\begin{align}\label{e:subordinate-killed-F-2}	
		F_0^{\gamma,\beta}(z) =\frac{1}{2} \big(F_0^{\gamma,\beta}(z) + F_0^{\gamma,\beta}(-z/(1+z_d))\big), 
			\quad \;\;z  = (\wt z,z_d) \in \bH_{-1}.
	\end{align}
Set 
	\begin{align}\label{e:defb}
b_{\gamma,\beta}:=
\begin{cases}
 \gamma/2  &\text{if }\gamma=2 \text{ or }\beta<1/2
 ;\\
 \gamma-\alpha & \text{otherwise},
\end{cases}
	\end{align}
and define
	\begin{align}
	 	\Phi_1^{\gamma,\beta}(r)&:= 
			(r \wedge 1)^{b_{\gamma,\beta}}  \label{e:example-Phi-1},\\
	 	\Phi_2^{\gamma,\beta}(r)&:= \begin{cases}
			r \wedge 1 &\mbox{ if $\gamma=2$;}  \\
			(r \wedge 1)^{\gamma/2-\alpha
			} &\mbox{ if  $\gamma<2$ and $\beta<1/2$};\\
			1 &\mbox{ if $\gamma<2$ and $\beta \ge 1/2$},
		\end{cases} \label{e:example-Phi-2}\\
		 \ell^{\gamma,\beta}(r)&:= \begin{cases}
			\log( e/ (r\wedge 1)) &\mbox{ if $\gamma<2$ and $\beta=1/2$;}  \\
			1 &\mbox{  otherwise}.\label{e:example-ell}
		\end{cases}
	\end{align}

The following is the main result of this subsection.

\begin{prop}\label{p:example-subordinate-killed}
The process $Y^D$ defined by \eqref{e:def-subordinate-killed} satisfies \hyperlink{B1}{{\bf (B1)}}, \hyperlink{B3}{{\bf (B3)}}, \hyperlink{B4-c}{{\bf (B4-c)}},   \hyperlink{K3}{{\bf (K3)}} and  \hyperlink{B5}{{\bf (B5)}}. More precisely, 
$Y^D$ satisfies \hyperlink{B4-c}{{\bf (B4-c)}} with  $\Phi_1=\Phi_1^{\gamma,\beta}$, $\Phi_2=\Phi_2^{\gamma,\beta}$ 
and $\ell=\ell^{\gamma,\beta}$,  \hyperlink{B5-I}{{\bf (B5-I)}}  with $\F_0=\F=F_0^{\gamma,\beta}$ and  any $\nu\in (0,1)$, 
and \eqref{e:C(alpha,p,F)} with  $p=\gamma/2$.
\end{prop}

Note that it was already proved in 
\cite[Lemma 3.2]{KSV18-b} and 
\cite[(2.9)]{KSV} that $Y^D$ satisfies \hyperlink{B3}{{\bf (B3)}}.   We repeat the argument from \cite{KSV}  below since it provides a passageway to the proofs \hyperlink{B5-I}{{\bf (B5-I)}} and \hyperlink{K3}{{\bf (K3)}}. 
 
 By the the strong Markov property and joint continuity of $q^{\gamma }$ and $q^{\gamma,D}$, 
$$
q^{\gamma } (s, |x- y|)-q^{\gamma,D}(s, x, y)=
\E_x
\Big[ q^{\gamma}\big(s - \tau^{(\gamma)}_D, |Z^\gamma_{\tau^{(\gamma)}_D}-y|\big) : \tau^{(\gamma)}_D < s\Big]
$$
Thus,  by \eqref{e:subordinate-killed-jump},  Fubini's theorem and the change of variables, we have tha  for any $x, y \in D$,
 \begin{align*}
j(|x- y|)-J^{ D}(x, y)
&=
c_\beta\E_x \int^\infty_{ \tau^{(\gamma)}_D} 
q^{\gamma}\big(s - \tau^{(\gamma)}_D,|Z^\gamma_{\tau^{(\gamma)}_D}-y|\big) 
s^{-1-\beta}ds
\nn\\
&=
c_\beta \E_x \int^\infty_{0} 
q^{\gamma}\big(v,|Z^\gamma_{\tau^{(\gamma)}_D}-y|\big) 
(v+\tau^{(\gamma)}_D)^{-1-\beta}ds\\
&\le
c_\beta\E_x \int^\infty_{0} 
q^{\gamma}\big(v,|Z^\gamma_{\tau^{(\gamma)}_D}-y|\big) 
v^{-1-\beta}dv\nn\\
&=
\E_x \Big [ j\big(|Z^\gamma_{\tau^{(\gamma)}_D}- y|\big)\Big] \le  j(\delta_D(y)),
\end{align*}
where the last inequality follows from  $|Z^\gamma_{\tau^{(\gamma)}_D}- y|\ge \delta_D(y)$.
Hence, 
\begin{align*}
&\sB^D(x, x)-\sB^D(x, y)=
c_{d, -\alpha}-|x-y|^{d+\alpha}J^D(x,y)\\
&=
|x-y|^{d+\alpha} \left(j(|x- y|)-J^{ D}(x, y)\right)\\ 
& \le
|x-y|^{d+\alpha} j(\delta_D(y)) =c_{d, -\alpha}\left(\frac{|x-y|}{ \delta_D(y)}\right)^{d+\alpha}.
\end{align*}
Thus  \hyperlink{B3}{{\bf (B3)}} holds with $\theta_0=d+\alpha>1$. 

As we have seen,  \hyperlink{B3}{{\bf (B3)}} can be proved by analyzing  the difference between the jumping kernel  of $Y^D$ and that of 
$\alpha$-stable process in
$\R^d$. 
However, the bound in 
\hyperlink{B5-I}{{\bf (B5-I)}} is much more delicate and we need more refined estimates to obtain the bound with the extra  vanishing term 
$( \delta_D(x)\vee\delta_D(y) \vee|x-y|)^{\theta_2}$. To prove \hyperlink{B5-I}{{\bf (B5-I)}} and \hyperlink{K3}{{\bf (K3)}},  
we will analyze  the difference between the jumping kernel and killing potential of $Y^D$ and those of $Y^{\bH}$, and use the sharp estimates of the transition density 
of subordinate killed stable processes in complement of balls.

Recall that $R=\wh R/8$, and  $E^{Q}_{\nu}(R)$, $S^{Q}(R)$ and $\wt S^Q(R)$ are defined by \eqref{e:def-Enu}. 
In the remainder of this subsection,  we fix $Q \in \partial D$, use the coordinate system CS$_{Q}$,  
and denote $E^{Q}_{\nu}(R)$, $S^{Q}(R)$ and $\wt S^Q(R)$ by $E_{\nu}$, $S$ and $\wt S$ respectively.

For a Borel set $A \subset \R^d$, let $\tau^{(\gamma)}_A:=\inf\{t>0:Z^\gamma_t \notin A\}$. By Lemma \ref{l:C11}(ii), the strong Markov property and joint continuity of $q^{\gamma,\R^d \setminus \wt S}$, we see that for all $t>0$ and $x,y \in \R^d \setminus \wt S$,
	\begin{align}\label{e:subordinate-heatkernel-diff}
		\begin{split}	|q^{\gamma,D}(t,x,y)- q^{\gamma,\bH}(t,x,y)| & \le q^{\gamma,\R^d \setminus \wt S}(t,x,y) -q^{\gamma,S}(t,x,y)\\
		&=\E_x \Big[q^{\gamma,\R^d \setminus \wt S}(t-\tau^{(\gamma)}_S, Z^\gamma(\tau^{(\gamma)}_S),y); \tau^{(\gamma)}_S<t \Big].\end{split}
	\end{align}
Hence, by \eqref{e:subordinate-killed-jump} and  Fubini's theorem, for any $x, y \in D\cap \bH$,
	\begin{align}\label{e:subordinate-killed-diff}	
		\begin{split}
		&\left| J^D(x,y)- J^\bH(x,y)\right| \\&\le c_\beta\int_0^\infty |q^{\gamma,D}(t,x,y)- q^{\gamma,\bH}(t,x,y)| \, t^{-1-\beta} dt \\
		& \le c_\beta\int_0^\infty  \E_x \Big[q^{\gamma,\R^d \setminus \wt S}(t-\tau^{(\gamma)}_S, Z^\gamma(\tau^{(\gamma)}_S),y);
			 \tau^{(\gamma)}_S<t \Big] t^{-1-\beta}dt\\
		& = c_\beta\E_x \bigg[ \int_{\tau^{(\gamma)}_S}^\infty q^{\gamma,\R^d \setminus \wt S}(t-\tau^{(\gamma)}_S,
			 Z^\gamma(\tau^{(\gamma)}_S),y) \, t^{-1-\beta}dt \bigg] \\
		&\le c_\beta \sup_{z \in \R^d \setminus (S \cup \wt S)}  \int_{0}^\infty q^{\gamma,\R^d \setminus \wt S}(t, z,y)
			 t^{-1-\beta}dt. 
		\end{split}
	\end{align}
Define   $q_\gamma:(0,\infty)\times \R^d \times \R^d \to (0,\infty)$ and 
$h_{\gamma,R}:(0,\infty)\times (\R^d \setminus \wt S) \to (0,\infty)$
by
	\begin{align*}
		q_\gamma(t,x,y)&= 
		\begin{cases}
			\displaystyle	t^{-d/2}e^{-|x-y|^2/(4t)} &\quad\mbox{ if } \gamma=2,\\
			\displaystyle	t^{-d/\gamma} \wedge \frac{t}{|x-y|^{d+\gamma}} &\quad \mbox{ if } \gamma<2,
		\end{cases}\\
		h_{\gamma,R}(t,x)&=\begin{cases}		\displaystyle  1 \wedge \frac{\delta_{\R^d \setminus \wt S}(x)^{\gamma/2}}{(t \wedge R^\gamma)^{1/2}} 					&\mbox{ if } d>\gamma,\\[10pt]	\displaystyle	 1 \wedge \frac{ \log(1+\delta_{\R^d \setminus \wt S}(x)/R)}{ \log(1+t^{1/2}/R)} 	 &\mbox{ if } d= \gamma=2.	\end{cases}
	\end{align*}
By the scaling property,  $	q^{\gamma,\R^d \setminus \wt S}(t,x,y)= R^{-d}	q^{\gamma,\R^d \setminus B(-\e_d, 1)}(t/R^\gamma,x/R,y/R)$ for all $t>0$ and  $x,y \in \R^{d}\setminus \wt S$. Hence, applying the  Dirichlet heat kernel estimates from \cite[Theorem 1.3]{CT11} (for $\gamma<2$), \cite[Subsection 5.2]{GS02} (for $\gamma=2$ and $d=2$) and \cite[Theorem 1.1]{Zha} (for $\gamma=2$ and $d\ge3$), we deduce that  there exist constants $c_1,c_2>0$ depending only on $d$ and $\gamma$ such that for all $t>0$ and $x,y \in \R^{d}\setminus \wt S$,
\begin{equation}\label{e:heatkernel-exterior}		
	\begin{split}
			 	q^{\gamma,\R^d \setminus \wt S}(t,x,y) \le c_1h_{\gamma,R}(t,x)h_{\gamma,R}(t,y) q_\gamma(c_2t,x,y).
	\end{split}		\end{equation}
In particular,  for all $t>0$ and $x,y \in \R^{d}\setminus \wt S$,
\begin{equation}\label{e:heatkernel-exterior-2}		
		q^{\gamma,\R^d \setminus \wt S}(t,x,y) \le c_1 q_\gamma(c_2t,x,y).	\end{equation}

\begin{lemma}\label{l:dirichlet-1}
	Suppose that $\gamma<2$.	There exists $C>0$ depending only on $d$ and $\gamma$  such that for all $\nu \in (0,1)$ and   $y \in E_{\nu}$,
	\begin{align*}
		\sup_{0<s\le 
			R^\gamma, \,z \in \R^d \setminus (S \cup \wt S)}q^{\gamma,\R^d \setminus \wt S}(s, z,y) 
		\le C \delta_D(y)^{-d} (\delta_D(y)/R)^{\frac{\gamma(1-\nu)}{2(1+\nu)}}.
	\end{align*}
\end{lemma}
\begin{proof}
	Let $y \in E_{\nu}$ and $z=(\wt z, z_d) \in \R^d \setminus (S \cup \wt S)$.
	We first note that, by Lemma \ref{l:C11}(iii), we have 
	\begin{align}\label{e:dist-wtS-y}
		\begin{split}
			&\delta_{\R^d \setminus \wt S} (y) \le  y_d + R - \sqrt{R^2-|\wt y|^2} \\
			&\le y_d+R -(R-R^{-1}|\wt y|^2) = y_d+R^{-1}|\wt y|^2 \le 5y_d/4 \le 2\delta_D(y)
		\end{split}
	\end{align}
	and
	\begin{align}\label{e:example-dist-yz}
		|y-w| \ge \delta_S(y) \ge 3\delta_D(y)/8 \quad \text{for all} \;\, w \in \R^d \setminus S.
	\end{align} 
	Thus, 
	\begin{align}\label{e:example-dist-yz1}
		\delta_{\R^d \setminus \wt S}(z) \le \delta_{\R^d \setminus \wt S}(y) + |y-z| \le  (19/3)|y-z|.
	\end{align}
	By  \eqref{e:heatkernel-exterior} and \eqref{e:dist-wtS-y}, we have
	\begin{align}\label{e:heatkernel-exterior<2} 
		\begin{split}
			\sup_{0<s\le 
				R^\gamma}q^{\gamma,\R^d \setminus \wt S}(s, z,y) 
			& \le c	\sup_{0<s\le 
				R^\gamma} 	\bigg( \frac{\delta_{\R^d \setminus \wt S}(z)^{\gamma/2}
				\delta_{\R^d \setminus \wt S}(y)^{\gamma/2}}{s} \frac{s}{|y-z|^{d+\gamma}} \bigg)\\
					& \le c \delta_{\R^d \setminus \wt S}(z)^{\gamma/2}|y-z|^{-d-\gamma} \delta_{D}(y)^{\gamma/2}.
		\end{split}
	\end{align}

	If $|y-z| > 4^{-1}R^{\nu/(1+\nu)} \delta_{D}(y)^{1/(1+\nu)}$, then since 
	$\delta_D(y)<R$,  
	using \eqref{e:example-dist-yz1} and \eqref{e:heatkernel-exterior<2}, we obtain
	\begin{align*} 
		\sup_{0<s\le 
			R^\gamma}q^{\gamma,\R^d \setminus \wt S}(s, z,y) 
		&\le c |y-z|^{-d-\gamma/2} \delta_{D}(y)^{\gamma/2}\\
		&\le c (R^{\nu/(1+\nu)} \delta_{D}(y)^{1/(1+\nu)})^{-d-\gamma/2} \delta_{D}(y)^{\gamma/2}\\
		&\le cR^{-(1-\nu)\gamma/(2+2\nu)} \delta_D(y)^{-d+ (1-\nu)\gamma/(2+2\nu)}.
	\end{align*}
	Suppose that  $|y-z| \le 4^{-1}R^{\nu/(1+\nu)} \delta_{D}(y)^{1/(1+\nu)}$. Then  $|\wt z| \le |\wt y| + R/4 \le R/2$. 
	Hence, it holds that
	$$
	\delta_{\R^d \setminus \wt S} (z) \le 2(R - \sqrt{R^2-|\wt z|^2}) \le  2R^{-1}|\wt z|^2.
	$$
	Using this and the triangle inequality in the first inequality below, $y\in E_{\nu}$ and  $|y-z| \le 4^{-1}R^{\nu/(1+\nu)} \delta_{D}(y)^{1/(1+\nu)}$ in the second, and Lemma \ref{l:C11}(iii) in the last, we get
	\begin{align}\label{e:dist-wtS-z}		
		\begin{split}		
			\delta_{\R^d \setminus \wt S} (z) & \le   4R^{-1}(|\wt y|^2 + |y-z|^2) \\		
			&\le (4R)^{-(1-\nu)/(1+\nu)}y_d^{2/(1+\nu)} + 4^{-1}R^{-(1-\nu)/(1+\nu)}\delta_D(y)^{2/(1+\nu)}\\
			& \le cR^{-(1-\nu)/(1+\nu)}\delta_D(y)^{2/(1+\nu)}.		
		\end{split}	
	\end{align}
	Using \eqref{e:heatkernel-exterior<2},  \eqref{e:example-dist-yz} and   \eqref{e:dist-wtS-z}, we arrive at
	\begin{align*}
		\sup_{0<s\le 
			R^\gamma}q^{\gamma,\R^d \setminus \wt S}(s, z,y) &\le c(R^{-(1-\nu)/(1+\nu)}\delta_D(y)^{2/(1+\nu)})^{\gamma/2} 
		(3\delta_D(y)/8)^{-d-\gamma} \delta_{D}(y)^{\gamma/2}.\\
		&= cR^{-(1-\nu)\gamma/(2+2\nu)} \delta_D(y)^{-d+ (1-\nu)\gamma/(2+2\nu)}.
	\end{align*}
	
\end{proof}

\begin{lemma}\label{l:dirichlet-2}
	There exists $C>0$ depending only on $d$  such that for all   $\nu \in (0,1)$ and $y \in E_{\nu}$,
	\begin{align*}
		\sup_{0<s\le
			R^2,\,z \in \R^d \setminus (S \cup \wt S)}q^{2,\R^d \setminus \wt S}(s, z,y) \le C 
		\delta_D(y)^{-d} (\delta_D(y)/R)^{ (1-\nu)/(1+\nu)}.
	\end{align*}
\end{lemma}
\begin{proof}
	Let $y \in E_{\nu}$ and $z=(\wt z, z_d) \in \R^d \setminus (S \cup \wt S)$. If $|y-z| > 4^{-1}R^{\nu/(1+\nu)}$ $\delta_{D}(y)^{1/(1+\nu)}$, then  by  
	\eqref{e:heatkernel-exterior-2},  
	since $R^{-1}\delta_D(y) \le 1$, we get
	\begin{align*} 
		&\sup_{
			s>0}\,q^{2,\R^d \setminus \wt S}(s, z,y)\le c\sup_{s>0}\, s^{-d/2} e^{-c|y-z|^2/s}= c|y-z|^{-d}\\
		&   \le cR^{-d\nu/(1+\nu)}\delta_{D}(y)^{-d/(1+\nu)} \le cR^{-(1-\nu)/(1+\nu)} \delta_D(y)^{-d+ (1-\nu)/(1+\nu)}.
	\end{align*}
	If $|y-z|\le 4^{-1}R^{\nu/(1+\nu)} \delta_{D}(y)^{1/(1+\nu)}$,  then using \eqref{e:heatkernel-exterior} with the fact that $\log(1+s) \asymp s$ for $s \in (0,1)$ in the first line below, \eqref{e:dist-wtS-y} and   \eqref{e:dist-wtS-z} in the second,  and  \eqref{e:example-dist-yz} in the last, we obtain
	\begin{align*}
		\sup_{0<s\le 
			R^2}\,q^{2,\R^d \setminus \wt S}(s, z,y)
		& \le 	c \delta_{\R^d \setminus \wt S}(z)\delta_{\R^d \setminus \wt S}(y)\sup_{0<s\le 
			R^2}\, s^{-1-d/2} e^{-c|y-z|^2/s}\\
		& \le 	c  R^{-(1-\nu)/(1+\nu)}\delta_D(y)^{(3+\nu)/(1+\nu)}\;\sup_{s>0}\, s^{-1-d/2} e^{-c|y-z|^2/s}\\
		&= c R^{-(1-\nu)/(1+\nu)}\delta_D(y)^{(3+\nu)/(1+\nu)}|y-z|^{-d-2} \\
		&\le cR^{-(1-\nu)/(1+\nu)}\delta_D(y)^{-d+(1-\nu)/(1+\nu)}.
	\end{align*}
\end{proof}

\begin{lemma}\label{l:dirichlet-smalltime}
	There exists $C>0$  depending only on $d$ and $\gamma$  such that for all   $\nu \in (0,1)$ and $y \in E_{\nu}$,
	\begin{align*}
	\sup_{z \in \R^d \setminus (S \cup \wt S)}
	\int_0^{\delta_D(y)^{\gamma} (\delta_D(y)/R)^{(1-\nu)\gamma/(2+2\nu)}}
	 q^{\gamma,\R^d \setminus \wt S}(t, z,y)t^{-1-\beta}dt \\
\le   C \bigg(\frac{\delta_D(y)}{R}\bigg)^{(1-\beta)(1-\nu)\gamma/(2+2\nu)}\frac{1}{\delta_D(y)^{d+\alpha}}. 
	\end{align*}
\end{lemma}
\begin{proof}
	Let $\eps:=(1-\nu)\gamma/(2+2\nu)$,  $y \in E_{\nu}$ and $z \in \R^d \setminus (S\cup\wt S)$. When $\gamma<2$, using 
	\eqref{e:heatkernel-exterior-2}  and \eqref{e:example-dist-yz}, since $\alpha=\gamma\beta$, we obtain
	\begin{align*}
		\int_0^{R^{-\eps}\delta_D(y)^{\gamma + \eps}} q^{\gamma,\R^d \setminus \wt S}(t, z,y)t^{-1-\beta}dt 
		&\le \frac{c_1}{|z-y|^{d+\gamma}}\int_0^{R^{-\eps}\delta_D(y)^{\gamma+ \eps}}  t^{-\beta}dt\\
		&\le c_2 R^{-(1-\beta)\eps} \delta_D(y)^{-d-\alpha + (1-\beta)\eps}.
	\end{align*}	
	When $\gamma=2$, using 
	\eqref{e:heatkernel-exterior-2}   and \eqref{e:example-dist-yz}, since  $\sup_{s>0}s^{d/2+1} e^{-s}<\infty$, 
	we get that
	\begin{align*}
		&\int_0^{R^{-\eps}\delta_D(y)^{2 + \eps}} q^{2,\R^d \setminus \wt S}(t, z,y)t^{-1-\beta}\,dt\\
		& \le c_3	\int_0^{R^{-\eps}\delta_D(y)^{2+ \eps}} t^{-d/2-1-\beta} e^{-c_4|z-y|^2/t}\,dt \\
		& \le c_5	\int_0^{R^{-\eps}\delta_D(y)^{2+ \eps}} t^{-d/2-1-\beta} e^{-c_6\delta_D(y)^2/t}\,dt \\
		&\le c_6^{-(d/2+1)} \ c_7 
		\int_0^{R^{-\eps}\delta_D(y)^{2+ \eps}} t^{-d/2-1-\beta}   \, (t/\delta_D(y)^2)^{d/2+1}\,dt\\
		&=c_8 R^{-(1-\beta)\eps} \delta_D(y)^{-d-2\beta + (1-\beta)\eps}. 
	\end{align*}
	The proof is complete. 
\end{proof}

We now analyze  the difference between the jumping kernel   of $Y^D$ and $Y^{\bH}$.

\begin{lemma}\label{l:gamma-beta-extra-decay}
	There exists $C>0$  depending only on $d$ and $\gamma$   such that for all   $\nu \in (0,1)$ and $x,y \in E_{\nu}$,
	\begin{align*}
		|J^D(x,y)-J^\bH(x,y)| 
		\le  C\bigg(\frac{\delta_D(x) \vee \delta_D(y)}{R}\bigg)^{\tfrac{(1-\beta)(1-\nu)\gamma}{2(1+\nu)}} 
		\frac{1}{(\delta_D(x)\vee\delta_D(y))^{d+\alpha}}.
	\end{align*}
\end{lemma}
\begin{proof}  Let $\eps:=(1-\nu)\gamma/(2+2\nu)$ and $x,y \in E_{\nu}$. By symmetry, without loss of generality, 
	we assume $\delta_D(x) \le \delta_D(y)$.  Using \eqref{e:subordinate-killed-diff}, \eqref{e:heatkernel-exterior} and Lemmas \ref{l:dirichlet-1}, \ref{l:dirichlet-2} and \ref{l:dirichlet-smalltime}, we get
	\begin{align*}	
		&\left| J^D(x,y) - J^\bH(x,y)\right|\\
		&\le c_\beta\sup_{z \in \R^d \setminus (S \cup \wt S)}  \bigg( \int_{0}^{R^{-\eps} \delta_D(y)^{\gamma + \eps}} 
		+ \int_{R^{-\eps}\delta_D(y)^{\gamma + \eps}}^
		{R^\gamma}  + \int_{
			R^\gamma }^\infty \bigg) q^{\gamma,\R^d \setminus \wt S}(t, z,y) 
		t^{-1-\beta}dt \\
		&\le c_\beta\sup_{z \in \R^d \setminus (S \cup \wt S)}   \int_{0}^{R^{-\eps} \delta_D(y)^{\gamma + \eps}} 
		q^{\gamma,\R^d \setminus \wt S}(t, z,y) t^{-1-\beta}dt \\
		& \quad + c_\beta	\sup_{0<s\le
			R^\gamma,\,z \in \R^d \setminus (S \cup \wt S)}q^{\gamma,\R^d \setminus \wt S}(s, z,y) 
		\int_{R^{-\eps}\delta_D(y)^{\gamma + \eps}}^{
			R^\gamma}\frac{dt}{t^{1+\beta}} + c \int_{
			R^\gamma}^\infty \frac{dt}{t^{d/\gamma+1+\beta}}\\
		&\le c(\delta_D(y)/R)^{(1-\beta)\eps} \delta_D(y)^{-d-\alpha} 
		+   cR^{-\eps} \delta_D(y)^{-d+\eps} (R^{-\eps}\delta_D(y)^{\gamma + \eps})^{-\beta} +c  R^{-d-\alpha}\\
		&=	c(\delta_D(y)/R)^{(1-\beta)\eps} \delta_D(y)^{-d-\alpha}  +c  R^{-d-\alpha}.
	\end{align*}  
	Since $\delta_D(y) <R
	$, we have 
	$(\delta_D(y)/R)^{(1-\beta)\eps} \delta_D(y)^{-d-\alpha} 
	>R^{-d-\alpha}$. The proof is complete.	 
\end{proof}

\textsc{Proof of Proposition \ref{p:example-subordinate-killed}.}  \hyperlink{B1}{{\bf (B1)}} clearly holds. 
As mentioned earlier,  \hyperlink{B3}{{\bf (B3)}} follows from \cite[(2.9)]{KSV}.
  Using \cite[Theorem 3.4]{So04} and \cite[Theorem 1.1]{CKP} if $\gamma=2$, and \cite[Theorem 1.1]{CKS10} and  \cite[Theorem 1.2]{CT11}  if $\gamma<2$, we see that $Z^{\gamma, D}$ satisfies either the condition {\bf HK$^{\text h}_{\text B}$} (if $D$ is bounded) 
or {\bf HK$^{\text h}_{\text U}$} (if $D$ is unbounded) in \cite{CKSV22} with $\Phi(r)=r^\gamma$, $C_0=\1_{\gamma \ne 2}$ and the boundary function 
	\begin{align}\label{e:def-hD-gamma}
		h^D_\gamma(t,x,y)= 	\bigg( 1 \wedge \frac{\delta_{D}(x)^{\gamma/2}}{t^{1/2}} \bigg)
			\bigg( 1 \wedge \frac{\delta_{D}(y)^{\gamma/2}}{t^{1/2}} \bigg).
	\end{align}
Thus, by \cite[Example 7.2]{CKSV22} (see also \cite[(1.1) and (1.2)]{KSV}), \hyperlink{B4-c}{{\bf (B4-c)}} holds with 
$\Phi_1=\Phi_1^{\gamma,\beta}$, $\Phi_2=\Phi_2^{\gamma,\beta}$ and  $\ell=\ell^{\gamma,\beta}$.
 For \hyperlink{B5-I}{{\bf (B5-I)}}, using  \eqref{e:subordinate-killed-F} and Lemma \ref{l:gamma-beta-extra-decay}, we get that  for all   $\nu \in (0,1)$ and $x,y \in E_{\nu}$,
 \begin{align*}		&\Big|\sB^D(x,y)- c_{d,-\alpha}F_0^{\gamma,\beta}((y-x)/x_d)\Big|= |x-y|^{d+\alpha} \left|J^D(x,y)-J^\bH(x,y)\right| \\	&\le c \bigg(	\frac{ |x-y|}{ \delta_D(x) \vee  \delta_D(y)}\bigg)^{d+\alpha} 	\bigg(\frac{\delta_D(x) \vee \delta_D(y)}{R}\bigg)^{\tfrac{(1-\beta)(1-\nu)\gamma}{2(1+\nu)}}\\		
 		&\le c \bigg(\frac{\delta_D(x) \vee \delta_D(y) \vee  |x-y|}{\delta_D(x) \wedge \delta_D(y) \wedge |x-y|} \bigg)^{d+\alpha}		\bigg( \frac{(\delta_D(x)\vee\delta_D(y) \vee |x-y|)}{R}\bigg)^{\tfrac{(1-\beta)(1-\nu)\gamma}{2(1+\nu)}}.
 \end{align*}
Hence, \hyperlink{B5-I}{{\bf (B5-I)}} holds true with  any $\nu\in (0,1)$, $\theta_1=d+\alpha$ 
and $\theta_2=(1-\beta)(1-\nu)\gamma/(2+2\nu)$. 

In view of \eqref{e:subordinate-killed-F-2}, it remains to  prove that  $\kappa^D$ satisfies \hyperlink{K3}{{\bf (K3)}} and \eqref{e:C(alpha,p,F)} with $\F=F_0^{\gamma,\beta}$ and $p=\gamma/2$. 
It is known, see \cite[(3.2)]{Pr81}, that there exists 
$
c_*>0$ such that for all  $w \in \R^d$, $r>0$ and $t>0$,
	\begin{align}\label{e:gamma-stable-EP}
		\P_w (\tau^{(\gamma)}_{B(w,r)} \le t)
		=\P_0 (\tau^{(\gamma)}_{B(0,r)} \le t)
		=\P_0 (\max_{0 \le s\le t}|Z^{\gamma}_s|  \ge r)
		 \le 
c_*tr^{-\gamma}.
	\end{align}	
From \eqref{e:gamma-stable-EP}, we see that
	\begin{align}\label{e:gamma-lifetime1}
		 \P_x(\tau^{(\gamma)}_D \le t ) 
			\le \P_x (\tau^{(\gamma)}_{B(x, \delta_D(x))}\le t ) \le 
c_* t \delta_D(x)^{-\gamma}.
	\end{align}
Applying 
\eqref{e:gamma-lifetime1} to \eqref{e:subordinate-killed-killing}, we have that for all $x \in D$ with $\delta_D(x) \ge R/2$,	
\begin{align}\label{e:kappanew8}
		\kappa^D(x) &
		\le  c_\beta \int_0^{\delta_D(x)^\gamma}\P_x(\tau^{(\gamma)}_D \le t ) t^{-1-\beta} dt  
			+ c_\beta \int_{\delta_D(x)^\gamma}^\infty t^{-1-\beta}dt\\
			&\le c\delta_D(x)^{-\gamma} \int_0^{\delta_D(x)^\gamma} t^{-\beta}dt + c \delta_D(x)^{-\alpha} = c\delta_D(x)^{-\alpha} 
			 	\le c R^{-\alpha}. \nn
	\end{align}

We now assume that $x \in D$ with $\delta_D(x)<R/2$. 
Without loss of generality, by choosing $Q_x\in \partial D$ such that $|x-Q_x|=\delta_D(x)$, we assume that $x=(\wt 0,x_d)=(\wt 0, \delta_D(x))$ in CS$_{Q_x}$ and denote $E^{Q_x}_{\nu}(R)$, $S^{Q_x}(R)$ and $\wt S^{Q_x}(R)$ by $E_{\nu}$, $S$ and $\wt S$ respectively.

Repeating the proof of \cite[Lemma 2.4(i)]{CKSV23}, we see that there exists  $c_1=c_1(\gamma,\beta)>0$ independent of $x$ such that  
	\begin{align}\label{e:example-kappa-1}
		\kappa^\bH(x)=c_1x_d^{-\alpha}= c_1\delta_D(x)^{-\alpha}.
	\end{align}
Recall $b_{\gamma,\beta}$ is defined in \eqref{e:defb}. We see that 
 \hyperlink{B4-a}{{\bf (B4-a)}} and \hyperlink{B4-b}{{\bf (B4-b)}} hold with $\Phi_0(r)=(r\wedge 1)^{b_{\gamma,\beta}}\ell(r)$ by Lemma \ref{l:A3-c}. Hence, by Lemmas \ref{l:constant} and \ref{l:C=infty},  $q\mapsto C(\alpha,q,F_0^{\gamma,\beta})$ is a strictly increasing continuous function on $[(\alpha-1)_+, \alpha+ b_{\gamma,\beta})$ with $C(\alpha,(\alpha-1)_+,F_0^{\gamma,\beta})=0$ 
and $\lim_{q \to b_{\gamma,\beta}} C(\alpha,q,F_0^{\gamma,\beta})=\infty$. Thus, there exists a unique constant 
$p \in [(\alpha-1)_+, \alpha+b_{\gamma,\beta}) \cap (0,\infty)$ such that
 	\begin{align}\label{e:example-kappa-2}
		C(\alpha,p,F_0^{\gamma,\beta})=c_1/c_{d, -\alpha}=c_1/\sB^D(x,x).
 	\end{align} 
 Set $\eps_1:=(1-\nu)\gamma/(8\beta+8\nu \beta)$ and $\eps_2:=(1-\nu)\gamma^2/(8d+8\nu d)$. 
 	By \eqref{e:example-kappa-1}, \eqref{e:example-kappa-2} and  \eqref{e:subordinate-killed-killing}, we get
 	\begin{align}\label{e:kappanew1}
 		&\big|\kappa^D(x)-
 		C(\alpha,p,F_0^{\gamma,\beta})\sB^D(x,x)
 		\delta_D(x)^{-\alpha}\big|= \left|\kappa^D(x)-\kappa^\bH(x)\right|\\
 		&\le c_\beta\int_0^{R^{-\eps_1} \delta_D(x)^{\gamma+\eps_1}} 
 		\big(	\P_x(\tau^{(\gamma)}_D \le t ) \vee \P_x(\tau^{(\gamma)}_\bH \le t ) \big)
 		t^{-1-\beta}dt\nn\\
 		&\quad + c_\beta\int_{R^{-\eps_1} \delta_D(x)^{\gamma+\eps_1}}^{R^{\eps_2} \delta_D(x)^{\gamma-\eps_2}} \bigg| 
 		\int_D q^{\gamma,D}(t,x,y)dy -  \int_\bH q^{\gamma,\bH}(t,x,y)dy \bigg| \,t^{-1-\beta}dt\nn \\
 		&\quad 
 		+ c_\beta\int_{R^{\eps_2} \delta_D(x)^{\gamma-\eps_2}}^\infty t^{-1-\beta}dt\nn \\
 		&=:I_1+I_2+I_3. \nn
 	\end{align}
 	For $I_3$, we have
 	\begin{align}\label{e:kappanew2}
 		I_3 \le c R^{-\beta\eps_2} 
 		\delta_D(x)^{-\alpha+\beta\eps_2}.
 	\end{align}

 	Using \eqref{e:gamma-stable-EP}, we see that
 	\begin{align}\label{e:gamma-lifetime}
 		\P_x(\tau^{(\gamma)}_D \le t ) \vee \P_x(\tau^{(\gamma)}_\bH \le t )
 		\le \P_x (\tau^{(\gamma)}_{B(x, \delta_D(x))}\le t ) \le 
 		c_* t \delta_D(x)^{-\gamma}.
 	\end{align}
 	Thus, we get
 	\begin{align}\label{e:kappanew3}
 		I_1 \le 
 		c_*\delta_D(x)^{-\gamma}\int_0^{R^{-\eps_1} \delta_D(x)^{\gamma+\eps_1}} t^{-\beta} dt 
 		= c R^{-(1-\beta)\eps_1} \delta_D(x)^{-\alpha+ (1-\beta)\eps_1}.
 	\end{align} 
 	Set $W:=B(x,  R^{\eps_2/\gamma}\delta_D(x)^{1-\eps_2/\gamma})$.  By Lemma \ref{l:C11}(ii), we have that
 	\begin{align}\label{e:kappanew4}
 		I_2 &\le c_\beta \int_{R^{-\eps_1} \delta_D(x)^{\gamma+\eps_1}}^{R^{\eps_2} \delta_D(x)^{\gamma-\eps_2}} 
 		\bigg( \int_{\R^d \setminus \wt S} q^{\gamma,\R^d \setminus \wt S}(t,x,y)dy 
 		-  \int_S q^{\gamma,S}(t,x,y)dy \bigg) \,t^{-1-\beta}dt \\
 		&\le  c_\beta\int_{R^{-\eps_1} \delta_D(x)^{\gamma+\eps_1}}^{R^{\eps_2} \delta_D(x)^{\gamma-\eps_2}}  
 		\int_{W \setminus \wt S} 
 		\left(q^{\gamma,\R^d \setminus \wt S}(t,x,y)- q^{\gamma,S}(t,x,y)\right) dy \,t^{-1-\beta}dt\nn\\
 		&\quad + c_\beta \int_{R^{-\eps_1} \delta_D(x)^{\gamma+\eps_1}}^{R^{\eps_2} \delta_D(x)^{\gamma-\eps_2}}
 		\int_{\R^d \setminus W}q^{\gamma,\R^d \setminus \wt S}(t,x,y) dy \,t^{-1-\beta}dt\nn\\
 		&=:I_{2,1}+I_{2,2}.\nn
 	\end{align}
 	For any 
 	$0<t\le R^\gamma$  and $y \in W\setminus \wt S$, since $x \in E_{\nu}$,
 	 using symmetry, \eqref{e:subordinate-heatkernel-diff} and Lemmas \ref{l:dirichlet-1}-\ref{l:dirichlet-2}, we have
 	\begin{align*}
 		q^{\gamma,\R^d \setminus \wt S}(t,x,y)- q^{\gamma,S}(t,x,y) 
 		&= 	q^{\gamma,\R^d \setminus \wt S}(t,y,x)- q^{\gamma,S}(t,y,x)\\
 		&=\E_y \Big[q^{\gamma,\R^d \setminus \wt S}(t-\tau^{(\gamma)}_S, Z^\gamma(\tau^{(\gamma)}_S),x); \tau^{(\gamma)}_S<t \Big]\\
 		&\le \sup_{0<s\le 
 			R^\gamma,\, z \in \R^d \setminus (S \cup \wt S)} q^{\gamma, \R^d \setminus \wt S}(s ,z, x) \\
 			&\le c
 	\delta_D(y)^{-d} (\delta_D(y)/R)^{\frac{\gamma(1-\nu)}{2(1+\nu)}}.
 	\end{align*}
 	Hence, we obtain
 	\begin{align}\label{e:kappanew5}
 		&	I_{2,1} \le c \delta_D(y)^{-d} (\delta_D(y)/R)^{\frac{\gamma(1-\nu)}{2(1+\nu)}}
		\int_{R^{-\eps_1} 	\delta_D(x)^{\gamma+\eps_1}}^{R^{\eps_2} \delta_D(x)^{\gamma-\eps_2}}  t^{-1-\beta}dt \int_{W}  dy \\
 			&\le c \delta_D(y)^{-d} (\delta_D(y)/R)^{\frac{\gamma(1-\nu)}{2(1+\nu)}}
			(R^{-\eps_1} \delta_D(x)^{\gamma+\eps_1})^{-\beta}( R^{\eps_2/\gamma}	\delta_D(x)^{1-\eps_2/\gamma})^d\nn\\
 				&=c \delta_D(y)^{-\alpha} (\delta_D(y)/R)^{\frac{\gamma(1-\nu)}{4(1+\nu)}}
				 \le c \delta_D(y)^{-\alpha} (\delta_D(y)/R)^{\beta\eps_2 } .\nn
 	\end{align}
 	For $I_{2,2}$, we see from \eqref{e:gamma-stable-EP} that for all $t>0$,
 	\begin{align*}
 		\int_{\R^d \setminus W} q^{\gamma,\R^d \setminus \wt S}(t,x,y) dy 
 		\le 	 \P_x (\tau^{(\gamma)}_{W} \le t) 
 		\le  ct  R^{-\eps_2} \delta_D(x)^{-\gamma+\eps_2}. 
 	\end{align*}
 	It follows that
 	\begin{align}\label{e:kappanew6}
 		&I_{2,2} \le  c  R^{-\eps_2} \delta_D(x)^{-\gamma+\eps_2} \int_{0}^{R^{\eps_2} \delta_D(x)^{\gamma-\eps_2}} t^{-\beta}dt 
 		=c R^{-\beta \eps_2} \delta_D(x)^{-\alpha+\beta\eps_2}.
 	\end{align}
 	Therefore, combining \eqref{e:kappanew1}, \eqref{e:kappanew2} and \eqref{e:kappanew3}--\eqref{e:kappanew6}, we 
 	get 
 	\begin{align*}
 		\big|\kappa^D(x)-
 		C(\alpha,p,F_0^{\gamma,\beta})\sB^D(x,x)
 		\delta_D(x)^{-\alpha}\big| \le c(R) \delta_D(x)^{-\alpha+\eta_0}
 	\end{align*}
 	where 
	 $\eta_0 := (1-\beta)\eps_1 \wedge \beta\eps_2>0$.  
 	From this and \eqref{e:kappanew8}, we conclude that \hyperlink{K3}{{\bf (K3)}} holds.

Lastly,	by comparing Theorem \ref{t:Green} with \cite[(7.10)]{CKSV22}, we deduce from \eqref{e:example-kappa-2} that \eqref{e:C(alpha,p,F)}  holds with $\F=F_0^{\gamma,\beta}$ and $p=\gamma/2$. The proof is complete. 
\qed

\bigskip
	
Below, we present two more examples, which are generalizations of the process $Y^D$ defined in \eqref{e:def-subordinate-killed}.

 Recall that the functions 
  $J^{\gamma,U,\beta}(x,y)$, $\kappa^{\gamma,D,\beta}$,     
$\sB^{\gamma,D,\beta}$, $F_0^{\gamma,\beta}$, $\Phi_1^{\gamma,\beta}$, $\Phi_2^{\gamma,\beta}$ 
and $\ell^{\gamma,\beta}$ are defined by 
\eqref{e:subordinate-killed-jump}
-\eqref{e:subordinate-killed-F-1} 
and \eqref{e:example-Phi-1}-\eqref{e:example-ell} respectively.  We also recall that the jump kernel of the  isotropic $\alpha$-stable process has density $c_{d,-\alpha}|x-y|^{-d-\alpha}$, and the $\beta$-stable subordinator with Laplace exponent $\lambda^\beta$  has  L\'evy density $c_\beta t^{-1-\beta}$.
    
\begin{example}\label{ex:sum-of-processes}  
Let $\alpha \in (0,2)$, $m \ge 2$ and $0<\gamma_1< \cdots< \gamma_m \le 2$.
Set $\beta_i:=\alpha/\gamma_i$ for $1\le i\le m$. Consider a process  $\wt Y$  corresponding to the generator 
	$$
		L=\sum_{i=1}^m-((-\Delta)^{\gamma_i/2}|_D)^{\beta_i}.
	$$
$\wt Y$ is an independent sum of subordinate killed stable processes whose infinitesimal generators have  
the same fractional order $\alpha$.
Note that the jump kernel  and the killing measure  of $\wt Y$ have densities $\wt J(x,y)$ and  $\wt \kappa(x)$ given by 
$\wt J(x,y) = \sum_{i=1}^mJ^{\gamma_i,D,\beta_i}(x,y)$ and $\wt\kappa(x)=\sum_{i=1}^m\kappa^{\gamma_i,D,\beta_i}(x)$.  	Set 
	$$
		\wt B(x,y):=\sum_{i=1}^{m}	\sB^{\gamma_i,D,\beta_i}(x,y).
	$$ 
In the following, we show that $\wt Y$ satisfies \hyperlink{B1}{{\bf (B1)}}, \hyperlink{B3}{{\bf (B3)}}, \hyperlink{B4-c}{{\bf (B4-c)}} with  
$\Phi_1=\Phi_1^{\gamma_1,\beta_1}$, $\Phi_2=\Phi_2^{\gamma_1,\beta_1}$ and $\ell=\ell^{\gamma_1,\beta_1}$,   \hyperlink{K3}{{\bf (K3)}}, and  \hyperlink{B5-I}{{\bf (B5-I)}} and \eqref{e:C(alpha,p,F)} with  
$\F_0=\F=\frac{1}{m}\sum_{i=1}^m F_0^{\gamma_i,\beta_i}$ and some $p \in (\gamma_1/2, \gamma_m/2)$.
  		
\medskip
  		
\noindent \hyperlink{B1}{{\bf (B1)}}: By symmetry,	\hyperlink{B1}{{\bf (B1)}} clearly holds.

\smallskip		

\noindent \hyperlink{B3}{{\bf (B3)}}: Since each $\sB^{\gamma_i,D,\beta_i}$, $1\le i \le m$, satisfies \hyperlink{B3}{{\bf (B3)}} by Proposition \ref{p:example-subordinate-killed} and
\begin{align*}
	|\wt \sB(x,x)-\wt \sB(x,y)|& \le  \sum_{i=1}^m	|\sB^{\gamma_i,D,\beta_i}(x,x)-\sB^{\gamma_i,D,\beta_i}(x,y)|,
\end{align*}
$\wt \sB$ satisfies \hyperlink{B3}{{\bf (B3)}}.

\smallskip
  		
\noindent \hyperlink{B4-c}{{\bf (B4-c)}}:  For $x,y \in D$, we let
  	\begin{align*}
  			&r_1^{x,y}:=\frac{\delta_D(x) \wedge \delta_D(y)}{|x-y|}, \qquad r_2^{x,y}:=\frac{\delta_D(x) \vee \delta_D(y) }{|x-y|}\\
  			& \text{and} \quad  r_3^{x,y}:=\frac{\delta_D(x) \wedge \delta_D(y)}{(\delta_D(x)\vee\delta_D(y)) \wedge |x-y|}.
  	\end{align*}	
 By Proposition \ref{p:example-subordinate-killed}, we get that for all $x,y \in D$,
  	\begin{align*}
  			\wt \sB(x,y) \asymp \sum_{i=1}^m \Phi_1^{\gamma_i,\beta_i}(r_1^{x,y}) \, 
  			\Phi_2^{\gamma_i,\beta_i}(r_2^{x,y}) \, \ell^{\gamma_i,\beta_i}(r_3^{x,y}).
  	\end{align*}
Hence, it suffices to show that there exists $c_1>0$ such that for all $2\le i\le m$ and $x,y \in D$,
  	\begin{align}\label{e:sum-of-processes-1}
  		I:=	\frac{\Phi_1^{\gamma_i,\beta_i}(r_1^{x,y}) \, \Phi_2^{\gamma_i,\beta_i}(r_2^{x,y}) \, \ell^{\gamma_i,\beta_i}(r_3^{x,y})}
  		{\Phi_1^{\gamma_1,\beta_1}(r_1^{x,y}) \, \Phi_2^{\gamma_1,\beta_1}(r_2^{x,y}) \, \ell^{\gamma_1,\beta_1}(r_3^{x,y})}\le c_1.
  	\end{align}
It suffices to prove \eqref{e:sum-of-processes-1} for $i=2$.   Since $\gamma_1<\gamma_2$, we have $\beta_1>\beta_2$.
If $\beta_1<1/2$, then 
	\begin{align*}
		I\le 	\frac{(r_1^{x,y}  \wedge 1)^{\gamma_2/2} \, (r_2^{x,y} \wedge 1)^{\gamma_2/2-\alpha}}{ (r_1^{x,y} 
			 \wedge 1)^{\gamma_1/2}\, (r_2^{x,y}  \wedge 1)^{\gamma_1/2-\alpha}} \le 1.
	\end{align*}
If $\beta_1 \ge 1/2> \beta_2$, then 
  	\begin{align*}
  		I\le 	\frac{(r_1^{x,y}  \wedge 1)^{\gamma_2/2} \, (r_2^{x,y} \wedge 1)^{\gamma_2/2-\alpha}}{ (r_1^{x,y}  
  		\wedge 1)^{\gamma_1 -\alpha}} \le 	\frac{(r_1^{x,y}  \wedge 1)^{\gamma_2/2}}{ (r_1^{x,y}  \wedge 1)^{\gamma_1/2}}  \le 1.
  	\end{align*}
 If $\beta_1> 1/2=\beta_2$, then 
  	\begin{align*}
  		I&\le 	\frac{(r_1^{x,y}  \wedge 1)^{\gamma_2-\alpha}}{ (r_1^{x,y}  \wedge 1)^{\gamma_1 -\alpha}}
  			\log \bigg( \frac{e}{r_3^{x,y}\wedge 1}\bigg) = (r_1^{x,y}  \wedge 1)^{\gamma_2-\gamma_1}
  			\log \bigg( \frac{e ( r_2^{x,y} \wedge 1)}{r_1^{x,y}\wedge 1}\bigg)  \\
  		&\le  (r_1^{x,y}  \wedge 1)^{\gamma_2-\gamma_1}\log \bigg( \frac{e}{r_1^{x,y}\wedge 1}\bigg)  
  			\le \sup_{0<s \le 1} s^{\gamma_2-\gamma_1} \log (e/s)=c_2.
  \end{align*}
If $\beta_1>\beta_2 > 1/2$, then 
  	\begin{align*}
		I&\le 	\frac{(r_1^{x,y}  \wedge 1)^{\gamma_2-\alpha}}{ (r_1^{x,y}  \wedge 1)^{\gamma_1 -\alpha}}\le 1.
	\end{align*}
Thus, \eqref{e:sum-of-processes-1} holds.

\smallskip
  		
\noindent \hyperlink{B5-I}{{\bf (B5-I)}}: Note that 
$\wt \sB(x,x)= m\sB^{\gamma_1,D,\beta_1}(x,x)=\cdots =m\sB^{\gamma_m,D,\beta_m}(x,x)$ for all $x\in D$. Hence, for all $x,y \in D$, we have
  	\begin{align*}
  		&\Big|\wt\sB(x,y)- \frac{1}{m}\wt \sB(x,x) \sum_{i=1}^m F_0^{\gamma_i,\beta_i}((y-x)/x_d)  \Big|\\
  		&\le \sum_{i=1}^m\Big|\sB^{\gamma_i,D,\beta_i}(x,y)-\sB^{\gamma_i,D,\beta_i}(x,x)  
  				F_0^{\gamma_i,\beta_i}((y-x)/x_d)  \Big|.
  	\end{align*}
Therefore,  since each $\sB^{\gamma_i,D,\beta_i}$, $1\le i \le m$, satisfies \hyperlink{B5-I}{{\bf (B5-I)}} by Proposition \ref{p:example-subordinate-killed}, $\wt \sB$ satisfies \hyperlink{B5-I}{{\bf (B5-I)}}.
  	
\smallskip
  	
\noindent \hyperlink{K3}{{\bf (K3)}} and \eqref{e:C(alpha,p,F)}:  By Proposition \ref{p:example-subordinate-killed}, 
each $\kappa^{\gamma_i,D,\beta_i}$, $1\le i\le m$, satisfies  \hyperlink{K3}{{\bf (K3)}} and \eqref{e:C(alpha,p,F)} with $p=\gamma_i/2$ and $\F=F_0^{\gamma_i,\beta_i}$. Hence, since  $\wt \sB(x,x)= m\sB^{\gamma_1,D,\beta_1}(x,x)=\cdots =m\sB^{\gamma_m,D,\beta_m}(x,x)$ for all $x \in D$, one sees that $\wt \kappa$ satisfies \hyperlink{K3}{{\bf (K3)}} 
with $C_9=\sum_{i=1}^m C(\alpha, \gamma_i/2, F_0^{\gamma_i,\beta_i})$. 
  	
Set $b:= \gamma_1/2$ if $\gamma_1=2$ or $\beta_1<1/2$, and set $b:=\gamma_1-\alpha_1$ otherwise. Then \hyperlink{B4-a}{{\bf (B4-a)}} and \hyperlink{B4-b}{{\bf (B4-b)}} hold with $\Phi_0(r)=(r\wedge 1)^{b}\ell^{\beta_1,\gamma_1}(r)$ by Lemma \ref{l:A3-c}. Using this and Lemmas \ref{l:constant} and \ref{l:C=infty},  we deduce that	 there exists a unique constant $
p \in [(\alpha-1)_+, \alpha+b) \cap (0,\infty)$ such that
  	\begin{align*}
  		C(\alpha,p,\frac{1}{m}\sum_{i=1}^m F_0^{\gamma_i,\beta_i})
  			= \frac{1}{m}\sum_{i=1}^m C(\alpha, \gamma_i/2, F_0^{\gamma_i,\beta_i}).
  	\end{align*} 
Using \eqref{e:def-killing-constant} and Lemma \ref{l:constant}, we also get that
	$$
  		C(\alpha,\gamma_1/2,\frac{1}{m}\sum_{i=1}^m F_0^{\gamma_i,\beta_i}) 
  			=\frac{1}{m}\sum_{i=1}^mC(\alpha, \gamma_1/2, F_0^{\gamma_i,\beta_i})<\frac{1}{m}\
  			sum_{i=1}^mC(\alpha, \gamma_i/2, F_0^{\gamma_i,\beta_i})
  	$$
  and
  	$$
  		C(\alpha,\gamma_m/2,\frac{1}{m}\sum_{i=1}^m F_0^{\gamma_i,\beta_i}) 
  			=\frac{1}{m}\sum_{i=1}^mC(\alpha, \gamma_m/2, F_0^{\gamma_i,\beta_i})>\frac{1}{m}
  			\sum_{i=1}^mC(\alpha, \gamma_i/2, F_0^{\gamma_i,\beta_i}).
  	$$
By Lemma \ref{l:constant}, it follows that $p \in (\gamma_1/2, \gamma_m/2)$.
\end{example}

\smallskip
		
\begin{example}\label{ex:subordinator} 
Suppose also that $D$ is bounded. Let $\gamma \in (0,2]$,  $\beta\in (0,1)$ and $\alpha:=\gamma\beta$. Let $\phi$ be a Bernstein function with  L\'evy triplet $(0,0,\Pi)$. Assume  that  $\Pi(dt)$ has a density $\Pi(t)dt$ satisfying the following property: 
			
\smallskip
			
\noindent \textit{There exist constants $t_0>0$ and $\theta \in (\gamma^{-1}(\alpha-1)_+, 1)$, and a  $\theta$-H\"older continuous function $k:(0,t_0) \to (0,\infty)$ with $k(0+) \in (0,\infty)$ such that}
		\begin{align}\label{e:def-mu-example}
			\Pi(t) = k(t) t^{-1-\beta} \quad \text{for} \;\, t \in (0,t_0).
		\end{align}
		
\noindent	Examples of such Bernstein functions include   $c_\beta\int_0^1 (1-e^{-\lambda t})t^{-1-\beta}dt$ and	
$(\lambda + m)^{\beta}-m^\beta$ $(m>0)$. We refer to \cite[Section 16]{SSV12} for more examples.
			
\smallskip
			
Let $Y^{\gamma,D,\phi}$ be a process corresponding to the generator 
		 $$
			 L=-\phi((-\Delta)^{\gamma/2}|_D).
		 $$ 
Equivalently, define $Y^{\gamma,D,\phi}$ by $Y^{\gamma,D,\phi}_t=Z^{\gamma,D}_{T^\phi_t}$, where $T^\phi$ is a subordinator with Laplace exponent $\phi$  independent of $Z^\gamma$.  	According to \cite[(2.8)-(2.9)]{Oku}, the jump kernel  and the killing measure  of $Y^{\gamma,D,\phi}$ have densities $J^{\gamma,D,\phi}(x,y)$ and  $\kappa^{\gamma,D,\phi}(x)$ given by 
		\begin{align}\label{e:subordinate-killed-general}
				J^{\gamma,D,\phi}(x,y)&=\int_0^\infty q^{\gamma,D}(t,x,y) \Pi(t)dt,\nn\\
				\kappa^{\gamma,D,\phi}(x)&=\int_0^\infty \bigg(1- \int_D q^{\gamma,D}(t,x,y)dy \bigg) \Pi(t)dt. 
		\end{align}

Define for $x,y \in D$,
		$$
			\sB^{\gamma,D,\phi}(x,y)=\begin{cases}
				|x-y|^{d+\alpha}J^{\gamma, D,\phi}(x,y) &\mbox{ if } x \neq y,\\
				c_{d,-\alpha}k(0+)/c_\beta &\mbox{ if } x = y.
			\end{cases}
		$$ 
In this example, we prove that $Y^{\gamma,D,\phi}$ satisfies \hyperlink{B1}{{\bf (B1)}},  \hyperlink{B3}{{\bf (B3)}}, \hyperlink{B4-c}{{\bf (B4-c)}} with  
$\Phi_1=\Phi_1^{\gamma,\beta}$, $\Phi_2=\Phi_2^{\gamma,\beta}$ and $\ell=\ell^{\gamma,\beta}$,  \hyperlink{K3}{{\bf (K3)}}, and  \hyperlink{B5-I}{{\bf (B5-I)}} and \eqref{e:C(alpha,p,F)} with  $\F_0=\F=F_0^{\gamma,\beta}$ and $p=\gamma/2$.
	
\medskip
				
\noindent \hyperlink{B1}{{\bf (B1)}}: Since $ q^{\gamma,D}(t,x,y)= q^{\gamma,D}(t,y,x)$, 	\hyperlink{B1}{{\bf (B1)}} clearly holds.
		
		\smallskip
		
\noindent \hyperlink{B4-c}{{\bf (B4-c)}}: Recall that $Z^{\gamma, D}$  satisfies  the condition {\bf HK$^{\text h}_{\text B}$}  in \cite{CKSV22} with $\Phi(r)=r^\gamma$, $C_0=\1_{\gamma \ne 2}$ and the boundary function $	h^D_\gamma(t,x,y)$ defined 
in \eqref{e:def-hD-gamma}. By \eqref{e:def-mu-example}, 
		\begin{align}\label{e:tail-mu-example}
			\Pi((t,\infty)) \asymp t^{-\beta} \quad \text{for} \;\, t \in (0,t_0/2).
		\end{align}
Thus, by \cite[Example 7.2]{CKSV22},   \hyperlink{B4-c}{{\bf (B4-c)}} holds with 
$\Phi_1=\Phi^{\gamma,\beta}_1, \Phi_2=\Phi^{\gamma,\beta}_2$ and $\ell=\ell^{\gamma,\beta}$.  
		
\smallskip
		
\noindent \hyperlink{B3}{{\bf (B3)}} and \hyperlink{B5-I}{{\bf (B5-I)}}: 	 Set $a_0:=k(0+)/c_\beta$. Then
		\begin{align}\label{e:example-diagonal}
			\sB^{\gamma,D,\phi}(x,x)=a_0 \sB^{\gamma,D,\beta}(x,x) = a_0 c_{d,-\alpha} 
			\quad \text{for all} \;\, x \in D.
		\end{align}
Since $t^{-d/2}e^{-r^2/(4t)} \le ct^{-d/2} (t/r^2)^{(d+2)/2}$ for all $t,r>0$, we have
		\begin{align}\label{e:hke-free-process}
			q^{\gamma,D}(t,x,y) \le  	q^{\gamma,\R^d}(t,x,y) \le c \bigg( t^{-d/\gamma} \wedge \frac{t}{|x-y|^{d+\gamma}}\bigg), 
			 \quad t>0, \, x,y \in D.
		\end{align}
By \eqref{e:subordinate-killed-jump} and \eqref{e:subordinate-killed-general}, we see that for all $x,y \in D$ with $x \ne y$,
		\begin{align*}
			& |	a_0\sB^{\gamma, D, \beta}(x,y) - 	\sB^{\gamma, D,\phi}(x,y)|\,|x-y|^{-d-\alpha}  \\
			&\le  \bigg( \int_0^{t_0} +\int_{t_0}^{\infty} \bigg) \, q^{\gamma,D}(t,x,y) \left|k(0+) t^{-1-\beta} - \Pi(t)\right|dt
			 =:I_1+I_2.
		\end{align*}
For $I_1$, using  \eqref{e:hke-free-process} and  \eqref{e:def-mu-example}, we get
	\begin{align*}
		I_1 &\le \frac{c}{|x-y|^{d+\gamma}}\int_0^{|x-y|^\gamma }  t^{-\beta} |k(0+)-k(t)|dt + c\int_{|x-y|^\gamma }^{t_0} 
			t^{-d/\gamma -1 - \beta} |k(0+)-k(t)|  dt\\
		&\le  \frac{c}{|x-y|^{d+\gamma}}\int_0^{|x-y|^\gamma }   t^{-\beta+\theta}dt + c\int_{|x-y|^\gamma }^{t_0}  
			t^{-d/\gamma -1 - \beta+\theta} dt \le c|x-y|^{-d-\alpha+\gamma\theta}.
	\end{align*}
Further,	 by \eqref{e:hke-free-process}, we have 
		$$
			I_2 \le c t_0^{-d/\gamma} \int_{t_0}^\infty (t^{-1-\beta} + \Pi(t))dt 
				= c \le c (\text{diam}(D))^{d+\alpha-\gamma\theta}\,|x-y|^{-d-\alpha+\gamma\theta}.
		$$ 
Therefore,  we deduce that 
		\begin{align}\label{e:example-diff-jump-phi-beta}
			|a_0	\sB^{\gamma,D,\beta}(x,y) - 	\sB^{\gamma, D,\phi}(x,y)|\le  c|x-y|^{\gamma\theta} \quad \text{for all} \;\; x,y \in D.
		\end{align}
	
If $\alpha \ge 1$, then	by \eqref{e:example-diagonal} and  \eqref{e:example-diff-jump-phi-beta}, since $\sB^{\gamma,D,\beta}$ satisifes \hyperlink{B3}{{\bf (B3)}} by Proposition \ref{p:example-subordinate-killed}, there exists a constant $\theta_0>\alpha-1$  such that for all $x,y \in D$, $x \ne y$, with $|x-y| < \delta_D(x) \wedge \delta_D(y)$,
	\begin{align*}
		&| \sB^{\gamma,D, \phi}(x,x) - \sB^{\gamma, D,\phi}(x,y)| = | a_0 \sB^{\gamma, D, \beta}(x,x) - \sB^{\gamma,D, \phi}(x,y) | \\
		& \le a_0| \sB^{\gamma, D, \beta}(x,x) - \sB^{\gamma, D, \beta}(x,y)| 
			+ | a_0\sB^{\gamma, D, \beta}(x,y) - \sB^{\gamma, D, \beta}(x,y)| \\
		&		\le  c  \bigg(\frac{|x-y|}{ \delta_D(x) \wedge \delta_D(y) \wedge R} \bigg)^{\theta_0} 
			+cR^{\gamma\theta}\bigg(\frac{|x-y|}{R} \bigg)^{\gamma\theta} 	
			\le  c  \bigg(\frac{|x-y|}{ \delta_D(x) \wedge \delta_D(y) \wedge R} \bigg)^{\theta_0 \wedge (\gamma\theta)}.
	\end{align*}
Since $\gamma\theta>(\alpha-1)_+$, by Remark \ref{r:A4},  we deduce that  \hyperlink{B3}{{\bf (B3)}} holds.
	
On the other hand, by \eqref{e:example-diagonal}, it holds that for all $x,y \in D$ and $z \in \bH_{-1}$,
		\begin{align*}
				&|	\sB^{\gamma,D, \phi}(x,y) -  a_0c_{d,-\alpha} F_0^{\gamma,\beta}(z) | \\
				&\le |\sB^{\gamma,D,\phi}(x,y)- a_0 \sB^{\gamma,D, \beta}(x,y) |
					+  a_0| \sB^{\gamma,D, \beta}(x,y) - \sB^{\gamma,D,\beta}(x,x)F_0^{\gamma,\beta}(z) | .
		\end{align*}
Since $\sB^{\gamma,D,\beta}$ satisfies \hyperlink{B5-I}{{\bf (B5-I)}} with $\F_0=\F=F_0^{\gamma,\beta}$ by Proposition 
\ref{p:example-subordinate-killed} and
		\begin{align*}
			 |\sB^{\gamma,D,\phi}(x,y)- a_0 \sB^{\gamma,D, \beta}(x,y) | 
			 	\le  cR^{\gamma\theta}  \bigg(\frac{\delta_D(x) \vee \delta_D(y) \vee |x-y|}{R} \bigg)^{\gamma\theta}
		\end{align*}
by \eqref{e:example-diff-jump-phi-beta}, we conclude that $\sB^{\gamma, D,\phi}$  satisfies  \hyperlink{B5-I}{{\bf (B5-I)}} 
with $\F_0=\F=F_0^{\gamma,\beta}$.
		
\smallskip

\noindent \hyperlink{K3}{{\bf (K3)}} and \eqref{e:C(alpha,p,F)}:   Set $\eps:=(\gamma \theta / (\theta+1)) \wedge (\alpha/\beta)$ and $R_2:=t_0/(2\,\text{diam}(D)^{\gamma-\eps})$.
Choose any $x \in D$. By \eqref{e:subordinate-killed-killing} and  \eqref{e:subordinate-killed-general}, we have
\begin{align*}
	&|a_0\kappa^{\gamma,D, \beta}(x) - 	\kappa^{\gamma,D, \phi}(x)| \nn\\
	&\le \bigg( \int_0^{R_2\delta_D(x)^{\gamma-\eps}}+\int_{R_2\delta_D(x)^{\gamma-\eps}}^\infty\bigg) \,
		\bigg(1- \int_D q^{\gamma,D}(t,x,y)dy \bigg)\, |k(0+) t^{-1-\beta} - \Pi(t)| dt \nn\\
	&=:I_1+I_2.
\end{align*}
By using \eqref{e:gamma-lifetime} and   \eqref{e:def-mu-example}, since 
$R_2\delta_D(x)^{\gamma-\eps} \le R_2 \, \text{diam}(D)^{\gamma-\eps} <t_0$, we have
	\begin{align*}
		I_1	&\le c\delta_D(x)^{-\gamma}\int_0^{R_2 \delta_D(x)^{\gamma-\eps}}  t^{\theta-\beta}dt  
			= c \delta_D(x)^{(\gamma-\eps)\theta - \alpha+\beta\eps-\eps }\le  c\delta_D(x)^{-\alpha+ \beta\eps}.
	\end{align*}
Further, we get from \eqref{e:tail-mu-example} that
	\begin{align*}
		I_2	&\le   \int_{R_2 \delta_D(x)^{\gamma-\eps}}^\infty k(0+)t^{-1-\beta}dt 
			+  \int_{R_2 \delta_D(x)^{\gamma-\eps}}^\infty  \Pi(t)dt  \le  c \delta_D(x)^{-\alpha+ \beta\eps}.
	\end{align*}
Combining the  two displays above, we obtain
	$$
		|a_0\kappa^{\gamma,D, \beta}(x) - 	\kappa^{\gamma, D,\phi}(x)|  \le c\delta_D(x)^{-\alpha+ \beta\eps}.
	$$
Using this and \eqref{e:example-diagonal}, we arrive at
	\begin{align*}
		&|\kappa^{\gamma,D,\phi}(x) - C(\alpha,\gamma/2, F_0^{\gamma,\beta})\sB^{\gamma,\phi}(x,x) \delta_D(x)^{-\alpha}| \\
		&\le a_0 |\kappa^{\gamma,D,\beta}(x) - C(\alpha,\gamma/2, F_0^{\gamma,\beta})\sB^{\gamma,D,\beta}(x,x) 
			\delta_D(x)^{-\alpha}|  + c\delta_D(x)^{-\alpha+\beta\eps}.
	\end{align*}
Therefore, since $Y^{\gamma,D,\beta}$ satisfies \hyperlink{K3}{{\bf (K3)}} and \eqref{e:C(alpha,p,F)} with $\F=F_0^{\gamma,\beta}$ and $p=\gamma/2$ by Proposition \ref{p:example-subordinate-killed}, 
one can conclude that   $Y^{\gamma,D,\phi}$ also satisfies \hyperlink{K3}{{\bf (K3)}} and \eqref{e:C(alpha,p,F)} 
with $\F=F_0^{\gamma,\beta}$ and $p=\gamma/2$. 
\end{example}

\subsection{General jump kernels with explicit boundary functions}
\label{s-example-general}

We start with a technical lemma that compares some quantities in $C^{1,1}$ open set with their analogs in the half-space $\bH$.

\begin{lemma}\label{l:general-coefficient}
Let $\nu \in (0,1]$. The following statements hold.
	
\noindent (i) For any $Q \in \partial D$ and $x,y \in E^Q_\nu$, we have
	\begin{align*}
		&\left|\frac{\delta_D(x) \wedge \delta_D(y)}{|x-y|}- \frac{x_d \wedge y_d}{|x-y|} \right| \vee 	
			\left|\frac{\delta_D(x) \vee \delta_D(y)}{|x-y|}- \frac{x_d \vee y_d}{|x-y|} \right|  \\
		&\le \frac{\delta_D(x) \vee \delta_D(y)}{|x-y|}\bigg(\frac{\delta_D(x) \vee \delta_D(y)}{R}\bigg)^{(1-\nu)/(1+\nu)}.
	\end{align*}
	(ii) There exists $C=C(\nu)>0$ such that	for any $Q \in \partial D$ and  $x,y \in E^Q_\nu$,
	\begin{align*}
		&\left|\frac{\delta_D(x) \wedge \delta_D(y)}{(\delta_D(x)\vee\delta_D(y))\wedge |x-y|}
			- \frac{x_d \wedge y_d}{(x_d \vee y_d) \wedge |x-y|} \right|\\
		& \le C\bigg(\frac{\delta_D(x) \vee \delta_D(y)}{(\delta_D(x) \vee  \delta_D(y)) \wedge |x-y|}\bigg)^2 
			\bigg(\frac{\delta_D(x) \vee \delta_D(y)}{R}\bigg)^{(1-\nu)/(1+\nu)}.
	\end{align*}
\end{lemma}
\begin{proof}  Let $x,y \in E^Q_{\nu}$. Without loss of generality, we assume that $Q=0$ and $\delta_D(x) \le \delta_D(y)$. By  Lemma \ref{l:C11}(i), (iii), since $x,y \in E^0_\nu$, we have
	\begin{align}\label{e:example2-dist}
		|\delta_D(x)-x_d| \vee |\delta_D(y)-y_d|  &\le R^{-1}(|\wt x| \vee |\wt y|)^2 \\
		&\le 4^{-2/(1+\nu)} R^{-(1-\nu)/(1+\nu)}(x_d\vee y_d)^{2/(1+\nu)}  \nn\\
		&\le 3^{-2/(1+\nu)}R^{-(1-\nu)/(1+\nu)} (\delta_D(x)\vee \delta_D(y))^{2/(1+\nu)}\nn\\
		& = 3^{-2/(1+\nu)}R^{-(1-\nu)/(1+\nu)}  \delta_D(y)^{2/(1+\nu)}.\nn
	\end{align}

(i) Since $\delta_D(x) \le \delta_D(y)$, we have
	\begin{align*}
		&|x_d-x_d\wedge y_d| = |y_d - x_d \vee y_d| = (x_d-y_d) \vee 0 \\
		&= (x_d-\delta_D(x) + \delta_D(y) - y_d + \delta_D(x)-\delta_D(y) ) \vee 0\\
		& \le |x_d- \delta_D(x)| + |\delta_D(y)-y_d|.
	\end{align*}
Hence,  by \eqref{e:example2-dist}, we get that
	\begin{align*}
		&\left| \delta_D(x) \wedge \delta_D(y) -x_d \wedge y_d\right| \vee 	
			\left| \delta_D(x) \vee \delta_D(y) -x_d \vee y_d\right|\nn\\
		&\le (|\delta_D(x) - x_d | + |x_d - x_d \wedge y_d|) \vee (|\delta_D(y)-y_d| + |y_d - x_d \vee y_d|) \\
		&\le 3(|\delta_D(x)-x_d| \vee |\delta_D(y)-y_d|)\\
		&\le 3^{-(1-\nu)/(1+\nu)}R^{-(1-\nu)/(1+\nu)}\delta_D(y)^{2/(1+\nu)}.
	\end{align*}

(ii) Since   $y_d \asymp \delta_D(y)$ by Lemma \ref{l:C11}(iii),  using (i),  we obtain
	\begin{align*}
		&\left| \frac{\delta_D(x)\wedge\delta_D(y)}{(x_d \vee y_d) \wedge |x-y|} 
			-  \frac{x_d \wedge y_d}{(x_d \vee y_d) \wedge |x-y|} \right|   \nn\\
		&\le \frac{c|\delta_D(x)\wedge\delta_D(y)- x_d \wedge y_d|}{\delta_D(y) \wedge |x-y|} 
			\le  \frac{c \delta_D(y)}{\delta_D(y)\wedge |x-y|}\bigg(\frac{\delta_D(y)}{R}\bigg)^{(1-\nu)/(1+\nu)}
	\end{align*}
and
	\begin{align*}
		&\left|\frac{\delta_D(x)\wedge\delta_D(y)}{(\delta_D(x)\vee\delta_D(y))\wedge |x-y|}
			- \frac{\delta_D(x)\wedge\delta_D(y)}{(x_d \vee y_d) \wedge |x-y|} \right|\nn\\
		&= \delta_D(x) \left|\frac{(x_d \vee y_d) \wedge |x-y|- (\delta_D(x)\vee\delta_D(y))\wedge |x-y| }
			{((\delta_D(x)\vee\delta_D(y)) \wedge |x-y|)((x_d \vee y_d) \wedge |x-y|) } \right| \nn\\
		& \le  c\delta_D(x) \left|\frac{(x_d \vee y_d) - (\delta_D(x)\vee\delta_D(y)) }{(\delta_D(y) \wedge |x-y|)^2} \right| \\ 
		&\le  \frac{c\delta_D(x)\delta_D(y)}{ (\delta_D(y) \wedge |x-y|)^2}\bigg(\frac{\delta_D(y)}{R}\bigg)^{(1-\nu)/(1+\nu)} \\
		&\le  c\bigg(\frac{\delta_D(y)}{ \delta_D(y) \wedge |x-y|}\bigg)^2\bigg(\frac{\delta_D(y)}{R}\bigg)^{(1-\nu)/(1+\nu)}.
	\end{align*}
Combining the  two displays above, we arrive at the result. 
\end{proof}

{\it
In this subsection,  we assume that 
 $\Phi_1, \Phi_2, \ell$ are  differentiable  and that
	\begin{align}\label{e:Phi-diff}
		\sup_{r>0} \bigg( \frac{|\Phi_1'(r)|}{r^{-1}\Phi_1(r)} +  \frac{|\Phi_2'(r)|}{r^{-1}\Phi_2(r)} 
			+  \frac{|\ell'(r)|}{r^{-1}\ell(r)} \bigg)<\infty.
	\end{align}
	}
	See Remark \ref{l:regularversion} below.

Let  $\alpha \in (0,2)$ and	$a:D \times D \to (0,\infty)$ be a Borel function satisfying the following properties: 

\smallskip

\noindent  \hypertarget{A1}{{\bf (A1)}} There exists  $C_{12}>1$ such that
	\begin{align*}
		C_{12}^{-1}\le a(x,y)=a(y,x) \le C_{12} \quad \text{for all} \;\, x,y \in D.
	\end{align*}

\noindent \hypertarget{A2}{{\bf (A2)}} If $\alpha \ge 1$, then there exist constants  $\theta_0'>\alpha-1$ and $C_{13}>0$ such that
	\begin{align}\label{e:B-expression-B3}
		|a(x,x) - a(x,y)| \le C_{13} \bigg(\frac{|x-y|}{\delta_D(x) \wedge \delta_D(y)  \wedge  R}\bigg)^{\theta_0'}  
			\quad \text{for all} \;\, x,y \in D.
	\end{align}

\noindent \hypertarget{A3}{{\bf (A3)}} There exist constants $\nu \in (0,1]$, $\theta_1',\theta_2',C_{14}>0$, a non-negative Borel function $f_0$ on $\bH_{-1}$   such that for any $Q \in \partial D$ and  $x,y\in E^Q_\nu(R)$  with $x=(\wt x,x_d)$ in CS$_Q$, 
	\begin{align*}
		&\big|a(x,y)- a(x,x)f_0((y-x)/x_d)\big| + \big|a(x,y)- a(y,y)f_0((y-x)/x_d)\big|\nn\\
		& \le  C_{14} \bigg(\frac{\delta_D(x)\vee\delta_D(y)\vee|x-y|}{\delta_D(x) \wedge \delta_D(y) \wedge |x-y|} \bigg)^{\theta_1'}
				\big( \delta_D(x)\vee\delta_D(y)\vee|x-y|\big)^{\theta_2'}.
			\end{align*}

\begin{remark}\label{r:Holder-continuous} 
Assume that $\theta_0'>(\alpha-1)_+$ and that $a \in C^{\theta_0'}(\overline D \times \overline D)$ 
is symmetric and bounded above and below by positive constants.  Then  $a$ satisfies  \hyperlink{A1}{{\bf (A1)}}, \hyperlink{A2}{{\bf (A2)}} and \hyperlink{A3}{{\bf (A3)}} with $f_0\equiv 1$.
\end{remark} 

We define $\sB^a:D \times D \to (0,\infty)$  by
	\begin{align}\label{e:B-expression}
		\sB^a(x,y)&=a(x,y) \Phi_1\bigg(\frac{\delta_D(x) \wedge \delta_D(y)}{|x-y|}\bigg)
			\Phi_2\bigg(\frac{\delta_D(x) \vee \delta_D(y)}{|x-y|}\bigg)\\
		&\qquad  \times \ell\bigg(\frac{\delta_D(x)\wedge \delta_D(y)}{(\delta_D(x)\vee\delta_D(y))\wedge |x-y|}\bigg).\nn
	\end{align}
For $y=x$, we interpret the above as $\sB^a(x,x)=a(x,x)$. 

\begin{prop}\label{p:example-coefficient-map}
Suppose that $a$ satisfies \hyperlink{A1}{{\bf (A1)}}, \hyperlink{A2}{{\bf (A2)}} and \hyperlink{A3}{{\bf (A3)}}. Then  the function 
$\sB^a$ defined by \eqref{e:B-expression} satisfies \hyperlink{B1}{{\bf (B1)}}, \hyperlink{B3}{{\bf (B3)}}, \hyperlink{B4-c}{{\bf (B4-c)}} and \hyperlink{B5-I}{{\bf (B5-I)}}.
\end{prop}
\begin{proof} \hyperlink{B1}{{\bf (B1)}} and  \hyperlink{B4-c}{{\bf (B4-c)}} are immediate by  \hyperlink{A1}{{\bf (A1)}}.  For \hyperlink{B3}{{\bf (B3)}}, we assume  $\alpha \ge 1$.  Then by  \hyperlink{A2}{{\bf (A2)}},  for all $x,y \in D$ with 
$|x-y|<\delta_D(x) \wedge \delta_D(y)$, 
	\begin{align*}
		|\sB^a(x,x)-\sB^a(x,y)| = |a(x,x)-a(x,y)| 
			\le   C_{13}\bigg(\frac{|x-y|}{\delta_D(x) \wedge \delta_D(y)  \wedge R}\bigg)^{\theta_0'}.
	\end{align*}
Hence, by Remark \ref{r:A4} (since \hyperlink{B2-a}{{\bf (B2-a)}} and \hyperlink{B4-a}{{\bf (B4-a)}} follows from 
\hyperlink{B4-c}{{\bf (B4-c)}}),  \hyperlink{B3}{{\bf (B3)}} holds.  By Lemma \ref{l:A5} below,  \hyperlink{B5-I}{{\bf (B5-I)}} also holds.   
\end{proof}

Define a kernel $K_0:\bH \times \bH \to [0,\infty)$ by
	\begin{align*}
		K_0(x,y)&=f_0\bigg(\frac{y-x}{x_d}\bigg) \Phi_1\bigg(\frac{x_d \wedge y_d}{|x-y|}\bigg)\,
			\Phi_2\bigg(\frac{x_d \vee y_d}{|x-y|}\bigg)\, \ell\bigg(\frac{x_d\wedge y_d}{(x_d\vee y_d)\wedge |x-y|}\bigg).
	\end{align*}
Observe that $K_0$ satisfies \eqref{e:ass-1-cond}. Hence, by Lemma \ref{l:ass-1}(i),  we have
	\begin{align}\label{e:B-expression-F}
		K_0(x,y)=F_0((y-x)/x_d) \quad \text{for all} \;\, x,y \in \bH,
	\end{align}
where 
	$$
		F_0(z):=K_0(\e_d,\e_d+z).
	$$ 

\begin{lemma}\label{l:A5}
Let $\nu \in (0,1]$ and $\nu_0 \in (0,\nu] \cap (0,1)$.	There exists $C>0$ such that	for any $Q \in \partial D$ and   
$x,y \in E^Q_{\nu_0}(R)$ with $x=(\wt x, x_d)$ in CS$_Q$,
	\begin{align*}
		&\left|\sB^a(x,y)- \sB^a(x,x)F_{0}((y-x)/x_d) \right| + \left|\sB^a(x,y)-  \sB^a(y,y)F_{0}((y-x)/x_d) \right|  \\
		&\le C\bigg(\frac{\delta_D(x)\vee \delta_D(y)\vee|x-y|}{\delta_D(x) \wedge \delta_D(y) \wedge  |x-y|} \bigg)^{\theta_1' \vee 2}
			\bigg( \frac{\delta_D(x) \vee \delta_D(y)\vee|x-y|}{ R}\bigg)^{\theta_2' \wedge ((1-\nu_0)/(1+\nu_0))},
	\end{align*}
where $F_{0}:\bH_{-1} \to [0,\infty)$ is defined as above,  and  $\theta_1',\theta_2'>0$ are the constants in  
\hyperlink{A3}{{\bf (A3)}}. Therefore, $\sB^a$ satisfies \hyperlink{B5-I}{{\bf (B5-I)}}.
\end{lemma}
\begin{proof} Let $Q \in \partial D$. By  \eqref{e:E-nu-inclusion}, we have $ E^Q_{\nu_0}(R) \subset  E^Q_{\nu}(R)$.	In this proof, we use the coordinate system CS$_Q$ and denote $E^Q_{\nu_0}(R)$ by $E_{\nu_0}$. Set $\lambda:=(1-\nu_0)/(1+\nu_0)$.
	
Let $x,y \in E_{\nu_0}$. We assume $\delta_D(x) \le \delta_D(y)$ without loss of  generality. Set
	\begin{align*}
		r_1:=\frac{\delta_D(x)}{|x-y|}, \quad r_2:=\frac{\delta_D(y)}{|x-y|}, \quad  r_3:=\frac{\delta_D(x)}{\delta_D(y) \wedge |x-y|}
	\end{align*}
and
	\begin{align*}
		r_1':=\frac{x_d \wedge y_d}{|x-y|}, \;\quad r_2':=\frac{x_d \vee y_d}{|x-y|}, \;
			\quad r_3':=\frac{x_d \wedge y_d}{(x_d \vee y_d) \wedge |x-y|}.
	\end{align*}
By using  \eqref{e:B-expression} and \eqref{e:B-expression-F},  we have
	\begin{align}\label{e:example-coefficient-1}
		&\sB^a(x,y)-  \sB^a(x,x) F_{0}((y-x)/x_d)=\sB^a(x,y)-  a(x,x) K_0(x,y)\\
		& =a(x,y)\Phi_1(r_1)\Phi_2(r_2)\ell(r_3)-  a(x,x) f_0((y-x)/x_d)\Phi_1(r_1')\Phi_2(r_2')\ell(r_3')\nn
	\end{align}
and
	\begin{align}\label{e:example-coefficient-2}
		&\sB^a(x,y)-  \sB^a(y,y) F_{0}((y-x)/x_d)=\sB^a(x,y)-  a(x,x) K_0(x,y)\\
		& =a(x,y)\Phi_1(r_1)\Phi_2(r_2)\ell(r_3)-  a(y,y) f_0((y-x)/x_d)\Phi_1(r_1')\Phi_2(r_2')\ell(r_3').\nn
	\end{align}

Since  $x_d \asymp \delta_D(x)$ and $y_d \asymp \delta_D(y)$ by Lemma \ref{l:C11}(iii), by using the scaling properties of $\Phi_1,\Phi_2,\ell$, we get
	\begin{align}\label{e:general-functions}
		\frac{r_1'}{r_1}\asymp \frac{r_2'}{r_2} \asymp \frac{r_3'}{r_3} \asymp 
			\frac{\Phi_1(r_1')}{\Phi_1(r_1)} \asymp \frac{\Phi_2(r_2')}{\Phi_2(r_2)}\asymp \frac{\ell(r_3')}{\ell(r_3)} \asymp 1. 
	\end{align}
By \eqref{e:B4-c-bounded} and \eqref{e:general-functions},  there exists $c_1>0$ independent of $Q,x$ and $y$ such that
	\begin{align}\label{e:general-bounded}
		M:=	\max\left\{ \Phi_1(a_1)\Phi_2(a_2)\ell(a_3) : a_i \in \{r_i,r_i'\}, \,1\le i \le 3\right\}	 \le c_1.
	\end{align}
Moreover, using the mean value theorem, \eqref{e:Phi-diff} and \eqref{e:general-functions},  we get
	\begin{align*}
		&\frac{\left|	\Phi_1(r_1)-\Phi_1(r_1') \right|}{\Phi_1(r_1)} 
			\le  \frac{|r_1-r_1'|}{\Phi_1(r_1)}\bigg(\sup_{r_1 \wedge r_1' \le u \le r_1 \vee r_1'} |\Phi_1'(u)|   \bigg) \nn\\
		&\le \frac{c|r_1-r_1'|}{\Phi_1(r_1)}\bigg(\sup_{r_1 \wedge r_1' 
			\le u \le r_1 \vee r_1'} \frac{\Phi_1(u)}{u}\bigg)  \le  \frac{c|r_1-r_1'|}{r_1}.
	\end{align*}
In the same way, we also get $|\Phi_2(r_2)-\Phi_2(r_2')|/\Phi_2(r_2) \le c |r_2-r_2'|/r_2$ and 
$|\ell(r_3)-\ell(r_3')|/\ell(r_3) \le c |r_3-r_3'|/r_3$. By Lemma \ref{l:general-coefficient}, it follows that
	\begin{align}\label{e:general-Phi-1}
		\frac{\left|	\Phi_1(r_1)-\Phi_1(r_1') \right|}{\Phi_1(r_1)}
		&\le   \frac{cr_2}{r_1}  \bigg(\frac{\delta_D(y)}{R}\bigg)^{\lambda}
			= \frac{c \delta_D(y)}{\delta_D(x)}  \bigg(\frac{\delta_D(y)}{R}\bigg)^{\lambda},
	\end{align}
	\begin{align}\label{e:general-Phi-2}
		&\frac{\left|	\Phi_2(r_2)-\Phi_2(r_2') \right|}{\Phi_2(r_2)}	
			\le  \frac{c|r_2-r_2'|}{r_2}\le  c\bigg(\frac{\delta_D(y)}{R}\bigg)^{\lambda} 
	\end{align}
and
	\begin{align}\label{e:general-ell}
		&\frac{\left| \ell(r_3)-\ell(r_3')\right|}{\ell(r_3)}\le  \frac{c}{r_3} 
			\bigg(\frac{ \delta_D(y)}{\delta_D(y) \wedge |x-y|}\bigg)^2 \bigg(\frac{\delta_D(y)}{R}\bigg)^{\lambda} \\
		&=  \frac{c\delta_D(y)^2}{\delta_D(x) (\delta_D(y) \wedge |x-y|)}  \bigg(\frac{\delta_D(y)}{R}\bigg)^{\lambda} 
			\le  c\bigg(\frac{\delta_D(y)}{\delta_D(x) \wedge |x-y|}\bigg)^2  \bigg(\frac{\delta_D(y)}{R}\bigg)^{\lambda}.\nn
	\end{align}
By using  \eqref{e:example-coefficient-1},   the triangle inequality, \eqref{e:general-Phi-1}-\eqref{e:general-ell} 
and  \hyperlink{A3}{{\bf (A3)}},  we obtain
	\begin{align*}
		&	\left|\sB^a(x,y)-   \sB^a(x,x) F_{0}((y-x)/x_d)\right| \\
		&\le a(x,y)\Phi_2(r_2)\ell(r_3)\big|\Phi_1(r_1)- \Phi_1(r_1')\big| \\
		&\quad  + a(x,y) \Phi_1(r_1')\ell(r_3) \big| \Phi_2(r_2)- \Phi_2(r_2')\big| \\
		& \quad  + a(x,y) \Phi_1(r_1')\Phi_2(r_2')\big|\ell(r_3)-\ell(r_3')\big|\\
		&\quad  +\Phi_1(r_1')\Phi_2(r_2')\ell(r_3') \big|a(x,y)- a(x,x)f_0((y-x)/x_d)\big|  \\
		&\le  cMa(x,y) \bigg(\frac{c \delta_D(y)}{\delta_D(x)}  \bigg(\frac{\delta_D(y)}{R}\bigg)^{\lambda}  
			 + \bigg(\frac{\delta_D(y)}{R}\bigg)^{\lambda} +\bigg(\frac{\delta_D(y)}{\delta_D(x) \wedge |x-y|}\bigg)^2 
			  	\bigg(\frac{\delta_D(y)}{R}\bigg)^{\lambda}\bigg)\\
		&\quad +  C_{14}M \bigg(\frac{\delta_D(y)\vee|x-y|}{\delta_D(x) \wedge |x-y|} \bigg)^{\theta_1'}
			\bigg( \frac{\delta_D(y)\vee|x-y|}{ R}\bigg)^{\theta_2'}.
	\end{align*}
Thus, by  \hyperlink{A1}{{\bf (A1)}} and \eqref{e:general-bounded}, we arrive at
	\begin{align*}
		&	\left|\sB^a(x,y)-   \sB^a(x,x) F_{0}((y-x)/x_d)\right| \\
		&\le  c \bigg(\frac{\delta_D(y)\vee|x-y|}{\delta_D(x) \wedge |x-y|} \bigg)^{ \theta_1' \vee 2}
			\bigg( \frac{\delta_D(y)\vee|x-y|}{ R}\bigg)^{\theta_2' \wedge \lambda}.
	\end{align*}

Similarly, we obtain the desired bound for $\left|\sB^a(x,y)-   \sB^a(y,y) F_{0}((y-x)/x_d)\right|$.
The proof is complete. 
\end{proof}

\begin{remark}\label{l:regularversion}
For a Borel function  $\Phi:(0,1) \to (0,\infty)$ such that
	\begin{align*}		
		a^{-1}(r/s)^{-k} \le	\Phi(r)/ \Phi(s) \le  a(r/s)^{k} \quad \text{for all} \;\, 0<s \le r <1,	
	\end{align*}	
with some constants $a>1$ and $k>0$, we define $[\Phi](r)=r^{-k-1}\int_0^r s^k\Phi(s)ds$. 
Then one can easily show that 
	\begin{align*}
		a^{-1} \le	\frac{\Phi(r)}{[\Phi](r)} \le (2k+1)a \quad \text{ and } \quad 
		\frac{|[\Phi]'(r)|}{r^{-1}[\Phi](r)} \le (2k+1)a \quad \text{ for all} \;\; 0<r < 1.
	\end{align*}

 Let $\wt \Phi_1:(0,\infty) \to (0,1)$ be an increasing differentiable function such that $\wt \Phi_1(r) = [\Phi_1](r)$ for $r \in (0,1/2)$ and $\wt \Phi_1(r)=1$ for $r \ge 1$. 
Define $\wt \Phi_2$ and $\wt \ell$ analogously.  By considering  
$\wt \Phi_1, \wt \Phi_2, \wt \ell$ instead of $\Phi_1,\Phi_2,\ell$ respectively, we see that the differentiability assumption and
\eqref{e:Phi-diff} are not  big restrictions.
\end{remark}

\medskip	
It follows from  Lemma \ref{l:C11}(ii) that 
for all $y\in D$ with $\delta_D(y)< R$, there is unique $Q_y \in \partial D$ such that $\delta_D(y)=|y-Q_y|$. For $y\in D$ with $\delta_D(y)<R$, let $\overline y$ be the reflection of $y$ with respect to $\partial D$, that is, 
$\overline y= 2Q_y -y$.  
	
\begin{example}\label{ex:Holder-continuous-2}
Let $\theta \in ((\alpha-1)_+,1)$,  and	   $h:D \times D \to [0,\infty)$ and $\Theta:[0,\infty) \to [0,\infty)$ be $\theta$-H\"older continuous functions. That is, 
		\begin{align}\label{e:h-holder}
			|h(x,y)-h(x',y')| \le C(|x-x'|+|y-y'|)^{\theta} \quad \text{for all} \;\; x,y,x',y' \in D
		\end{align}
and
		\begin{align}\label{e:Theta-holder}
			|\Theta(r)-\Theta(s)| \le C|r-s|^{\theta} \quad \text{for all} \;\; r,s\ge 0
		\end{align}
for some $C>0$.		Suppose that $a: D \times D \to (0,\infty)$ is a Borel function satisfying the following properties: 
There exists $C>0$ such that for all $x,y \in D$,	
		\begin{equation}\label{e:example-general-coefficient-1}
			C^{-1}\le 	a(x,y)=a(y,x) \le C,
		\end{equation}
		\begin{equation}\label{e:example-general-coefficient-2}
				\left|  a(x,x)-a(x,y) \right| \le C|x-y|^{\theta} \quad \mbox{if}\;\, \delta_D(x) \wedge \delta_D(y) >R/2,
		\end{equation}
and
		\begin{equation}
			\bigg| 	a(x,y) - h(x,y) \Theta\bigg(  \frac{|x-y|}{|x-\overline y|}\bigg) \bigg| 
				\le C|x-y|^{\theta} \quad\mbox{if}\;\, \delta_D(x) \vee \delta_D(y) <R.\label{e:example-general-coefficient-3}
		\end{equation}
Then   $a$ satisfies  \hyperlink{A1}{{\bf (A1)}},  \hyperlink{A2}{{\bf (A2)}} and  \hyperlink{A3}{{\bf (A3)}}.
		
Indeed, \hyperlink{A1}{{\bf (A1)}} immediately follows from \eqref{e:example-general-coefficient-1}. For  \hyperlink{A2}{{\bf (A2)}}, we assume that  $\alpha \ge 1$ for the moment. Since  \hyperlink{A1}{{\bf (A1)}} holds, it suffices to show that 
\eqref{e:B-expression-B3} holds for $x,y \in D$ with $|x-y|<\delta_D(x)/4$
(cf.~Remark \ref{r:A4}). Let $x,y  \in D$ with $|x-y|<\delta_D(x)/4$. Suppose that $\delta_D(x)\ge 2R/3$. 
Then $\delta_D(y) \ge \delta_D(x) - |x-y| >3\delta_D(x)/4 \ge R/2$. Thus,  by \eqref{e:example-general-coefficient-2}, we get
		\begin{align*}
			\left|  a(x,x)-a(x,y) \right| \le c|x-y|^{\theta} \le  c R^{\theta} 
				\bigg(\frac{|x-y|}{\delta_D(x) \wedge \delta_D(y) \wedge R}\bigg)^{\theta}.
		\end{align*} 
Now assume that  $\delta_D(x)<2R/3$. Since $|x-y|<\delta_D(x)/4$, we have 
$\delta_D(y) \in [(3/4)\delta_D(x), (5/4)\delta_D(x)] \subset (0, R)$ and 
		\begin{align}\label{e:example-reflection} 
			\frac{|x-y|}{|x-\overline y|}\le 	\frac{|x-y|}{\delta_{D^c}(\overline y)} 
				= \frac{|x-y|}{\delta_{D}(y)} <  \frac{\delta_D(x)}{4\delta_{D}(y)}\le 1/3.
		\end{align}
Note that $a(x,x)=h(x,x) \Theta(0)$ by \eqref{e:example-general-coefficient-3} and hence $\Theta(0)>0$ by 
\eqref{e:example-general-coefficient-1}. Thus, using \eqref{e:h-holder}, \eqref{e:Theta-holder},  
\eqref{e:example-general-coefficient-1}, \eqref{e:example-general-coefficient-3} and \eqref{e:example-reflection}, we obtain 
		\begin{align}\label{e:coefficient-B3-2}
			|a(x,x)- a(x,y)| &\le  \frac{a(x,x)}{\Theta(0)}	
				\left| 	\Theta(0)-\Theta\bigg(  \frac{|x-y|}{|x-\overline y|}\bigg) \right| 
				+  \Theta\bigg(  \frac{|x-y|}{|x-\overline y|}\bigg)|h(x,x)-h(x,y)|  \\
				&\quad  +	\left| 	a(x,y) -h(x,y)\Theta\bigg(  \frac{|x-y|}{|x-\overline y|}\bigg) \right| \nn\\
				&\le \frac{c}{\Theta(0)} \bigg(  \frac{|x-y|}{|x-\overline y|}\bigg)^{\theta} 
					+  c\big( \sup_{s \in [0,1/3]} \Theta(s)\big)|x-y|^{\theta} + c|x-y|^{\theta} \nn\\
				& \le  c\big( \Theta(0)^{-1} + R^{\theta} \sup_{s \in [0,1/3]} \Theta(s) 
					+ R^{\theta}\big) \bigg(\frac{|x-y|}{ \delta_D(y) \wedge R}\bigg)^{\theta} \nn\\
				&\le  c \bigg(\frac{|x-y|}{\delta_D(x) \wedge \delta_D(y) \wedge R}\bigg)^{\theta}.\nn
		\end{align}
Thus, \hyperlink{A2}{{\bf (A2)}} holds.
			
Now we show that  \hyperlink{A3}{{\bf (A3)}} holds. Define 
	$$
		f_1(z)=\Theta(|z|/|(\wt z, - z_d-2)|) \quad \text{and} \quad f_0(z)=f_1(z)/\Theta(0), \quad \;\;z \in \bH_{-1}.
	$$
Since $|z|/|(\wt z, - z_d-2)| \le 1$ for all $z \in \bH_{-1}$, we have
		\begin{align}\label{e:example-reflection-2}
			\sup_{z \in \bH_{-1}}|f_1(z)|\le \sup_{s\in [0,1]}\Theta(s)=:c_1<\infty
		\end{align}
Moreover, since the map $(x,y) \mapsto \Theta (|x-y|/|x-(\wt y, -y_d)|)$ satisfies \eqref{e:ass-1-cond}, 
		\begin{align}\label{e:f1-Theta}
			f_1((y-x)/x_d) = \Theta(|x-y|/|x-(\wt y, -y_d)|) \quad \text{for all} \;\, x,y \in \bH
		\end{align} 		
by Lemma \ref{l:ass-1}(i). Fix $Q \in \partial D$ and let $x=(\wt x,x_d),\, y=(\wt y,y_d) \in E^Q_{1/2}(R)$ in CS$_Q$.  By  \eqref{e:example-general-coefficient-3}, we have  $a(x,x)=h(x,x)\Theta(0)$ and $a(y,y)=h(y,y)\Theta(0)$. Hence, using 
\eqref{e:h-holder} and  \eqref{e:example-reflection-2}, we obtain
		\begin{align}\label{e:B-expression-coefficient-check-1}
			\begin{split}
				&\big|a(x,y)- a(x,x)f_0((y-x)/x_d)\big| + \big|a(x,y)- a(y,y)f_0((y-x)/x_d)\big|\\
				& = \big|a(x,y)- h(x,x)f_1((y-x)/x_d)\big| + \big|a(x,y)- h(y,y)f_1((y-x)/x_d)\big|\\
				&\le 2\big|a(x,y)- h(x,y)f_1((y-x)/x_d)\big|  \\
				&\quad + f_1((y-x)/x_d) \big(  | h(x,x)-h(x,y)| + | h(y,y)-h(x,y)|\big) \\
				& \le  2\big|a(x,y)- h(x,y)f_1((y-x)/x_d)\big|  +c |x-y|^{\theta}.
			\end{split}
		\end{align}
Let $Q_y=(\wt w, w_d) \in \partial D$ in CS$_Q$ be such that $\delta_D(y)=|y-Q_y|$. Then 
$\wt w - \wt y = (y_d-w_d) \nabla \Psi(\wt w)$, where $\Psi=\Psi_Q$ is the function in \eqref{e:local-coordinate}. Thus,  we have
		\begin{align*}
			&|\wt w - \wt y| \le |y_d-w_d||\nabla \Psi(\wt w)-\nabla \Psi(\wt 0)| \
				le \Lambda |y_d-w_d||\wt w| \le \Lambda |y_d-w_d|  (|\wt y| + |\wt w - \wt y|).
		\end{align*}
Since $\Lambda|y_d -w_d| \le (2R)^{-1}|y_d-w_d| \le (2R)^{-1}\delta_D(y)< 1/2$ and $y \in E_{1/2}^Q(R)$, it follows that
		\begin{align}\label{e:B-expression-coefficient-check-2}
			|\wt w - \wt y| \le 2\Lambda |y_d-w_d|  |\wt y|  \le R^{-1} \delta_D(y)   (R y_d^2)^{1/3}
				\le c R^{-2/3} \delta_D(y)^{5/3}.
		\end{align}
We used  Lemma \ref{l:C11}(iii) in the last inequality above. Further, by using   \eqref{e:Psi-bound}, 
\eqref{e:B-expression-coefficient-check-2} and Lemma \ref{l:C11}(iii), we get that
		\begin{align}\label{e:B-expression-coefficient-check-3}
			|w_d| &= |\Psi(\wt w)| \le (4R)^{-1}  (|\wt y|^2 + |\wt w - \wt y|^2) \\
			&\le (4R)^{-1}(Ry_d^2)^{2/3}  + (4R)^{-1} R^{-2/3} \delta_D(y)^{5/3} \nn\\
			&\le  cR^{-1/3}\delta_D(y)^{4/3}.\nn
		\end{align}
Combining  \eqref{e:B-expression-coefficient-check-2} with \eqref{e:B-expression-coefficient-check-3}, since 
$\overline y = (2\wt w- \wt y, 2w_d-y_d)$ and $\delta_D(y)<R$, we deduce that
		\begin{align}\label{e:distance-reflection-0}
			|\overline y- (\wt y, -y_d)|  = 2| ( \wt w - \wt y, w_d)| \le cR^{-1/3}\delta_D(y)^{4/3}.
		\end{align}
By using \eqref{e:Theta-holder} in the first inequality, the facts that $\overline y \in D^c$ and 
$(\wt y, -y_d) \in \wt E_{1/2}^Q(R) \subset D^c$ (see Lemma \ref{l:C11}(ii)) in the third, and \eqref{e:distance-reflection-0} in the last inequality below, it follows that 
		\begin{align}\label{e:distance-reflection}
			\left|   \Theta\bigg(  \frac{|x-y|}{|x-\overline y|}\bigg) - \Theta\bigg(  \frac{|x-y|}{|x-(\wt y,-y_d)|}\bigg)\right| 
			&\le c \bigg|  \frac{|x-y|}{|x-\overline y|} -   \frac{|x-y|}{|x-(\wt y,-y_d)|} \bigg|^{\theta}\\
			&\le c|x-y|^{\theta}\bigg( \frac{||x-(\wt y,-y_d)| 
				- |x-\overline y| | }{|x-\overline y|\,|x-(\wt y,-y_d)|} \bigg)^{\theta}  \nn\\
			&\le c \frac{|x-y|^{\theta} |\overline y- (\wt y, -y_d)| ^{\theta}}{\delta_D(x)^{2\theta}}  \nn\\
			&\le c \bigg( \frac{|x-y|}{\delta_D(x)}\bigg)^{\theta}   \bigg( \frac{\delta_D(y)}{\delta_D(x)}\bigg)^{\theta}  
				\bigg( \frac{\delta_D(y)}{R}\bigg)^{\theta/3}.\nn
		\end{align}
Since  $a(x,x)=h(x,x) \Theta(0)$,  using  \eqref{e:example-general-coefficient-1} and  \eqref{e:h-holder}, we see that
		\begin{equation}\label{e:h-bound}
			h(x,y)\le h(x,x)+ |h(x,x)-h(x,y)|\le \Theta(0)^{-1}a(x,x) + c|x-y|^{\theta} \le  c.
		\end{equation}
		
Now, using  \eqref{e:f1-Theta} and the triangle inequality in the first inequality below, \eqref{e:example-general-coefficient-3}, \eqref{e:h-bound} and  \eqref{e:distance-reflection} in the second, $|x-y|<R<1$ in the last, we obtain
		\begin{align*}
			&\big|a(x,y)- h(x,y)f_1((y-x)/x_d)\big|\\
			&\le 	\left| 	a(x,y) -h(x,y)\Theta\bigg(  \frac{|x-y|}{|x-\overline y|}\bigg) \right|
				 + h(x,y)\left|   \Theta\bigg(  \frac{|x-y|}{|x-\overline y|}\bigg) - \Theta\bigg(  \frac{|x-y|}{|x-(\wt y,-y_d)|}
				 \bigg)\right|\\
			&\le c|x-y|^{\theta} +c \bigg( \frac{|x-y|}{\delta_D(x)}\bigg)^{\theta}   
				\bigg( \frac{\delta_D(y)}{\delta_D(x)}\bigg)^{\theta}  
				\delta_D(y)^{\theta/3}   \\
			&\le c \bigg(\frac{\delta_D(x)\vee\delta_D(y)\vee|x-y|}{\delta_D(x) \wedge \delta_D(y) \wedge |x-y|} \bigg)^{2\theta}
				\big( \delta_D(x)\vee\delta_D(y)\vee|x-y|\big)^{\theta/3}. 
		\end{align*}
Combining this with \eqref{e:B-expression-coefficient-check-1}, we conclude that  \hyperlink{A3}{{\bf (A3)}} holds. 
\end{example}

\smallskip

\begin{example}\label{ex:censored}
Assume that $\alpha \in (1,2)$. Consider a non-local operator
		\begin{align*}
			L^\sB_\alpha f(x)=\text{p.v.}\int_D (f(y)-f(x)) \frac{\sB(x,y)}{|x-y|^{d+\alpha}}dy,
		\end{align*}
where $\sB$ is a  Borel function on $D \times D$ such that 
	\begin{align}\label{e:example-censored-1}
		C^{-1}\le 	\sB(x,y)=\sB(y,x) \le C \quad \text{for all} \;\, x,y \in D
	\end{align}			
for some $C\ge 1$. When $\sB(x,y) \equiv c$ is a constant,  the operator $L^\sB_\alpha$ is called the 
\textit{regional fractional Laplacian} in $D$ and the corresponding process $Y^0$ is called the \textit{censored $\alpha$-stable process} on $D$.

Let  $\theta \in (\alpha-1,1)$.   Suppose that there exist   $C>0$ and  $\theta$-H\"older continuous functions 
$h_1:D \times D \to [0,\infty)$, $h_2:D \times D \to [0,\infty)$ and  $\Theta: [0,\infty) \to [0,\infty)$   such that 
$\sup_{x \in D}h_2(x,x)<\infty$ and  for all $x,y \in D$,
	\begin{align}\label{e:censored-condition}
		\!	\begin{cases}
			\left| \sB(x,x)-\sB(x,y) \right| \le C|x-y|^{\theta} &\mbox{if }\delta_D(x) \wedge \delta_D(y) >R/2,\\[4pt]
			 \displaystyle 	\bigg| 	\sB(x,y) - h_1(x,y) - h_2(x,y) \Theta\bigg(  \frac{|x-y|}{|x-\overline y|}\bigg) \bigg| 
			 \le C|x-y|^{\theta}  &\mbox{if }\delta_D(x) \vee \delta_D(y) <R.
		\end{cases}
	\end{align}
We will prove that \hyperlink{B1}{{\bf (B1)}}, \hyperlink{B3}{{\bf (B3)}},  \hyperlink{B4-c}{{\bf (B4-c)}} and \hyperlink{B5-II}{{\bf (B5-II)}} hold  under \eqref{e:example-censored-1} and \eqref{e:censored-condition}.  
Assume these for the moment. Then we deduce from Theorem \ref{t:BHPnew} that for any subcritical killing  
potential $\kappa(x)$
 satisfying \hyperlink{K3}{{\bf (K3)}} with $C_9=0$  (including 
no killing, i.e.,  $\kappa(x) \equiv 0$), the operator $L^\sB_\alpha-\kappa$ satisfies  the  boundary Harnack principle \eqref{e:TAMSe1.8new} with $p=\alpha-1$.  

\smallskip

Now, we show that $\sB$ satisfies \hyperlink{B1}{{\bf (B1)}},  \hyperlink{B3}{{\bf (B3)}}, \hyperlink{B4-c}{{\bf (B4-c)}} and \hyperlink{B5-II}{{\bf (B5-II)}}.  \hyperlink{B1}{{\bf (B1)}} and \hyperlink{B4-c}{{\bf (B4-c)}} (with $\Phi_1=\Phi_2=\ell\equiv1$) clearly hold by  \eqref{e:example-censored-1}. 

\smallskip

\noindent \hyperlink{B3}{{\bf (B3)}}: Let $x,y \in D$ with $|x-y|<\delta_D(x)/4$. If $\delta_D(x)\ge 2R/3$, then 
$\delta_D(y)>3\delta_D(x)/4>R/2$ so that $	\left| \sB(x,x)-\sB(x,y) \right| \le c|x-y|^{\theta}$ by \eqref{e:censored-condition}. If $\delta_D(x)<2R/3$, then by following the arguments for \eqref{e:coefficient-B3-2}, we get from  \eqref{e:censored-condition} that
	\begin{align*}
		&|\sB(x,x)- \sB(x,y)| = |h_1(x,x)+h_2(x,x)\Theta(0)-\sB(x,y)|\\
		&\le |h_1(x,x)-h_1(x,y)| + h_2(x,x)	\left| 	\Theta(0)-\Theta\bigg(  \frac{|x-y|}{|x-\overline y|}\bigg) \right|   \nn\\
		&\quad +  \Theta\bigg(  \frac{|x-y|}{|x-\overline y|}\bigg)|h_2(x,x)-h_2(x,y)|  
			+	\left| 	\sB(x,y)-h_1(x,y) -h_2(x,y)\Theta\bigg(  \frac{|x-y|}{|x-\overline y|}\bigg) \right| \nn\\
		&\le c|x-y|^{\theta} + c \bigg(  \frac{|x-y|}{|x-\overline y|}\bigg)^{\theta} 
			+  c\big( \sup_{s \in [0,1/3]} \Theta(s)\big)|x-y|^{\theta} + c|x-y|^{\theta} \nn\\
		&\le  c \bigg(\frac{|x-y|}{\delta_D(x) \wedge \delta_D(y) \wedge R}\bigg)^{\theta}.
	\end{align*}
In the second inequality above, we used 	$\sup_{x \in D}h_2(x,x)<\infty$. By Remark \ref{r:A4}, we deduce that \hyperlink{B3}{{\bf (B3)}} holds.

\smallskip

\noindent \hyperlink{B5-II}{{\bf (B5-II)}}:  Define $\mu^1(x)=h_1(x,x)$ and $\mu^2(x)=h_2(x,x)$ for $x \in D$, and
	$$
		\F^1_0(z)=1 \quad \text{and} \quad \F^2_0(z)=\Theta(|z|/|(\wt z, - z_d-2)|) \quad \text{for}  \;\,z \in \bH_{-1}.
	$$
Fix $Q \in \partial D$ and let $x=(\wt x,x_d),\, y=(\wt y,y_d) \in E^Q_{1/2}(R)$ in CS$_Q$.  Using \eqref{e:f1-Theta} in the equality below and  \eqref{e:example-reflection-2} in the first inequality, we obtain
	\begin{align*}
		&	\bigg|\sB(x,y)- \sum_{i=1}^{2}\mu^i(x)\F_0^i((y-x)/x_d) \bigg| + \bigg|\sB(x,y)
			- \sum_{i=1}^{2}\mu^i(y)\F_0^i((y-x)/x_d) \bigg|\nn\\
		& = 	\bigg|\sB(x,y)- h_1(x,x) - h_2(x,x) \Theta \bigg( \frac{|x-y|}{|x-(\wt y, -y_d)|}\bigg)  \bigg| \nn\\
		&\qquad + \bigg|\sB(x,y)- h_1(y,y) - h_2(y,y) \Theta \bigg( \frac{|x-y|}{|x-(\wt y, -y_d)|}\bigg) \bigg|\nn\\
		&\le 	2\bigg|\sB(x,y)- h_1(x,y) - h_2(x,y) \Theta \bigg( \frac{|x-y|}{|x-(\wt y, -y_d)|}\bigg)  \bigg|  \nn\\
		&\qquad + |h_1(x,x)-h_1(x,y)| +  |h_1(y,y)-h_1(x,y)|\nn\\
		&\qquad + \sup_{s \in [0,1]}\Theta(s)\big(|h_2(x,x)-h_2(x,y)| +  |h_2(y,y)-h_2(x,y)|\big)\,  \nn\\
		&\le 2\bigg|\sB(x,y)- h_1(x,y) - h_2(x,y) \Theta \bigg( \frac{|x-y|}{|x-(\wt y, -y_d)|}\bigg)  \bigg| + c|x-y|^\theta.
	\end{align*}
Since $h_2(x,y) \le \sup_{v \in D} h_2(v,v) + cR^\theta \le c$ and $|x-y|<R<1$, using \eqref{e:censored-condition} and  \eqref{e:distance-reflection}, we get
	\begin{align*}
		&\bigg|\sB(x,y)- h_1(x,y) - h_2(x,y) \Theta \bigg( \frac{|x-y|}{|x-(\wt y, -y_d)|}\bigg)  \bigg|\\
		&\le \bigg|\sB(x,y)- h_1(x,y) - h_2(x,y) \Theta\bigg(  \frac{|x-y|}{|x-\overline y|}\bigg)  \bigg| \\
		&\qquad + h_2(x,y) \bigg|\Theta\bigg(  \frac{|x-y|}{|x-\overline y|}\bigg)
			-\Theta \bigg( \frac{|x-y|}{|x-(\wt y, -y_d)|}\bigg)\bigg|\\
		&\le c|x-y|^\theta  + c \bigg( \frac{|x-y|}{\delta_D(x)}\bigg)^{\theta}   \bigg( \frac{\delta_D(y)}{\delta_D(x)}\bigg)^{\theta}  		
			\delta_D(y)^{\theta/3}\\
		&\le c \bigg(\frac{\delta_D(x)\vee\delta_D(y)\vee|x-y|}{\delta_D(x) \wedge \delta_D(y) \wedge |x-y|} \bigg)^{2\theta}
			\big(\delta_D(x)\vee\delta_D(y)\vee|x-y|\big)^{\theta/3}. 
	\end{align*}
Putting the above two displays together, we conclude that \hyperlink{B5-II}{{\bf (B5-II)}} holds.  
\end{example}	

\textbf{Acknowledgements:} The major part of this work was done while 
Zoran Vondra\v{c}ek was Guest Professor of the College of Natural Sciences at Seoul National University within the Brain Pool Program of the National Research Foundation of Korea (NRF). The hospitality of Department of Mathematical Sciences and Research Institute of Mathematics, Seoul National University, is gratefully acknowledged.

The authors thank Ren\'e Schilling for the helpful discussion about the presentation.

\vspace{.1in}


\end{document}